\documentclass[10pt,reqno]{amsart}
\usepackage[utf8]{inputenc}
\usepackage{lmodern}
\usepackage[T1]{fontenc}
\usepackage{amsmath,amssymb}
\usepackage{hyperref}
\usepackage{mathrsfs}
\usepackage{amsfonts}
\usepackage{mathtools}
\usepackage{tikz}
\usepackage{lmodern}
\usepackage{bigints}
\usepackage{relsize}
\usepackage{bbm}
\usepackage{amsthm}
\usepackage{mathtools}
\mathtoolsset{showonlyrefs}
\usepackage{geometry}
\usepackage{scalerel,stackengine}
\stackMath
\newcommand\reallywidehat[1]{
\savestack{\tmpbox}{\stretchto{
  \scaleto{
    \scalerel*[\widthof{\ensuremath{#1}}]{\kern-.6pt\bigwedge\kern-.6pt}
    {\rule[-\textheight/2]{1ex}{\textheight}}WIDTH-LIMITED BIG WEDGE
  }{\textheight} 
}{0.5ex}}
\stackon[1pt]{#1}{\tmpbox}
}
\numberwithin{equation}{subsection}
\newcommand{\norm}[1]{\left\lVert#1\right\rVert}
\newtheorem{theorem}{Theorem}[section]
\newtheorem{corollary}[theorem]{Corollary}
\newtheorem{lemma}[theorem]{Lemma}
\newtheorem{proposition}[theorem]{Proposition}
\newtheorem{definition}[theorem]{Definition}
\newtheorem{notation}[theorem]{Notation}
\newtheorem{remark}[theorem]{Remark}
\theoremstyle{definition}
\DeclareMathOperator{\sppp}{Span}

\DeclareMathOperator{\I}{Im}

\DeclareMathOperator{\sech}{sech}

\DeclareMathOperator{\Ra}{Range}

\title[Multi-solitons for $1$d NLS]{Asymptotic Stability of Multi-Solitons for the One-Dimensional $L^2$-Supercritical Nonlinear Schr\"odinger Equation}
\author[G. Chen]{Gong Chen}
\author[A. Moutinho]{Abdon Moutinho}
\email{gc@math.gatech.edu}
\email{aneto8@gatech.edu}

\address{School of Mathematics, Georgia Institute of Technology, Atlanta, GA 30332, USA}
\thanks{GC  was partially supported by NSF grant DMS-2350301,  Simons foundation MP-TSM00002258 and the Stefan Bergman Fellowship}
\date{\today}

\begin{document}
\begin{abstract}
We consider the one-dimensional $L^2$-supercritical nonlinear Schr\"odinger equation
\[
i\partial_t \psi + \partial_x^2 \psi + |\psi|^{2k}\psi = 0, \qquad k>2.
\]
In this regime solitary waves are spectrally unstable and dispersion is weak. In the pioneering work of Krieger and Schlag~\cite{KriegerSchlag}, asymptotic stability of a single soliton was established on a codimension-one center-stable manifold. We prove asymptotic stability of well-separated multi-solitons on a finite-codimension center-stable manifold. Specifically, for $k>\frac{11}{4}$, perturbations lying on a codimension-$m$ Lipschitz manifold around a superposition of $m$ solitons with distinct velocities converge in $H^1(\mathbb{R})$ to a sum of modulated solitons plus dispersive radiation.

Due to the comparatively weak dispersion in one dimension, the analysis of multi-soliton dynamics is considerably more delicate. Existing full-line asymptotic stability results for multi-solitons in non-integrable dispersive equations  treat only two solitons and rely on strong relative velocity assumptions together with additional structural conditions on the nonlinearity. Our proof combines a modulation analysis, a refined linear theory for one-dimensional matrix charge transfer models developed in our earlier works~\cite{dispanalysis1,dispanalysis2} together with carefully designed norms that capture the interactions of multiple moving solitons. Our result applies to arbitrarily many solitons with the natural power-type nonlinearity in the $L^2$-supercritical regime under the sole requirements of distinct velocities and sufficient spatial separation.

\end{abstract}
\maketitle
\tableofcontents 

\section{Introduction}
\subsection{Background}
We consider the $L^{2}$-supercritical one-dimensional Schr\"odinger equation
\begin{equation}\label{NLS1d}
    i\partial_{t}\psi+\partial^{2}_{x}\psi+\vert \psi\vert^{2k}\psi=0 \text{, $k>2$}
\end{equation}
The study of solutions to the nonlinear Schr\"odinger equations \eqref{NLS1d} has important applications in the fields of optics and plasma physics, see \cite{optics1}, \cite{book0} and \cite{book2} for example. It is classical  that \eqref{NLS1d} is locally well-posed in $H^{1}_{x}(\mathbb{R}
).$  For each $\alpha>0$,  equation \eqref{NLS1d} has  standing  wave solutions given by, 
\begin{equation}\label{sta}\tag{Standing waves}
    e^{i\alpha^{2} t}\phi_{\alpha}(x),
\end{equation}
such that $\phi_{\alpha}\in H^{2}_{x}(\mathbb{R})$ is a real solution of the following elliptic ordinary differential equation.
\begin{equation}\label{elipt.11}
    {-}\phi^{\prime\prime}_{\alpha}+\alpha^{2}\phi_{\alpha}=\phi_{\alpha}^{2k+1}.
\end{equation}
In particular, for each $\alpha>0,$ the ground-state solution of \eqref{elipt.11} is unique and equal to  
\begin{equation}\label{gr}\tag{Ground states}
    \phi_{\alpha}(x)=\alpha^{\frac{1}{k}}(k+1)^{\frac{1}{2k}}\sech^{\frac{1}{k}}{(k \alpha x)}.
\end{equation}
Applying the Galilean symmetry, one obtains the family of traveling solitary wave solutions to  \eqref{NLS1d} which are given by 
\begin{equation}\label{Solwave}\tag{Solitary waves} 
    \psi(t,x)=e^{i(\frac{vx}{2}-\frac{v^{2}t}{4})+i\alpha^{2}t+i\gamma}\phi_{\alpha}(x-vt-y).
\end{equation}
Due to the nonlinear nature of equation \eqref{NLS1d}, the superposition of multiple soliton waves which is called a multi-soliton, $m\in\mathbb{N}$
\begin{equation}\label{eq:mulisoliton}
    \sum_{\ell=1}^me^{i(\frac{v_\ell x}{2}-\frac{v_\ell^{2}t}{4})+i\alpha_\ell^{2}t+i\gamma_\ell}\phi_{\alpha}(x-v_\ell t-y_\ell).
\end{equation}
Such a superposition is not an exact solution due to the nonlinearity, but approximate multi-soliton dynamics can still be constructed. More precisely, one can always construct  solutions $\psi$ to \eqref{NLS1d} such that as $t \rightarrow \infty$, 
\begin{equation}
    \psi \rightarrow   \sum_{\ell=1}^me^{i(\frac{v_\ell x}{2}-\frac{v_\ell^{2}t}{4})+i\alpha_\ell^{2}t+i\gamma_\ell}\phi_{\alpha}(x-v_\ell t-y_\ell),
\end{equation}we refer to C\^ote, Martel and Merle \cite{CoteMartelMerle} for the construction.


The asymptotic behavior of solutions to nonlinear dispersive equations is a central theme in modern PDE theory. A fundamental question — underpinning the Soliton Resolution Conjecture — is whether general solutions decompose, as $t\rightarrow\infty$, into a superposition of coherent structures (solitons) and a radiative part that disperses. Concerning references on the soliton resolution conjecture, see, for example,
\cite{solitonresol2}, \cite{solitonresol1}, 
\cite{solresolwavemap} and references therein. Understanding this decomposition has motivated the research of the dynamics of solitons and multi-solitons. 

In this paper, we are interested in the asymptotic stability of the multi-soliton \eqref{eq:mulisoliton}. The asymptotic stability  refers to the situation in which small perturbations not only remain small, but in fact, disperse.  Due to the $L^2$-supercritical nature of the equation, it is well known that solitary waves for \eqref{NLS1d} are unstable, and the same instability holds for multi-solitons.  Under the well-separation condition for  centers, we establish the existence of a center-stable manifold of finite codimension around multi-solitons for which asymptotic stability holds with perturbations on this center-stable manifold: if the initial perturbation is in this manifold, then the remainder of the solution will scatter as $t\rightarrow\infty$. For technical reasons, we restrict our attention to the case  $k>2+\frac{3}{4}=\frac{11}{4}.$  We do not claim that this condition on $k$ is optimal by our approach. 

The literature on asymptotic stability of a single soliton is extensive.  Without aiming to be exhaustive, we refer the reader to \cite{KriegerSchlag}, \cite{annalsunstable}, 
\cite{NLS3dcritical}, \cite{collotger}, \cite{GCnls} \cite{Busper1}, \cite{Mizumachi}, \cite{nls3soliton} for results on asymptotic stability of single solitons for Schr\"odinger equations. Furthermore, in  \cite{weakasymptotic1}, \cite{weakasymptotic2} and \cite{weakasymptotic3}, asymptotic stability of a single soliton on the subspace $L^{2}_{x}(K)$ was obtained for any compact set $K$ of $\mathbb{R}$ for a class of $1d$ Schr\"odinger models having nonlinearity of cubic order.  We also refer to the survey articles  by Cuccagna \cite{CuccagnasurveyI}, Cuccagna-Maeda \cite{CuccagnasurveyII}.

\par Concerning the asymptotic stability of multi-solitons for nonlinear dispersive equations in dimensions $d\geq 3,$ we refer to, for example, the asymptotic stability of multi-solitons for  nonlinear Schr\"odinger equations in \cite{nlsstates} by Schlag, Soffer and Rodnianski and in \cite{Perelman4} by Perelman. Moreover, in \cite{JacekGongKleingordon}, the first author and Jendrej proved the asymptotic stability of multi-solitons for nonlinear Klein–Gordon equations on a center-stable manifold. See also the recent article \cite{Zahkarovmulti} by Pilod and Valet on the asymptotic stability of multi-solitons for Zakharov–Kuznetsov equation.

\par Returning to the $1$d setting as in this paper, one important feature is that the dispersion is weak, so that the analysis of multi-soliton is  substantially more delicate than in higher dimensions.  Concerning the study of the asymptotic stability of multi-solitons solutions for one-dimensional nonlinear partial differential equations, again one can use  monotonicity arguments and virial identities, or one can employ dispersive pointwise decay. Without trying to be exhaustive again, we refer to, for example, the article \cite{gKdV} by Martel, Merle, and Tsai, and in \cite{gkdv2} by Martel and Merle where 
the asymptotic stability of multi-solitons for subcritical gKdV  equations was proved  using a monotonicity argument and virial identities. For the second view perspective, we refer to Mizumachi \cite{Mizumachikdvtwosoliton} for  two solitons of similar speeds in the setting of  stable gKdV equations. Finally, the asymptotic stability of two fast stable solitons for Schr\"odinger models was obtained in \cite{perelmanasym} by Perelman under certain assumptions on nonlinearities and spectrum. We emphasize that even in this stable setting, Perelman's approach \emph{can not} be directly generalized to the problem with more than two solitons, and it is highly non-trivial to build a satisfactory linear theory for more than two potentials, see \cite{dispanalysis1}.

   

In our earlier works \cite{dispanalysis1,dispanalysis2}, under the assumptions of distinct velocities and well-separated centers, we developed a general linear theory for Schr\"odinger equations with multiple potentials, accommodating the presence of threshold resonances and unstable eigenvalues under very general conditions.
In this paper, we will use linear dispersive estimates we developed in \cite{dispanalysis1,dispanalysis2} to study the multi-soliton in the $L^2$ supercritical setting. While our approach is inspired by \cite{KriegerSchlag}, the combination of weak 1D dispersion, orbital instability of the solitons, and the multi-potential interaction requires substantially new estimates at each step. Crucially,  we point out one important difficulty in the multi-soliton setting.  A crucial estimate in \cite{KriegerSchlag} is the local improved decay of the form
\begin{equation}
    \norm{\langle x\rangle^{-1}e^{i(t-s)(-\partial_x^2+V)}P_c F(s)}_{L^\infty_x}\lesssim \frac{1}{|t-s|^{\frac{3}{2}}}\norm{\langle x\rangle F(s)}_{L^1_x}
\end{equation}
with a generic potential $V$, where $P_c$ is the projection onto the continuous spectrum. This estimate  basically says that if there are no quantum particles with zero velocity, all quantum particles will leave a localized region quickly. The corresponding estimate for multiple potentials is much more involved since the generic condition only rules out zero-velocity particles, but when one considers the interaction among potentials with different potentials, clearly particles with the same velocities with potentials will always hit the localized region around the origin. Therefore, the local improved decay will not be as strong as the estimate above. Putting this in technical terms, in this setting, the inhomogeneous term $F$ is localized around the centers of potentials, which can cause a mismatch with the weight $\langle x \rangle$. This complicates our analysis and estimates. 

To conclude this general introduction, we highlight the main advances of the present work. Earlier approaches based on monotonicity or virial identities, such as Martel--Merle--Tsai~\cite{martel2006stability}, do not apply in the full-line setting considered here. Existing full-line asymptotic stability results, including~\cite{Mizumachikdvtwosoliton,perelmanasym}, treat \emph{only} two solitons and require strong restrictions on the relative velocities; moreover, the work~\cite{perelmanasym} imposes additional structural assumptions on the nonlinearity (see also Remark~\ref{rem:compare}).

A central novelty of the present paper is the treatment of \emph{arbitrarily} many unstable solitons in the natural one-dimensional $L^2$-supercritical regime \emph{without} imposing large relative velocity assumptions or additional structural conditions on the nonlinearity. The key new ingredient is a refined linear theory for one-dimensional matrix charge transfer models developed in our earlier works~\cite{dispanalysis1,dispanalysis2}. This framework provides robust dispersive and localized decay estimates in the presence of multiple moving potentials under the sole assumption of distinct velocities. Combined with a modulation analysis, it enables us to control the unstable directions and close the nonlinear argument to show the scattering of the dispersive tails. We should emphasize that in contrast to the single-soliton theory of Krieger--Schlag~\cite{KriegerSchlag}, the multi-soliton setting involves the interaction of several moving potentials, for which classical dispersive estimates centered at a single potential are no longer sufficient. To overcome this difficulty, we introduce carefully designed norms that exploit the structure of the problem when $k>\frac{11}{4}$.

\subsection{Main results}
We now introduce the main results of this paper. To state them precisely, we first set up the necessary notation.
\subsubsection{Some notations}
Throughout this paper, we fix $k$ as a real number $k\geq \frac{11}{4}$.

We set 
\begin{equation*}
    \mathfrak{p}_{1}=\begin{bmatrix}
        0 & 1\\
        1 & 0
    \end{bmatrix},\,\mathfrak{p}_3=\begin{bmatrix}1 & 0\\
0 & -1\end{bmatrix}
\end{equation*}     
as  the standard first Pauli matrix and the third Pauli matrix.

\par Given $m\in\mathbb{N}$,  we denote
\begin{equation}
     [m]:=\{1,\,2\,,...\,,m\}.
\end{equation}
Given a collection of $m$ vectors in $\mathbb{R}^2\times\mathbb{R}^+\times\mathbb{R}$: $\sigma=\{(v_{\ell},y_{\ell},\alpha_{\ell},\gamma_{\ell})\}_{\ell\in [m]},$ we are interested in the superposition of $m$ solitons with parameters given from $\sigma$:
\begin{equation}\label{eq:Qsigma}
\mathcal{Q}_\sigma:=\sum_{\ell=1}^{m}e^{(i\frac{v_{\ell}x}{2}+\gamma_{\ell})}\phi_{\alpha_{\ell}}(x-y_{\ell})
\end{equation}
where $\phi_{\alpha_\ell}$ solves \eqref{elipt.11} and is given by \eqref{gr}.

The linearization \eqref{NLS1d} around each soliton results in a matrix linear operator.

When $\alpha_\ell=1$, we set
\begin{equation}\label{Hconst}
   \mathcal{H}_{1}:=\begin{bmatrix}
    {-}\partial^{2}_{x}+1-(k+1)\phi^{2k}_{1}(x) &  {-}k \phi^{2k}_{1}(x)\\
    k \phi^{2k}_{1}(x) & \partial^{2}_{x}-1+(k+1)\phi^{2k}_{1}(x)
    \end{bmatrix}.
\end{equation}
For a general $\alpha_{\ell}>0$, one defines
\begin{equation}\label{eq:Hell}
\mathcal{H}_{\ell}:=\begin{bmatrix}
    {-}\partial^{2}_{x}+\alpha_\ell^2-(k+1)\phi^{2k}_{\alpha_\ell}(x) &  {-}k \phi^{2k}_{\alpha_\ell}(x)\\
    k \phi^{2k}_{\alpha_{\ell}}(x) & \partial^{2}_{x}-\alpha_\ell^2+(k+1)\phi^{2k}_{\alpha_\ell}(x)
    \end{bmatrix}.
\end{equation}

 For each $\ell\in[m],$  $P_{d,\alpha_{\ell}}$ is the projection onto the discrete spectrum of $\mathcal{H}_{\ell}.$  We set the projection onto linear center-unstable space as $P_{d,\alpha_{\ell},\mathrm{cu}}$:
\begin{equation*}
\Ra P_{d,\alpha_{\ell},\mathrm{cu}}=\sppp\left\{\vec{z}\,\Big\vert\,
    \mathcal{H}_{\ell}^{2}\vec{z}=0 \text{ or } \mathcal{H}_{\ell}\vec{z}=\lambda z \text{ and $\I{\lambda}>0$} \right\}
\end{equation*}
which is of dimension $5$. Then we define the projection onto the linear unstable subspace as $P_{d,\alpha_{\ell},\mathrm{u}}$:
\begin{equation*}
\Ra P_{d,\alpha_{\ell},\mathrm{u}}=\sppp\left\{\vec{z}\,\Big\vert\,
   \mathcal{H}_{\ell}\vec{z}=\lambda z \text{ and $\I{\lambda}>0$} \right\}
\end{equation*}which is of dimension $1$.
Finally we define the projection onto the  generalized kernel of $\mathcal{H}_\ell$ as $P_{d,\alpha_{\ell},\mathrm{root}}$
\begin{equation*}
\Ra  P_{d,\alpha_{\ell},\mathrm{root}}=\sppp\left\{\vec{z}\,\Big\vert\,
  \mathcal{H}_{\ell}^{2}\vec{z}=0 \right\}
\end{equation*}which is of dimension $4$.
For more information on the spectrum of $H_\ell$, see \S\ref{subsubsec:spectral}.

Let $\alpha(t)$ be a continuous function on $t.$ If $r\in L^{2}_{x}(\mathbb{R}),$ the rescaled version $r_{\alpha(t)}(x)$ or $r(\alpha(t),x)$ is defined by
\begin{equation*}
r(\alpha(t),x)=r_{\alpha(t)}(x)\coloneqq \alpha(t)^{\frac{1}{k}}r\left(\alpha(t)x\right),   
\end{equation*}
for all $t$ in the domain of $\alpha,$ and $x\in\mathbb{R}.$

An indispensable tool to study multi-solitons is the Galilean transformation. 
For any $t\in\mathbb{R},$ we denote the Galilean transformation associated to a constant vector $(v_{\ell},y_{\ell},\alpha_{\ell},\gamma_{\ell})\in\mathbb{R}^{2}\times \mathbb{R}^+\times \mathbb{R}$ applied to  a function $\vec{f}=(f_{1},f_{2})\in L^{2}(\mathbb{R},\mathbb{C}^{2})$ by
\begin{equation}\label{galileantransf}
    \mathfrak{g}_{\ell}(\vec{f})(t,x)=e^{i\mathfrak{p}_3\left(\frac{v_{\ell}x}{2}-\frac{v_{\ell}^{2}t}{4}+\alpha_{\ell}^{2}t+\gamma_{\ell}\right)}\vec{f}(x-v_{\ell}t-y_{\ell})=\begin{bmatrix}
        e^{i\frac{v_{\ell}x}{2}-i\frac{v_{\ell}^{2}t}{4}+i\alpha_{\ell}^{2}t+i\gamma_{\ell}}f_{1}(x-v_{\ell}t-y_{\ell})\\
        e^{{-}i\frac{v_{\ell}x}{2}+i\frac{v_{\ell}^{2}t}{4}-i\alpha_{\ell}^{2}t-i\gamma_{\ell}}f_{2}(x-v_{\ell}t-y_{\ell})
    \end{bmatrix}.
\end{equation}
We define $\Sigma\subset L^{2}_{x}(\mathbb{R})$ to be the set of all the functions $r$ satisfying
\begin{equation}\label{eq:Sigmaspace}
\norm{r(x)}_{\Sigma}:=\norm{r(x)}_{H^{1}_{x}(\mathbb{R})}+\norm{\langle x\rangle r(x)}_{L^{2}_{x}(\mathbb{R})}+\norm{\langle x\rangle r(x)}_{W^{1,1}_{x}(\mathbb{R})}<{+}\infty.
\end{equation}
\subsubsection{Asymptotic stability}
With notations above, we now introduce the standing hypotheses and statements of asymptotic stability.

Throughout this paper, we impose the following \emph{hypotheses}.
\begin{itemize}
 \item [(H1)] We assume that velocities are different, and we order them as
  \begin{equation}\label{eq:velocity}
    v_{1}>v_{2}>...>v_{m}.
\end{equation}
    \item [(H2)] We assume that the centers of the solitons are well-separated:
    \begin{equation}\label{eq:sepcenter}
        \min_{\ell}y_{\ell}-y_{\ell+1}>L(\vec{\alpha_{\ell}},\min_{h}v_{h}-v_{h+1},m)
    \end{equation}for a positive parameter $L(\vec{\alpha}_\ell,\min_{h}v_{h}-v_{h+1}, m)=O\left(\max_{\ell}\{1,\frac{1}{\alpha_{\ell}},\frac{1}{\min_{h}v_{h}-v_{h+1}},m\}\right).$
   
\end{itemize}


Informally, we show that perturbations lying on a suitable codimension-$m$ center-stable manifold converge to a sum of modulated solitons plus dispersive radiation. 
The main result of this article is the following.
\begin{theorem}\label{asy}Assume that  hypotheses  $\mathrm{(H1)}$ and $\mathrm{(H2)}$ hold. Let $\delta_0\in (0,1)$ be a small constant only depending on the prescribed constants 
\begin{equation}\label{deltachoice}
    \delta_0:= \delta_0\Big( \max_{\ell}(\vert v_{\ell} \vert), \vert y_1- y_{m} \vert,L\Big)\ll 1.
\end{equation}
%
Consider the linear stable space: $$ 
 \mathrm{S}_m:=\Sigma\cap \left(\mathfrak{p}_3\bigoplus_{\ell=1}^{m}\Ra P_{d,\alpha_\ell,\mathrm{cu}}\right)^{\perp},$$
and  a small ball inside it
\begin{equation}\label{epsilonBsmall}
    \mathcal{B}_{\delta^{2}}:=\{r\in \mathrm{S}_m\vert\, \norm{r(x)}_{\Sigma}
    <\delta^{2}\}.
\end{equation}
If $\delta<\delta_0,$ then there exists
a Lipschitz map\footnote{Starting from Section \ref{sec:prelim}, without additional confusion, we will  drop the dependence on $\sigma$ in the subscript in $g_\sigma$.} $$g_\sigma:\mathcal{B}_{\delta^2}\to  \bigoplus_{\ell=1}^{m}\Ra P_{d,\alpha_{\ell},\mathrm{cu}}$$ 
satisfying
\begin{equation*}
  \norm{g_\sigma(r)-g_\sigma(r_{1})}_{L^{2}_{x}(\mathbb{R})}\lesssim  \norm{r-r_{1}}_{L^{2}_{x}(\mathbb{R})},\, \norm{g_\sigma(r)}_{L^{2}_{x}(\mathbb{R})}\lesssim\left[\delta_{0}^{2}+ \norm{r}_{L^{2}_{x}(\mathbb{R})}^{2}\right]
\end{equation*}
such that for\begin{equation}\label{eq:sigmanormsmall}
  \psi_{0}(x)=\sum_{\ell=1}^{m}e^{i(\frac{v_{\ell}x}{2}+\gamma_{\ell})}\phi_{\alpha_{\ell}}(x-y_{\ell})+r_{0}(x)\,\,\text{with}\,\,    r_0\in B_{\delta^2}, 
\end{equation}
%
the solution of the initial value problem
\begin{equation*}
    \begin{cases}
        i\partial_{t}\psi+\partial^{2}_{x}\psi+\vert \psi\vert^{2k}\psi=0,\\
        \psi(0,x)=\psi_{0}(x)+g(r_{0}(x))
    \end{cases}
\end{equation*}
such that the following estimate holds for  all $t\geq 0$ 
\begin{equation}\label{H1remainderinfty}
    \norm{\psi(t)-\sum_{\ell=1}^{m}e^{i\left(\frac{v_{\ell,\infty}x}{2}-\frac{v_{\ell,\infty}^{2}t}{4}+\alpha_{\ell,\infty}^{2}t+\gamma_{\ell,\infty}\right)}\phi_{\alpha_{\ell,\infty}}\left(x-v_{\ell,\infty}t-D_{\ell,\infty}\right)-e^{it\partial^{2}_{x}}f(x)}_{H^{1}_{x}(\mathbb{R})}\lesssim \frac{\delta_{0}}{(1+t)^{\frac{1}{4}}}
\end{equation}
for a function $f\in H^{1}_{x}(\mathbb{R}),$ for constant parameters $\alpha_{\ell,\infty},\,v_{\ell,\infty},\,D_{\ell,\infty}$ and $\gamma_{\ell,\infty}$, $\ell\in [m]$. In particular, the solution converges to a modulated multi-soliton plus a free dispersive term.
\end{theorem}
We first give comments to put our asymptotic stability result in the current literature.





\begin{remark}\label{rem:compare}
To the best of our knowledge, Theorem \ref{asy} is the first result in the 1D non-integrable setting that establishes full-line asymptotic stability for multi-solitons beyond the two-soliton case.


For the nonlinear Schr\"odinger equation, Perelman \cite{perelmanasym} considered only \emph{two} solitons with \emph{large} relative speeds, assuming both a stability condition and a nonlinearity that is \emph{very flat near the origin} (i.e., of high-order vanishing).  We remark that, to our knowledge, no explicit nonlinearity satisfying all the hypotheses of \cite{perelmanasym}  has been exhibited in the literature.  The seminal work of Martel–Merle–Tsai \cite{martel2006stability} proved orbital stability of multi-solitons in dimensions $1$–$3$ under restrictions relating the ratios of relative velocities and relative masses.  For the generalized Korteweg–de Vries equation, Mizumachi \cite{Mizumachikdvtwosoliton} obtained stability for \emph{two} solitons with very \emph{small} relative speed.

Compared with previous works, our result provides two principal advances:

\begin{enumerate}
  \item We consider a natural power-type, $L^2$-supercritical nonlinearity.
  \item We assume only that the soliton velocities are distinct; in particular, no large-separation  condition on the velocities is required.
\end{enumerate}
\end{remark}

Theorem \ref{asy} i follows from the following more precise asymptotic description.
\begin{theorem}\label{Tsigma(t)}
    Let $r_{0}\in 
     \Sigma\cap \left(\mathfrak{p}_3\bigoplus_{\ell=1}^{m}\Ra P_{d,\alpha_\ell,\mathrm{cu}}\right)^{\perp},$ $p\in (1,2)$ be close enough to $1$ and $\epsilon=\frac{3}{4}+\frac{3}{2}(1-\frac{2-p}{p})>\frac{3}{4}.$ Let 
\begin{equation*}
    \psi_{0}(x)=\sum_{\ell=1}^{m}e^{i(\frac{v_{\ell}x}{2}+\gamma_{\ell})}\phi_{\alpha_{\ell}}(x-y_{\ell})+r_{0}(x),
\end{equation*}
If  hypotheses  $\mathrm{(H1)}$ and $\mathrm{(H2)}$ hold, \eqref{epsilonBsmall}, then there exists a  Lipschitz map $g_\sigma:\mathcal{B}_{\delta^2}\to  \bigoplus_{\ell=1}^{m}\Ra P_{d,\alpha_{\ell}}$ and $\delta\in(0,\delta_0)$ 
satisfying
\begin{equation}\label{p0intial}
   \norm{g_\sigma(r)}_{L^{2}_{x}(\mathbb{R})}\leq \left[\delta_{0}^{2}+ \norm{r}_{L^{2}_{x}(\mathbb{R})}^{2}\right]
\end{equation}
such that if
\begin{equation*}
    \norm{ r_{0}(x)}_{\Sigma}\leq \delta^{2}, 
\end{equation*}
then the solution of the initial value problem
\begin{equation}\label{IVPmultisolitons}
    \begin{cases}
i\partial_{t}\psi+\partial^{2}_{x}\psi+\vert \psi\vert^{2k}\psi=0,\\
        \psi(0,x)=\psi_{0}(x)+g(r_{0}(x))
    \end{cases}
\end{equation}
satisfies the following asymptotics:  for some $C^{1}$ functions $\{(v_{\ell}(t),y_{\ell}(t),\alpha_{\ell}(t),\gamma_{\ell}(t)\}_{\ell\in[m]}$, one can decompose $\psi(t)$ as
\begin{equation}
    \psi(t)=\sum_{\ell=1}^{m}e^{i\left(\frac{v_{\ell}(t)x}{2}+\gamma_{\ell}(t)\right)}\phi_{\alpha_{\ell}(t)}\left(x-y_{\ell}(t)\right)+u(t)
\end{equation}such that the remainder term $u$ satisfies the following estimates:
\begin{itemize}
    \item  the sharp $1$d dispersive decay:
    \begin{align}
    \norm{u(t)}_{L^{\infty}_{x}(\mathbb{R})}\lesssim & \frac{\delta_0}{(1+t)^{\frac{1}{2}}};
    \end{align}
    \item local $L^2$ decay: for some small positive $\omega$, 
    \begin{align}
        \max_{\ell}\norm{\frac{1}{\langle x-y_{\ell}(t) \rangle^{\frac{3}{2}+\omega} }u(t)}_{L^{2}_{x}(\mathbb{R})}\lesssim & \frac{\delta_{0}}{(1+t)^{\frac{1}{2}+\epsilon}};
    \end{align}
    \item local $H^{1}$ decay: for some small positive $\omega$ and $p^{*}\in (2,{+}\infty)$ a large positive constant,
    \begin{align}
\max_{\ell}\norm{\frac{\partial_{x}\vec{u}(t,x)}{\langle x-y_{\ell}(t) \rangle^{1+\frac{p^{*}-2}{2p^{*}}+\omega}}}_{L^{2}_{x}(\mathbb{R})}\lesssim &\frac{\delta_{0}}{(1+t)^{\frac{1}{2}+\epsilon}};
    \end{align}
    \item Moreover, the following orthogonality conditions hold: 
    \begin{align}
           \langle \vec{u}(t),\mathfrak{p}_3e^{i\mathfrak{p}_3\left(\frac{v_{\ell}(t)x}{2}+\gamma_{\ell}(t)\right)}\vec{z}(\alpha_{\ell}(t),x-y_{\ell}(t)) \rangle=&0
    \end{align}for all $t\geq 0,$ and any $\vec{z}\in \ker\mathcal{H}^{2}_{\ell}.$ Here  
$\vec{u}=\begin{bmatrix}u \\
\Bar{u}\end{bmatrix}$.
\end{itemize}
Finally, the modulation parameters satisfy\begin{align}\label{odesinfty}
  \max_{\ell}\vert \dot \alpha_{\ell}(t) \vert&+ \max_{\ell} \vert \dot y_{\ell}(t)-v_{\ell}(t) \vert+\max_{\ell}\vert \dot v_{\ell}(t)\vert\\
  &+  \max_{\ell}\left\vert  \dot \gamma_{\ell}(t)-\alpha_{\ell}(t)^{2}+\frac{v_{\ell}(t)^{2}}{4}+\frac{y_{\ell}(t)\dot v_{\ell}(t)}{2}\right\vert
   \lesssim  \frac{\delta_{0}}{(1+t)^{1+2\epsilon}}.
\end{align}
\end{theorem}
Moreover, we can conclude that  the error term $u$ above has a refined scattering behavior:
\begin{corollary}\label{corollll}
Let $\psi(t,x)$ be solution of the initial value problem \eqref{IVPmultisolitons} given by  Theorem \ref{Tsigma(t)}. There exists a unique $\vec{\phi}_{\infty}\in L^{2}_{k}(\mathbb{R},\mathbb{C}^{2})$ belonging to the domain of the dispersive map $\mathcal{S}$ associated to $\sigma_{\infty}$ defined in Definition \ref{s0def} such that
\begin{equation*}
\norm{\psi(t)-\sum_{\ell=1}^{m}e^{i\left(\frac{v_{\ell}(t)x}{2}+\gamma_{\ell}(t)\right)}\phi_{\alpha_{\ell}(t)}\left(x-y_{\ell}(t)\right)-\mathcal{S}_{1}(\vec{\phi}_{\infty})(t,x)}_{H^{1}_{x}(\mathbb{R})}\lesssim \frac{\delta_{0}}{(1+t)^{\frac{3}{4}}} \text{, for all $t\geq 0,$}
\end{equation*}
where $\mathcal{S}_{1}(\vec{\phi}_{\infty})(t,x)$ is given by the first row of $\mathcal{S}(\vec{\phi}_{\infty})(t,x)$:
\begin{equation*}
    \mathcal{S}(\vec{\phi}_{\infty})(t,x)=\begin{bmatrix}
        \mathcal{S}_{1}(\vec{\phi}_{\infty})(t,x)\\
        \mathcal{S}_{2}(\vec{\phi}_{\infty})(t,x)
    \end{bmatrix} \in H^{1}_{x}(\mathbb{R},\mathbb{C}^{2}),
\end{equation*}
for all $t\geq 0.$
\end{corollary}
The proof of Corollary \ref{corollll} is in Appendix  \ref{ApB}.
\begin{proof}[Proof of Theorem \ref{asy} using Theorem \ref{Tsigma(t)}]
 First, Theorem \ref{Tsigma(t)} implies the existence of real constants $v_{\ell,\infty},\,y_{\ell,\infty},\,\alpha_{\ell,\infty}$ and $\gamma_{\ell,\infty}$ such that
 \begin{align*}
     \max_{\ell}\left\vert \gamma_{\ell}(t)-\alpha_{\ell,\infty}^{2}t+\frac{v_{\ell,\infty}^{2}t}{4}-\gamma_{\ell,\infty}\right \vert+\max_{\ell}\vert y_{\ell}(t)-v_{\ell,\infty}t-y_{\ell,\infty}  \vert\lesssim & \frac{\delta_{0}}{(1+t)^{2\epsilon-1}},\\
     \max_{\ell}\vert v_{\ell}(t)-v_{\ell,\infty}\vert+\max_{\ell}\vert \alpha_{\ell}(t)-\alpha_{\ell,\infty}\vert\lesssim & \frac{\delta_{0}}{(1+t)^{2\epsilon}}. 
 \end{align*}
To simplify more the notation used in the argument below, we define
 \begin{align}\label{thetainfty}
     \theta_{\ell,\infty}(t,x)=&\frac{v_{\ell,\infty}x}{2}-\frac{v_{\ell,\infty}^{2}t}{4}+\alpha_{\ell,\infty}^{2}t+\gamma_{\ell,\infty},\\ \label{thetanormal}
     \theta_{\ell}(t,x)=&\frac{v_{\ell}(t)x}{2}+\gamma_{\ell}(t).
 \end{align}
Setting $N(z)=\vert z\vert^{2k}z$ 
and using  equation \eqref{NLS1d} satisfied by $\psi(t),$ we can verify that the function
 \begin{equation}\label{rinfty}
     u(t,x)=\psi(t)-\sum_{\ell=1}^{m}e^{i\theta_{\ell}(t,x)}\phi_{\alpha_{\ell}(t)}(x-y_{\ell}(t))
 \end{equation}
is a  solution of the following  equation
\begin{multline}\label{forceeqqq}
    \begin{aligned}
    i\partial_{t}u(t)+\partial^{2}_{x}u(t,x)
    =&{-}N\left(\sum_{\ell=1}^{m}e^{i\theta_{\ell}(t,x)}\phi_{\alpha_{\ell}(t)}(x-y_{\ell}(t))\right)
    +\sum_{\ell=1}^{n}N\left(e^{i\theta_{\ell}(t,x)}\phi_{\alpha_{\ell}(t)}(x-y_{\ell}(t))\right)\\
    &{-}\sum_{\ell=1}^{m}\left[\left(i\partial_{t}+\partial^{2}_{x}\right)e^{i\theta_{\ell}(t,x)}\phi_{\alpha_{\ell}(t)}(x-y_{\ell}(t))+N\left(e^{i\theta_{\ell}(t,x)}\phi_{\alpha_{\ell}(t)}(x-y_{\ell}(t))\right)\right]\\
    &{-}N\left(\sum_{\ell=1}^{m}e^{i\theta_{\ell}(t,x)}\phi_{\alpha_{\ell}(t)}(x-y_{\ell}(t))+u(t)\right)\\&{+}N\left(\sum_{\ell=1}^{m}e^{i\theta_{\ell}(t,x)}\phi_{\alpha_{\ell}(t)}(x-y_{\ell}(t))\right)\\
    =&Forc(t).
\end{aligned}
\end{multline}
Using the Duhamel formula, the function $u(t)$ satisfies the following integral equation.
\begin{equation}
u(t,x)=e^{it\partial^{2}_{x}}u(0)+\int_{0}^{t}e^{i(t-s)\partial^{2}_{x}}Forc(s)\,ds.
\end{equation}
If the estimates 
\begin{align}\label{inftyest1}
    \norm{\int_{t}^{+\infty}e^{{-}is\partial^{2}_{x}}Forc(s)\,ds}_{L^{2}_{x}(\mathbb{R})}\lesssim & \frac{\delta_{0}}{(1+t)^{\frac{1}{4}}},\\ \label{inftyest2}
     \norm{\int_{t}^{+\infty}e^{{-}is\partial^{2}_{x}}Forc(s)\,ds}_{H^{1}_{x}(\mathbb{R})}\lesssim & \frac{\delta_{0}}{(1+t)^{\frac{1}{4}}}
\end{align}
hold, we can obtain using \eqref{odesinfty} and Lemma \ref{dinftydt} that Theorem \ref{asy} is true.
More precisely, estimates \eqref{inftyest1} and \eqref{inftyest2} would imply that the function $f(x)$ denoted by
\begin{equation*}
    f(x)=u(0,x)+\int_{0}^{+\infty}e^{{-}is\partial^{2}_{x}}Forc(s)\,ds
\end{equation*}
satisfies the statement of Theorem \ref{asy}.
\par Furthermore, using the estimates \eqref{odesinfty} and the exponential decay satisfied by all derivatives of the function $\phi_{\alpha}(x)$ defined in \eqref{gr}, we can verify from the chain rule of derivative that
\begin{equation}\label{1solitonpart}
\max_{\ell}\norm{\left(i\partial_{t}+\partial^{2}_{x}\right)e^{i\theta_{\ell}(t,x)}\phi_{\alpha_{\ell}(t)}(x-y_{\ell}(t))+N\left(e^{i\theta_{\ell}(t,x)}\phi_{\alpha_{\ell}(t)}(x-y_{\ell}(t))\right)}_{H^{1}_{x}(\mathbb{R})}\lesssim \frac{\delta_{0}}{(1+t)^{1+2\epsilon}}.
\end{equation}
Indeed, \eqref{1solitonpart} is a consequence of Theorem \ref{Tsigma(t)} and the following identity.
\begin{multline*}
\left(i\partial_{t}+\partial^{2}_{x}\right)e^{i\theta_{\ell}(t,x)}\phi_{\alpha_{\ell}(t)}(x-y_{\ell}(t))+N\left(e^{i\theta_{\ell}(t,x)}\phi_{\alpha_{\ell}(t)}(x-y_{\ell}(t))\right)\\
   \begin{aligned} 
    =&{-}(\dot y_{\ell}(t)-v_{\ell}(t))\begin{bmatrix}
ie^{i(\frac{v_{\ell}(t)x}{2}+\gamma_{\ell}(t))}\partial_{x}\phi_{\alpha_{\ell}(t)}(x-y_{\ell}(t))\\
ie^{{-}i(\frac{v_{\ell}(t)x}{2}+\gamma_{\ell}(t))}\partial_{x}\phi_{\alpha_{\ell}(t)}(x-y_{\ell}(t))
 \end{bmatrix}\\
   &{-} \dot v_{\ell}(t)
       \frac{(x-y_{\ell}(t))}{2}e^{i(\frac{v_{\ell}(t)x}{2}+\gamma_{\ell}(t))}\phi_{\alpha_{\ell}(t)}(x-y_{\ell}(t))\\
&{+}\dot \alpha_{\ell}(t)
ie^{i(\frac{v_{\ell}(t)x}{2}+\gamma_{\ell}(t))}\partial_{\alpha}\phi_{\alpha_{\ell}(t)}(x-y_{\ell}(t))\\
    &{-}\left(\dot \gamma_{\ell}(t)-\alpha_{\ell}(t)^{2}+\frac{v_{\ell}(t)^{2}}{4}+\frac{y_{\ell}(t)\dot v_{\ell}(t)}{2}\right)
    e^{i(\frac{v_{\ell}(t)x}{2}+\gamma_{\ell}(t))}\phi_{\alpha_{\ell}(t)}(x-y_{\ell}(t)).
\end{aligned}
\end{multline*}
\par Furthermore, Lemma \ref{multisolitonsinteractionsize}, \eqref{odesinfty} and \eqref{deltachoice}  imply that
\begin{multline}\label{H1forceterm}
\norm{N\left(\sum_{\ell=1}^{m}e^{i\theta_{\ell}(t,x)}\phi_{\alpha_{\ell}(t)}(x-y_{\ell}(t))\right)
    -\sum_{\ell=1}^{n}N\left(e^{i\theta_{\ell}(t,x)}\phi_{\alpha_{\ell}(t)}(x-y_{\ell}(t))\right)}_{H^{1}_{x}(\mathbb{R})}\\
    \begin{aligned}
    \lesssim &\delta_{0} e^{{-}\frac{\min_{\ell,j}\alpha_{j}(0)[y_{\ell}(t)-y_{\ell+1}(t)]}{5}}\\
    \lesssim &\frac{\delta_{0}}{(1+t)^{20}}.  
    \end{aligned}
\end{multline}
\par Next, let $Q(t)$ be the following function for all $t\geq 0.$
\begin{align}\label{QQQQ}
    Q(t,x)=N\left(\sum_{\ell=1}^{m}e^{i\theta_{\ell}(t,x)}\phi_{\alpha_{\ell}(t)}(x-y_{\ell}(t))+u(t)\right){-}N\left(\sum_{\ell=1}^{m}e^{i\theta_{\ell}(t,x)}\phi_{\alpha_{\ell}(t)}(x-y_{\ell}(t))\right)
    .
\end{align}
Since $N(z)=\vert z\vert^{2k}z$ is in $C^{2}$ when $k>\frac{11}{4}$ and $\vert u(t,x)\vert\lesssim\delta\ll 1$ from Theorem \ref{Tsigma(t)}, we can verify using the fundamental theorem of calculus the following estimates.
\begin{align}
    \left\vert Q(t,x) \right\vert\lesssim & \max_{\ell}\phi_{\alpha_{\ell}(t)}(x-y_{\ell}(t))\vert u(t,x)\vert+\vert u(t,x) \vert^{2k+1},\\
    \left\vert \partial_{x}Q(t,x) \right\vert\lesssim & \max_{\ell}(1+\vert v_{\ell}\vert)\phi_{\alpha_{\ell}(t)}(x-y_{\ell}(t))[\vert u(t,x)\vert+\vert \partial_{x}u(t,x)\vert]+ \vert u(t,x)^{2k} \partial_{x}u(t,x) \vert\\
    &{+} \max_{\ell}\vert\partial_{x}\phi_{\alpha_{\ell}(t)}(x-y_{\ell}(t)) u(t,x)\vert,
\end{align}
Consequently, we can deduce using the $L^{\infty}$ estimate, and the local $L^{2}$ and $H^{1}$ decay estimates from Theorem \ref{Tsigma(t)} that
\begin{equation}\label{Qestimatee}
    \norm{Q(t,x)}_{H^{1}_{x}(\mathbb{R})}\lesssim\frac{\delta_{0}^{2k+1}}{(1+t)^{k+\frac{1}{2}}}+\max_{\ell}(1+\vert v_{\ell}\vert) \frac{\delta_{0}}{(1+t)^{\frac{1}{2}+\epsilon}} \text{, for all $t\geq 0.$}
\end{equation}
As a consequence, using the inequalities \eqref{1solitonpart}, \eqref{H1forceterm} and \eqref{Qestimatee}, we can conclude from the fundamental theorem of calculus and the definition of $Forc(t,x)$ in \eqref{forceeqqq} that estimates \eqref{inftyest1} and \eqref{inftyest2} are true for all $t\geq 0.$ 
\par Therefore, using Lemma \ref{dinftydt}, the fact that $\epsilon>\frac{3}{4},$ and the estimates \eqref{inftyest1}, \eqref{inftyest2}, we can deduce that \eqref{H1remainderinfty} is true for all $t\geq 0.$ The other properties described of $(\vec{u}(t),\sigma(t))$ in the statement of Theorem \ref{asy} were proved in the proof of Theorem \ref{Tsigma(t)}.
\end{proof}

\subsubsection{Center-stable manifolds}
Note that Theorem \ref{asy} and Theorem \ref{Tsigma(t)} give asymptotic stability on a codimension $5m$ center-stable manifold. From the point of view of the  number of unstable eigenvalues, there are only $m$ unstable directions. Using the implicit function theorem, indeed, we can gain $4m$ dimensions back and conclude the following.

\begin{theorem}\label{thm:localmanifold}Assume that  hypotheses  $\mathrm{(H1)}$ and $\mathrm{(H2)}$ hold. Let $0<\delta<\delta_0$ where $\delta_0$ is defined in \eqref{deltachoice}. Using the notations from Theorem \ref{asy}, there is a Lipschitz manifold $\mathcal{N}_\sigma$ of  codimension $m$ in the the space $\Sigma$, \eqref{eq:Sigmaspace}, of size $\delta^2$ around $\mathcal{Q}_\sigma$, see \eqref{eq:Qsigma}, such that for any initial data $\psi(0)\in\mathcal{N}_\sigma$,  equation \eqref{NLS1d} has a global solution such that 
    for some $C^{1}$ functions $\{(v_{\ell}(t),y_{\ell}(t),\alpha_{\ell}(t),\gamma_{\ell}(t)\}_{\ell\in[m]}$, one can decompose $\psi(t)$ as
\begin{equation}
    \psi(t)=\sum_{\ell=1}^{m}e^{i\left(\frac{v_{\ell}(t)x}{2}+\gamma_{\ell}(t)\right)}\phi_{\alpha_{\ell}(t)}\left(x-y_{\ell}(t)\right)+u(t)
\end{equation}such that the remainder term $u$ and $\{(v_{\ell}(t),y_{\ell}(t),\alpha_{\ell}(t),\gamma_{\ell}(t)\}_{\ell\in[m]}$ satisfy all estimates in Theorem \ref{Tsigma(t)}. In particular, the conclusion of Theorem \ref{asy}, the decomposition \eqref{H1remainderinfty} holds.
\end{theorem}
\begin{proof}
 This is a direct consequence of Theorem \ref{Tsigma(t)} and Theorem \ref{asy} with a standard application of the implicit function theorem: we sketch the argument for completeness, following Theorem 4.4 in \cite{JacekGongKleingordon}.

Given  $\delta<\delta_0$, we set the $\delta^2$ neighborhood of $\mathcal{Q}_\sigma$  in $\Sigma$ as $\mathcal{B}_{\delta^2}(\mathcal{Q}_\sigma)$. Take any $\Tilde{\psi}_0\in \mathcal{B}_{\delta^2}(\mathcal{Q}_\sigma) $, by the implicit function theorem, one can find a map Lipschitz map
\begin{equation}\label{eq:tildesigma}
    \mathfrak{m}(\Tilde{\psi}_0,\mathcal{Q}_\sigma)=:\mathcal{Q}_{\Tilde{\sigma}}
\end{equation}such that with the new modulation parameter $\Tilde{\sigma}$,
\begin{equation}
    P_{d,\Tilde{\alpha}_{\ell},\mathrm{root}}(\Tilde{\psi}_0-\mathcal{Q}_{\Tilde{\sigma}})=0.
\end{equation}
Denote 
\begin{equation}
    \mathfrak{R}(\Tilde{\psi}_0)=\Tilde{\psi}_0-\mathcal{Q}_{\Tilde{\sigma}}=\Tilde{\psi}_0-\mathfrak{m}(\Tilde{\psi}_0,\mathcal{Q}_\sigma)
\end{equation}which is clearly a Lipschitz function in $\mathcal{B}_{\delta^2}(\mathcal{Q}_\sigma)$.

We define a codimension $m$ linear center stable space as
\begin{equation}
 \mathfrak{N}(\mathcal{Q}_\sigma):=\left\{ \Tilde{\psi}_0\in \mathcal{B}_{\delta^2}(\mathcal{Q}_\sigma) |P_{d,\Tilde{\sigma}_\ell,\mathrm{u}} (\mathfrak{R}(\Tilde{\psi}_0))=0,\ell\in[m]\right\}.   
\end{equation}
Then $\mathcal{N}_\sigma$ is defined as 
\begin{equation}
 \mathcal{N}_\sigma:=\left\{ \Tilde{\psi}_0+g_{\Tilde{\sigma}}( \mathfrak{R}(\Tilde{\psi}_0)),\,\,\Tilde{\psi}_0\in  \mathfrak{N}(\mathcal{Q}_\sigma),\quad\text{where $\Tilde{\sigma}$ is from \eqref{eq:tildesigma}} \right\} 
\end{equation}where $g_{\Tilde{\sigma}}$ is the function constructed in Theorem \ref{asy} with respect to the parameter $\Tilde{\sigma}$.

Define
\begin{align}
    \mathfrak{F}_\sigma(\Tilde{\psi}_0):=\Tilde{\psi}_0+g_{\Tilde{\sigma}}( \mathfrak{R}(\Tilde{\psi}_0)).
\end{align}One can check that the Jacobian of the map $\mathfrak{F}_\sigma$ is non-degenerate. One can conclude that   $\mathcal{N}_\sigma$ is the image of $\mathfrak{N}(\mathcal{Q}_\sigma)$ under the bi-Lipschitz invertible map $\mathfrak{F}_\sigma$. So $\mathcal{N}_\sigma$ is indeed a Lipschitz manifold of codimension $m$. The desired behaviors of solutions with initial condition in $\mathcal{N}_\sigma$ follow Theorem \ref{asy} and Theorem \ref{Tsigma(t)}.
\end{proof}

Finally, with a patching argument, one can extend the local manifold construction to the neighborhood of the family of well-separated $m$ multi-solitons.

Given $m\in \mathbb{N}$, and a constant $L$ from \eqref{eq:sepcenter}, we define the multi-soliton family:
\begin{equation}
\mathfrak{F}_{m,C,L}=\left\{\mathcal{Q}_\sigma| \,\,\sigma\, \text{ satisfies \eqref{eq:velocity}\,\eqref{eq:sepcenter} }  \right\}
\end{equation}
where $\mathcal{Q}_\sigma$ is defined by \eqref{eq:Qsigma}.
\begin{theorem}
Given a multi-soliton family $\mathfrak{F}_{m,C,L}$ above,  there
exists a codimension $m$ Lipschitz center-stable manifold $\mathcal{N}$ around the well-separated multi-soliton
family $\mathfrak{F}_{m,C,L}$ which is invariant for $t\geq0$ such that for any choice of initial data $\psi(0)\in\mathcal{N}$ , the
solution $\psi$ to \eqref{NLS1d} with initial data $\psi(0)$ exists globally, and it scatters to the multi-soliton
family: there exist a function $f\in H^{1}_{x}(\mathbb{R}),$ for constant parameters $\alpha_{\ell,\infty},\,v_{\ell,\infty},\,D_{\ell,\infty}$ and $\gamma_{\ell,\infty}$, $\ell\in [m]$, the following inequality for all $t\geq 0$ 
\begin{equation}
    \norm{\psi(t)-\sum_{\ell=1}^{m}e^{i\left(\frac{v_{\ell,\infty}x}{2}-\frac{v_{\ell,\infty}^{2}t}{4}+\alpha_{\ell,\infty}^{2}t+\gamma_{\ell,\infty}\right)}\phi_{\alpha_{\ell,\infty}}\left(x-v_{\ell,\infty}t-D_{\ell,\infty}\right)-e^{it\partial^{2}_{x}}f(x)}_{H^{1}_{x}(\mathbb{R})}\lesssim \frac{\delta_{0}}{(1+t)^{\frac{1}{4}}}
\end{equation}
with $$\mathcal{Q}_{\sigma_\infty}=\sum_{\ell=1}^{m}e^{i\left(\frac{v_{\ell,\infty}x}{2}-\frac{v_{\ell,\infty}^{2}t}{4}+\alpha_{\ell,\infty}^{2}t+\gamma_{\ell,\infty}\right)}\phi_{\alpha_{\ell,\infty}}\left(x-v_{\ell,\infty}t-D_{\ell,\infty}\right)\in\mathfrak{F}_{m,C,L}.$$ Moreover, the more precise description of the solution from Theorem \ref{Tsigma(t)} also holds.
\end{theorem}
\begin{proof}
    This is a direct consequence of Theorem \ref{thm:localmanifold} above. The proof is independent of the structure of the equation. We refer to the proof of   Theorem 4.5 in \cite{JacekGongKleingordon}.
\end{proof}

\subsection{Organization and notations}
\subsubsection{Organization}
In Section \ref{sec:prelim}, we introduce basic notations and linear estimates which will be crucial in our nonlinear analysis. We also introduce the main ideas including two main propositions which together imply Theorem \ref{Tsigma(t)}. In Section \ref{sec:undecay}, we prove Proposition \ref{undecays} to find a sequence of solutions $u_n$ satisfying decay properties with their unstable mode terminated at $T_n$ with $T_n\rightarrow\infty$. Then we show the convergence of the sequence in the previous section in Section \ref{sectioncauchy} by Proposition \ref{propun-un-1}. In Appendix \ref{B}, weighted $L^2$ estimates for linear flow are provided. Finally, we show Corollary \ref{corollll} in Appendix \ref{ApB}.

\subsubsection{Notations}
Throughout this article,  we use $\Diamond$ as a placeholder for a dummy variable when the specific choice is clear from context.

As usual, “$A := B$” or “$B =: A$” is the definition of $A$ by means of the expression
$B$.  

We use $\langle \Diamond \rangle:=\sqrt{1+\Diamond^2}$.
%
$\chi_A$ for some set $A$ is always denoted as a smooth indicator function adapted to the set $A$. 

Throughout, we use $u_t=\partial_t u:=\frac{\partial}{\partial_t}u$ and $u_x=\partial_x u:=\frac{\partial}{\partial_x} u$ interchangeably. 

For non-negative $X$, $Y$, we write $X\lesssim Y$ if $X \leq C Y$ ,
and we use the notation $X\ll Y$ to indicate that the implicit constant should be regarded as small.
Furthermore, for nonnegative $X$ and arbitrary $Y$ , we use the shorthand notation $Y = \mathcal{O}(X)$ if
$|Y|\leq C X$. 

Given a complex function $f$, $\vec{f}$ is always used to denote
\begin{equation}
    \vec{f}=\begin{bmatrix}
        f\\
        \Bar{f}
    \end{bmatrix}.
\end{equation}

Let $F \in \mathbb{C}^{2}$ be vector 
\begin{equation}\label{Fdefinition}
    F(z)\coloneqq 
    \begin{bmatrix}
        {-}\vert z\vert^{2k}z\\
        \vert z \vert^{2k}\overline{z}
    \end{bmatrix}.
\end{equation}
\noindent {\it Inner products.}
In terms of the $L^2$ inner product of complex-valued functions, we use
\begin{equation}\label{eq:L2inner}
    \langle f, g\rangle = \int_\mathbb{R} f\overline{g}\, dx.
\end{equation} 
Given two pairs of complex-valued vector functions $\vec{f} = (f_1, f_2)$ and $\vec{g} = (g_1, g_2)$, their  inner product is given by
\begin{equation}\label{eq:L2L2inner}
    \langle \vec{f}, \vec{g}\rangle:=  \int_\mathbb{R} \bigl( f_1 \overline{g_1} + f_2 \overline{g_2} \bigr) \, d x.
\end{equation}

\subsection{Acknowledgement}
 We would like to thank Jacek Jendrej and 
 Joachim Krieger 
 for valuable comments and feedback.

\section{Preliminaries and main ideas}\label{sec:prelim}
\subsection{Spectral theory}\label{subsubsec:spectral}
We start with the spectral theory for the matrix Schr\"odinger operator
\begin{equation*}
    \mathcal{H}_{1}=\begin{bmatrix}
    {-}\partial^{2}_{x}+1-(k+1)\phi^{2k}_{1}(x) &  {-}k \phi^{2k}_{1}(x)\\
    k \phi^{2k}_{1}(x) & \partial^{2}_{x}-1+(k+1)\phi^{2k}_{1}(x)
    \end{bmatrix},
\end{equation*}
where $\phi_{1}(x)=(k+1)^{\frac{1}{2k}}\sech^{\frac{1}{k}}(kx),$ see \eqref{gr}.
The spectral properties are well-known. 
The spectrum of $\mathcal{H}_{1}=\sigma_{d}\mathcal{H}_{1}\bigoplus \sigma_{e}\mathcal{H}_{1},$ where $\sigma_{d}\mathcal{H}_{1}$ means the discrete spectrum of $\mathcal{H}_{1},$ and $\sigma_{e}\mathcal{H}_{1}$ means the essential spectrum of $\mathcal{H}_{1}.$
More precisely:
\begin{itemize}
    \item [(a)] $\sigma_{d}\mathcal{H}_{1}=\{0,\pm i\lambda_{0}\},$ for a constant $\lambda_{0}>0.$
    \item [(b)] $\sigma_{e}\mathcal{H}_{1}= ({-}\infty,1]\cup [1,{+}\infty).$
    \item [(c)] $\ker\mathcal{H}_{1}^{2}=\ker\mathcal{H}^{n}_{1},$ for all natural number $n\geq 2.$
    \item [(d)] $\dim \ker\mathcal{H}_{1}=2.$ In particular
    \begin{equation*}
        \ker \mathcal{H}_{1}=\sppp\left\{\begin{bmatrix}
            \partial_{x}\phi_{1}(x)\\
            \partial_{x}\phi_{1}(x)
        \end{bmatrix},\, \begin{bmatrix}
            i\phi_{1}(x)\\
            {-}i\phi_{1}(x)
        \end{bmatrix}\right\}.
    \end{equation*}
    \item [(e)] $\dim \ker\mathcal{H}_{1}^{2}=4,$ and
    \begin{equation}\label{ker2basis}
        \ker \mathcal{H}_{1}^{2}=\sppp\left\{\begin{bmatrix}
            \partial_{x}\phi_{1}(x)\\
            \partial_{x}\phi_{1}(x)
        \end{bmatrix},\, \begin{bmatrix}
            i\phi_{1}(x)\\
            {-}i\phi_{1}(x)
        \end{bmatrix},\,
        \begin{bmatrix}
        ix\phi_{1}(x)\\
        {-}ix\phi_{1}(x)
        \end{bmatrix},\,
        \begin{bmatrix}
        \partial_{\alpha}\phi_{1}(x)\\
        \partial_{\alpha}\phi_{1}(x)
        \end{bmatrix}
        \right\},
    \end{equation}
    see \eqref{kernel2proj} for more details.
    \item [(f)] $\dim \ker [\mathcal{H}_{1}\pm i\lambda_{0}\mathrm{Id}]=1.$ In particular, there exists a unitary vector $\vec{Z}_{+}\in L^{2}_{x}(\mathbb{R},\mathbb{C}^{2})$ such that for all $n\in\mathbb{N}_{\geq 1}$
    \begin{align}\label{z11}
        \ker [\mathcal{H}_{1}- i\lambda_{0}\mathrm{Id}]^{n}=\ker [\mathcal{H}_{1}- i\lambda_{0}\mathrm{Id}]=&\sppp\left\{\vec{Z}_{+}(x)\right\},\\ \nonumber
        \ker [\mathcal{H}_{1}-i\lambda_{0}\mathrm{Id}]^{n}=\ker [\mathcal{H}_{1}+ i\lambda_{0}\mathrm{Id}]=&\sppp\left\{\overline{\vec{Z}_{+}(x)}\right\}.
    \end{align}
\end{itemize}
\begin{remark}
The existence of the eigenvector $\vec{Z}_{+}(x)$ satisfying $\mathcal{H}_{1}\vec{Z}_{+}(x)=i\lambda_{0}\vec{Z}_{+}(x)$ implies the spectral instability of the operator $\mathcal{H}_{1}.$ From  the article \cite{solisch} of Grillakis, Shatah and Strauss,  as a consequence, the soliton solution $e^{it}\phi_{1}(x)$ of \eqref{NLS1d} is  orbitally unstable, due to the presence of unstable mode $\vec{Z}_{+}$ of the operator $\mathcal{H}_{1}.$
\end{remark}
\par Concerning information about the subspace of $L^{2}_{x}(\mathbb{R},\mathbb{C})$ generated by Riesz projection on the essential spectrum of $\mathcal{H}_{1},$ see Section $6$ of \cite{KriegerSchlag}.

\begin{remark}
    By a simple rescaling, for $\mathcal{H}_\ell$, see \eqref{eq:Hell}, the essential spectrum is given by $\sigma_{e}\mathcal{H}_{\ell}= ({-}\infty,-\alpha_\ell]\cup [\alpha_\ell,{+}\infty)$ and its discrete spectrum is given by $\sigma_{d}\mathcal{H}_{\ell}=\{0,\pm i \alpha_\ell\lambda_{0}\}$. One also knows  that $\ker \mathcal{H}_{\ell}^{2}=\ker \mathcal{H}_{\ell}^{n}$ for all $n\geq 2,$ and $\dim \ker \mathcal{H}_{\ell}^{2}=4.$ In particular, the following identities are true:
\begin{align}\nonumber
    \ker \mathcal{H}_{\ell}&=\sppp\left\{\begin{bmatrix}
        \partial_{x}\phi_{\alpha_{\ell}}(x)\\
        \partial_{x}\phi_{\alpha_{\ell}}(x)
    \end{bmatrix},\begin{bmatrix}
i\phi_{\alpha_{\ell}}(x)\\ \label{kerident}
{-}i\phi_{\alpha_{\ell}}(x)
    \end{bmatrix}\,\right\},\\ \mathcal{H}_{\ell}\begin{bmatrix}
    \partial_{\alpha}\phi_{\alpha_{\ell}}(x)\\
    \partial_{\alpha}\phi_{\alpha_{\ell}}(x)
    \end{bmatrix}&={-}2\alpha_{\ell}\begin{bmatrix}
\phi_{\alpha_{\ell}}(x)\\
{-} \phi_{\alpha_{\ell}}(x)     \end{bmatrix},\, \mathcal{H}_{\ell}\begin{bmatrix}
    x\partial_{x}\phi_{\alpha_{\ell}}(x)\\
{-}x\partial_{x}\phi_{\alpha_{\ell}}(x)
    \end{bmatrix}={-}2\begin{bmatrix}
\partial_{x}\phi_{\alpha_{\ell}}(x)\\
\partial_{x}\phi_{\alpha_{\ell}}(x)       \end{bmatrix}.
\end{align}
\end{remark}

From now on, for each $\ell\in[m],$ we consider the subset $\sigma_{d,\mathrm{stab}}\mathcal{H}_{\ell}$ of $\sigma_{d}\mathcal{H}_{\ell}$ to be 
\begin{equation}\label{sigmastab}
\sigma_{d,\mathrm{stab}}\mathcal{H}_{\ell}\coloneq\{\lambda \in \sigma_{d} \mathcal{H}_{\ell}\vert\,\, \I\lambda<0 \}.
\end{equation}
\subsubsection{Distorted Fourier transforms}
It is known, see  Lemma $6.3$ of \cite{KriegerSchlag},  that there exist generalized eigenfunctions $\mathcal{F}(x,\xi),\,\mathcal{G}(x,-\xi)$ of $\mathcal{H}_1$ in $L^{\infty}_{x}(\mathbb{R})$ satisfying 
$$\mathcal{H}_{1}\left(\mathcal{F}(x,\xi)\right)=(1+\xi^{2})\mathcal{F}(x,\xi),\,\mathcal{H}_{1}(\mathcal{G}(x,{-}\xi))=(1+\xi^{2})\mathcal{G}(x,{-}\xi),$$with following asymptotic behavior for some $\gamma>0.$
\begin{align}\label{asy1}
    \mathcal{G}(x,{-}\xi)=&\overline{s(\xi)}\left[e^{i\xi x}
    \begin{bmatrix}
        1\\
        0
    \end{bmatrix}+O\left(\frac{e^{\gamma x}}{(1+\vert \xi \vert)}\right)\right] \text{, as $x\to {-}\infty,$}\\ \label{asy2}
    \mathcal{G}(x,{-}\xi)=&e^{i\xi x}
    \begin{bmatrix}
        1\\
        0
    \end{bmatrix}+\overline{r(\xi)}e^{{-}i\xi x}
    \begin{bmatrix}
        1\\
        0
    \end{bmatrix}+O\left(\frac{e^{{-}\gamma x}}{(1+\vert \xi \vert)}\right)
    \text{, as $x\to {+}\infty,$}\\ \label{asy3}
    \mathcal{F}(x,\xi)=& s(\xi)\left[e^{i\xi x}
    \begin{bmatrix}
        1\\
        0
    \end{bmatrix}+O\left(\frac{e^{{-}\gamma x}}{(1+\vert \xi \vert)}\right)\right] \text{, as $x\to {+}\infty,$}\\ \label{asy4}
    \mathcal{F}(x,\xi)=&e^{i\xi x}
    \begin{bmatrix}
        1\\
        0
    \end{bmatrix}+r(\xi)e^{{-}i\xi x}
    \begin{bmatrix}
        1\\
        0
    \end{bmatrix}+O\left(\frac{e^{\gamma x}}{(1+\vert \xi \vert)}\right)
    \text{, as $x\to {-}\infty.$}
\end{align}
It is also shown in \cite{KriegerSchlag} that  $\vert r(\xi)\vert^{2}+\vert s(\xi)\vert^{2}=1.$ The smooth functions $r$ and $s$ are called reflection and transmission coefficients of the Sch\"odinger operator $\mathcal{H}_{1}$ respectively.  Concerning more details about the properties of $r,\,s$ and the scattering theory of $\mathcal{H}_{1},$ see sections $5$ and $6$ of  \cite{KriegerSchlag}.

It shown in \cite{dispanalysis1}, the range of essential spectrum projection of $\mathcal{H}_{1}$ is equal to the set
    \begin{multline*}
        \left\{\int_{\mathbb{R}}\frac{\mathcal{G}(x,{-}\xi)u_{1}(\xi)}{s({-}\xi)}+\mathfrak{p}_{1}\frac{\mathcal{G}(x,{-}\xi)u_{2}(\xi)}{s({-}\xi)}\,d\xi\in L^{2}_{x}(\mathbb{R},\mathbb{C}^{2})\vert\,\, u_{1},\,u_{2}\in L^{2}_{\xi}(\mathbb{R})\right\}\\
        = \left\{\int_{\mathbb{R}}\frac{\mathcal{F}(x,\xi)f_{1}(\xi)}{s(\xi)}+\mathfrak{p}_{1}\frac{\mathcal{F}(x,\xi)f_{2}(\xi)}{s(\xi)}\,d\xi\in L^{2}_{x}(\mathbb{R},\mathbb{C}^{2})\vert\,\, f_{1},\,f_{2}\in L^{2}_{\xi}(\mathbb{R})\right\}.
    \end{multline*}

\subsection{Linear theory for one-dimensional charge transfer models}
In this subsection, we revisit the linear theory for one-dimensional charge transfer models developed in \cite{dispanalysis1,dispanalysis2}. We recall the dispersive theory  which forms the backbone of our nonlinear analysis.
We state all estimates  here in the precise form needed for the nonlinear argument.
\begin{definition}\label{def:evolution}
Let $\sigma=\{(v_{\ell},y_{\ell},\alpha_{\ell},\gamma_{\ell})\}_{\ell\in [m]}$ be a set constant vectors of $\mathbb{R}^{2}\times \mathbb{R}^+\times \mathbb{R}.$ For any $t\geq \tau\geq 0,$ we denote the solution of the following  system
{\footnotesize \begin{multline}\label{chargetransfer}
 \begin{cases}  
   i\partial_{t}\vec{u}(t,x)+\mathfrak{p}_3\partial^{2}_{x}\vec{u}(t,x)+\\\sum_{\ell=1}^{m}\begin{bmatrix}
       (k+1)\phi^{2k}_{\alpha_{\ell}}(x-v_{\ell}t-y_{\ell}) & ke^{i\left(\frac{v_{\ell}x}{2}-\frac{v_{\ell}^{2}t}{4}+\alpha_{\ell}^{2}t+\gamma_{\ell}\right)}\phi^{2k}_{\alpha_{\ell}}(x-v_{\ell}t-y_{\ell})\\
      {-} ke^{{-}i\left(\frac{v_{\ell}x}{2}-\frac{v_{\ell}^{2}t}{4}+\alpha_{\ell}^{2}t+\gamma_{\ell}\right)}\phi^{2k}_{\alpha_{\ell}}(x-v_{\ell}t-y_{\ell}) & (k+1)\phi^{2k}_{\alpha_{\ell}}(x-v_{\ell}t-y_{\ell}) 
   \end{bmatrix}\vec{u}(t,x)=0,\\
   \vec{u}(\tau,x)=\vec{u}_{0}(x)\in L^{2}_{x}(\mathbb{R},\mathbb{C}^{2}),
\end{cases}
\end{multline}}
by $\mathcal{U}_{\sigma}(t,\tau)(\vec{u}_{0})(x).$
\end{definition}
It is known that there exist solutions of \eqref{chargetransfer} converging to the range of $\mathfrak{g}_{\ell}\left(P_{d,\alpha_{\ell}}\right)(t)$ as $t$ approaches ${+}\infty.$ More precisely, the following theorem holds.

\begin{theorem}[Solutions in the discrete space]\label{tdis}
Given any $1\leq \ell \leq m$, let $\mathfrak{v}_{\ell,\lambda}\in L^{2}_{x}\left(\mathbb{R},\mathbb{C}^{2}\right)$ be any element in $\ker [\mathcal{H}_{\ell}-\lambda \mathrm{Id}]$
for some eigenvalue 
$\lambda $ of $\mathcal{H}_{\ell}$ satisfying $\I \lambda\leq 0.$ There exist constants $K>0,\,L,\,C\gg 1,\,\beta>0$ depending on $\{(\alpha_{\ell},\gamma_{\ell})\}_{\ell}$ and $m$ such that if $$\min_{\ell}y_{\ell}-y_{\ell+1}>L\quad\text{and}\quad\min_{\ell}v_{\ell}-v_{\ell+1}>C,$$ then using the matrix Galilei transformation, there is a unique solution
\begin{align}\label{eq:Gv}
\mathfrak{G}_{\ell}(\mathfrak{v}_{\alpha_{\ell},\lambda})(t,x)&\coloneq e^{i\mathfrak{p}_3(\frac{v_{\ell}x}{2}+\gamma_{\ell})}e^{{-}i((\lambda-\alpha_{\ell}^{2}t) t+\frac{\mathfrak{p}_3v_{\ell}^{2}t}{4})}\mathfrak{v}_{\alpha_{\ell},\lambda}(x-y_{\ell}-v_{\ell}t)+r(t,x)\\
&=:\mathfrak{g}_{\ell}(\mathfrak{v}_{\alpha_{\ell},\lambda})(t,x)+r(t,x)
\end{align} of \eqref{chargetransfer} satisfying
\begin{equation}
 \norm{r(t,x)}_{L^{2}_{x}(\mathbb{R})}\leq K e^{{-}\min_{\ell,j}\alpha_{j}(y_{\ell}-y_{\ell+1}+v_{\ell}-v_{\ell+1}) t)} \text{, for all $t\geq 0.$}   
\end{equation}
Moreover, if $\mathfrak{z}_{\ell}\in \ker \mathcal{H}_{\ell}^{2}$ and $\mathcal{H}_{\ell}(\mathfrak{z}_{\ell})=i\mathfrak{v}_{\alpha_{\ell},0},$ then there exists a unique solution
\begin{equation}\label{eq:Gz}
\mathfrak{G}_{\ell}(\mathfrak{z}_{\ell})(t,x)\coloneq\mathfrak{g}_{k}(\mathfrak{z}_{\ell})(t,x)+t\mathfrak{g}_{\ell}(\mathfrak{v}_{\ell,0})(t,x)+r(t,x)
\end{equation}
of \eqref{chargetransfer} satisfying
$
 \norm{r(t,x)}_{L^{2}_{x}(\mathbb{R})}\leq K e^{{-}\min_{\ell,j}\alpha_{j}(y_{\ell}-y_{\ell+1}+v_{\ell}-v_{\ell+1}) t)} \text{, for all $t\geq 0.$}   
$
\end{theorem} 
\begin{proof}
 See Section $7$ of \cite{dispanalysis1} and Section $6$ of Section \cite{dispanalysis2}.   
\end{proof}

\begin{definition}
 If $\alpha>0,$ we consider 
 \begin{equation*}
     \mathcal{G}_{\alpha}(x,{-}\xi)\coloneqq \mathcal{G}\left(\alpha x,{-}\alpha^{{-}1}\xi\right),\,\,s_{\alpha}(\xi)\coloneqq s(\alpha^{{-}1}\xi),\,r_{\alpha}(\xi)=r(\alpha^{{-}1}\xi). 
 \end{equation*}
We define the map 
\begin{align}
\label{DisG1}
    \hat{G}_{\alpha}\left(\overrightarrow{u}\right)(x)\coloneqq &\frac{1}{\sqrt{2\pi}}\int_{\mathbb{R}}\frac{1}{s_{\alpha}({-} \xi)}\left[\mathcal{G}_\alpha(x,{-}\xi)\,\,\sigma_{1} \mathcal{G}_\alpha(x,{-}\xi)\right]\overrightarrow{u}(\xi)\,d\xi \in L^{2}(\mathbb{R},\mathbb{C}^{2}).
 \end{align}
\end{definition}
For more properties of the map above, we refer to \cite{dispanalysis2}.
\begin{definition}[Dispersive map]\label{s0def}
Given $v_{1}>v_{2}>...>v_{m},\,\delta y_{\ell}=y_{\ell-1}-y_{\ell}\gg 1,$ for  any given
\begin{equation*}
\vec{\phi}(\xi)= 
    \begin{bmatrix}
     \phi_{1}(\xi)\\
     \phi_{2}(\xi)
    \end{bmatrix}\in L^{2} (\mathbb{R},\mathbb{C}^{2}),
\end{equation*}
we define the following formula
\begin{align}\label{eq:Sphi}
   \mathcal{S}(\vec{\phi})(t,x):= & \sum_{\ell=1}^{m}e^{i\left(\frac{v_{\ell}x}{2}-\frac{v_{\ell}^{2}t}{4}+\alpha^2_{\ell}t+\gamma_{\ell}\right)\sigma_{3}}\hat{G}_{\alpha_\ell}\left(
   e^{{-}it(\xi^{2}+\alpha^2_\ell)\sigma_{3}}e^{{-}i\gamma_{\ell}\sigma_{3}} \begin{bmatrix}
e^{iy_{\ell}\xi}\phi_{1,\ell}\left(\xi+\frac{v_{\ell}}{2}\right)\\
       e^{iy_{\ell}\xi}\phi_{2,\ell}\left(\xi-\frac{v_{\ell}}{2}\right)
    \end{bmatrix}\right)(x-y_{\ell}-v_{\ell}t)\\
    & {-}\frac{1}{\sqrt{2 \pi}}\int_{\mathbb{R}}e^{{-}it \xi^{2}\sigma_{3}}
    \begin{bmatrix}
       \varphi_{1}(\xi)\\
       \varphi_{2}(\xi)
    \end{bmatrix} e^{i\xi}\,d\xi,\nonumber
\end{align}
where the sequence $\{\overrightarrow{\phi_{\ell}}\}_{\ell=1}^m$ and $\vec{\varphi}$ are constructed recursively from $\vec{\phi}(\xi)$ via the following conditions
\begin{enumerate}
    \item [a)] $\begin{bmatrix}
        \phi_{1,1}(\xi)\\
        \phi_{2,1}(\xi)
    \end{bmatrix}=\vec{\phi}(\xi); $
    \item [b)] for each $\ell\geq 1,$
    \begin{equation*}
      e^{{-}i\gamma_{\ell+1}\sigma_{3}} \begin{bmatrix}
           \phi_{1,\ell+1}(\xi)\\
           \phi_{2,\ell+1}(\xi)
       \end{bmatrix}
        =e^{{-}i\gamma_{\ell}\sigma_{3}}
        \begin{bmatrix}
        \frac{\phi_{1,\ell}\left(\xi\right)-r_{\alpha_{\ell}}\left(\xi-\frac{v_{\ell}}{2}\right)e^{{-}i2y_{\ell}(\xi-\frac{v_{\ell+1}}{2})+iy_{\ell}(v_{\ell}-v_{1+\ell})}\phi_{1,\ell}\left({-}\xi+v_{\ell}\right)}{s_{\alpha_{\ell}}\left(\xi-\frac{v_{\ell}}{2}\right)}\\
        \frac{\phi_{2,\ell}\left(\xi\right)-r_{\alpha_{\ell}}\left(\xi+\frac{v_{\ell}}{2}\right)e^{{-}2iy_{\ell}(\xi+\frac{v_{\ell+1}}{2})+iy_{\ell}(v_{\ell+1}-v_{\ell})}\phi_{2,\ell}\left({-}\xi-v_{\ell}\right)}{s_{\alpha_{\ell}}\left(\xi+\frac{v_{\ell}}{2}\right)}
    \end{bmatrix};
    \end{equation*}
    \item [c)] and
    \begin{equation*}
        \begin{bmatrix}
          \varphi_{1}(\xi)\\
          \varphi_{2}(\xi)
        \end{bmatrix}=\sum_{\ell=1}^{m-1}\begin{bmatrix}
           \phi_{1,\ell}(\xi)\\
           \phi_{2,\ell}(\xi)
       \end{bmatrix}.
    \end{equation*}

\end{enumerate} 
 For the convenience of notations, we use $S(t)$ to denote
\begin{equation}\label{eq:St}
    \mathcal{S}(t)\vec{\phi}:= \mathcal{S}(\vec{\phi})(t,x).
\end{equation}
\end{definition}
\begin{remark}\label{RRR}
 It was verified in \cite{dispanalysis1} when $N>1$ is large, and $m\in\mathbb{N}_{\geq 1}$  enough that
 \begin{equation*}
    \max_{j\in\{0,1\}}\norm{\langle x \rangle^{m}\frac{\partial^{j}}{\partial x^{j}}\left[\hat{G}_{\alpha_{\ell}}\left(\vec{f}(\xi)\right)(x)-\int_{\mathbb{R}}e^{i\xi}\frac{\vec{f}(\xi)}{\sqrt{2\pi}}\,d\xi\right]}_{L^{2}_{x}({-}\infty,{-}N)}\lesssim N^{m}e^{{-}\gamma N}\norm{\vec{f}(\xi)}_{L^{2}_{\xi}(\mathbb{R})}, 
 \end{equation*}
where $\gamma>0$ is related to  the exponential decay rate  of $V_{\ell}(x).$
\end{remark}
The following theorem provides the canonical decomposition of solutions into dispersive, root, stable, and unstable components.
\begin{theorem}\label{stablecase}
There exists $L>1$ such that if $\min_{\ell}y_{\ell}-y_{\ell+1}>L>1,$ then for any $t\geq 0,$ function $\vec{f}\in L^{2}_{x}(\mathbb{R},\mathbb{C}^{2})$ has a unique representation of the form
\begin{align}\label{princ11}
   \vec{f}(x)=& \mathcal{S}
\left(\vec{\phi}\right)(t,x)+\sum_{\ell=1}^{m}\sum_{j=1}^{\dim \ker \mathcal{H}_{\ell}^{2} } b_{j,\ell,0}\mathfrak{G}_{\ell}(\mathfrak{z}_{\ell})(t,x)\\&{+}\sum_{\ell=1}^{m}\sum_{\lambda\in \sigma_{d,\mathrm{stab}}(\mathcal{H}_{\ell})}b_{\ell,\lambda}(t)\mathfrak{G}_{\ell}(\mathfrak{v}_{\alpha_{\ell},\lambda})(t,x)\\
\\&{+}\sum_{\ell=1}^{m}b_{\ell,+}(t)e^{i\theta_{\ell}(t,x)\sigma_{3}}\alpha_{\ell}^{\frac{1}{k}}\vec{Z}_{+}\left(\alpha_{\ell}[x-v_{\ell}t-y_{\ell}]\right).
\end{align}

\end{theorem}
\begin{proof}
    See Lemma $5.1$ from \cite{dispanalysis1}, and Corollary $1.11$ from \cite{dispanalysis2}.
\end{proof}

In particular, any function $\vec{f}(t,x)\in L^{2}_{x}(\mathbb{R},\mathbb{C}^{2})$ has for any $t\geq 0$ a unique representation of the form.
\begin{align}\label{princ111}
   \vec{f}(t,x)=& \mathcal{S}
\left(\vec{\phi}(t)\right)(t,x)+\sum_{\ell=1}^{m}\sum_{j=1}^{\dim \ker \mathcal{H}_{\ell}^{2} } b_{j,\ell,0}(t)\mathfrak{G}_{\ell}(\mathfrak{z}_{\ell})(t,x)\\&{+}\sum_{\ell=1}^{m}\sum_{\lambda\in \sigma_{d,\mathrm{stab}}(\mathcal{H}_{\ell})}b_{\ell,\lambda}(t)\mathfrak{G}_{\ell}(\mathfrak{v}_{\alpha_{\ell},\lambda})(t,x)\\
\\&{+}\sum_{\ell=1}^{m}b_{\ell,+}(t)e^{i\theta_{\ell}(t,x)\sigma_{3}}\alpha_{\ell}^{\frac{1}{k}}\vec{Z}_{+}(\alpha_{\ell}[x-v_{\ell}t-y_{\ell}]),  
\end{align}
As a consequence, we can deduce  the following projections.

\begin{definition}[Projections onto the stable hyperbolic modes]\label{Hyperbolicspace}
 Let $\sigma=\{(v_{\ell},y_{\ell},\alpha_{\ell},\gamma_{\ell})\}_{\ell\in[m]}$ be a set of constant vectors in $\mathbb{R}^{2}\times \mathbb{R}^+\times \mathbb{R}.$ The projection $P_{\mathrm{stab},\ell,\sigma}(t)$ onto $\sppp \{\mathfrak{G}_{\ell}(\vec{z}_{\ell})(t,x)\vert\,\I \lambda<0,\,\vec{z}_{\ell}\in \ker[\mathcal{H}_{\ell}-\lambda \mathrm{Id}]  \}$ is defined for any $\vec{f}(t,x)\in L^{2}_{x}(\mathbb{R},\mathbb{C}^{2})$ is equal by the finite 
 \begin{equation*}
 \sum_{\lambda\in \sigma_{d,\mathrm{stab}}(\mathcal{H}_{\ell})}a_{\ell,\lambda}(t)\mathfrak{G}_{\ell}(\mathfrak{v}_{\alpha_{\ell},\lambda})(t,x)
 \end{equation*} 
satisfying \eqref{princ111}.
Finally, the  projection $P_{\mathrm{stab},\sigma}$ is defined for any $t\geq 0$ by
\begin{equation}\label{eq:Psu}
     P_{\mathrm{stab},\sigma}(t)=\bigoplus_{\ell=1}^{m} P_{\mathrm{stab},\ell,\sigma}(t).
\end{equation}
\end{definition}
\begin{definition}[Projection onto the root space]\label{rootspace}
 Let $\sigma=\{(v_{\ell},y_{\ell},\alpha_{\ell},\gamma_{\ell})\}_{\ell\in[m]}$ be a set of constant vectors in $\mathbb{R}^{2}\times \mathbb{R}^+\times \mathbb{R}.$ The map $P_{\mathrm{root},\ell,\sigma}(t):L^{2}_{x}(\mathbb{C},\mathbb{R})\to L^{2}_{x}(\mathbb{R},\mathbb{C}^{2})$ is the unique  projection onto $\sppp\{\mathfrak{G}_{\ell}(\vec{z}_{\ell})(t,x)\vert\,\vec{z}_{\ell}\in \ker\mathcal{H}_{\ell}^{2}  \}$ such that $P_{\mathrm{root},\ell,\sigma}(t)\vec{f}(t)$ for any is equal to the unique term of the form
 \begin{equation*}
      \sum_{j=1}^{\dim \ker \mathcal{H}_{\ell}^{2} } b_{j,\ell,0}(t)\mathfrak{G}_{\ell}(\mathfrak{z}_{\ell})(t,x)
 \end{equation*}
satisfying \eqref{princ111}. The projection $P_{\mathrm{root},\sigma}(t)$ is defined by
\begin{equation}\label{eq:Proot}
     P_{\mathrm{root},\sigma}(t)=\bigoplus_{\ell=1}^{m} P_{\mathrm{root},\ell,\sigma}(t).
\end{equation}
\end{definition}
\begin{definition}[Projection onto the unstable hyperbolic modes]\label{unstspace} Let $\sigma=\{(v_{\ell},y_{\ell},\alpha_{\ell},\gamma_{\ell})\}_{\ell\in[m]}$ be a set of constant vectors in $\mathbb{R}^{2}\times \mathbb{R}^+\times \mathbb{R}.$ The map $P_{\mathrm{unst},\ell,\sigma}(t):L^{2}_{x}(\mathbb{R},\mathbb{C}^{2})\to L^{2}_{x}(\mathbb{R},\mathbb{C}^{2})$ is the unique  projection onto $\sppp\{e^{i\theta_{\ell}(t,x)\sigma_{3}}\mathfrak{v}_{j,\omega_{\ell},\lambda_{\ell,n}}(x-v_{\ell}t-y_{\ell})\vert \, \I\lambda_{\ell,n}>0,\,\mathfrak{v}_{j,\omega_{\ell},\lambda_{\ell,n}}\in \ker[\mathcal{H}_{\ell}-\lambda_{\ell,n}Id]  \}$ such that, for any $\vec{f}(t,x)\in L^{2}_{x}(\mathbb{R},\mathbb{C}^{2}),$ $P_{\mathrm{unst},\ell,\sigma}(t)\vec{f}(t)$ is equal to the unique term of the form
 \begin{equation*}
    \sum_{\ell=1}^{m}a_{\ell}(t)e^{i\theta_{\ell}(t,x)\sigma_{3}}\alpha_{\ell}^{\frac{1}{k}}\vec{Z}_{+}\left(\alpha_{\ell}[x-v_{\ell}t-y_{\ell}]\right)
 \end{equation*}
   satisfying \eqref{princ111}. 
\end{definition}
From now on, we define the projection $P_{c,\sigma}(t)$ onto the range of $\mathcal{T}(t)$ to be equals to
\begin{equation}\label{Pct}
    P_{c,\sigma}(t)\vec{f}(t)=\vec{f}(t)-P_{\mathrm{stab},\sigma}(t)\vec{f}-P_{\mathrm{unst},\sigma}(t)\vec{f}(t)-P_{\mathrm{root},\sigma}(t)\vec{f}(t).
\end{equation}
Moreover, given the unique decomposition of $\vec{f}$ in \eqref{princ111}, we consider for each $j\in[m]$ that
\begin{equation}\label{pcj}
    P_{c,j,\sigma}(t)\vec{f}(t)=e^{i\left(\frac{v_{j}x}{2}-\frac{v_{j}^{2}t}{4}+\alpha^2_{j}t+\gamma_{j}\right)\sigma_{3}}\hat{G}_{\alpha_{j}}\left(
   e^{{-}it(\xi^{2}+\alpha^2_{j})\sigma_{3}}e^{{-}i\gamma_{j}\sigma_{3}} \begin{bmatrix}
e^{iy_{j}\xi}\phi_{1,j}\left(\xi+\frac{v_{j}}{2}\right)\\
       e^{iy_{j}\xi}\phi_{2,j}\left(\xi-\frac{v_{j}}{2}\right)
    \end{bmatrix}\right)(x-y_{j}-v_{j}t).
\end{equation}
Clearly, \eqref{princ111} implies that $P_{c}(t)$ is well defined for any $t\geq 0.$
\par Furthermore, using the dispersive estimates and weighted estimates for the semigroup operator $e^{it\mathcal{H}_{\ell}}$ obtained in \cite{KriegerSchlag}, the following theorem was proved in \cite{dispanalysis1,dispanalysis2}. 
\begin{theorem}\label{Decesti1}
 Let $\sigma=\{(v_{\ell},y_{\ell},\alpha_{\ell},\gamma_{\ell})\}_{\ell\in[m]}$ be a set of constant vectors in $\mathbb{R}^{2}\times \mathbb{R}^+\times \mathbb{R},$ and let
 \begin{align*}
     D_{\ell,\mathrm{mid},+}(t)
     =&
     \begin{cases}
     \frac{(v_{\ell}+v_{\ell- 1})t}{2}+\frac{y_{\ell}+y_{\ell-1}}{2}, \text{ if $\ell\neq 1,$}\\
    ={+}\infty \text{, otherwise.} 
     \end{cases},\\
     D_{\ell,\mathrm{mid},-}(t)
     =&\begin{cases}
     \frac{(v_{\ell}+v_{\ell+ 1})t}{2}+\frac{y_{\ell}+y_{\ell-1}}{2}, \text{ if $\ell\neq m,$}\\
    ={-}\infty \text{, otherwise.} 
     \end{cases}.
 \end{align*}
If $\min_{\ell}y_{\ell}-y_{\ell+1}>L$ for a large positive $L,$ the following estimates are true for constants $K>1,\,\beta>0,$ and any $t>\tau\geq 0.$
{\footnotesize  \begin{align}\nonumber
\norm{\mathcal{S}(\vec{\phi})(t,x)}_{H^{j}_{x}(\mathbb{R})}\leq & K \norm{\mathcal{S}(\vec{\phi})(\tau,x)}_{H^{j}_{x}(\mathbb{R})} \text{, for any $j\in\{0,1,2\},$}\\ \label{Q1}
\norm{\mathcal{S}(\vec{\phi})(t,x)}_{L^{\infty}_{x}(\mathbb{R})}\leq & \max_{\ell} \frac{K}{(t-\tau)^{\frac{1}{2}}}\left[\norm{\mathcal{S}(\vec{\phi})(\tau,x)}_{L^{1}_{x}(\mathbb{R})}+e^{{-}\min_{j,\ell}\frac{\alpha_{j}(y_{\ell}+v_{\ell}\tau-y_{\ell+1}-v_{\ell+1}\tau)}{2}}\norm{\mathcal{S}(\vec{\phi})(\tau,x)}_{L^{2}_{x}(\mathbb{R})}\right],\\ \label{Q2}
\norm{\frac{\mathcal{S}(\vec{\phi})(t,x)}{(1+\vert x-y_{\ell}-v_{\ell}t\vert)}}_{L^{\infty}_{x}(\mathbb{R})}\leq & \frac{K (y_{1}-y_{m}+\tau)}{(t-\tau)^{\frac{3}{2}}} \norm{\mathcal{S}(\vec{\phi})(\tau,x)}_{L^{1}_{x}(\mathbb{R})}\\ \nonumber 
&{+}\frac{K}{(t-\tau)^{\frac{3}{2}}}\max_{\ell}\norm{(1+\vert x-y_{\ell}-v_{\ell}\tau\vert )\mathcal{S}(\vec{\phi})(\tau,x)}_{L^{1}_{x}[D_{\ell,\mathrm{mid},-}(\tau),D_{\ell,\mathrm{mid},+}(\tau)]}\\ \nonumber 
&{+}\frac{e^{{-}\min_{j,\ell} \alpha_{j}[(v_{\ell}-v_{\ell+1})\tau+(y_{\ell}-y_{\ell+1})]}\norm{\mathcal{S}(\vec{\phi})(\tau,x)}_{L^{2}_{x}(\mathbb{R})}}{(t-\tau)^{\frac{3}{2}}}. 
 \end{align}}
Moreover, the constant $K>1$ depends only on the set $\min_{\ell}(v_{\ell}-v_{\ell-1}),$ and the set $\{\alpha_{\ell}\}$ and the number $m.$ 
\end{theorem}
\begin{proof}
See the proof of Theorem $6.7$ from \cite{dispanalysis1} in Section $6,$ and the proof of Theorem $3.4$ from \cite{dispanalysis2} for the case where $\min_{\ell}v_{\ell}-v_{\ell+1}>0$ is  small.
\end{proof}
\begin{remark}\label{stabel}
From the definition of $P_{\mathrm{stab},\sigma}$ and the fact that all eigenfunctions of $\mathcal{H}_{1}$ are Schwartz functions having exponential decay, we can verify from Theorem \ref{stablecase} and the definition of $P_{\mathrm{stab},\sigma}$ that there exists a constant $K>1$ depending only on $\{(v_{\ell},\alpha_{\ell})\}_{\ell\in [m]}$ and $m$ satisfying
\begin{equation*}
    \max_{q\in [2,\infty]}\norm{\mathcal{U}_\sigma(t,s)P_{\mathrm{stab},\ell,\sigma}(s)\vec{\psi}_{0}}_{W^{2,q}_{x}(\mathbb{R})}\leq K e^{{-}\vert \lambda_{\ell}\vert (t-s)}\min_{q\in\{1,2\}}\norm{\vec{\psi}_{0}}_{L^{q}_{x}(\mathbb{R})},
\end{equation*}
which is much stronger than all estimates of Theorem \ref{Decesti1}.
\end{remark}
To study the nonlinear problem,  we also establish a weighted decay estimate in the space derivative of $\mathcal{U}_\sigma(t,s)f$ as it is explained in the following theorem holds.
\begin{theorem}\label{interpolation est.}
Assume that $\min_{\ell}y_{\ell}-y_{\ell+1}>L$ for a large positive number $L.$  If $p\in (1,2),$ $\omega\in (0,1)$ and $p^{*}=\frac{p}{p-1},$ then there exists a constant $C_{p,\omega}$ depending only on $\{(v_{\ell},\alpha_{\ell})\},p$ and $\omega$ satisfying
{\footnotesize \begin{align}\label{interpweder}
\max_{\ell}\norm{\frac{\partial_{x}\mathcal{S}(\vec{\phi})(t,x)}{\langle x-v_{\ell}t-y_{\ell} \rangle^{1+\frac{p^{*}-2}{2p^{*}}+\omega}}}_{L^{2}_{x}}
\leq & \frac{C_{p,\omega}\max_{\ell}\norm{(1+\vert x-y_{\ell}-v_{\ell}s\vert)\chi_{\ell}(s,x)\langle \partial_{x}\rangle\mathcal{S}(\vec{\phi})(s,x)}_{L^{1}_{x}(\mathbb{R})}^{\frac{2-p}{p}}\norm{\mathcal{S}(\vec{\phi})(s,x)}_{H^{1}}^{\frac{2(p-1)}{p}}}{(t-s)^{\frac{3}{2}(\frac{1}{p}-\frac{1}{p^{*}})}} \\
 &{+}C_{p,\omega}\frac{(s+y_{1}-y_{m})}{(t-s)^{\frac{3}{2}(\frac{1}{p}-\frac{1}{p^{*}})}}\norm{\mathcal{S}(\vec{\phi})(s,x)}_{W^{1,1}_{x}(\mathbb{R})}^{\frac{2-p}{p}}\norm{\mathcal{S}(\vec{\phi})(s,x)}_{H^{1}}^{\frac{2(p-1)}{p}}\\
 &{+}C_{p,\omega}\frac{e^{{-}\min_{j,\ell} \alpha_{j}[(v_{\ell}-v_{\ell+1})s+(y_{\ell}-y_{\ell+1})]}}{(t-s)^{\frac{3}{2}(\frac{1}{p}-\frac{1}{p^{*}})}}\norm{\mathcal{S}(\vec{\phi})(s,x)}_{L^{2}_{x}(\mathbb{R})},
\end{align}}
where $
    \chi_{\ell}(s,x)=\chi_{\left[\frac{y_{\ell}+y_{\ell+1}+(v_{\ell}+v_{\ell+1})s}{2},\frac{y_{\ell}+y_{\ell-1}+(v_{\ell}+v_{\ell-1})s}{2}\right]}(x).
$
Moreover, there exists a positive constant $C$ depending only on $\{(v_{\ell},\alpha_{\ell})\}$ satisfying for all $t>s\geq 0$
\begin{equation}\label{derivalinftyweig}
    \norm{\partial_{x}\mathcal{S}(\vec{\phi})(t,x)}_{L^{\infty}_{x}(\mathbb{R})}\leq \frac{C\norm{\mathcal{S}(\vec{\phi})(s,x)}_{W^{1,1}_{x}(\mathbb{R})}}{(t-s)^{\frac{1}{2}}}. 
\end{equation}
\end{theorem}
\begin{proof}
    See Theorem $1.17$ from \cite{dispanalysis2} for the proof of the first inequality. The proof of inequality \eqref{derivalinftyweig} follows from Lemma $\mathrm{C.1}$ and Proposition $\mathrm{C.2}$ from \cite{dispanalysis1}, and the Minkowski inequality.
\end{proof}
\par Moreover, we will need the following estimate on the growth of the $L^{2}$ norm of $\vec{u}$ with a polynomial weight.
\begin{proposition}\label{growthweightl2}
If $\min_{\ell}y_{\ell}-y_{\ell+1}>L$ for a large positive number $L,$ then the evolution operator $\mathcal{U}_\sigma(t,s)$ satisfies for all $t\geq s\geq 0$
\begin{multline*}
    \max_{\ell}\norm{\left\vert x-v_{\ell}t-y_{\ell}\right\vert \mathcal{S}(\vec{\phi})(t,x)}_{L^{2}_{x}[D_{\ell,\mathrm{mid},-}(t),D_{\ell,\mathrm{mid},+}(t)]}\\
    \begin{aligned}
      \leq &C\left[\max_{\ell}(v_{\ell}-v_{\ell+1})\right]\langle t-s\rangle \norm{\mathcal{S}(\vec{\phi})(s,x)}_{H^{1}_{x}(\mathbb{R})}\\
      {+}&C [\max_{\ell}(y_{\ell}-y_{\ell+1})]\norm{\mathcal{S}(\vec{\phi})(s,x)}_{H^{1}_{x}(\mathbb{R})}
      \\{+}&C\norm{\left\vert x-y_{\ell}\right\vert \mathcal{S}(\vec{\phi})(s,x)}_{L^{2}_{x}[D_{\ell,\mathrm{mid},-}(0),D_{\ell,\mathrm{mid},+}(0)]},
    \end{aligned}
\end{multline*}
for a constant $C>1$ depending only on $\{(v_{\ell},\alpha_{\ell})\}_{\ell\in [m]}.$
\end{proposition}
\begin{proof}
    See  Appendix  \ref{B}.
\end{proof}

\subsection{Some technical preparations}
Since all trajectories of potentials are nonlinear, one has to replace trajectories of solitons by their linear approximations. Here, we introduce the potentials moving with given paths and their linear approximations.
\begin{definition}\label{Def1.1}
Let $\alpha_{\ell}(t),\,y_{\ell}(t),\,v_{\ell}(t)$ and $\gamma_{\ell}(t)$ be continuous functions on $t$ for each $\ell\in  [m].$ For each set $\sigma=\{(v_{\ell},y_{\ell},\alpha_{\ell},\gamma_{\ell})\}_{\ell},$ the operator $\mathcal{H}_{\ell,\sigma}(t)$ is defined by
\begin{equation*}
    \mathcal{H}_{\ell,\sigma}(t)=\begin{bmatrix}
    {-}\partial^{2}_{x}+\alpha_{\ell}(t)^{2}-(k+1)\phi^{2k}_{\alpha(t)}(x-y_{\ell}(t)) &  {-}k e^{i(\frac{v_{\ell}(t)x}{2}+\gamma_{\ell}(t))}\phi^{2k}_{\alpha(t)}(x-y_{\ell}(t))\\
    k e^{{-}i(\frac{v_{\ell}(t)x}{2}+\gamma_{\ell}(t))}\phi^{2k}_{\alpha(t)}(x-y_{\ell}(t)) & \partial^{2}_{x}-\alpha_{\ell}(t)^{2}+(k+1)\phi^{2k}_{\alpha(t)}(x-y_{\ell}(t))
    \end{bmatrix}.
\end{equation*}

\par Next, given $\sigma(t)=\{(v_{\ell}(t),y_{\ell}(t),\alpha_{\ell}(t),\gamma_{\ell}(t))\}_{\ell\in [m]}$
\begin{equation}\label{sigmaTTT}
    \sigma^{T}_{\ell}(t)\coloneqq\left\{(v_{\ell}(T),y_{\ell}(T)+v_{\ell}(T)(t-T),\alpha_{\ell}(T),\gamma_{\ell}(T)-\frac{v_{\ell}(T)^{2}(t-T)}{4}+\alpha_{\ell}(T)^{2}(t-T))\right\}_{\ell\in[m]},
\end{equation}
for any $t\geq 0.$
\par Furthermore, for any $\ell\in [m]$ and any path $\sigma=\{(v_{\ell}(t),y_{\ell}(t),\alpha_{\ell}(t),\gamma_{\ell}(t))\}_{\ell\in [m]}\in C([0,T],\mathbb{R}^{2}\times \mathbb{R}^+\times \mathbb{R})$ for some $T>0,$ we denote 
for 
\begin{align*}
    \theta_{\ell,\sigma}(t,x)=&\frac{v_{\ell}(t)x}{2}+\gamma_{\ell}(t),\\
    \theta^{T}_{\ell,\sigma}(t,x)=&\frac{v_{\ell}(T)x}{2}+\gamma_{\ell}(T)-\frac{v_{\ell}(T)^{2}(t-T)}{4}+\alpha_{\ell}(T)^{2}(t-T),
\end{align*}
that
\begin{align*}
    V_{\ell,\sigma}\coloneqq &
    \begin{bmatrix}
        {-}(k+1)\phi^{2k}_{\alpha_{\ell}(t)}(x-y_{\ell}(t)) & {-}k e^{i\left(\frac{v_{\ell}(t)x}{2}+\gamma_{\ell}(t)\right)}\phi^{2k}_{\alpha_{\ell}(t)}(x-y_{\ell}(t))\\
        k e^{{-}i\left(\frac{v_{\ell}(t)x}{2}+\gamma_{\ell}(t)\right)}\phi^{2k}_{\alpha_{\ell}(t)}(x-y_{\ell}(t)) &  (k+1)\phi^{2k}_{\alpha_{\ell}(t)}(x-y_{\ell}(t))
    \end{bmatrix},\\
    V^{T}_{\ell,\sigma}\coloneqq &
    \begin{bmatrix}
        {-}(k+1)\phi^{2k}_{\alpha_{\ell}(T)}(x-v_{\ell}(T)t-(y_{\ell}(T)-v_{\ell}(T)T)) & {-}k e^{i\theta^{T}_{\ell,\sigma}(t,x)}\phi^{2k}_{\alpha_{\ell}(t)}(x-v_{\ell}(T)t-(y_{\ell}(T)-v_{\ell}(T)T))\\
        k e^{{-}i\theta^{T}_{\ell,\sigma}(t,x)}\phi^{2k}_{\alpha_{\ell}(T)}(x-v_{\ell}(T)t-(y_{\ell}(T)-v_{\ell}(T)T)) &  {-}(k+1)\phi^{2k}_{\alpha_{\ell}(T)}(x-v_{\ell}(T)t-(y_{\ell}(T)-v_{\ell}(T)T))
    \end{bmatrix}
\end{align*}
\end{definition}
To save space when we study the time derivatives of modulation terms, we introduce the following short-hand notations.
\begin{definition}\label{lambdasigma}
Let $\sigma=\{(v_{\ell},y_{\ell},\alpha_{\ell},\gamma_{\ell})\}_{\ell\in [m]}$ be an element of $ C^{1}(\mathbb{R}_{\geq 0},\mathbb{R}^{4m}).$ For each $\ell\in [m]$ and $t\geq 0,$ $\Lambda \dot \sigma_{\ell}(t)=(\Lambda \dot v_{\ell}(t),\Lambda \dot y_{\ell}(t),\Lambda \dot \alpha_{\ell}(t),\Lambda \dot \gamma_{\ell}(t))$ is the element of $\mathbb{R}^{2}\times \mathbb{R}^+\times \mathbb{R}$ defined by
\begin{align*}
    \Lambda \dot y_{\ell}(t)\coloneqq & \dot y_{\ell}(t)-v_{\ell}(t),\\
    \Lambda \dot v_{\ell}(t)\coloneqq &\dot v_{\ell}(t),\,\,
    \Lambda \dot \alpha_{\ell}(t)\coloneqq \dot \alpha_{\ell}(t),\\
    \Lambda \dot \gamma_{\ell}(t)\coloneqq  & \dot \gamma_{\ell}(t)-\alpha_{\ell}(t)^{2}+\frac{v_{\ell}(t)^{2}}{4}+\frac{y_{\ell}(t)\dot v_{\ell}(t)}{2}.
\end{align*}
\end{definition}
\begin{remark}
    The motivation for the definition of $\Lambda \dot \sigma (t)$ is because this term is included
    in the forcing term of the  equation satisfied by the error term  around the multi-soliton and the solution $\psi(t)$ to  equation \eqref{NLS1d}.
\end{remark}
\par Finally, we will use the following technical propositions to help in the proof of the estimates of the next sections.
\begin{lemma}\label{interactt}
Let $(x)_+$ denote the positive part of $x$. For any real numbers $x_{2},x_{1}$, such that $\zeta=x_{2}-x_{1}>0$ and $\alpha,\,\beta,\,m>0$ with $\alpha\neq \beta$ the following bound holds:
\begin{equation*}
    \int_{\mathbb{R}}\vert x-x_{1}\vert ^{m} e^{-\alpha(x-x_{1})_{+}}e^{-\beta(x_{2}-x)_{+}}\lesssim_{\alpha,\beta,m} \max\left(\left(1+\zeta^{m}\right)e^{-\alpha \zeta},e^{-\beta \zeta}\right),
\end{equation*}
In particular, for any $\alpha>0$, the following bound holds
\begin{equation*}
    \int_{\mathbb{R}}\vert x-x_{1}\vert^{m} e^{-\alpha(x-x_{1})_{+}}e^{-\alpha(x_{2}-x)_{+}}\lesssim_{\alpha}\left[1+\zeta^{m+1}\right] e^{-\alpha \zeta}.
\end{equation*}
 \end{lemma}
\begin{proof}
    Elementary computations.
\end{proof}

\begin{lemma}\label{interpol}
For any $\alpha,\beta\in\mathbb{R},$ we have if $t\geq 0,$ then
\begin{equation*}
    \int_{0}^{t}\frac{1}{(1+t-s)^{\alpha}(1+s)^{\beta}}\,ds\sim
    \begin{cases}
        \max\left(\frac{1}{(1+t)^{\alpha+\beta-1}},\frac{1}{(1+t)^{\alpha}},\frac{1}{(1+t)^{\beta}}\right) \text{if $\alpha\neq 1$ and $\beta \neq 1,$}\\
       \max\left(\frac{1}{(1+t)},\frac{\ln{(1+t)}}{(1+t)^{\alpha}}\right) \text{, if $\beta=1,$}\\
       \max\left(\frac{1}{(1+t)},\frac{\ln{(1+t)}}{(1+t)^{\beta}}\right) \text{, if $\alpha=1.$}
    \end{cases}
\end{equation*}
\end{lemma}
\begin{proof}
Elementary computations.
\end{proof}
\subsection{Main ideas for the proof of Theorem \ref{Tsigma(t)}}
\par Due to the unstable nature of solitons and equations, the iteration to construct of the solution is more involved. The proof of Theorem \ref{asy} follows from an  iterative argument in a finite time interval $[0, T_n]$ to find a solution $u_n$ whose unstable mode is terminated at $T_n$ and then we pass $u_n$ to a limit.  This is  inspired by the method of \cite{KriegerSchlag}.

Compared with the single-soliton problem, one has to be cautious with unstable modes. In the setting without the large-separation condition for speeds, we can not construct solutions which asymptotically approach
unstable modes of each potential, so we do not have invariant projections for the unstable components here. When we design the iteration procedure, we have to make sure that the stabilization conditions for unstable modes involve only functions from the previous iteration. In some other problems, a Lyapunov-Schmidt argument might be involved. Here since we only run the iteration on  $[0,T_m]$, it allows us to work on an iteration scheme directly.

More precisely, in the notation of Theorem \ref{asy}, considering $\vec{u}_{0}(t,x)\coloneqq (r_{0}(x),\overline{r_{0}(x)}),\,T_{0}=\frac{1}{\delta},$ and linear trajectory
\begin{equation*}
    \sigma_{0}(t)=\{(v_{\ell}(0),y_{\ell}(0)+v(0)t,\alpha_{\ell}(0),\gamma_{\ell}(0)-\frac{v_{\ell}(0)^{2}t}{4}+\alpha_{\ell}(0)^{2}t)\}_{\ell\in[m]},
\end{equation*}
we will construct a sequence of functions $(\vec{u}_{n}(t),\sigma_{n}(t))\in L^{2}(\mathbb{R},\mathbb{C}^{2})\times \mathbb{R}^{2}\times \mathbb{R}^+\times \mathbb{R}$ satisfying for any $t\in \left[0, \frac{1}{\delta}+n\right]$ an explicit system of linear equations of the following kind
\begin{align*}
   i\partial_{t}\vec{u}_{n}(t)-\sum_{\ell=1}^{m}\left[\mathcal{H}_{\ell,n-1}(t)-\alpha_{\ell,n-1}(t)^{2}\mathfrak{p}_3\right]\vec{u}_{n}(t)=&G(\dot \sigma_{n}(t),\vec{u}_{n-1}(t),\sigma_{n-1}(t)),\\
   \langle \vec{u}_{n}(t),\mathfrak{p}_3e^{i(\frac{v_{\ell,n-1}(t)x}{2}+\gamma_{\ell,n-1}(t))\mathfrak{p}_3}\vec{z}(\alpha_{\ell,n-1}(t),x-y_{\ell,n-1}(t)) \rangle=&0 \text{, for all $\vec{z}\in \ker\mathcal{H}^{2}_{1},$}
\end{align*}
for more details see Proposition \ref{undecays} below.

\par Furthermore, in Section \ref{sectioncauchy}, it will be verified that for any $T>0$ that 
\begin{equation*}
    \lim_{n\to{+}\infty}\norm{\vec{u}_{n}-\vec{u}_{*}}_{L^{\infty}([0,T],L^{2}_{x}(\mathbb{R}))}+\lim_{n\to{+}\infty}\norm{\sigma_{n}(t)-\sigma_{*}(t)}_{L^{\infty}([0,T])}=0.
\end{equation*}
 Consequently, using the decay estimates satisfied by each $\vec{u}_{n}$ and the dispersive decay estimates obtained in \cite{dispanalysis1}, we will deduce that the limit function $\vec{u}_{*}(t)$ scatters when $t$ approaches infinity.
\begin{align*}
    \lim_{t\to {+}\infty}\norm{e^{{-}it\partial^{2}_{x}}\vec{u}_{*}(t)-f(x)}_{L^{2}_{x}(\mathbb{R})}.
\end{align*}
The properties of $\{(\vec{u}_{n},\sigma_{n})\}_{n}$ and the convergence of this sequence are well-explained in the two propositions of the two following subsubsections.

\subsubsection{Linearized equation}
First, using  equation \eqref{NLS1d}, from the ansatz 
\begin{equation*}
    \psi(t,x)=\sum_{\ell=1}^{m}e^{i(\frac{v_{\ell(t)x}}{2}+\gamma_{\ell}(t))}\phi_{\alpha_{\ell}(t)}(x-y_{\ell}(t))+u(t,x)
\end{equation*}
is a  solution of \eqref{NLS1d} for a set of $C^{1}$ functions $\{(v_{\ell},y_{\ell},\alpha_{\ell},\gamma_{\ell})\}$ if $\vec{u}(t)\coloneqq (u(t),\overline{u(t)})$ is a strong solution of the following system.
{\footnotesize \begin{multline}\label{gqq}
 i\partial_{t}\vec{u}(t,x)+\mathfrak{p}_3\partial^{2}_{x}\vec{u}(t,x)+\sum_{\ell=1}^{m}V^{T}_{\ell,\sigma}(t,x)\vec{u}(t,x)\\
 \begin{aligned}
=&\begin{bmatrix}
     {-}\vert \psi(t,x)-u(t,x) \vert^{2k}[  \psi(t,x)-u(t,x)]+\sum_{\ell}e^{i(\frac{v_{\ell(t)x}}{2}+\gamma_{\ell}(t))}\phi^{2k+1}_{\alpha_{\ell}(t)}(x-y_{\ell}(t))\\
     \vert \psi(t,x)-u(t,x) \vert^{2k}\overline{\psi(t,x)-u(t,x)}-\sum_{\ell}e^{{-}i(\frac{v_{\ell(t)x}}{2}+\gamma_{\ell}(t))}\phi^{2k+1}_{\alpha_{\ell}(t)}(x-y_{\ell}(t))
 \end{bmatrix}
\\
&{+}\sum_{\ell=1}^{m}\left[ V^{T}_{\ell,\sigma}(t,x)-V_{\ell,\sigma}(t,x)\right]\vec{u}(t,x)\\
&{+}F\left(\sum_{\ell} e^{i(\frac{v_{\ell(t)x}}{2}+\gamma_{\ell}(t))}\phi_{\alpha_{\ell}(t)}(x-y_{\ell}(t))+\vec{u}\right)-F\left(\sum_{\ell} e^{i(\frac{v_{\ell(t)x}}{2}+\gamma_{\ell}(t))}\phi_{\alpha_{\ell}(t)}(x-y_{\ell}(t))\right)\\
&{-}\sum_{\ell=1}^{m}F^{\prime}\left(^{i(\frac{v_{\ell(t)x}}{2}+\gamma_{\ell}(t))}\phi_{\alpha_{\ell}(t)}(x-y_{\ell}(t))\right)\vec{u}\\
&{+}\sum_{\ell}(\dot y_{\ell}(t)-v_{\ell}(t))\begin{bmatrix}
ie^{i(\frac{v_{\ell}(t)x}{2}+\gamma_{\ell}(t))}\partial_{x}\phi_{\alpha_{\ell}(t)}(x-y_{\ell}(t))\\
ie^{{-}i(\frac{v_{\ell}(t)x}{2}+\gamma_{\ell}(t))}\partial_{x}\phi_{\alpha_{\ell}(t)}(x-y_{\ell}(t))
 \end{bmatrix}\\
   &{+}\sum_{\ell} \dot v_{\ell}(t)\begin{bmatrix}
       \frac{(x-y_{\ell}(t))}{2}e^{i(\frac{v_{\ell}(t)x}{2}+\gamma_{\ell}(t))}\phi_{\alpha_{\ell}(t)}(x-y_{\ell}(t))\\
       {-}\frac{(x-y_{\ell}(t))}{2}e^{{-}i(\frac{v_{\ell}(t)x}{2}+\gamma_{\ell}(t))}\phi_{\alpha_{\ell}(t)}(x-y_{\ell}(t)
   \end{bmatrix}\\
&{-}\sum_{\ell}\dot \alpha_{\ell}(t)\begin{bmatrix}
ie^{i(\frac{v_{\ell}(t)x}{2}+\gamma_{\ell}(t))}\partial_{\alpha}\phi_{\alpha_{\ell}(t)}(x-y_{\ell}(t))\\
ie^{{-}i(\frac{v_{\ell}(t)x}{2}+\gamma_{\ell}(t))}\partial_{\alpha}\phi_{\alpha_{\ell}(t)}(x-y_{\ell}(t)
    \end{bmatrix}\\
    &{+}\sum_{\ell}\left(\dot \gamma_{\ell}(t)-\alpha_{\ell}(t)^{2}+\frac{v_{\ell}(t)^{2}}{4}+\frac{y_{\ell}(t)\dot v_{\ell}(t)}{2}\right)\begin{bmatrix}
    e^{i(\frac{v_{\ell}(t)x}{2}+\gamma_{\ell}(t))}\phi_{\alpha_{\ell}(t)}(x-y_{\ell}(t))\\
{-}e^{{-}i(\frac{v_{\ell}(t)x}{2}+\gamma_{\ell}(t))}\phi_{\alpha_{\ell}(t)}(x-y_{\ell}(t)
    \end{bmatrix}.
\end{aligned}
\end{multline}}
Inspired by the equation above and the formulation of $(\vec{u}_{n},\sigma_{n})$ constructed in \cite{KriegerSchlag}, we consider the following equation that will be used to construct each $(\vec{u}_{n},\sigma_{n})=(\vec{u}_{n},\{(y_{\ell,n},v_{\ell,n},\alpha_{\ell,n},\gamma_{\ell,n})\}_{\ell\in[m]})\in C(\mathbb{R}_{\geq 0},L^{2}(\mathbb{R},\mathbb{C}^{2})\times \mathbb{R}^{4m})$ from $(\vec{u}_{n-1},\sigma_{n-1})$ and from the data $\psi_{0}$ that satisfy the hypotheses of Theorem \ref{asy}.  

More precisely, we study the following iterative equation:
{\footnotesize \begin{multline}\label{unequation}
 i\partial_{t}\vec{u}_{n}(t,x)+\mathfrak{p}_3\partial^{2}_{x}\vec{u}_{n}(t,x)+\sum_{\ell=1}^{m}V_{\ell,\sigma_{n-1}}(t,x)\vec{u}_{n}(t,x)\\
 \begin{aligned}
=&\begin{bmatrix}
     {-}\vert \psi_{n-1}(t,x)-u_{n-1}(t,x) \vert^{2k}[  \psi_{n-1}(t,x)-u_{n-1}(t,x)]+\sum_{\ell}e^{i(\frac{v_{\ell,n-1}(t)x}{2}+\gamma_{\ell,n-1}(t))}\phi^{2k+1}_{\alpha_{\ell,n-1}(t)}(x-y_{\ell,n-1}(t))\\
     \vert \psi_{n-1}(t,x)-u_{n-1}(t,x) \vert^{2k}[\overline{\psi_{n-1}(t,x)-u_{n-1}(t,x)}]-\sum_{\ell}e^{{-}i(\frac{v_{\ell,n-1}(t)x}{2}+\gamma_{\ell,n-1}(t))}\phi^{2k+1}_{\alpha_{\ell,n-1}(t)}(x-y_{\ell,n-1}(t))
 \end{bmatrix}
\\
&{+}F\left(\sum_{\ell} e^{i(\frac{v_{\ell,n-1}(t)x}{2}+\gamma_{\ell,n-1}(t))}\phi_{\alpha_{\ell,n-1}(t)}(x-y_{\ell,n-1}(t))+\vec{u}_{n-1}\right)\\&{-}F\left(\sum_{\ell} e^{i(\frac{v_{\ell,n-1}(t)x}{2}+\gamma_{\ell,n-1}(t))}\phi_{\alpha_{\ell,n-q}(t)}(x-y_{\ell,n-1}(t))\right)\\
&{-}\sum_{\ell=1}^{m}F^{\prime}\left(^{i(\frac{v_{\ell,n-1}(t)x}{2}+\gamma_{\ell,n-1}(t))}\phi_{\alpha_{\ell,n-1}(t)}(x-y_{\ell,n-1}(t))\right)\vec{u}_{n-1}\\&{+}\sum_{\ell}(\dot y_{\ell,n}(t)-v_{\ell,n}(t))\begin{bmatrix}
ie^{i(\frac{v_{\ell,n-1}(t)x}{2}+\gamma_{\ell,n-1}(t))}\partial_{x}\phi_{\alpha_{\ell,n-1}(t)}(x-y_{\ell,n-1}(t))\\
ie^{{-}i(\frac{v_{\ell,n-1}(t)x}{2}+\gamma_{\ell,n-1}(t))}\partial_{x}\phi_{\alpha_{\ell,n-1}(t)}(x-y_{\ell,n-1}(t))
 \end{bmatrix}\\
   &{+}\sum_{\ell} \dot v_{\ell,n}(t)\begin{bmatrix}
       \frac{(x-y_{\ell,n-1}(t))}{2}e^{i(\frac{v_{\ell,n-1}(t)x}{2}+\gamma_{\ell}(t))}\phi_{\alpha_{\ell}(t)}(x-y_{\ell,n-1}(t))\\
       {-}\frac{(x-y_{\ell,n-1}(t))}{2}e^{{-}i(\frac{v_{\ell,n-1}(t)x}{2}+\gamma_{\ell,n-1}(t))}\phi_{\alpha_{\ell}(t)}(x-y_{\ell,n-1}(t)
   \end{bmatrix}\\
&{-}\sum_{\ell}\dot \alpha_{\ell,n}(t)\begin{bmatrix}
ie^{i(\frac{v_{\ell,n-1}(t)x}{2}+\gamma_{\ell,n-1}(t))}\partial_{\alpha}\phi_{\alpha_{\ell,n-1}(t)}(x-y_{\ell,n-1}(t))\\
ie^{{-}i(\frac{v_{\ell,n-1}(t)x}{2}+\gamma_{\ell,n-1}(t))}\partial_{\alpha}\phi_{\alpha_{\ell,n-1}(t)}(x-y_{\ell,n-1}(t)
    \end{bmatrix}\\
    &{+}\sum_{\ell}\left(\dot \gamma_{\ell,n}(t)-\alpha_{\ell,n}(t)^{2}+\frac{v_{\ell,n}(t)^{2}}{4}+\frac{y_{\ell,n}(t)\dot v_{\ell,n}(t)}{2}\right)\begin{bmatrix}
    e^{i(\frac{v_{\ell,n-1}(t)x}{2}+\gamma_{\ell,n-1}(t))}\phi_{\alpha_{\ell,n-1}(t)}(x-y_{\ell,n-1}(t))\\
{-}e^{{-}i(\frac{v_{\ell,n-1}(t)x}{2}+\gamma_{\ell,n-1}(t))}\phi_{\alpha_{\ell,n-1}(t)}(x-y_{\ell,n-1}(t))
    \end{bmatrix}\\
    =:& G(t,\sigma_{n}(t),\sigma_{n-1}(t),\vec{u}_{n-1}).
\end{aligned}
\end{multline}}
In particular, the first term on the right-hand side of  equation \eqref{unequation} is simpler to estimate its Sobolev norms, since they depend only on the solitons $\phi_{\alpha_{\ell,n-1}(t)},$ and on the modulation parameters $\{(v_{\ell,n-1},y_{\ell,n-1},\alpha_{\ell,n-1},\gamma_{\ell,n-1})\}_{\ell\in [m]}.$ More precisely: 
\begin{proposition}\label{multisolitonsinteractionsize}
Let $F$ be the function defined in \eqref{Fdefinition} having $k>2.$ If 
\begin{equation*}
\sigma_{n-1}(t)=\{(v_{\ell,n-1}(t),y_{\ell,n-1}(t),\alpha_{\ell,n-1}(t),\gamma_{\ell,n-1}(t))\}_{\ell\in[m]}    
\end{equation*}
satisfies the hypothesis $\mathrm{(H1)}$, then the following estimates holds
\begin{multline*}
\max_{q\in \{1,2\}j\in[m]}\Bigg\vert \Bigg\vert\langle x-y_{j,n-1}(t)\rangle^{2} \Bigg[F\left(\sum_{\ell=1}^{m}e^{i\mathfrak{p}_3\theta_{\ell,\sigma_{n-1}(t)}(t,x)}\phi_{\alpha_{\ell,n-1}(t)}(x-y_{\ell,n-1}(t))\right)\\{-}\sum_{\ell=1}^{m}  F(e^{i\mathfrak{p}_3\theta_{\ell,\sigma_{n-1}(t)}(t,x)}\phi_{\alpha_{\ell,n-1}(t)}(x-y_{\ell,n-1}(t)))\Bigg]\Bigg\vert\Bigg\vert_{L^{q}_{x}(\mathbb{R})}\\
\leq C e^{{-}\frac{99}{100}\min_{j,\ell}\alpha_{j,n-1}(t)[y_{\ell,n-1}(t)-y_{\ell+1,n-1}(t)]}.
\end{multline*}
Furthermore, the following estimate holds when $k>2.$
\begin{multline}\label{estiofF-Fderiv}
\max_{h\in\{0,1\},q\in \{1,2\},j\in[m]}\Bigg\vert \Bigg\vert\langle x-y_{j,n-1}(t)\rangle^{2} \frac{\partial^{h}}{\partial x^{h}}\Bigg[F^{\prime}\left(\sum_{\ell=1}^{m}e^{i\mathfrak{p}_3\theta_{\ell,\sigma_{n-1}(t)}(t,x)}\phi_{\alpha_{\ell,n-1}(t)}(x-y_{\ell,n-1}(t))\right)\\{-}\sum_{\ell=1}^{m}  F^{\prime}(e^{i\mathfrak{p}_3\theta_{\ell,\sigma_{n-1}(t)}(t,x)}\phi_{\alpha_{\ell,n-1}(t)}(x-y_{\ell,n-1}(t)))\Bigg]\Bigg\vert\Bigg\vert_{L^{q}_{x}(\mathbb{R})}\\
\leq C e^{{-}\frac{99}{100}\min_{j,\ell}\alpha_{j,n-1}(t)[y_{\ell,n-1}(t)-y_{\ell+1,n-1}(t)]}.
\end{multline}
\end{proposition}
\begin{proof}
First, using the fundamental theorem of calculus and the fact that $F(0)=0,\,F^{\prime}(0)=0,\,F^{\prime \prime}(0)=0,$ and $F\in C^{2},$ we can verify that
\begin{align}\label{idF-sumF}
    \left\vert F(\sum_{\ell=1}^{m}\vec{f}_{\ell}(x))-\sum_{\ell}F(\vec{f}_{\ell}(x))\right\vert
=&\left\vert \sum_{\ell=1}^{m}\int_{0}^{1}\Bigg(F^{\prime}(\theta [\sum_{j=1}^{m}\vec{f}_{j}(x)])-F^{\prime}(\theta f_{\ell}(x))\Bigg)\vec{f}_{\ell}(x)\,d\theta\right\vert\\ \nonumber
\leq & C(m)\max_{j,\ell,j\neq \ell}\vert [\vec{f}_{j}(x) \vert ^{2}] \vec{f}_{\ell}(x)\vert 
\end{align}
 for a constant $C(m)>1$ depending only on $m$.
In particular, for
\begin{equation*}
    \vec{f}_{\ell}(t,x)=e^{i\mathfrak{p}_3\theta_{\ell,\sigma_{n-1}}(t,x)}\phi_{\alpha_{\ell,n-1}(t)}(x-y_{\ell,n-1}(t)),
\end{equation*}
we can verify from the estimate above using Lemma \ref{interactt} and the definition of $\phi_{\alpha}$ in \eqref{gr} that  
\begin{equation*}
\max_{q\in\{1,2\}}\norm{F(\sum_{\ell=1}^{m}\vec{f}_{\ell}(t,x))-\sum_{\ell}F(\vec{f}_{\ell}(t,x))}_{L^{q}_{x}(\mathbb{R})}\leq C(v,\alpha)e^{{-}\min_{j,\ell}\alpha_{j,n-1}(t)[y_{\ell,n-1}(t)-y_{\ell+1,n-1}(t)]}, 
\end{equation*}
for some constant $C(v,\alpha)>1$ depending only on $\{(v_{\ell}(0),\alpha_{\ell}(0))\}_{\ell\in [m]}.$
\par Furthermore, using the fact that $\sigma_{n-1}$ satisfies the hypothesis $\mathrm{(H2)}$ and Lemma \ref{interactt} again, we can verify that
\begin{multline*}
    \max_{h\in[m],\,q\in\{1,2\}}\norm{\langle x-y_{h,n-1}(t) \rangle^{2} F(\sum_{\ell=1}^{m}\vec{f}_{\ell}(t,x))-\sum_{\ell}F(\vec{f}_{\ell}(t,x))}_{L^{q}_{x}(\mathbb{R})}\\
    \leq C(v,\alpha)e^{{-}\frac{99}{100}\min_{j,\ell}\alpha_{j,n-1}(t)[y_{\ell,n-1}(t)-y_{\ell+1,n-1}(t)]}. 
\end{multline*}
\par The proof of inequality \eqref{estiofF-Fderiv} is completely similar.
\end{proof}

\par For \emph{new modulation parameters}, we consider that $\sigma_{n}$ is defined to adjust the forcing term in \eqref{unequation} so that the solution $\vec{u}_{n}$ satisfies the following orthogonality conditions.
\begin{equation}\label{odesigma}
    \langle \vec{u}_{n}(t,x),\mathfrak{p}_3 e^{i\mathfrak{p}_3(\theta_{\ell,n-1}(t,x))}z(\alpha_{\ell,n-1}(t),x-y_{\ell,n-1}(t))\rangle=0, \text{ for any $z\in \ker \mathcal{H}^{2}_{1}.$}
\end{equation}
Next, to simplify the notations of the argument in the paper, we set the following definition.
\begin{definition}\label{cutlinearpath}
Let $n-1\geq 0.$ using \eqref{sigmaTTT} with $T=T_{n-1}$ and $\alpha (t)=\alpha_{n-1}(t)$, we get
 \begin{equation*}
    \sigma^{T_{n-1}}_{n-1}(t):=\{\sigma^{T_{n-1}}_{\ell,n-1}\}_{\ell\in[m]} \text{, for any $t\geq 0.$}
 \end{equation*}
\par Next, the function $\chi_{\ell,n-1}$ is defined by
 \begin{align*}
     \chi_{\ell,n-1}(\tau,x)=&\chi_{\left[\frac{y^{T_{n-1}}_{\ell,n-1}(\tau)+y_{\ell+1,n-1}(\tau)}{2},\frac {y^{T_{n-1}}_{\ell,n-1}(\tau)+y_{\ell-1,n-1}(\tau)}{2}\right]}(x) \text{, if $\ell\neq 1 $ and $\ell\neq m,$}\\
     \chi_{1,n-1}(\tau,x)=&\chi_{\left(\frac{y^{T_{n-1}}_{1,n-1}(\tau)+y^{T_{n-1}}_{2,n-1}(\tau)}{2},{+}\infty\right)}(x),\\
     \chi_{m,n-1}(\tau,x)=&\chi_{\left({-}\infty,\frac{y^{T_{n-1}}_{m,n-1}(\tau)+y^{T_{n-1}}_{m-1,n-1}(\tau)}{2}\right)}(x). 
 \end{align*}
\end{definition}

\par The sequence $\{(\vec{u}_{n},\sigma_{n})\}$ will be chosen to satisfy the following proposition. 
\begin{proposition}\label{undecays}
Assume that  hypotheses $\mathrm{(H1)}$  and $\mathrm{(H2)}$ are true. Let $T_{n}=\frac{1}{\delta}+n$ for any $n\geq 0,$
and 
\begin{align*}
    \sigma_{0}(t)=&\{(y_{\ell}(0)+v_{\ell}(0)t,v_{\ell}(0),\alpha_{\ell}(0),\gamma_{\ell}(0)-\frac{v_{\ell}(0)^{2}t}{4}+\alpha_{\ell}(0)^{2}t)\}_{\ell\in [m]},\\
    \vec{u}_{0}(t,x)=&
    \begin{cases}
       \mathcal{U}_{\sigma_0}(t,0)[\vec{r}_{0}-\sum_{\ell=1}^{m}P_{\mathrm{unst},\ell,\sigma_{0}}(0)\vec{r}_{0}(x)] \text{, if $t\in [0,\frac{1}{\delta}],$}\\
        0 \text{, otherwise.}
    \end{cases}
\end{align*}
Suppose $r_{0}\in H^{2}_{x}(\mathbb{R})\cap \langle x \rangle L^{1}_{x}(\mathbb{R})\cap \Sigma\cap (\mathfrak{p}_3\bigoplus_{\ell=1}^{m}\Ra P_{d,\alpha_\ell,+})^{\perp}$ and
\begin{equation}\label{r0cond}
    \norm{r_{0}(x)}_{H^{2}_{x}(\mathbb{R})}+\norm{\langle x\rangle r_{0}(x)}_{L^{1}_{x}(\mathbb{R})}\leq \delta^{2}\ll 1.
\end{equation}
If $\min_{\ell}y_{\ell}(0)-y_{\ell+1}(0)$ satisfies the separation condition \eqref{eq:sepcenter}  depending on the set $\{(\alpha_{\ell}(0))_{\ell\in [m]},\min_{h}v_{h}(0)-v_{h+1}(0)\},$ and 
\begin{equation*}
0<\delta<\delta_0\Big( \max_{\ell}(\vert v_{\ell} \vert), \vert y_1- y_{m} \vert,\min_{\ell}y_{\ell}-y_{\ell+1}\Big)\ll 1,    
\end{equation*}
then there exist a unique sequence $(\vec{u}_{n},\sigma_{n})_{n\geq 1}$
satisfying the differential equations \eqref{unequation} and orthogonality conditions \eqref{odesigma} for any $t \in [0,T_{n}],$
and   the stabilization condition
\begin{equation*}
    P_{\mathrm{unst},\sigma^{T_{n-1}}_{n-1}}(T_{n})(\vec{u}_{n}(T_{n}))=0.
\end{equation*}
After $T_n$, we set
\begin{equation*}
    \vec{u}_{n}(t)\equiv 0, \, \Lambda \dot\sigma_{n}(t)\equiv 0 \text{ when $t>T_{n},$}
\end{equation*}
Moreover, $\vec{u}_{n}(0,x)$ has a unique representation when $t=0$ of the form
\begin{align}\label{un(0)formula}
    \vec{u}_{n}(0)=&r_{0}(x)+\sum_{\ell=1}^m h_{\ell,n}(0)e^{i(\frac{v_{\ell,n-1}(0)x}{2}+\gamma_{\ell,n-1}(0))\mathfrak{p}_3}\Vec{Z}_{+}(\alpha_{\ell,n-1}(T_{n-1}),x-y_{\ell,n-1}(0))\\ \nonumber
    &{+}\sum_{\ell=1}^{m}e^{i(\frac{v_{\ell,n-1}(0)x}{2}+\gamma_{\ell,n-1}(0))\mathfrak{p}_3}\Vec{E}_{\ell,n}(\alpha_{\ell,n-1}(T_{n-1}),x-y_{\ell,n-1}(0)),
\end{align}
 such that
\begin{equation*}
    \langle \vec{u}_{n}(t),\mathfrak{p}_3\vec{E}(\alpha_{\ell,n-1}(t),x-y_{\ell,n-1}(t))\rangle=0 \text{, for all $\vec{E}\in \ker \mathcal{H}_{1}^{2},$}
\end{equation*}
for some constants $h_{\ell,n}(0)$, $\vec{Z}_{+}\in \ker \left(\mathcal{H}_{1}-i\lambda_{0}\mathrm{Id}\right)$ is the function defined in \eqref{z11} and some functions $\vec{E}_{\ell}\in\ker\mathcal{H}_{1}^2.$ 
Furthermore, for  $p\in (1,2)$ close enough to $1,$ $\omega\in (0,1)$ a small constant, and $p^{*}=\frac{p}{p-1},$  the following estimates are true for any $n\geq 0,$ and $t\geq 0.$ 
\begin{align}\label{decay1}
    (1+t)^{\frac{1}{2}+\frac{3}{4}+\frac{3}{2}\left(1-\frac{2-p}{p}\right)}\max_{\ell}\norm{\chi_{\ell}^{n-1}(t,x)\frac{\partial_{x}\vec{u}_{n}(t,x)}{\langle x-y^{T_{n-1}}_{\ell,n-1}(t) \rangle^{1+\frac{p^{*}-2}{2p^{*}}+\omega}}}_{L^{2}_{x}(\mathbb{R})}\leq &\delta_{0},\\ \label{decay2}
    \norm{\vec{u}_{n}(t)}_{H^{1}_{x}(\mathbb{R})}\leq &\delta_{0},\\ \label{decay3} \max_{\ell}\frac{\norm{\chi_{\ell,n-1}(t,x)\vert x-y^{T_{n-1}}_{\ell,n-1}(t)\vert\vec{u}_{n}(t)}_{L^{2}_{x}(\mathbb{R})}}{[\max_{\ell}\vert v_{\ell}\vert+1](1+t)}\leq \delta,\, \quad \left[(1+t)^{\frac{1}{2}}\right]\norm{\vec{u}_{n}(t)}_{L^{\infty}_{x}(\mathbb{R})}\leq &\delta_{0},\\ \label{decay4}
    \left[(1+t)^{\frac{1}{2}+\frac{3}{4}+\frac{3}{2}\left(1-\frac{2-p}{p}\right)}\right]\max_{\ell}\norm{\frac{\chi_{\ell,n-1}(t,x)\vec{u}_{n}(t,x)}{(1+\vert x-y^{T_{n-1}}_{\ell,n-1}(t)\vert)^{\frac{3}{2}+\omega}}}_{L^{2}_{x}(\mathbb{R})}\leq &\delta_{0},\\ \label{decay5}
    (1+t)^{\frac{1}{2}+\epsilon}\max_{h\in\{\mathrm{root},\mathrm{unst}\}}\vert P_{h,\sigma^{T_{n-1}}_{n-1}}(t)\vec{u}_n(t)\vert\leq  \delta_{0} 
      \end{align}
for
\begin{equation}\label{eq:epsilon}
    \epsilon=\frac{3}{4}+\frac{3}{2}\left(2-\frac{2}{p}\right).
\end{equation} The modulation parameters satisfy $ (y_{\ell,n}(0),v_{\ell,n}(0),\gamma_{\ell,n}(0),\alpha_{\ell,n}(0))=(y_{\ell}(0),v_{\ell}(0),\gamma_{\ell}(0),\alpha_{\ell}(0)),$ and 
\begin{align}\label{odes}
\max_{\ell}\vert \dot y_{\ell,n}(t)-v_{\ell,n}(t)\vert+
\vert\dot \alpha_{\ell,n}(t)\vert +
|\dot v_{\ell,n}(t)|\lesssim \frac{\delta_{0}}{(1+t)^{1+2\epsilon}},\\ \nonumber
\left\vert\dot \gamma_{\ell,n}(t)-\alpha_{\ell,n}(t)^{2}+\frac{v_{\ell,n}(t)^{2}}{4}+\frac{y_{\ell,n}(t)\dot v_{\ell,n}(t)}{2}\right\vert\lesssim &\frac{\delta_{0}}{(1+t)^{1+2\epsilon}}.
\end{align}
\end{proposition}
From now on, to simplify the notation, we consider the following.
\begin{notation}
In notation of Proposition \ref{undecays}, for any $n\in\mathbb{N},$ and $h\in\{\mathrm{stab},\mathrm{unst},\mathrm{root}\}$ we denote $P_{h,\ell,n}$ to be the unique  projections satisfying for all $\vec{f}\in L^{2}_{x}(\mathbb{R},\mathbb{C}^{2})$
\begin{align*}
    P_{h,\ell,n}(t)\vec{f}(x)=&P_{h,\ell,\sigma^{T_{n}}_{n}}(t)\vec{f}(x) \text{, for all $\vec{f}\in L^{2}_{x}(\mathbb{R},\mathbb{C}^{2}),$}\\
    P_{h,n}(t)\vec{f}(x)=&\sum_{\ell=1}^{m}P_{h,\ell,\sigma^{T_{n}}_{n}}(t)\vec{f}(x) \text{, for all $\vec{f}\in L^{2}_{x}(\mathbb{R},\mathbb{C}^{2}).$}
\end{align*}
We denote the projection $P_{c,n}$ by the unique  projection satisfying
\begin{equation*}
    P_{c,n}(t)\vec{f}(x)=P_{c,\sigma^{T_{n}}_{n}}(t)\vec{f}(x),
\end{equation*}
for all $\vec{f}\in L^{2}_{x}(\mathbb{R},\mathbb{C}^{2}).$
\end{notation}
In particular, Proposition \ref{undecays} implies the following estimate on the difference of $V_{\ell,\sigma_{n}}$ and $V^{T_{n}}_{\ell,\sigma_{n}}.$
\begin{lemma}\label{dinftydt}
Assume that $\sigma_{n}$ satisfies \eqref{odes}, and
let $\alpha_{1}>0$ be any number in $(0,1),$ and
\begin{align}\label{DnT}
D^{T_{n}}_{\ell,n}=&y_{\ell,n}(T_{n})-v_{\ell,n}(T_{n})T_{n},\\ \label{gammanT}
\gamma^{T_{n}}_{\ell,n}(t)= &\gamma_{\ell,n}(T_{n})+\frac{v_{\ell,n}(T_{n})^{2}t}{4}-\alpha_{\ell,n}(T_{n})^{2}t.
\end{align}
There exist a constants $K(\alpha,v,q,\omega)$ depending only on $\{(v_{\ell}(0),\alpha_{\ell}(0))\}_{\ell},\,q\in [1,\infty],$ and $\omega\in (0,1)$ satisfying for any $t\in [0,T_{n+1}]$ the following inequality
\begin{align}\label{Diffesti}
   \max_{j\in\{0,1,2\}}\norm{\left<x-v_{\ell,n}(T_{n})t-D^{T_{n}}_{\ell,n}\right>^{\frac{3}{2}+\omega}\frac{\partial^{j}}{\partial x^{j}}[V_{\ell,\sigma_{n}}(t,x)-V^{T_{n}}_{\ell,\sigma_{n}}(t,x)]}_{L^{q}_{x}}
  \leq \frac{ K(\alpha,v,q,\omega)\delta_{0}}{(1+t)^{2\epsilon-1}}
\end{align}
where $\epsilon$ is given by \eqref{eq:epsilon}.
\end{lemma}
\begin{proof}
    \par First, from the estimates \eqref{odes} and the fact that $T_{n+1}=T_{n}+1,$ we can verify using the fundamental theorem of calculus for any $t\in [0,T_{n}]$ that there exists a constant $C>1$ satisfying
      \begin{align}\label{dyn11}
          \left\vert y_{\ell,n}(t)-v_{\ell,n}(T_{n})t-D^{T_{n}}_{\ell,n}\right\vert\leq & \frac{C\delta_{0}}{(1+t)^{2\epsilon-1}},\\ \label{dyn12}
          \vert \alpha_{\ell,n}(t)-\alpha_{\ell,n}(T_{n})\vert+\left\vert v_{\ell,n}(t)-v_{\ell,n}(T_{n})\right\vert\leq & \frac{C\delta_{0}}{(1+t)^{2\epsilon}},
          \end{align}
    for any $t\in [0,T_{n+1}].$ Moreover, using Definition \ref{Def1.1}, we can verify using the fundamental theorem of calculus that
    \begin{align}\label{dyn13}
      \vert \theta_{\ell,\sigma_{n}}(t,x)-\theta^{T_{n}}_{\ell,\sigma_{n}}(t,x) \vert \leq &\left\vert \frac{\left[v_{\ell,n}(t)-v_{\ell,n}(T_{n})\right]( x-y_{\ell,n}(t))}{2} \right\vert\\ \nonumber
      &{+}\frac{\vert y_{\ell,n}(t) \vert \vert v_{\ell,n}(t)-v_{\ell,n}(T_{n})\vert}{2}\\ \nonumber
      &{+}\left\vert \gamma_{\ell,n}(t)-\gamma_{\ell,n}(T_{n})+\frac{v_{\ell,n}(T_{n})^{2}(t-T_{n})}{4}-\alpha_{\ell,n}(T_{n})^{2}(t-T_{n}) \right\vert\\ \nonumber
      \lesssim & \langle t \rangle\left[\int_{t}^{T_{n}}\frac{\delta_{0}}{(1+s)^{1+2\epsilon}}\,ds\right]+\vert x-y_{\ell,n}(t)\vert\int_{t}^{T_{n}}\frac{\delta_{0}}{(1+s)^{1+2\epsilon}}\,ds
      \\ \nonumber
      \lesssim & \frac{\delta_{0}}{(1+t)^{2\epsilon-1}}+\frac{\delta_{0}\vert x-y_{\ell,n}(t) \vert}{(1+t)^{2\epsilon}},
    \end{align}
    for all $t\in [0,T_{n}].$
    \par Next, from the Definition \ref{Def1.1}, we can verify that $V^{T}_{\ell,\sigma}(t,x)$ and $V_{\ell,\sigma}(t,x)$ are Schwartz functions having all of its derivatives decaying exponentially. 
    \par Therefore, 
    using the estimates \eqref{dyn11}, \eqref{dyn12} and \eqref{dyn13}, we can obtain the result of
    Lemma \ref{dinftydt} as an application of the fundamental theorem of calculus.   
\end{proof}
\subsubsection{Convergence of the sequence $(\vec{u}_{n},\sigma_{n})$}\label{uncauchysection}
Let $\{(\vec{u}_{n},\sigma_{n})\}$ be a sequence satisfying Proposition \ref{undecays}. The function $\vec{u}(t)$ satisfying Theorem \ref{asy} will be the limit of $\vec{u}_{n}(t)$ on $L^{2}_{x}(\mathbb{R})$ for all $t\geq 0.$
\par To study the convergence of the sequence, we consider the following norm applied to the subspace $C^{1}(\mathbb{R}_{\geq 0},H^{2}_{x}(\mathbb{R})\times  \mathbb{R}^{4m}).$
\begin{definition}\label{Yndef}
 The norm $\norm{\Diamond}_{Y_{n}}$ is defined for any element $(\vec{u},\sigma)$ of $C^{1}(\mathbb{R}_{\geq 0},H^{2}_{x}(\mathbb{R})\times  \mathbb{R}^{4m})$ by
 \begin{align*}
\norm{(u,\sigma)}_{Y_{n}}=&\max_{t\in  [0,T_{n}]}\langle t \rangle^{{-}1}\norm{\vec{u}(t,x)}_{L^{2}_{x}(\mathbb{R})}\\
     &{+}\max_{t\in  [0,T_{n}]}\langle t \rangle^{1+\frac{\epsilon}{2}-\frac{3}{8}}\max_{\ell}\left\vert \Lambda \dot\sigma_{\ell}(t)\right\vert\\
     &{+}\max_{t\in  [0,T_{n}]}\langle t \rangle^{{-}\frac{1}{4}}\norm{\frac{\chi_{\ell,n-1}(t)\vec{u}(t,x)}{\langle x-y^{T_{n-1}}_{\ell,n-1}(t) \rangle}}_{L^{\infty}_{x}(\mathbb{R})},
 \end{align*}
where $\Lambda\dot \sigma_{\ell}(t)$ is the function defined in Definition \ref{lambdasigma}.
\end{definition}
The main result of \S \ref{uncauchysection} is the following proposition.
\begin{proposition}\label{propun-un-1}
Let $\{(\vec{u}_{n},\sigma_{n})\}$ be the sequence defined in Proposition \ref{undecays}, $\delta_{0}$ defined in \eqref{deltachoice}, and $T_{n}=\frac{1}{\delta_{0}}+n.$ There exists a constant $C>1$ independent of $\delta_{0}\in (0,1)$ such that the following inequality is true for all $n\in\mathbb{N},\,n\geq2.$
\begin{equation}
    \norm{(\vec{u}_{n}-\vec{u}_{n-1},\sigma_{n}-\sigma_{n-1})}_{Y_{n}}\leq C\delta_{0} \norm{(\vec{u}_{n-1}-\vec{u}_{n-2},\sigma_{n-1}-\sigma_{n-2})}_{Y_{n-1}}+\frac{C}{T_{n}^{\frac{1}{2}+\epsilon}}.
\end{equation}
\end{proposition}
\begin{remark}
Proposition \ref{propun-un-1} is inspired by  the Proposition $4.5$ from the article \cite{KriegerSchlag}  by Krieger and Schlag. The main difference from the proposition in \cite{KriegerSchlag} and Proposition \ref{propun-un-1} is the choice for the norm $\norm{\Diamond}_{Y_{n}},$ which is weaker  compared to the one in Proposition $4.5$ of \cite{KriegerSchlag}. The main motivation for our choice on $\norm{\Diamond}_{Y_{n}}$ is the decay estimates satisfied by $\mathcal{U}_\sigma(t,s)$ in Theorem \ref{Decesti1} that are weaker compared to the evolution of the semigroup $e^{it\mathcal{H}_{1}}$ associated to a single stationary potential.
\end{remark}
\subsubsection{Proof of Theorem \ref{Tsigma(t)} using Propositions \ref{undecays} and \ref{propun-un-1}}\label{teo1.2}
\par First, Proposition \ref{undecays} implies the existence of a sequence $\{(u_{n},\sigma_{n})\}_{n\in\mathbb{N}}$ that satisfies all the inequalities \eqref{decay1}- \eqref{odes}, and all functions $\sigma_{n}(t)$ satisfy the same initial condition below.
\begin{equation}\label{initial00}
    \sigma_{n}(0)=\{(v_{\ell}(0),y_{\ell}(0),\alpha_{\ell}(0),\gamma_{\ell}(0))\}_{\ell\in [m]} \text{ for all $n\in\mathbb{N}.$}
\end{equation}
In particular, $\sigma_{1}(t)$ satisfies the decay \eqref{odes} for all $t\in [0,\frac{1}{\delta_{0}}+1].$
\par Moreover, since the path $\sigma_{0}$ chosen in Proposition \ref{undecays} is linear, the following estimate holds.  
\begin{equation}\label{Lambdasigma000}
\left\vert \Lambda \dot\sigma_{0}(t) \right\vert\equiv 0 \text{, for all $t\geq 0.$}
\end{equation}
Therefore, using \eqref{initial00}, estimate \eqref{odes} satisfied by $\sigma_{1}(t),$ and the fundamental theorem of calculus, we can verify for all $t\geq 0$ that
\begin{align}
 \max_{\ell}\vert y_{\ell,1}(t)-y_{\ell,0}(t) \vert   =\max_{\ell}\vert y_{\ell,1}(t)-y_{\ell}(0)-v_{\ell}(0)t \vert\leq &\max_{\ell}\int_{0}^{t}\vert v_{\ell,1}(s)-v_{\ell}(0) \vert\,ds\\
    \leq &
\max_{\ell}\int_{0}^{t}\int_{0}^{s}\vert \dot v_{\ell,1}(s_{1})\vert\,ds_{1}\,ds\\
    \leq &\max_{\ell}\int_{0}^{t}\int_{0}^{s}\vert \dot v_{\ell,1}(s_{1})\vert\,ds_{1}\,ds\\
    \lesssim &\delta_{0} t.
\end{align}
In particular, $
    \max_{\ell\in[m],t\in [0,\frac{1}{\delta}+1]}\left\vert y_{\ell,1}(t) -y_{\ell,0}(t)\right\vert\leq 2.
$
Consequently, we can verify that
\begin{equation}\label{m0}
    \max_{\ell\in[m],t\in [0,\frac{1}{\delta}+1]}\left\vert y_{\ell,1}(t) -y_{\ell,0}(t)\right\vert\leq 2.
\end{equation}
 Moreover, using Corollary \ref{diffprop2}, we can prove similarly to \eqref{m0} that
\begin{align}\label{m2}
     \max_{\ell\in[m],t\in [0,\frac{1}{\delta}+n_{1}]}\left\vert y_{\ell,n_{1}}(t) -y_{\ell,n_{1}-1}(t)\right\vert\leq &t\max_{\ell\in [m],s\in [0,t]}\langle s\rangle^{1+\frac{\epsilon}{2}-\frac{3}{8}}\left\vert \Lambda \dot\sigma_{n_{1}+1,\ell}(s)-  \Lambda \dot\sigma_{n_{1},\ell}(s) \right\vert\\
     \leq & t\norm{(\vec{u}_{n_{1}}-\vec{u}_{n_{1}-1},\sigma_{n_{1}}-\sigma_{n_{1}-1})}_{Y_{n_{1}}}.
\end{align}
\par Next, Proposition \ref{propun-un-1} implies that the quantity $A_{n}=\norm{(\vec{u}_{n}-\vec{u}_{n-1},\sigma_{n}-\sigma_{n-1})}_{Y_{n}}$ satisfies the following recursive estimate
\begin{equation}\label{Ansequence}
    A_{n}\leq C\delta_{0} A_{n-1}+\frac{C}{T_{n}^{\frac{1}{2}+\epsilon}},
\end{equation}
for $C>1$ a uniform constant independent of $\delta,$ and $\delta\in (0,1)$ small enough. \par Furthermore, from the definition of $(\vec{u}_{0},\sigma_{0})$ in Proposition \ref{undecays}, we can verify the following estimates for all $t\geq 0.$ 

\begin{align*}
 \norm{\vec{u}_{1}(t)-\vec{u}_{0}(t)}_{L^{2}_{x}(\mathbb{R})}\leq\norm{\vec{u}_{1}(t)}_{L^{2}_{x}(\mathbb{R})}+\norm{\vec{u}_{0}(t)}_{L^{2}_{x}(\mathbb{R})}\leq & 2\delta_{0},\\
 \norm{\frac{\chi_{\ell,0}(t)(\vec{u}_{1}(t)-\vec{u}_{0}(t))}{\langle x-y_{\ell}(0)-v_{\ell}(0)t\rangle}}_{L^{\infty}_{x}(\mathbb{R})}\lesssim \norm{\vec{u}_{1}(t)}_{H^{1}_{x}(\mathbb{R})}+\norm{\vec{u}_{0}(t)}_{H^{1}_{x}(\mathbb{R})}\lesssim & \delta_{0},\\
 \langle t\rangle^{1+\frac{\epsilon}{2}-\frac{3}{8}}\left\vert \Lambda\dot\sigma_{1}(t)-\Lambda\dot\sigma_{0}(t) \right\vert=\langle t\rangle^{1+\frac{\epsilon}{2}-\frac{3}{8}}\left\vert \Lambda\dot\sigma_{1}(t) \right\vert\lesssim \frac{\delta_{0}}{(1+t)^{\frac{3\epsilon}{2}-\frac{3}{8}}}\ll & \delta_{0}.
\end{align*}
Therefore, we can deduce from Definition \ref{Yndef} that there is a uniform constant $K>1$ satisfying 
\begin{equation}\label{A1}
  A_{1}=\norm{(\vec{u}_{1}-\vec{u}_{0},\sigma_{1}-\sigma_{0})}_{Y_{1}}\leq K\delta_{0}.  
\end{equation}
We recall that $\delta_{0}\ll 1$ is defined in \eqref{deltachoice}.
\par Consequently, applying the method of generating functions on the inequalities \eqref{Ansequence} and \eqref{A1}, we can check that $A_{n}$ satisfies the following estimate for all $n\in\mathbb{N}_{\geq 1}.$
\begin{align*}
    A_{n}\leq & A_{1}(C\delta_{0})^{n-1}+C^{n-1}\sum_{j=1}^{n} \frac{\delta_{0}^{j}}{T_{n-j}^{\frac{1}{2}+\epsilon}}\\
   \leq & KC^{n-1}\delta_{0}^{n-1}+C^{n-1}\sum_{j=1}^{n} \frac{\delta_{0}^{j}}{T_{n-j}^{\frac{1}{2}+\epsilon}}\\
   \leq &KC^{n-1}\delta_{0}^{n-1}+C^{n-1}\left(\frac{n\delta_{0}^{\frac{n}{2}+\frac{1}{2}+\epsilon}}{2}+\frac{\delta_{0}}{(1-\delta)\left[\frac{1}{\delta_{0}}+\frac{n}{2}\right]^{\frac{1}{2}+\epsilon}}\right),  
\end{align*}
where the last inequality above was obtained using the estimates
\begin{align*}
    \sum_{j=1}^{\lfloor \frac{n}{2}\rfloor}\frac{\delta_{0}^{j}}{T_{n-j}^{\frac{1}{2}+\epsilon}}\leq & \frac{\delta_{0}}{(1-\delta_{0})}\left[\frac{1}{\frac{1}{\delta_{0}}+\lfloor \frac{n}{2} \rfloor}\right]^{\frac{1}{2}+\epsilon},\\
    \sum_{j=\lceil \frac{n}{2}\rceil}^{n-1}\frac{\delta_{0}^{j}}{T_{n-j}^{\frac{1}{2}+\epsilon}}\leq & \frac{n\delta_{0}^{\frac{n}{2}+\frac{1}{2}+\epsilon}}{2}.
\end{align*}
\par Consequently, since $\frac{1}{2}+\epsilon>1$ and $T_{n}=\frac{1}{\delta}+n,$ we can verify from \eqref{m2} the existence of a uniform constant $K_{1}>1$ satisfying
\begin{equation}\label{final}
    \max_{n\in\mathbb{N}_{\geq 1}} \max_{\ell\in[m],t\in [0,\frac{1}{\delta}+n]}\left\vert y_{\ell,n}(t) -y_{\ell,n-1}(t)\right\vert\leq \max_{n\in\mathbb{N}_{\geq 1}}A_{n}T_{n}\leq K_{1}.
\end{equation}
\par Moreover, when $\delta\in (0,1)$ is small enough, it is not difficult to verify that
\begin{equation}\label{Annn}
\sum_{n}A_{n}\leq K\left[\sum_{n=1}^{{+}\infty} C^{n-1}\left[\delta_{0}^{n-1}+\frac{n\delta_{0}^{\frac{n+1}{2}+\epsilon}}{2}+\frac{\delta_{0}}{(1-\delta_{0})\left[\frac{1}{\delta_{0}}+\frac{n}{2}\right]^{\frac{1}{2}+\epsilon}}\right]\right]<{+}\infty.
\end{equation}
Therefore, Definition \ref{Yndef} implies that the sequence $\{\vec{u}_{n}(t)\}_{n\in\mathbb{N}}$ is a Cauchy sequence in $L^{2}_{x}(\mathbb{R},\mathbb{C}^{2})$ for any number $t\geq 0.$ 
\par Furthermore, since $\sigma_{n}(0)=\sigma_{0}(0)$ for all $n\in\mathbb{N}_{\geq 1},$  we can verify using estimates \eqref{m2}, \eqref{Annn}, the fundamental theorem of calculus, and identity $A_{n}=\norm{(\vec{u}_{n}-\vec{u}_{n-1},\sigma_{n}-\sigma_{n-1})}_{Y_{n}}$ that
$\sigma_{n}(t),\,\Lambda\dot\sigma_{n}(t)$ converge in $\mathbb{R}^{2}\times \mathbb{R}^+\times \mathbb{R}$ to unique values $\sigma(t),\,\Lambda\dot\sigma(t)$ respectively for any $t\geq 0$ when $n$ approaches ${+}\infty.$ As a consequence, we obtain that $\sigma(t)$ satisfies \eqref{odes} for all $t\geq 0.$ \par The proof that $\vec{u}(0,x)$ satisfies \eqref{p0intial} follows from Proposition \ref{undecays} and from the fact that $\lim_{n\to{+}\infty}\norm{\vec{u}_{n}(0)-\vec{u}(0)}_{L^{2}_{x}(\mathbb{R})}=0.$ 

Since $\lim_{n\to{+}\infty}\vert \sigma_{n}(t)-\sigma(t) \vert=0,$ we obtain from Proposition \ref{undecays} that
\begin{equation*}
    \langle \vec{u}(t,x),\mathfrak{p}_3e^{i\mathfrak{p}_3\left(\frac{v_{\ell}(t)x}{2}+\gamma_{\ell}(t)\right)}\vec{z}_{\alpha_{\ell}(t)}(x-y_{\ell}(t)) \rangle=0 \text{, for all $\vec{z}\in\ker\mathcal{H}^{2}_{1}$ when $t\geq 0.$}
\end{equation*}
\par Furthermore, the proof that $(\vec{u}(t),\sigma(t))=\lim_{n\to{+}\infty}(\vec{u}_{n}(t),\sigma_{n}(t))$ satisfies \eqref{decay1}-\eqref{odes} can be obtained using the Banach-Alaoglu theorem.
\par Finally, Lemmas \ref{blemma} and estimates \eqref{l2root}, \eqref{decay4} imply for any $n\in\mathbb{N}_{\geq 1}$ that
\begin{equation*}
   \norm{P_{\mathrm{root},n-1}(t)\vec{u}_{n}(t)}_{L^{2}_{x}(\mathbb{R})}+\norm{P_{\mathrm{unst},n-1}(t)\vec{u}_{n}(t)}_{L^{2}_{x}(\mathbb{R})}\lesssim \delta_{0}^{2} \text{, for all $t\in [0,T_{n}].$}
\end{equation*}
Consequently, using the Banach-Alaoglu theorem, we can verify that the function $g$ defined in Theorem \ref{Tsigma(t)} satisfies \eqref{p0intial}. 
\par The proof that the $g$ map is Lipschitz is similar to the proof of Lemma $4.10$ of \cite{KriegerSchlag}. More precisely, let $(\vec{u}_{n,r_{0}},\sigma_{n,r_{0}})$ and $(\vec{u}_{n,\hat{r}_{0}},\sigma_{n,\hat{r}_{0}})$ be functions defined in Proposition \ref{undecays} for $r_{0},\,\hat{r}_{0}\in \mathcal{B}_{\delta^{2}}$ satisfying
\begin{equation*}
    \max\left(\norm{r_{0}(x)}_{\Sigma},\norm{\hat{r}_{0}(x)}_{\Sigma}\right)\leq \delta^{2}.
\end{equation*}
Based on the argument used in Lemma $4.10$ of \cite{KriegerSchlag}, let $\norm{\Diamond}_{Y_{n,*}}$ be the following norm.
\begin{align*}
    \norm{(\vec{u},\sigma)}_{Y_{n,*}}=&\max_{t\in  \left[0,\min\left(T_{n},\norm{r_{0}-\hat{r}_{0}}_{L^{2}_{x}(\mathbb{R})}^{{-}1}\right)\right]}\langle t \rangle^{{-}1}\norm{\vec{u}(t,x)}_{L^{2}_{x}(\mathbb{R})}\\
     &{+}\max_{t\in  \left[0,\min\left(T_{n},\norm{r_{0}-\hat{r}_{0}}_{L^{2}_{x}(\mathbb{R})}^{{-}1}\right)\right]}\langle t \rangle^{1+\frac{\epsilon}{2}-\frac{3}{8}}\max_{\ell}\left\vert \Lambda \dot\sigma_{\ell}(t)\right\vert\\
     &{+}\max_{t\in  \left[0,\min\left(T_{n},\norm{r_{0}-\hat{r}_{0}}_{L^{2}_{x}(\mathbb{R})}^{{-}1}\right)\right]}\langle t \rangle^{{-}\frac{1}{4}}\norm{\frac{\chi_{\ell,n-1}(t)\vec{u}(t,x)}{\langle x-y^{T_{n-1}}_{\ell,n-1}(t) \rangle}}_{L^{\infty}_{x}(\mathbb{R})},
\end{align*}
Similarly to the proof of Proposition \ref{propun-un-1}, we can verify the following estimate for a uniform constant $C>1.$
\begin{align}
\label{Puncauchy}
    \norm{(\vec{u}_{n,r_{0}}-\vec{u}_{n,\hat{r}_{0}},\sigma_{n,r_{0}}-\sigma_{n,\hat{r}_{0}})}_{Y_{n,*}}&\leq C\delta_{0} \norm{(\vec{u}_{n-1,r_{0}}-\vec{u}_{n-1,\hat{r}_{0}},\sigma_{n-1,r_{0}}-\sigma_{n-1,\hat{r}_{0}})}_{Y_{n-1,*}}\\
    &+\frac{C}{T_{n}^{\frac{1}{2}+\epsilon}}+C\norm{r_{0}-\hat{r}_{0}}_{L^{2}_{x}(\mathbb{R})}.
\end{align}
Consequently, repeating the argument in the proof of Lemma $4.10$ from \cite{KriegerSchlag}, we can verify from the estimate \eqref{Puncauchy} that $g$ is Lipschitz on $\mathcal{B}_{\delta^{2}}.$

\section{Proof of Proposition \ref{undecays}}\label{sec:undecay}

\par We prove  Proposition \ref{undecays} by induction on $n$.  Given $(\vec{u}_{n-1 },\sigma_{n-1})$ satisfying  satisfying the estimates of Proposition \ref{undecays} for any $n\geq 1$, we will construct a map $A_{n-1}:C([0,T_{n-1}],L^{2}_{x}(\mathbb{R},\mathbb{C}^{2}))\times  C([0,T_{n-1}],\mathbb{R}^{2}\times \mathbb{R}^+\times \mathbb{R})\to ([0,T_{n-1}],L^{2}_{x}(\mathbb{R},\mathbb{C}^{2}))\times  C([0,T_{n-1}],\mathbb{R}^{2}\times \mathbb{R}^+\times \mathbb{R}).$ The map $A_{n-1}$ will have $(\vec{u}_{n},\sigma_{n})$ as its unique fixed point, from which we will obtain that $\vec{u}_{n}$ satisfies  equation \eqref{unequation} and all of the decays estimates of Proposition \ref{undecays}.  
\par In this section, to simplify  notations, we consider the map $\mathcal{U}_\sigma(t,s)$ to be the evolution operator associated to the flow $\sigma^{T_{n-1}}_{\ell,n-1}$ defined in \eqref{sigmaTTT} for the map $\sigma_{n-1}.$ We will also simply denote $A_{n-1}$ as $A$.
\subsection{Definition of the contraction map $A$}
 We define a map $A$ with the input $(\vec{u}^*,\sigma^{*})$ and the output $(\vec{u}(t),\sigma(t))$:
 \begin{equation*}
     A(\vec{u}_{*},\sigma^{*})(t)=(\vec{u}(t),\sigma(t)) \text{, for all $t\geq 0.$}
 \end{equation*}
 Precisely, $(\vec{u}(t),\sigma(t))$ are given as following. 
 
 \noindent{\bf Initial conditions for parameters:}
 The map $\sigma$ satisfies $\sigma(0)=\{(v_{\ell},y_{\ell}(0),\alpha_{\ell}(0),\gamma_{\ell}(0))\}_{\ell}$ for any $\ell\in [m].$
 
 \noindent{\bf Initial conditions for $\vec{u}$:}
 Let
 \begin{equation*}
     b_{\ell,+,*}(t)=P_{\mathrm{unst},\ell,n-1}(t)\vec{u}_{*}(t).
 \end{equation*}

 The function $\vec{u}(0)$ is the unique function of the form
 \begin{align}\label{u0form}
  \vec{u}(0,x)=&\vec{r}_{0}(x)+\sum_{\ell}e^{i\mathfrak{p}_3(\frac{v_{\ell}(0)x}{2}+\gamma_{\ell}(0))}h_{\ell}(0)Z_{+}(\alpha_{\ell,n-1}(T_{n-1}),x-y_\ell(0))
  \\ \nonumber
  &{+}\sum_{\ell}e^{i\mathfrak{p}_3(\frac{v_{\ell}(0)x}{2}+\gamma_{\ell}(0))}\vec{\mathcal{E}}_{\ell}(\alpha_{\ell,n-1}(T_{n-1}),x-y_{\ell}(0))  
\end{align}

satisfying
 \begin{align}\label{initial+data}
     P_{\mathrm{unst},\ell,n-1}(0)\vec{u}(0)=&i\int_{0}^{T_{n}}e^{i\lambda_{\ell}s}P_{\mathrm{unst},\ell,n-1}(s)\Bigg[G(s,\sigma^{*}(s),\sigma_{n-1}(s),\vec{u}_{n-1})\\ \nonumber
     &{+}\sum_{j=1}^{m}e^{i\lambda_{j}s}[V^{T_{n-1}}_{j,\sigma_{n-1}}(t,x)-V_{j,\sigma_{n-1}}(t,x)]\vec{u}_{*}(s)\\
     &{-}\sum_{h=1}^{m}\sum_{j=1,j\neq h}^{m} b_{h,+,*}(s)V^{T_{n-1}}_{j,\sigma_{n-1}}(s,x)e^{i\theta^{T_{n-1}}_{\sigma,\ell}(s,x)}\vec{Z}_{+}\left(\alpha_{h}(T_{n-1}),x-y^{T_{n-1}}_{h,\sigma_{n-1}}(s)\right)
     \\
     &{-}\sum_{j=1}^{m}V^{T_{n-1}}_{j,\sigma_{n-1}}(s,x)[P_{c,n-1}(s)\vec{u}_{*}(s)-P_{c,j,n-1}(s)\vec{u}_{*}(s)]\Bigg]\,ds,
 \end{align}
 where the functions $\vec{\mathcal{E}}_{\ell}\in \ker\mathcal{H}^{2}_{1}$ uniquely determined so that 
the following orthogonality conditions is satisfied:
\begin{equation*}
    \langle \vec{u}(0,x),\mathfrak{p}_3e^{i(\frac{v_{\ell,n-1}(0)x}{2}+\gamma_{\ell,n-1}(0))}\vec{z}(\alpha_{\ell,n-1}(0),x-y_{\ell,n-1}(0)) \rangle=0.
\end{equation*}

\noindent{\bf Equations for $\vec{u}(t)$:}
Next, we define $\vec{u}(t)$  any $t\in [0,T_{n}]$ by 
\begin{equation*}
\vec{u}(t)=\vec{u}_{c}+\vec{u}_{\mathrm{root}}+\vec{u}_{\mathrm{unst},n-1}+\vec{u}_{\mathrm{stab},n-1}(t)
\end{equation*}
such that
\begin{align}\label{disperpart}
    \vec{u}_{\mathrm{stab},n-1}(t,x)=&\mathcal{U}_\sigma(t,0)(P_{\mathrm{stab,n-1}}(0)\vec{u}(0,x))\\ \nonumber
    &{-}i\int_{0}^{t}\mathcal{U}_\sigma(t,s)P_{\mathrm{stab},n-1}(s)G(s,\sigma^{*}(s),\sigma_{n-1}(s),\vec{u}_{n-1})\,ds\\ \nonumber
    &{-}i\int_{0}^{t} \mathcal{U}_\sigma(t,s)P_{\mathrm{stab},n-1}(s)[\sum_{\ell}[V^{T_{n-1}}_{\ell,\sigma_{n-1}}(t,x)-V_{\ell,\sigma_{n-1}}(t,x)]\vec{u}_{*}(s)]\,ds
    \\ \nonumber
    &{+}i\int_{0}^{t}\mathcal{U}_{\sigma}(t,s)P_{\mathrm{stab},n-1}(s)\\&\times \sum_{h=1}^{m}\sum_{j=1,j\neq h}^{m}V^{T_{n-1}}_{j,\sigma_{n-1}}(s,x) b_{h,+,*}(s)e^{i\theta^{T_{n-1}}_{\sigma,\ell}(s,x)}\vec{Z}_{+}\left(\alpha_{h}(T_{n-1}),x-y^{T_{n-1}}_{h,\sigma_{n-1}}(s)\right)\,ds\\
     &{+}i\int_{0}^{t} \mathcal{U}_{\sigma}(t,s)P_{\mathrm{stab},n-1}(s)\sum_{j=1}^{m}V^{T_{n-1}}_{j,\sigma_{n-1}}(s,x)[P_{c,n-1}(s)\vec{u}_{*}(s)-P_{c,j,n-1}(s)\vec{u}_{*}(s)]\,ds
     \\=&\mathcal{U}_\sigma(t,0)(P_{\mathrm{stab},n-1}(0)\vec{u}(0,x))-i\int_{0}^{t}\mathcal{U}_\sigma(t,s)P_{\mathrm{stab},n-1}(s)H(s)\,ds,
\end{align}
\begin{align*}
    \vec{u}_{\mathrm{unst},n-1}(t,x)=&{i}\int_{t}^{T_{n}}e^{{-}i\lambda_{\ell}(t-s)}P_{\mathrm{unst},\ell,n-1}(s)\Bigg[G(s,\sigma^{*}(s),\sigma_{n-1}(s),\vec{u}_{n-1})\\ \nonumber
     &{+}\sum_{j=1}^{m}[V^{T_{n-1}}_{j,\sigma_{n-1}}(t,x)-V_{j,\sigma_{n-1}}(t,x)]\vec{u}_{*}(s)\\
     &{-}\sum_{h=1}^{m}\sum_{j=1,j\neq h}^{m} b_{h,+,*}(s)V^{T_{n-1}}_{j,\sigma_{n-1}}(s,x)e^{i\theta^{T_{n-1}}_{\sigma,\ell}(s,x)}\vec{Z}_{+}\left(\alpha_{h}(T_{n-1}),x-y^{T_{n-1}}_{h,\sigma_{n-1}}(s)\right)
     \\
     &{-}\sum_{j=1}^{m}V^{T_{n-1}}_{j,\sigma_{n-1}}(s,x)[P_{c,n-1}(s)\vec{u}_{*}(s)-P_{c,j,n-1}(s)\vec{u}_{*}(s)]\Bigg]\,ds\\
=&i\int_{t}^{T_{n}}e^{{-}i\lambda_{\ell}(t-s)}P_{\mathrm{unst},\ell,n-1}H(s)\,ds,
\end{align*}
and 
\begin{align}\label{contpart}
    \vec{u}_{c}(t,x)=\mathcal{S}(t)\circ \mathcal{S}^{{-}1}(0)P_{c,n-1}(0)\vec{u}(0,x)-i\int_{0}^{t}\mathcal{S}(t)\circ \mathcal{S}^{{-}1}(s)P_{c,n-1}(s)H(s)\,ds.
\end{align}
The function $\vec{u}_{\mathrm{root}}(t)$ is the unique element of $\Ra P_{\mathrm{root},n-1}(t)$ that satisfies
\begin{equation}\label{ortoconditionglob}
    \langle \vec{u}(t),\mathfrak{p}_3e^{i\mathfrak{p}_3(\frac{v_{\ell,n-1}(t)x}{2}+\gamma_{\ell,n-1}(t))}\vec{z}(\alpha_{\ell,n-1}(t),x-y_{\ell,n-1}(t))\rangle=0 \text{, for any $t\in [0,T_{n}],$ and $\vec{z}\in \ker\mathcal{H}^{2}_{1}.$}
\end{equation}
Finally, the function $\vec{u}(t)$ is defined by 
\begin{equation}\label{definitionofu}
    \vec{u}(t)=\begin{cases}
    \vec{u}_{c}+\vec{u}_{\mathrm{root}}+\vec{u}_{\mathrm{unst},n-1}+\vec{u}_{\mathrm{stab},n-1}(t) \text{ if $t\in [0,T_{n}],$}\\
    0 \text{ otherwise.}
    \end{cases}
\end{equation}
\noindent{\bf Equations for $\sigma(t)$:}
Finally, we define the map $\sigma(t)$ as the unique map satisfying the identity $\sigma(0)=\{(v_{\ell}(0),y_{\ell}(0),\alpha_{\ell}(0),\gamma_{\ell}(0))\}_{\ell}$ and the following ordinary differential system for any $\vec{z}\in\ker\mathcal{H}^{2}_{1},$ and $\ell\in [m].$
\begin{multline}\label{ODEofsigma}
 \left\langle {-}iG(t,\sigma(t),\sigma_{n-1}(t),\vec{u}_{n-1}),\mathfrak{p}_3e^{i\mathfrak{p}_3\left(\frac{v_{\ell,n-1}(t)x}{2}+\gamma_{\ell,n-1}(t)\right)}\vec{z}(\alpha_{\ell,n-1}(t),x-y_{\ell,n-1}(t))\right\rangle\\
 {+}\left\langle \vec{u}_{*}(t,x),\mathfrak{p}_3\left(\partial_{t}-i\mathfrak{p}_3\partial^{2}_{x}-iV_{\ell,\sigma_{n-1}}(t,x)\right)\left[e^{i\mathfrak{p}_3\left(\frac{v_{\ell,n-1}(t)x}{2}+\gamma_{\ell,n-1}(t)\right)}\vec{z}(\alpha_{\ell,n-1}(t),x-y_{\ell,n-1}(t))\right] \right \rangle\\
 {+} \left\langle \vec{u}_{n-1}(t,x),{-}i\mathfrak{p}_3\left[\sum_{j\neq \ell}V_{j,\sigma_{n-1}}(t,x)\right]\left[e^{i\mathfrak{p}_3\left(\frac{v_{\ell,n-1}(t)x}{2}+\gamma_{\ell,n-1}(t)\right)}\vec{z}(\alpha_{\ell,n-1}(t),x-y_{\ell,n-1}(t))\right] \right \rangle=0,
\end{multline}
see \eqref{unequation} for the definition of the $G$ function.
\par In the next subsections, we will estimate the $L^{\infty}$ norm and the $L^{2}$ norm for $\vec{u}(t)$ while $t\in [0,T_{n}].$ 
The main motivation for this is to verify the following proposition. 

\begin{proposition}\label{Aiscontraction}
Let $B_{n,\delta_{0}}\subset C([0,T_{n}],L^{2}_{x}(\mathbb{R},\mathbb{C}^{2})\times (\mathbb{R}^{2}\times \mathbb{R}^+\times \mathbb{R})^{m})$ be the subset of all elements $(\vec{u},\sigma)$ satisfying for $\sigma(t)=\{(v_{\ell}(t),y_{\ell}(t),\alpha_{\ell}(t),\gamma_{\ell}(t))\}$ equipped with the norms   
\begin{align*}
    \max_{t\in[0,T_{n}]}(1+t)^{\frac{1}{2}+\frac{3}{4}+\frac{3}{2}\left(1-\frac{2-p}{p}\right)}\max_{\ell}\norm{\chi_{\ell,n-1}(t,x)\frac{\partial_{x}\vec{u}(t,x)}{\langle x- y^{T_{n-1}}_{\ell,n-1}(t) \rangle^{1+\frac{p^{*}-2}{2p^{*}}+\omega}}}_{L^{2}_{x}(\mathbb{R})}\lesssim &\delta_{0},\\
    \max_{t\in [0,T_{n}]}\norm{\vec{u}(t,x)}_{H^{1}_{x}(\mathbb{R})}\lesssim &\delta_{0},\\ \max_{t\in[0,T_{n}],\ell}\frac{\norm{\chi_{\ell,n-1}(t,x)\vert x-y^{T_{n-1}}_{\ell,n-1}(t)\vert\vec{u}(t,x)}_{L^{2}_{x}(\mathbb{R})}}{[\max_{\ell}\vert v_{\ell}(0)\vert+1](1+t)}\leq \delta_{0},\, \max_{t\in[0,T_{n}]}\left[(1+t)^{\frac{1}{2}}\right]\norm{\vec{u}(t,x)}_{L^{\infty}_{x}(\mathbb{R})}\lesssim &\delta_{0},\\
\max_{t\in[0,T_{n}]}\left[(1+t)^{\frac{1}{2}+\frac{3}{4}+\frac{3}{2}\left(1-\frac{2-p}{p}\right)}\right]\max_{\ell}\norm{\frac{\chi_{\ell,n-1}(t,x)\vec{u}(t,x)}{(1+\vert x-y^{T_{n-1}}_{\ell,n-1}(t)\vert)^{\frac{3}{2}+\omega}}}_{L^{2}_{x}(\mathbb{R})}\lesssim &\delta_{0}, 
     \end{align*}
and 
\begin{align}\label{dotsigma*}
\max_{\ell,t\in [0,T_{n}]}\vert \dot y_{\ell}(t)-v_{\ell}(t)\vert\leq & \frac{\delta_{0}}{(1+t)^{1+2\epsilon}},\\ \nonumber
\max_{\ell,t\in [0,T_{n}]}\vert\dot \alpha_{\ell}(t)\vert \leq &\frac{\delta_{0}}{(1+t)^{1+2\epsilon}},\\
\max_{\ell,t\in [0,T_{n}]}|\dot v_{\ell}(t)|\leq & \frac{\delta_{0}}{(1+t)^{1+2\epsilon}},\\ \nonumber
\max_{\ell,t\in [0,T_{n}]}\left\vert\dot \gamma_{\ell}(t)-\alpha_{\ell}(t)^{2}+\frac{v_{\ell}(t)^{2}}{4}+\frac{y_{\ell}(t)\dot v_{\ell}(t)}{2}\right\vert\leq &\frac{\delta_{0}}{(1+t)^{1+2\epsilon}}
\end{align}where $\chi_{\ell,n-1}$ is given in Definition \ref{cutlinearpath}.

The map $A: B_{n,\delta_{0}}\to B_{n,\delta_{0}}$ is a contraction.
\end{proposition}
\begin{remark}\label{pointfixisun}
It is not difficult to verify that the fixed point of $A$, $(\vec{u}_{A},\sigma_{A})$,  is a solution of  equation \eqref{unequation} and equation \eqref{odesigma} when $t\in [0,T_{n}].$ Furthermore, the facts that $(\vec{u}_{A},\sigma_{A})\in B_{n,\delta}$ and $A(\vec{u}_{A},\sigma_{A})=(\vec{u}_{A},\sigma_{A})$ imply that $(\vec{u}_{n},\sigma_{n})=(\vec{u}_{A},\sigma_{A})$ satisfies Proposition \ref{undecays}.
\end{remark}
\begin{proof}[Proof of Proposition \ref{undecays} assuming Proposition \ref{Aiscontraction}.]
First, Remark \ref{pointfixisun} implies that Proposition \ref{undecays} is true for $n$ if it is true for $n-1.$ Therefore, by induction, it is enough to prove that $(\vec{u}_{0},\sigma_{0})$ satisfies Proposition \ref{undecays} for $n=0,$ which is true from the definition of $((\vec{u}_{0},\sigma_{0}))$ in Proposition \ref{undecays}, Theorem\ref{Decesti1} and theorem \ref{interpolation est.}.     
\end{proof}
\subsection{Basic setting for \emph{a priori} estimates}
In order to show that $A$ is a contraction, we have to do \emph{a priori} estimates and difference estimates. 
Notice that from \eqref{disperpart}, the right-hand side is linear in $\vec{u}^*$. The corresponding difference estimates will follow from \emph{a priori} estimates after taking the difference. So we focus on \emph{a priori} estimates.

From now on, we consider in the next subsections that $(\vec{u}_{*},\sigma_{*})\in B_{n,\delta}$ and that
\begin{equation}\label{vu=A}
    (\vec{u},\sigma)=A(\vec{u}_{*},\sigma_{*}).
\end{equation}
Moreover, using the decomposition formula \eqref{princ111}, Theorem \ref{tdis}, $\lambda_{\ell}=i\lambda_{0}\alpha_{\ell,n-1}(T_{n-1})^{2},$ and the linear path $\sigma^{T_{n-1}}_{n-1}(t)\in\mathbb{R}^{4m}$ defined at \eqref{sigmaTTT}, we can decompose the function $\vec{u}(t)$ uniquely in the following form $t\in [0,T_{n}].$
\begin{align}\label{decomp}
    \vec{u}(t)=&\vec{u}_{c}(t)+\sum_{\ell=1}^{n}b_{\ell,+}(t)e^{i\theta^{T_{n-1}}_{\ell,\sigma_{n-1}}(t,x)\sigma_{3}}\vec{Z}_{+}\left(\alpha_{\ell}(T_{n-1}),x-y^{T_{n-1}}_{\ell,\sigma_{n-1}}(t)\right)+\sum_{\ell=1}^{n}b_{\ell,-}(t)\mathfrak{G}_{\ell}(\mathfrak{v}_{\alpha_{\ell,n-1},\overline{\lambda_{\ell}}})(t,x)\\ \nonumber
    &{+}\vec{u}_{\mathrm{root}}(t,x),
\end{align}
such that  
\begin{align}\label{udisp}
    \vec{u}_{c}\in \Ra P_{c,\sigma^{T_{n-1}}_{n-1}},\\ \label{stabpart}
    \sppp\{\mathfrak{G}_{\ell}(\mathfrak{v}_{\alpha_{\ell,n-1},{-}i\lambda_{0}\alpha_{\ell,n-1}(T_{n-1})^{2}})(t,x)\}=&\Ra P_{\mathrm{stab},\ell,\sigma^{T_{n-1}}_{n-1}}(t),\\
    \vec{u}_{\mathrm{root}}\in \Ra P_{\mathrm{root},\sigma^{T_{n-1}}_{n-1}},
\end{align}
and $\vec{u}_{\mathrm{root}}(t)$ is uniquely linearly determined from $\vec{u}_{c}(t),\,\{b_{\ell,\pm}(t)\}$ to satisfy 
\begin{equation}\label{orthouu}
    \langle \vec{u}(t),\mathfrak{p}_3e^{i\mathfrak{p}_3\theta_{\ell,n-1}(t,x)}z_{\alpha_{\ell,n-1}(t)}(x-y_{\ell,n-1}(t))\rangle=0 \text{, for any $t\geq 0,$}
\end{equation}
for more details see the definition of the map $A$ in the previous subsection. 

\par Concerning the proof that $\vec{u}(t)\in B_{\delta,\sigma},$ we will check separately that 
\begin{equation*}
\vec{u}_{c}(t),\,\vec{u}_{\mathrm{root}}(t),\,P_{\mathrm{unst},\ell,n-1}(t)\vec{u}(t),\,  P_{\mathrm{stab},\ell,n-1}(t)\vec{u}    
\end{equation*}
 satisfy all the decays estimates \eqref{decay1}, \eqref{decay2}, \eqref{decay3}, \eqref{decay4} and \eqref{decay5}. Finally, using the hypotheses on $\vec{u}_{n-1},$ the fact that $\sigma(t)$ satisfies \eqref{odes} and the estimates \eqref{decay1} to \eqref{decay5}, we will obtain that $\sigma(t)=\{(y_{\ell}(t),v_{\ell}(t),\alpha_{\ell}(t),\gamma_{\ell}(t))\}_{\ell\in [m]} $ satisfy all the inequalities in Proposition \ref{Aiscontraction}.
\par The proof that $A$ is a contraction on $B_{n,\sigma}$ will done after the computation of the norm of $(\vec{u}(t),\sigma(t))$ which is much lower than the norm of $(\vec{u}_{*},\sigma_{*}).$ For more details regarding the proof of the contraction of $A,$ see Subsection \ref{Aiscont}.  Before going into details, we present some technical preparations.

We first check the projections  induced by the approximate trajectory $\sigma^{T_{n-1}}_{n-1}$ applied to roots space induced by the original trajectory $\sigma_{n-1}$.  
Theorem \ref{princ11} implies the following proposition.
\begin{proposition}\label{phrootissmall}
Assume that $\sigma_{n-1}(t)$ satisfies the estimates of Proposition \ref{undecays} for any $t\geq 0.$ There exists $C(\alpha,m)$ depending only on $\sigma(0)$ and $m$ satisfying for any $t\geq 0$ and $\mathfrak{v}(1,x)\in \Ra \ker\mathcal{H}^{2}_{1}$ the following estimate.    
\begin{multline*}
\max_{\ell\in[m],h\in\{\mathrm{stab},\mathrm{unst},c\}}\norm{P_{h,\sigma^{T_{n-1}}_{n-1},\ell}(t)[e^{i\mathfrak{p}_3\left(\frac{v_{\ell,n-1}(t)x}{2}+\gamma_{\ell,n-1}(t)\right)}\mathfrak{v}(\alpha_{\ell,n-1}(t),x-y_{\ell,n-1}(t))]}_{L^{2}_{x}(\mathbb{R})}\\\leq C\delta(\alpha,m) \left[\norm{\mathfrak{v}(1,x)}_{H^{1}_{x}(\mathbb{R})}+\norm{\langle x \rangle \mathfrak{v}(1,x)}_{L^{2}_{x}(\mathbb{R})}\right]. 
\end{multline*}
\end{proposition}
\begin{proof}
First, from Theorem \ref{princ11}, we have that
\begin{equation*}
    P_{h,\ell,n-1}(t)\mathfrak{G}_{\ell}(\mathfrak{v}_{\alpha_{\ell,n-1}(T_{n-1}),0})(t,x)=0 \text{, for any $h\in\{\mathrm{stab},\mathrm{unst},c\},$ }.
\end{equation*}
Next, using Theorem \ref{tdis} and definition of $\delta>0$ in Proposition \ref{undecays}, we can verify from the estimates \eqref{odes} satisfied by $\sigma_{n-1}$ and the fundamental theorem of calculus that
\begin{multline*}
    \norm{\mathfrak{G}_{\ell}(\mathfrak{v}_{\alpha_{\ell,n-1}(T_{n-1}),0})(t,x)-e^{i\mathfrak{p}_3\left(\frac{v_{\ell,n-1}(t)x}{2}+\gamma_{\ell,n-1}(t)\right)}\mathfrak{v}(\alpha_{\ell,n-1}(t),x-y_{\ell,n-1}(t))}_{L^{2}_{x}(\mathbb{R})}\\
    \lesssim_{\{(\alpha_{\ell}(0),v_{\ell}(0))\}}\delta [\norm{\mathfrak{v}(1,x)}_{H^{1}_{x}(\mathbb{R})}+\norm{\langle x \rangle \mathfrak{v}(1,x)}_{L^{2}_{x}(\mathbb{R})}].
\end{multline*}
for any $t\geq 0.$
\par In conclusion, using the Minkowski inequality, we obtain the result of Proposition \ref{phrootissmall} from the two estimates above. 
\end{proof}

\par  We next record localized estimates for solitons.
\begin{lemma}\label{diffcutV}
Let $\{(v_{\ell},y_{\ell
},\alpha_{\ell},\gamma_\ell)\}_{\ell\in[m]}$ be a set satisfying  hypotheses $\mathrm{(H1)},\, \mathrm{(H2)},$ and $\min_{\ell}y_{\ell}-y_{\ell+1}>10,\,$ and $A(\alpha)=\min_{\ell}\alpha_{\ell}>0.$ If $\omega\in (0,1),$ there exists a constant $K_{\omega}(\alpha,m)>1$ depending only on the set $\{(\alpha_{\ell})\}_{\ell}$ and $m$ satisfying the following inequality for any $n\in\{0,1,2\}$.
\begin{multline*}
    \max_{t\in\mathbb{R}_{\geq 0},\,q\in\{1,\infty\}}\norm{\chi_{\left[\frac{(v_{j}+v_{j+1})t+(y_{j}+y_{j+1})}{2},\frac{(v_{j}+v_{j-1})t+(y_{j}+y_{j-1})}{2}\right]}(x)\langle x-v_{j}t-y_{j}\rangle^{4+2\omega}\frac{d^{n}}{dx^{n}}\phi_{\alpha_{\ell}}(x-v_{\ell}t-y_{\ell})}_{L^{q}_{x}(\mathbb{R})}\\
    \leq K_{\omega}(\alpha,m)\begin{cases}
        1 \text{, if $j=\ell,$}\\
        e^{{-}\frac{95}{100}\min_{\ell,j}\alpha_{j}[(y_{\ell}-y_{\ell+1})+(v_{\ell}-v_{\ell+1})t]}
    \end{cases}.
\end{multline*}
Furthermore, if a set $\{(\alpha^{*}_{\ell}(t)\}_{\ell\in[m]}$ satisfies 
\begin{equation*}
\max_{t\in\mathbb{R}_{\geq 0}}\vert \alpha^{*}_{\ell}(t)-\alpha_\ell\vert\leq 1, 
\end{equation*}
the following inequality holds
\begin{multline*}
    \max_{t\in\mathbb{R}_{\geq 0},\,q\in\{1,\infty\}}\norm{\chi_{\left[\frac{(v_{j}+v_{j+1})t+(y_{j}+y_{j+1})}{2},\frac{(v_{j}+v_{j-1})t+(y_{j}+y_{j-1})}{2}\right]}(x)\langle x-v_{j}t-y_{j}\rangle^{4+2\omega}\frac{d^{n}}{dx^{n}}\phi_{\alpha^{*}_{\ell}(t)}(x-v_{\ell}t-y_{\ell})}_{L^{q}_{x}(\mathbb{R})}\\
    \leq K_{\omega}(\alpha,m)\begin{cases}
        1 \text{, if $j=\ell,$}\\
        e^{{-}\frac{95}{100}\min_{\ell,j}\alpha_{j}(y_{\ell}-y_{\ell+1}+(v_{\ell}-v_{\ell+1})t)}
    \end{cases}.
\end{multline*}
\end{lemma}
\begin{proof}
  The proof when of both inequalities when $j=\ell$ follows directly from the fact that $\phi_{\alpha}$ is a Schwartz function with of all of its derivative on $x$ having exponential decay. 
\par When $j\neq \ell,$ the proof follows from the elementary estimate for a constant $C_{n}>1$ depending only on $n\in\mathbb{N}$
    \begin{equation*}
        \left\vert \frac{d^{n}}{dx^{n}}\phi_{\alpha}(x)\right\vert\leq C_{n}e^{{-}\alpha\vert x \vert} \text{, for all $n\in\mathbb{N}$}  
    \end{equation*}
and the fact that 
\begin{equation*}
    \vert y_{j}-y_{j+1} \vert^{4+2\omega}e^{{-}\min_{j,\ell} \alpha_{j}(y_{\ell}-y_{\ell+1})}\leq e^{\frac{(4+2\omega)\min_{\ell,j}\alpha_{j}(y_{\ell}-y_{\ell+1})}{600}}e^{{-}\min_{\ell,j}\alpha_{j}(y_{\ell}-y_{\ell+1})}\leq e^{{-}\frac{95}{100}\min_{\ell,j}\alpha_{j}(y_{\ell}-y_{\ell+1})},
\end{equation*}
because of the hypothesis $\mathrm{(H2)}.$
\end{proof}
As a consequence, we can deduce the following corollary for the localized nonlinear terms for $\vec{u}$.
\begin{corollary}\label{quadratipotentialestimate}
  Let $\{(v_{\ell},y_{\ell
},\alpha_{\ell}(t),\gamma_{\ell}\}_{\ell\in[m]}$ be a set satisfying  hypotheses $\mathrm{(H1)},\, \mathrm{(H2)},$ and $\min_{\ell}y_{\ell}-y_{\ell+1}>10,\,$ and $A(\alpha)=\min_{\ell}\alpha_{\ell}>0.$ If $\omega\in (0,1)$ and $\alpha^{*}_{\ell}(t)>0$ satisfies
\begin{equation*}
    \max_{\ell,t\geq 0}\vert \alpha^{*}_{\ell}(t)-\alpha_{\ell}(0) \vert<1,\,\min_{\ell}\alpha_{\ell}(0)>0,
\end{equation*}
then there exists a constant $K_{\omega}(\alpha,m,k)>1$ depending only on $k\in\mathbb{N}_{\geq 2},$ the set $\{\alpha_{\ell}(0)\}_{\ell\in[m]}$ and $m$ satisfying for any $\vec{u}\in L^{2}_{x}(\mathbb{R},\mathbb{C}^{2})$ the following estimates for any $d\in[2,2k]$
\begin{multline*}
   \max_{d\in\{2k,2\},\,\ell_{1},\ell_{2}\in[m]}\norm{\chi_{\left[\frac{(v_{\ell_{1}}+v_{\ell_{1}+1})t+(y_{\ell_{1}}+y_{\ell_{1}+1})}{2},\frac{(v_{\ell_{1}}+v_{\ell_{1}-1})t+(y_{\ell_{1}}+y_{\ell_{1}-1})}{2}\right]}(x)\phi_{\alpha^{*}_{\ell_{2}(t)}}(x-v_{\ell_{2}}t-y_{\ell_{2}})\vert \vec{u}(t,x)\vert^{d}}_{L^{1}_{x}(\mathbb{R})}\\
   \leq K_{\omega}(\alpha,m,k) \max_{\ell}\norm{\frac{\chi_{\left[\frac{(v_{\ell}+v_{\ell+1})t+(y_{\ell}+y_{\ell+1})}{2},\frac{(v_{\ell}+v_{\ell-1})t+(y_{\ell}+y_{\ell-1})}{2}\right]}(x)\vec{u}(t,x)}{\langle x- v_{\ell}t-y_{\ell}\rangle^{\frac{3}{2}+\omega}}}_{L^{2}_{x}(\mathbb{R})}^{2}\norm{\vec{u}(t)}_{L^{\infty}_{x}(\mathbb{R})}^{d-2},  
\end{multline*}
and
\begin{multline*}
\max_{\ell_{1},\ell_{2}\in[m]}\norm{\chi_{\left[\frac{(v_{\ell_{1}}+v_{\ell_{1}+1})t+(y_{\ell_{1}}+y_{\ell_{1}+1})}{2},\frac{(v_{\ell_{1}}+v_{\ell_{1}-1})t+(y_{\ell_{1}}+y_{\ell_{1}-1})}{2}\right]}(x)\phi_{\alpha^{*}_{\ell_{2}(t)}}(x-v_{\ell_{2}}t-y_{2})\vert \vec{u}(t,x)\vert^{d}}_{L^{2}_{x}(\mathbb{R})}\\
   \leq K_{\omega}(\alpha,m,k) \max_{\ell}\norm{\frac{\chi_{\left[\frac{(v_{\ell}+v_{\ell+1})t+(y_{\ell}+y_{\ell+1})}{2},\frac{(v_{\ell}+v_{\ell-1})t+(y_{\ell}+y_{\ell-1})}{2}\right]}(x)\vec{u}(t,x)}{\langle x- v_{\ell}t-y_{\ell}\rangle^{\frac{3}{2}+\omega}}}_{L^{2}_{x}(\mathbb{R})}\norm{\vec{u}(t)}_{L^{\infty}_{x}(\mathbb{R})}^{d-1}.  
\end{multline*}
\end{corollary}
\begin{proof}
    The proof for the case when $d=2$ and $d=2k>2$ is a consequence of Lemma \ref{diffcutV} and H\"older's inequality. The proof for the case $d\in (2,2k)$ follows from the previous cases and interpolation.
\end{proof}
\subsection{Estimate of unstable components }In this subsection, we establish estimates for the $L^\infty$ and $L^2$ norms for  $P_{\mathrm{unst},\ell,n-1}(t)\vec{u}(t)$.
First, using the decomposition formula \eqref{decomp}, we write
\begin{equation}\label{b+id}
    b_{\ell,+}(t)e^{i\theta^{T_{n-1}}_{\ell,\sigma_{n-1}}(t,x)\sigma_{3}}\vec{Z}_{+}\left(\alpha_{\ell}(T_{n-1}),x-y^{T_{n-1}}_{\ell,\sigma_{n-1}}(t)\right)=P_{\mathrm{unst},\ell,n-1}(t)\vec{u}(t),
\end{equation}
for any $\ell \in [m].$ 
\par Next, to simplify more our notation, let
\begin{align}\label{Forcingtermn-1ustar}
    F_{n-1}(s,\sigma^{*},\vec{u}_{*})=&G(s,\sigma^{*}(s),\sigma_{n-1}(s),\vec{u}_{n-1})+\sum_{\ell}[V^{T_{n-1}}_{\ell,\sigma_{n-1}}(t,x)-V_{\ell,\sigma_{n-1}}(t,x)]\vec{u}_{*}(s),\\ \nonumber
    Int_{\mathrm{unst},n-1}(s,\vec{u}_{*})=&{-}\sum_{h=1}^{m}\sum_{j=1,j\neq h}^{m} b_{h,+,*}(s)V^{T_{n-1}}_{j,\sigma_{n-1}}(s,x)e^{i\theta^{T_{n-1}}_{\sigma,\ell}(s,x)}\vec{Z}_{+}\left(\alpha_{h}(T_{n-1}),x-y^{T_{n-1}}_{h,\sigma_{n-1}}(s)\right)\\ \label{intunstn-1b0}
    &{-}\sum_{j=1}^{m}V^{T_{n-1}}_{j,\sigma_{n-1}}(s,x)[P_{c,\sigma^{T_{n-1}}_{n-1}}(s)\vec{u}_{*}(s)-P_{c,j,\sigma^{T_{n-1}}_{n-1}}(s)\vec{u}_{*}(s)],
\end{align} 
and $Forc_{\mathrm{unst},n-1}(s,\sigma^{*},\vec{u}_{*})=F_{n-1}(s,\sigma^{*},\vec{u}_{*})+Int_{\mathrm{unst},n-1}(s,\vec{u}_{*})$ where the function $G$ is the one defined in \eqref{unequation}.
From the identity \eqref{initial+data} in the previous section, we can verify for any $\ell\in [m],\,\lambda_{\ell}=i\alpha_{\ell,n-1}(T_{n-1})^{2}\lambda_{0},$ and any $t\in [0,T_{n}]$ that
\begin{multline}\label{unstbustar(t)}
    b_{\ell,+}(t)e^{i\theta^{T_{n-1}}_{\sigma,\ell}(s,x)}\vec{Z}_{+}\left(\alpha_{h}(T_{n-1}),x-y^{T_{n-1}}_{h,\sigma_{n-1}}(s)\right)\\= i\int_{t}^{T_{n}}e^{(t-s)\vert\lambda_{\ell}\vert}P_{\mathrm{unst},\ell,n-1}(s)\left(Forc_{\mathrm{unst},n-1}(s,\sigma^{*},\vec{u}_{*})\right)\,ds,
\end{multline}
for any $t\in[0,T_{n}].$ In particular, given the element $(\vec{u}_{*},\sigma_{*})\in B_{n,\sigma},$ the existence of unique functions $(b_{\ell,+}(t))_{\ell\in [m]}$ satisfying the integral equation \eqref{unstbustar(t)} is obtained using the Picard–Lindelöf Theorem for system of ordinary differential equations. 
\par Consequently, from \eqref{stabpart}, we can verify for all $t\in [0,T_{n}]$
the existence of a uniform constant $C>1$ satisfying \begin{align*}
  \vert b_{\ell,+}(t) \vert\leq &Ce^{{-}\frac{\vert \lambda_{\ell} \vert t}{2}}\max_{s\in[0,\frac{t}{2}]}\norm{ P_{\mathrm{unst},\ell,n-1}(s)\left(F_{n-1}(s,\sigma^{*},\vec{u}_{*})\right)}_{L^{2}_{x}(\mathbb{R})}\\&{+}C\max_{s\in[\frac{t}{2},t]}\norm{ P_{\mathrm{unst},\ell,n-1}(s)\left(F_{n-1}(s,\sigma^{*},\vec{u}_{*})\right)}_{L^{2}_{x}(\mathbb{R})}\\
  &{+}e^{{-}\min_{j,\ell}\alpha_{j,n-1}(T_{n-1})[y^{T_{n-1}}_{\ell,\sigma_{n-1}}(t)-y^{T_{n-1}}_{\ell+1,\sigma_{n-1}}(t)]}\max_{h\in[m]} \norm{b_{h,+,*}(s)}_{L^{\infty}_{s}[t,T_{n}]}\\
  &{+}e^{{-}\min_{j,\ell}\alpha_{j,n-1}(T_{n-1})[y^{T_{n-1}}_{\ell,\sigma_{n-1}}(t)-y^{T_{n-1}}_{\ell+1,\sigma_{n-1}}(t)]} \max_{s\in[t,T_{n}]}\norm{P_{c,\sigma^{T_{n-1}}_{n-1}}(s)\vec{u}_{*}(s)}_{L^{2}_{x}(\mathbb{R})}.
\end{align*}
Therefore, there exist a uniform constant $C>1$ satisfying for all $t\in [0,T_{n}]$ the following estimate
\begin{align}\label{bupperbound}
\max_{\ell\in [m]}\norm{ b_{\ell,+}(s)}_{L^{\infty}_{s}[t,T_{n}]} \leq &Ce^{{-}\frac{\vert \lambda_{\ell} \vert t}{2}}\max_{\ell\in [m]}\max_{s\in[0,\frac{t}{2}]}\norm{ P_{\mathrm{unst},\ell,n-1}(s)\left(F_{n-1}(s,\sigma^{*},\vec{u}_{*})\right)}_{L^{2}_{x}(\mathbb{R})}\\ \nonumber
&{+}C\max_{\ell\in [m]}\max_{s\in[\frac{t}{2},t]}\norm{ P_{\mathrm{unst},\ell,n-1}(s)\left(F_{n-1}(s,\sigma^{*},\vec{u}_{*})\right)}_{L^{2}_{x}(\mathbb{R})}\\
&{+}Ce^{{-}\min_{j,\ell}\alpha_{j,n-1}(T_{n-1})[y^{T_{n-1}}_{\ell,\sigma_{n-1}}(t)-y^{T_{n-1}}_{\ell+1,\sigma_{n-1}}(t)]} \max_{h\in [m]}\norm{b_{h,+,*}(s)}_{L^{\infty}_{s}[t,T_{n}]}
\\ \nonumber
&{+}C\max_{\ell\in [m]}\max_{s\in[\frac{t}{2},t]}\norm{ P_{\mathrm{unst},\ell,n-1}(s)\left(F_{n-1}(s,\sigma^{*},\vec{u}_{*})\right)}_{L^{2}_{x}(\mathbb{R})}\\
&{+}Ce^{{-}\min_{j,\ell}\alpha_{j,n-1}(T_{n-1})[y^{T_{n-1}}_{\ell,\sigma_{n-1}}(t)-y^{T_{n-1}}_{\ell+1,\sigma_{n-1}}(t)]} \max_{s\in[t,T_{n}]}\norm{P_{c,\sigma^{T_{n-1}}_{n-1}}(s)\vec{u}_{*}(s)}_{L^{2}_{x}(\mathbb{R})}.
\end{align}
\par Next, using Corollary \ref{quadratipotentialestimate} and the definition of $F$ in \eqref{Fdefinition}, we can verify from the fundamental theorem of calculus and $(\vec{u},\sigma)\in B_{n,\delta}$ that
\begin{multline}\label{l2ofquadraticpart}
\Big\vert\Big\vert F\left(\sum_{\ell} e^{i(\frac{v_{\ell,n-1}(t)x}{2}+\gamma_{\ell,n-1}(t))}\phi_{\alpha_{\ell,n-1}(t)}(x-y_{\ell,n-1}(t))+\vec{u}_{n-1}\right)\\
{-}F\left(\sum_{\ell} e^{i(\frac{v_{\ell,n-1}(t)x}{2}+\gamma_{\ell,n-1}(t))}\phi_{\alpha_{\ell,n-1}(t)}(x-y_{\ell,n-1}(t))\right)
\\
{-}\sum_{\ell=1}^{m}F^{\prime}\left(e^{i(\frac{v_{\ell,n-1}(t)x}{2}+\gamma_{\ell,n-1}(t))}\phi_{\alpha_{\ell,n-1}(t)}(x-y_{\ell,n-1}(t))\right)\vec{u}_{n-1}(t)+\begin{bmatrix}
   \vert \vec{u}_{n-1}(t) \vert^{2k} \vec{u}_{n-1}(t)\\
   {-}\vert \vec{u}_{n-1}(t) \vert^{2k} \vec{u}_{n-1}(t)
\end{bmatrix}\Big\vert\Big\vert_{L^{2}_{x}(\mathbb{R})}\\
\leq C(\alpha,v)\max_{\ell}\norm{\frac{\chi_{\ell,n-1}(t)\vec{u}(t)}{\langle x-y^{T_{n-1}}_{\ell,n-1}(t)\rangle^{\frac{3}{2}+\omega}}}_{L^{2}_{x}(\mathbb{R})}\left[\norm{\vec{u}(t)}_{L^{\infty}_{x}(\mathbb{R})}^{2}+\norm{\vec{u}(t)}_{L^{\infty}_{x}(\mathbb{R})}^{2k}\right], 
\end{multline}
for a constant $C(\alpha,v)>1$ depending only on the $\{(v_{\ell}(0),\alpha_{\ell}(0))\},$ which is constant according to the assumptions of Theorem \ref{asy}. 
\par Furthermore, using Lemma \ref{dinftydt}, and estimates \eqref{dyn11}, \eqref{dyn12},
we can verify the existence of a constant $C(\alpha,v)>1$ depending only on the set $\{v_{\ell}(0),\alpha_{\ell}(0)\}$ satisfying the following inequality.
\begin{equation}\label{L2differenceofVlinftyu}
    \norm{\sum_{\ell}[V^{T_{n-1}}_{\ell,\sigma_{n-1}}(t,x)-V_{\ell,\sigma_{n-1}}(t,x)]\vec{u}_{*}(s)}_{L^{2}_{x}(\mathbb{R})}\leq \frac{C(\alpha,v)\delta_{0}}{(1+t)^{2\epsilon-1}}\max_{\ell}\norm{\frac{\chi_{\ell,n-1}(t,x)\vec{u}_{*}(t,x)}{(1+\vert x-y^{T_{n-1}}_{\ell,n-1}(t)\vert)^{\frac{3}{2}+\omega}}}_{L^{2}_{x}(\mathbb{R})},  
\end{equation}
for any $t\in [0,T_{n+1}].$ 
\par Consequently, using Propositions \ref{multisolitonsinteractionsize}, \ref{phrootissmall} and estimates \eqref{l2ofquadraticpart}, \eqref{L2differenceofVlinftyu},  we can verify from the definition of $G$ given in  equation \eqref{unequation} and the definition of $F_{n-1}$ that
\begin{multline}\label{forcingunstabstab}
\max_{h\in\{\mathrm{unst},\mathrm{stab}\}}\norm{ P_{h,\ell,n-1}(s)\left(F_{n-1}(s,\sigma^{*},\vec{u}_{*})\right) }_{L^{2}_{x}(\mathbb{R})}\\
\begin{aligned}
\leq &  C(v,\alpha)\frac{\delta_{0}}{(1+s)^{2\epsilon-1}} \norm{\frac{\chi_{\ell}(s)\vec{u}_{*}(s)}{\langle x-y^{T_{n-1}(s)}_{\ell,n-1}(s) \rangle^{\frac{3}{2}+\omega}}}_{L^{2}_{x}(\mathbb{R})}\\
&{+}C(v,\alpha)\max_{\ell}\norm{\frac{\chi_{\ell,n-1}(s)\vec{u}_{n-1}(s)}{\langle x-y^{T_{n-1}}_{\ell,n-1}(s) \rangle^{\frac{3}{2}+\omega}}}_{L^{2}_{x}(\mathbb{R})}^{2}\\
&{+}C(v,\alpha)\delta_{0} \max_{\ell}\left\vert \Lambda \dot \sigma^{*}_{\ell}(t)   \right\vert
\\
&{+}C(v,\alpha)e^{{-}\frac{99}{100} \min_{j}\alpha_{j,n-1}(s)\min_{\ell}(y_{\ell,n-1}(s)-y_{\ell+1,n-1}(s))},
\end{aligned}
\end{multline}
for some constant $C(v,\alpha)>1$ depending only $\{(v_{\ell}(0),\alpha_{\ell}(0))\}_{\ell}.$
\par The estimates \eqref{bupperbound} and \eqref{forcingunstabstab} imply the following proposition.
\begin{lemma}\label{blemma}
If $\norm{\vec{u}_{n-1}(t)}_{(L^{\infty}[0,{+}\infty);H^{1}_{x}(\mathbb{R}))}\lesssim \delta_{0},\,\max_{\ell}\vert \Lambda\dot \sigma_{\ell}(s) \vert\lesssim \delta_{0},$
\begin{equation*}
    \norm{\frac{\chi_{\ell}(s)\vec{u}_{*}(s)}{\langle x-y^{T_{n-1}}_{j,n-1}(s) \rangle^{\frac{3}{2}+\omega}}}_{L^{2}_{x}}\lesssim \frac{\delta_{0}}{(1+s)^{\frac{1}{2}+\epsilon}} \text{ for $s\in [0,T_{n}],$}
\end{equation*}
then, for any $t\geq 0,$ then $b_{\ell,\pm}(t)$ satisfies for a constant $C(v)>1,\,\alpha=\frac{\vert \lambda_{\ell} \vert}{2}$ the following estimates
\begin{multline*}
    \begin{aligned}
    \vert b_{\ell,+}(t)\vert
    \leq & C\Bigg(\max_{j,\tau \geq t}\norm{\frac{\chi_{j}(\tau,x)\vec{u}_{*}(\tau,x)}{(1+\vert x-y^{T_{n-1}}_{j,n-1}(\tau)\vert)^{\frac{3}{2}+\omega}}}_{L^{2}_{x}(\mathbb{R})}^{2}{+}\delta_{0}\max_{j,\tau\geq t}\norm{\frac{\chi_{j,n-1}(\tau,x)\vec{u}_{*}(\tau,x)}{(1+\vert x-y^{T_{n-1}}_{j,n-1}(\tau)\vert)^{\frac{3}{2}+\omega}}}_{L^{2}_{x}(\mathbb{R})}
   \\&{+}e^{{-}\frac{99}{100} \min_{j}\alpha_{j,n-1}(s)\min_{\ell}(y_{\ell,n-1}(s)-y_{\ell+1,n-1}(s))}\Bigg).
    \end{aligned}
\end{multline*}
\end{lemma}
Similarly to the proof of Lemma \ref{blemma}, but considering now $A(\vec{u}_{*},\sigma^{*})=(\vec{u}_{*,A},\sigma^{*}_{A}),$ $A(\vec{u}_{**},\sigma^{**})=(\vec{u}_{**,A},\sigma^{**}_{A})$ and the difference of the equations \eqref{unequation} satisfied by $\vec{u}_{*,A}$ and $\vec{u}_{**,A} ,$ we can verify the following proposition.

\begin{lemma}\label{blemma2}
If $(\vec{u}_{*},\sigma^{*})$ and $(\vec{u}_{**} ,\sigma^{**})$ are elements of $B_{n,\delta},$ then, for any $t\in[0,T_{n}],$ the following estimate holds
\begin{multline*}
    \max_{\ell}\norm{ P_{\mathrm{unst},\ell,n-1}(t)\left[\vec{u}_{*,A}-\vec{u}_{**,A} \right](t)}_{L^{2}_{x}(\mathbb{R})}
    \\
    \begin{aligned}
    \leq  C\Bigg(&\max_{j,\tau\geq t}\norm{\frac{\chi_{j}(\tau,x)[\vec{u}_{*}-\vec{u}_{**} ](\tau,x)}{(1+\vert x-y^{T_{n-1}}_{j,\sigma_{n-1}}(\tau)\vert)^{\frac{3}{2}+\omega}}}_{L^{2}_{x}(\mathbb{R})}^{2}\\&{+}\delta_{0}\max_{\tau\geq t} e^{{-}\min_{\ell,j}\alpha_{j,n-1} (v_{\ell,n-1}-v_{\ell+1},n-1)\tau}\norm{\vec{u}_{*}(\tau)-\vec{u}_{**}(\tau)}_{L^{2}_{x}(\mathbb{R})}\\
    &{+}\delta_{0}\max_{j,\tau\geq t}\norm{\frac{\chi_{j}(\tau,x)[\vec{u}_{*}-\vec{u}_{**} ](\tau,x)}{(1+\vert x-y^{T_{n-1}}_{j,\sigma_{n-1}}(\tau)\vert)^{\frac{3}{2}+\omega}}}_{L^{2}_{x}(\mathbb{R})}
    {+}\delta_{0} \vert \Lambda\dot \sigma^{*}_{A}(t)-\Lambda \dot \sigma^{**}_{A}(t)\vert\Bigg).
    \end{aligned}
\end{multline*}
\end{lemma}

\subsection{$L^{\infty}$ estimates}
In this subsection, we will estimate the size of $\norm{\vec{u}(t)}_{L^{\infty}_{x}(\mathbb{R})}$ while $t\in [0,T_{n}].$ Since the right-hand side of  equation \eqref{unequation} has many different terms, we will estimate each of these terms differently in each part of this subsection.  
\subsubsection{Estimate of $\norm{\mathcal{U}_\sigma(t,0)\vec{u}(0,x)}_{L^{\infty}}$}\label{linearpartt0}
The proof that 
\begin{equation*}
    \norm{P_{c,n-1}\mathcal{U}_\sigma(t,0)\vec{u}(0,x)}_{L^{\infty}_{x}(\mathbb{R})}\leq C\frac{\delta^{2}}{(1+t)^{\frac{1}{2}}}\ll \frac{\delta_{0}}{(1+t)^{\frac{1}{2}}},
\end{equation*}
follows from Theorem \ref{Decesti1}, the definition of $\vec{u}(0,x)$ in \eqref{u0form} and  hypotheses satisfied by $\vec{r}_{0}$ in the statement of Theorem \ref{asy}.
\subsubsection{Estimate of $\norm{\vec{u}_{\mathrm{root}}}_{L^{\infty}_{x}(\mathbb{R})},\,\norm{\vec{u}_{\mathrm{unst},\sigma^{T_{n-1}}_{\sigma_{n-1}}}}_{L^{\infty}_{x}(\mathbb{R})},\,\norm{\vec{u}_{\mathrm{stab},\sigma^{T_{n-1}}_{\sigma_{n-1}}}}_{L^{\infty}_{x}(\mathbb{R})}$}
\par  First, the estimate of 
\begin{equation*}
\max_{\ell}\vert b_{\ell,+}(t) \vert    \end{equation*}
 in Lemma \ref{blemma} the previous subsection implies that
\begin{equation*}
\norm{P_{\mathrm{unst},\sigma^{T_{n-1}}_{n-1}}(t)\vec{u}(t)}_{L^{\infty}_{x}(\mathbb{R})}\leq C \frac{\delta_{0}^{2}}{(1+t)^{\frac{1}{2}}}\ll \frac{\delta_{0}}{(1+t)^{\frac{1}{2}}}.    
\end{equation*}
Finally, since $\vec{u}(t)$ satisfies \eqref{ortoconditionglob} for any $t\in [0,T_{n}],$ we obtain from formula \eqref{decomp} that
\begin{equation}\label{uppuroot}
    \norm{ \vec{u}_{\mathrm{root}}(t)}_{L^{\infty}_{x}(\mathbb{R})} \leq C\delta_{0}\max\left(\max_{\ell}\vert b_{\ell,+}(t)\vert,\norm{\vec{u}_{c}(t)}_{L^{\infty}_{x}(\mathbb{R})}\right),
\end{equation}
for some constant $C>1$ depending only on $(v_{\ell}(0),\alpha_{\ell}(0))_{\ell\in[m]}.$
\par Moreover, using Lemma \ref{diffcutV}, and the estimates \eqref{dyn11}, \eqref{dyn12}, we can verify  using \eqref{decomp} that
\begin{equation}\label{l2root}
    \norm{ \vec{u}_{\mathrm{root}}(t)}_{L^{2}_{x}(\mathbb{R})} \leq C\delta_{0}\max\left(\max_{\ell}\vert b_{\ell,+}(t)\vert,\max_{\ell}\norm{\frac{\chi_{\ell,n-1}(t,x)\vec{u}_{c}(t)}{\langle x-y^{T_{n-1}}_{\ell,n-1}(t)\rangle^{\frac{3}{2}+\omega}}}_{L^{2}_{x}(\mathbb{R})}\right).
\end{equation}
\par Furthermore, Remark \ref{stabel} implies that $P_{\mathrm{stab},\sigma^{T_{n-1}}_{n-1}}(t)\}\mathcal{U}_\sigma(t,\tau)$ has a stronger decay than any of the decay estimates of $P_{c,\sigma^{T_{n-1}}_{n-1}}(t)\}\mathcal{U}_\sigma(t,\tau)$ in Theorem \ref{Decesti1}.
\par Consequently, because Remark \ref{stabel}, estimate \eqref{uppuroot}, and Lemma \ref{blemma}, it is enough to verify that
\begin{equation*}
    \norm{P_{c,n-1}(t)\vec{u}(t)}_{L^\infty_{x}(\mathbb{R})}\ll \frac{\delta_{0}}{(1+t)^{\frac{1}{2}}}
\end{equation*}
to obtain that $\vec{u}(t)$ satisfies $L^{\infty}$ from Proposition \ref{Aiscontraction}.
\subsubsection{$[V^{T_{n-1}}_{\ell,\sigma_{n-1}}(s)-V_{\ell,\sigma_{n-1}}(s)]\vec{u}_{*}$ term}
now we consider the $L^{\infty}$ norm of the following term 
\begin{align*}
&i\sum_{\ell}\int_{0}^{t}\mathcal{S}(t)\circ \mathcal{S}^{{-}1}(s)P_{c,n-1}(s)\left[V^{T_{n-1}}_{\ell,\sigma_{n-1}}(s,x)-V_{\ell,\sigma_{n-1}}(s,x)\right]\vec{u}_{*}(s,x)\,ds
\end{align*}
\par First, using Lemma \ref{dinftydt}, we can verify that
\begin{equation*}
    \norm{\left[V^{T_{n-1}}_{\ell,\sigma_{n-1}}(s,x)-V_{\ell,\sigma_{n-1}}(s,x)\right]\vec{u}_{*}(s,x)}_{L^{1}_{x}(\mathbb{R})}\lesssim \frac{\delta_{0}}{(1+s)^{2\epsilon-1}}\norm{\frac{\max_{\ell}\chi_{\ell,n-1}(s,x)\vec{u}_{*}(s,x)}{\langle x-y^{T_{n-1}}_{\ell,n-1}(s)\rangle^{\frac{3}{2}+\omega}}}_{L^{2}_{x}(\mathbb{R})}
\end{equation*}
\par Therefore, using Theorem \ref{Decesti1}, we can verify the following estimate.
\begin{multline*}
  \norm{\int_{0}^{t}\mathcal{S}(t)\circ \mathcal{S}^{{-}1}(s)P_{c,n-1}(s)\left[V^{T_{n-1}}_{\ell,\sigma_{n-1}}(s,x)-V_{\ell,\sigma_{n-1}}(s,x)\right]\vec{u}_{*}(s,x)\,ds}_{L^{\infty}_{x}(\mathbb{R})}\\
  \begin{aligned}
  \leq & C\int_{0}^{t}\frac{\delta}{(t-s)^{\frac{1}{2}}(1+s)^{2\epsilon-1}}\norm{\frac{\max_{\ell}\chi_{\ell,n-1}(s,x)\vec{u}_{*}(s,x)}{\langle x-y^{T_{n-1}}_{\ell,n-1}(s)\rangle^{\frac{3}{2}+\omega}}}_{L^{2}_{x}(\mathbb{R})}\,ds.
  \end{aligned}
\end{multline*}

\par Consequently, since $\vec{u}_{*}\in B_{n,\delta}.$ we obtain for $\epsilon\geq \frac{3}{4}$ that
\begin{multline*}
   \norm{\int_{0}^{t}\mathcal{S}(t)\circ \mathcal{S}^{{-}1}(s)P_{c,n-1}(s)\left[V^{T_{n-1}}_{\ell,\sigma_{n-1}}(s,x)-V_{\ell,\sigma_{n-1}}(s,x)\right]\vec{u}_{*}(s,x)\,ds}_{L^{\infty}_{x}(\mathbb{R})} \\
   \begin{aligned}
   \leq & K\delta_{0}^{2}\int_{0}^{t}\frac{1}{(t-s)^{\frac{1}{2}}}\frac{1}{(1+s)^{3\epsilon-\frac{1}{2}}}\,ds\\
   \leq & K\frac{\delta_{0}}{N}\left[\min\left(\frac{\sqrt{2}}{t^{\frac{1}{2}}},2t^{\frac{1}{2}}\right)+\frac{4\sqrt{2}}{(1+t)^{\frac{5}{4}}}\right]\ll \frac{\delta_{0}}{(1+t)^{\frac{1}{2}}} \text{, for a $N\gg 1,$ if $\delta_{0}\in (0,1)$ is small enough.}
   \end{aligned}
\end{multline*}
Similarly, we can verify that 
\begin{multline}\label{DiffVu*}
   \norm{\int_{0}^{t}\mathcal{S}(t)\circ \mathcal{S}^{{-}1}(s)P_{c,n-1}(s)\left[V^{T_{n-1}}_{\ell,\sigma_{n-1}}(s,x)-V_{\ell,\sigma_{n-1}}(s,x)\right]\vec{u}_{*}(s,x)\,ds}_{L^{\infty}_{x}(\mathbb{R})}\\
   \leq \frac{6\sqrt{2}K\delta_{0}}{(1+t)^{\frac{1}{2}}}\max_{s\in[0,T_{n}]}(1+s)^{\frac{1}{2}+\epsilon}\max_{\ell}\norm{\frac{\chi_{\ell,n-1}(s)\vec{u}_{*}(s)}{\langle x-y^{T_{n-1}}_{\ell,n-1}(s)\rangle^{\frac{3}{2}+\omega}}}_{L^{2}_{x}(\mathbb{R})}.
\end{multline}
For the last estimate above, we divide the integral in $\int_{0}^{\frac{t}{2}}$ and $\int_{\frac{t}{2}}^{t},$ and estimate them separately.  
\subsubsection{Interaction of unstable and scattering modes with potentials}\label{ssss}
We recall the function $Int_{\mathrm{unst},n-1}(s,\vec{u}_{*})$ defined at \eqref{intunstn-1b0}. Using Lemma \ref{interactt} and Lemma $4.1$ from \cite{dispanalysis1}, we can verify from the Cauchy-Schwarz inequality for any $j\in[m]$ that
\begin{multline*}
    \max_{q\in\{1,2\}}\norm{V^{T_{n-1}}_{h,\sigma_{n-1}}(s,x)[P_{c,\sigma^{T_{n-1}}_{n-1}}(s)\vec{u}_{*}(s)-P_{c,h,\sigma^{T_{n-1}}_{n-1}}(s)\vec{u}_{*}(s)]}_{L^{q}_{x}(\mathbb{R})}\\
    \lesssim e^{{-}\min_{j,\ell}\alpha_{j,n-1}(T_{n-1})[y^{T_{n-1}}_{\ell,\sigma_{n-1}}(t)-y^{T_{n-1}}_{\ell+1,\sigma_{n-1}}(t)]}\norm{P_{c,\sigma^{T_{n-1}}_{n-1}}(s)\vec{u}_{*}(s,x)}_{L^{2}_{x}(\mathbb{R})}. 
\end{multline*}
\par Next, using Lemma \ref{interactt}, we can deduce the following estimate
\begin{multline*}
\max_{q\in\{1,2\}}\norm{\sum_{h=1}^{m}\sum_{j=1,j\neq h}^{m} b_{h,+,*}(s)V^{T_{n-1}}_{j,\sigma_{n-1}}(s,x)e^{i\theta^{T_{n-1}}_{\sigma,\ell}(s,x)}\vec{Z}_{+}\left(\alpha_{h}(T_{n-1}),x-y^{T_{n-1}}_{h,\sigma_{n-1}}(s)\right)}_{L^{q}_{x}(\mathbb{R})}\\
\begin{aligned}
\lesssim & e^{{-}\min_{j,\ell}\alpha_{j,n-1}(T_{n-1})[y^{T_{n-1}}_{\ell,\sigma_{n-1}}(s)-y^{T_{n-1}}_{\ell+1,\sigma_{n-1}}(s)]}\max_{h\in[m]}\vert b_{h,+,*}(s)\vert\\
\lesssim & e^{{-}\min_{j,\ell}\alpha_{j,n-1}(T_{n-1})[y^{T_{n-1}}_{\ell,\sigma_{n-1}}(s)-y^{T_{n-1}}_{\ell+1,\sigma_{n-1}}(s)]}\norm{\vec{u}_{*}(s,x)}_{L^{2}_{x}(\mathbb{R})}.
\end{aligned}
\end{multline*}
In particular,
\begin{equation}\label{b**}
    \max_{s\geq 0,h}\vert b_{h,+,*}(s)\vert\lesssim \delta_{0}.
\end{equation}
Consequently, we obtain that
\begin{align*}
  \norm{\int_{0}^{t}\mathcal{S}(t)\circ \mathcal{S}^{{-}1}(s)P_{c,n-1}(s)\left[Int_{\mathrm{unst},n-1}(s,\vec{u}_{*})\right]\vec{u}_{*}(s,x)\,ds}_{L^{\infty}_{x}(\mathbb{R})}
  \lesssim &\frac{\delta_{0}^{2}\max_{s\in[0,t]}\norm{\vec{u}_{*}(s,x)}_{L^{2}_{x}(\mathbb{R})}}{(1+t)^{\frac{1}{2}}}\\
  \ll & \frac{\delta_{0}}{(1+t)^{\frac{1}{2}}}.
\end{align*}

\subsubsection{Localized nonlinear terms}
Let
\begin{multline}\label{quadratic function}
\begin{aligned}
    F_{2}(t,\sigma_{n-1},\vec{u}_{n-1},x)=&F\left(\sum_{\ell} e^{i(\frac{v_{\ell,n-1}(t)x}{2}+\gamma_{\ell,n-1}(t))}\phi_{\alpha_{\ell,n-1}(t)}(x-y_{\ell,n-1}(t))+\vec{u}_{n-1}(t)\right)\\
   &{-}F\left(\sum_{\ell} e^{i(\frac{v_{\ell,n-1}(t)x}{2}+\gamma_{\ell,n-1}(t))}\phi_{\alpha_{\ell,n-1}(t)}(x-y_{\ell,n-1}(t))\right)
\\&{-}\sum_{\ell=1}^{m}F^{\prime}\left(^{i(\frac{v_{\ell,n-1}(t)x}{2}+\gamma_{\ell,n-1}(t))}\phi_{\alpha_{\ell,n-1}(t)}(x-y_{\ell,n-1}(t))\right)\vec{u}_{n-1}\\
&{+}\sigma_{3}\begin{bmatrix}
   \vert \vec{u}_{n-1}(t) \vert^{2k} \vec{u}_{n-1}(t)
\end{bmatrix}.
\end{aligned}
\end{multline}
We estimate the $L^{\infty}_{x}$ norm of the following term. 
\begin{equation*}
{-}i\int_{0}^{t}\mathcal{S}(t)\circ \mathcal{S}^{{-}1}(s)P_{c,n-1}(s) F_{2}(s,\sigma_{n-1},\vec{u}_{n-1})\,ds,
\end{equation*}
Since we are assuming that $\vec{u}_{n-1}$ satisfies Proposition \ref{undecays} for $n-1,$ the function satisfies
\begin{equation}\label{linftyn-1}
   \norm{\vec{u}_{n-1}(t,x)}_{L^{\infty}_{x}(\mathbb{R})}\leq \frac{\delta_{0}}{(1+t)^{\frac{1}{2}}} \text{, for any $t\geq 0.$}
\end{equation}
From the fact that $F(0)=F^{\prime}(0)=F^{\prime\prime}(0)=0$ from \eqref{Fdefinition} and the assumption $k>2,$ we can deduce from \eqref{quadratic function} and the fundamental theorem of calculus that
\begin{multline}\label{est000}
  \max_{h\in\{0,1\}} \left\vert \frac{\partial^{h} }{\partial x^{h}} F_{2}(s,\sigma_{n-1},\vec{u}_{n-1},x)\right\vert\\
  \begin{aligned}
 \lesssim  &  \max_{h}\left\vert \frac{\partial^{h} }{\partial x^{h}}\vec{u}_{n-1}(s,x)\right\vert^{2}\left[\max_{\ell}\vert\phi_{\alpha_{\ell,n-1}(t)}(x-y_{\ell,n-1}(t))\vert\right]\\ &{+} \max_{h\in\{0,1\}}\left\vert \frac{\partial^{h} }{\partial x^{h}}\vec{u}_{n-1}(s,x)\right\vert \min_{\ell\neq j} \phi_{\alpha_{\ell,n-1}(t)}(x-y_{\ell,n-1}(t))\phi_{\alpha_{j,n-1}(t)}(x-y_{j,n-1}(t)) \\
  &{+}\left\vert\vec{u}_{n-1}(s,x)\right\vert \left\vert \frac{\partial^{h} }{\partial x^{h}}\min_{\ell\neq j} \phi_{\alpha_{\ell,n-1}(t)}(x-y_{\ell,n-1}(t))\phi_{\alpha_{j,n-1}(t)}(x-y_{j,n-1}(t)) \right\vert,  
\end{aligned}
\end{multline}
for any $t\in [0,T_{n}].$
\par Consequently, using Cauchy-Schwarz inequality, estimate \eqref{dyn11} for $n-1,$ Corollary \ref{quadratipotentialestimate}, and Lemma \ref{interactt}, we obtain for
\begin{equation*}
    \beta=\min_{\ell,j}\frac{\alpha_{j}(0)[v_{\ell}(0)-v_{\ell+1}(0)]}{2}>0
\end{equation*}
the following estimate
\begin{multline*}
 \norm{\int_{0}^{t}\mathcal{S}(t)\circ \mathcal{S}^{{-}1}(s)P_{c,n-1}(s)F_{2}(s,\sigma_{n-1},\vec{u}_{n-1})\,ds}_{L^{\infty}_{x}(\mathbb{R})}\\
 \begin{aligned}
 \leq & C(v,\alpha)\int_{0}^{t}\frac{1}{(t-s)^{\frac{1}{2}}}\max_{\ell}\norm{\frac{\chi_{\ell,n-1}(s)\vec{u}_{n-1}(s)}{\langle x-y^{T_{n-1}}_{\ell,n-1}(s)\rangle^{\frac{3}{2}+\omega}}}_{L^{2}_{x}(\mathbb{R})}^{2}\,ds\\
  &{+} C(v,\alpha)\int_{0}^{t}\frac{\delta_{0} e^{{-}\beta s}}{(t-s)^{\frac{1}{2}}}\max_{\ell}\norm{\frac{\chi_{\ell,n-1}(s)\vec{u}_{n-1}(s)}{\langle x-y^{T_{n-1}}_{\ell,n-1}(s)\rangle^{\frac{3}{2}+\omega}}}_{L^{2}_{x}(\mathbb{R})},
\end{aligned}
\end{multline*}
for any $t\in[0,T_{n}]$.
Therefore, since $\vec{u}_{n-1}(t)$ satisfies the following assumption in Proposition \ref{undecays} 
\begin{equation}\label{weightedassum}
\norm{\frac{\chi_{\ell,n-1}(t,x)\vec{u}_{n-1}(t)}{\left\langle x-y^{T_{n-1}}_{\ell,n-1}(t)\right\rangle^{\frac{3}{2}+\omega} }}_{L^{2}_{x}}\leq \frac{\delta_{0}}{(1+t)^{\frac{1}{2}+\epsilon}} \text{, for any $t\geq 0,$}
\end{equation}
with  $\epsilon>\frac{3}{4},$ we deduce that there exists a constant $K>1$ satisfying
\begin{align*}
\norm{\int_{0}^{t}\mathcal{S}(t)\circ \mathcal{S}^{{-}1}(s)P_{c,n-1}(s)\left[F_{2}(s,\sigma_{n-1},\vec{u}_{n-1})\right]\,ds}_{L^{\infty}_{x}(\mathbb{R})}
 \leq & K\int_{0}^{t}\frac{1}{(t-s)^{\frac{1}{2}}}\frac{\delta_{0}^{2}}{(1+s)^{1+\frac{3}{4}}}\,ds\\ 
 \leq & K\delta_{0}^{2}\left[\min\left(\frac{\sqrt{2}}{t^{\frac{1}{2}}},2t^{\frac{1}{2}}\right)+\frac{4\sqrt{2}}{(1+t)^{\frac{5}{4}}}\right].
\end{align*}
In particular,
\begin{equation}\label{linftyquadratic}
\norm{\int_{0}^{t}\mathcal{S}(t)\circ \mathcal{S}^{{-}1}(s)P_{c,n-1}(s)\left[F_{2}(s,\sigma_{n-1},\vec{u}_{n-1})\right]\,ds}_{L^{\infty}_{x}(\mathbb{R})}\ll \frac{\delta_{0}}{(1+t)^{\frac{1}{2}}}
\end{equation}
if $\delta_{0}>0$ is small enough.
\subsubsection{Full nonlinear term}
In this subsection,  we consider the following term.
\begin{equation*}
{-}i\int_{0}^{t}\mathcal{S}(t)\circ \mathcal{S}^{{-}1}(s)P_{c,n-1}(s)\left[\vert\vec{u}_{n-1}(s,x)\vert^{2k}\vec{u}_{n-1}(s,x) \right]\,ds.    
\end{equation*}
By the  assumptions    $\vec{u}_{n-1}$ satisfying 
\begin{equation}\label{l2linftyn-1}
\norm{\vec{u}_{n-1}(t)}_{L^{2}_{x}(\mathbb{R})}\leq \delta_{0},\,\, \norm{\vec{u}_{n-1}(t)}_{L^{\infty}_{x}(\mathbb{R})}\leq \frac{\delta_{0}}{(1+t)^{\frac{1}{2}}},
\end{equation} for all $t\geq 0$, we can verify from Theorem \ref{Decesti1} that
\begin{multline*}
\norm{\int_{0}^{t}\mathcal{S}(t)\circ \mathcal{S}^{{-}1}(s)P_{c,n-1}(s)\left[\vert \vec{u}_{n-1}(s,x)\vert^{2k}\vec{u}_{n-1}(s,x) \right]\,ds}_{L^{\infty}_{x}(\mathbb{R})}
  \\
  \begin{aligned}
  \leq  & C(v,\alpha)\max_{j\in\{1,2\}}\int_{0}^{t}\frac{\norm{\vec{u}_{n-1}(s,x)}_{L^{2}_{x}(\mathbb{R})}^{j}\norm{\vec{u}_{n-1}(x)}_{L^{\infty}_{x}(\mathbb{R})}^{2k+1-j}}{(t-s)^{\frac{1}{2}}}\,ds\\
  \leq  &  C(v,\alpha)\int_{0}^{t}\frac{\delta_{0}^{2k+1}}{(t-s)^{\frac{1}{2}}(1+s)^{k-\frac{1}{2}}}\,ds,
\end{aligned}
\end{multline*}
for some constant $C(v,\alpha)>1$ depending only on $\{(v_{\ell}(0),\alpha_{\ell}(0))\}_{\ell\in[m]}.$
\par In conclusion, we can verify that if $\min_{\ell} y_{\ell}-y_{\ell+1}$ is large enough, $k\geq 2$, and $\delta\ll 1,$ then we conclude that
\begin{equation}\label{linftyfullnonlinear}
\norm{\int_{0}^{t}\mathcal{S}(t)\circ \mathcal{S}^{{-}1}(s)P_{c,n-1}(s)\left[\vert \vec{u}_{n-1}(s,x)\vert^{2k}\vec{u}_{n-1}(s,x) \right]\,ds}_{L^{\infty}_{x}(\mathbb{R})}\ll \frac{\delta_{0}^{2}}{(1+t)^{\frac{1}{2}}},
\end{equation}
for all $t\in [0,T_{n}].$
\subsubsection{ODE terms}\label{ODEsection}
First, for $\sigma_{*}(t)=\{(v_{\ell,*}(t),y_{\ell,*}(t),\alpha_{\ell,*}(t),\gamma_{\ell,*}\}_{\ell\in [m]}$ and each $\ell\in [m],$ we consider the following set of functions 
\begin{align}\label{Omegadt}
    \Omega\dot\sigma_{*}(t)=&\Bigg\{(\dot y_{\ell,\sigma_{*}}(t)-v_{\ell,\sigma_{*}}(t))e^{i\theta_{\ell,\sigma_{n-1}}(t)\mathfrak{p}_3}\begin{bmatrix}
i\partial_{x}\phi_{\alpha_{\ell,n-1}(t)}(x-y_{\ell,n-1}(t))\\
i\partial_{x}\phi_{\alpha_{\ell,n-1}(t)}(x-y_{\ell,n-1}(t))
 \end{bmatrix}\\ \nonumber
    &\dot v_{\ell,\sigma_{*}}(t)e^{i\theta_{\ell,\sigma_{n-1}}(t)\mathfrak{p}_3}\begin{bmatrix}
       \frac{(x-y_{\ell,n-1}(t))}{2}e^{i(\frac{v_{\ell,n-1}(t)x}{2}+\gamma_{\ell}(t))}\phi_{\alpha_{\ell}(t)}(x-y_{\ell,n-1}(t))\\
       {-}\frac{(x-y_{\ell,n-1}(t))}{2}e^{{-}i(\frac{v_{\ell,n-1}(t)x}{2}+\gamma_{\ell,n-1}(t))}\phi_{\alpha_{\ell}(t)}(x-y_{\ell,n-1}(t)
   \end{bmatrix}\\ \nonumber
&\dot \alpha_{\ell,\sigma_{*}}(t)e^{i\theta_{\ell,\sigma_{n-1}}(t,x)\mathfrak{p}_3}\begin{bmatrix}
i\partial_{\alpha}\phi_{\alpha_{\ell,n-1}(t)}(x-y_{\ell,n-1}(t))\\
i\partial_{\alpha}\phi_{\alpha_{\ell,n-1}(t)}(x-y_{\ell,n-1}(t)
    \end{bmatrix}\\ \nonumber
    &\left(\dot \gamma_{\ell,\sigma_{*}}(t)-\alpha_{\ell,\sigma_{*}}(t)^{2}+\frac{v_{\ell,\sigma_{*}}(t)^{2}}{4}+\frac{y_{\ell,\sigma_{*}}(t)\dot v_{\ell,\sigma_{*}}(t)}{2}\right)e^{i\theta_{\ell,\sigma_{n-1}}(t,x)\mathfrak{p}_3}\begin{bmatrix}
    \phi_{\alpha_{\ell,n-1}(t)}(x-y_{\ell,n-1}(t))\\
{-}\phi_{\alpha_{\ell,n-1}(t)}(x-y_{\ell,n-1}(t).
\end{bmatrix}
\Bigg\}_{\ell\in [m]}.
\end{align}
Moreover, using Proposition \ref{phrootissmall}, we deduce that any element $\vec{f}(t,x)$ of $\Omega \dot\sigma_{*}(t)$ satisfies the following estimate
\begin{equation*}
\norm{\int_{0}^{t}\mathcal{S}(t)\circ \mathcal{S}^{{-}1}(s)P_{c,n-1}(s)\vec{f}(s,x)\,ds}_{L^{\infty}_{x}(\mathbb{R})}\leq C(v,\alpha)\int_{0}^{t}  \frac{\delta_{0} \max_{\ell}\vert \Lambda \dot \sigma_{*}(t)\vert}{(t-s)^{\frac{1}{2}}}\,ds,
\end{equation*}
from which we deduce using $(\vec{u}_{*},\sigma_{*})\in B_{n,\delta}$ and estimates \eqref{odes} satisfied by $\sigma_{*}$ that
\begin{align}\label{odeestimatelinfty}
\norm{\int_{0}^{t}\mathcal{S}(t)\circ \mathcal{S}^{{-}1}(s)P_{c,n-1}(s)\vec{f}(s,x)\,ds}_{L^{\infty}_{x}(\mathbb{R})}\leq & \frac{K(v,\alpha)\delta_{0}[\max_{s\in [0,t]}\vert \Lambda \dot\sigma_{*}(s) \vert\langle s \rangle^{1+2\epsilon}]}{(1+t)^{\frac{1}{2}}}\\
\leq  & \frac{K(v,\alpha)\delta_{0}^{2}}{(1+t)^{\frac{1}{2}}}\ll \frac{\delta_{0}}{(1+t)^{\frac{1}{2}}}.
\end{align}
\subsubsection{Interaction of multi-solitons}\label{interforce} Using Theorem \ref{Decesti1}, Proposition \ref{multisolitonsinteractionsize} and the choice of $\delta_{0}\in (0,1)$ in \eqref{deltachoice}, we can verify when $\min_{\ell}y_{\ell}(0)-y_{\ell+1}(0)>1$ is large enough that the function 
\begin{align}\label{interactionfunction}
    Int_{n-1}(t,x)=&F\left(\sum_{\ell=1}^{m}e^{i\mathfrak{p}_3\theta_{\ell,\sigma_{n-1}(t)}(t,x)}\phi_{\alpha_{\ell,n-1}(t)}(x-y_{\ell,n-1}(t))\right)\\&{-}\sum_{\ell=1}^{m}  F(e^{i\mathfrak{p}_3\theta_{\ell,\sigma_{n-1}(t)}(t,x)}\phi_{\alpha_{\ell,n-1}(t)}(x-y_{\ell,n-1}(t)))
\end{align}
satisfies
\begin{equation*}
\norm{\int_{0}^{t}\mathcal{S}(t)\circ \mathcal{S}^{{-}1}(s)P_{c,n-1}(s)Int_{n-1}(s,x)\,ds}_{L^{\infty}_{x}(\mathbb{R})}\leq C\frac{\delta_{0}^{\frac{99}{90}}}{(1+t)^{\frac{1}{2}}}\ll \frac{\delta_{0}}{(1+t)^{\frac{1}{2}}}.
\end{equation*}

\subsubsection{Conclusion}
As a consequence of the previous subsections and the remark in Subsection \ref{linearpartt0}, we conclude from \eqref{decomp} and the equation satisfied by $\vec{u}$ that
\begin{equation}\label{linftyforu}
    \norm{\vec{u}(t)}_{L^{\infty}_{x}(\mathbb{R})}\leq \frac{\delta_{0}}{(1+t)^{\frac{1}{2}}} \text{, for all $t\in [0,T_{n}].$}
\end{equation}
\subsection{Localized  $L^{2}$ norm of $\vec{u}$}\label{weighteduu}
Now we examine the localized weighted $L^{2}$ norm of $\vec{u}$.
\subsubsection{Weighted $L^{2}$ norm of $\mathcal{U}_\sigma(t,0)\vec{u}(0,x)$} Using the last estimate in the statement of Theorem \ref{Decesti1}, the definition of $\vec{u}(0,x)$ in \eqref{u0form} and the hypotheses satisfied by $\vec{r}_{0},$ we can verify the existence of a constant $C>1$ depending on the initial data $\sigma(0)$ satisfying  
\begin{equation*}
   \norm{\frac{\chi_{\ell,n-1}(t)\mathcal{S}(t)\circ \mathcal{S}^{{-}1}(0)P_{c,n-1}(0)\vec{u}(0,x)}{\langle x-y^{T_{n-1}}_{\ell,n-1}(t)\rangle^{\frac{3}{2}+\omega}} }_{L^{2}_{x}(\mathbb{R})}\leq C\frac{\delta^{2}}{(1+t)^{\frac{1}{2}+\epsilon}}\ll \frac{\delta_{0}}{(1+t)^{\frac{1}{2}+\epsilon}}. 
\end{equation*}
\subsubsection{Weighted norm of $\vec{u}_{\mathrm{root}}(t),\, \vec{u}_{\mathrm{stab},\sigma^{T_{n-1}}_{n-1}}(t)$  $\vec{u}_{\mathrm{unst},\sigma^{T_{n-1}}_{n-1}}(t)$}
Using Lemma \ref{blemma} and estimate \eqref{l2root}, and the fact that the eigenfunctions of $\mathcal{H}_{1}$ are Schwartz functions with exponential decay, we can verify similarly to the previous subsection that
\begin{align*}
   \max_{\ell}\vert b_{\ell,+}(t)\vert\ll &\frac{\delta}{(1+t)^{\frac{1}{2}+\epsilon}},\\  
\max_{\ell}\norm{\frac{\chi_{\ell,n-1}(t)\vec{u}_{\mathrm{root}}(t)}{\langle x-y^{T_{n-1}}_{\ell,n-1}(t)\rangle^{\frac{3}{2}+\omega}}  }_{L^{2}_{x}(\mathbb{R})}\leq &C\frac{\delta_{0}^{2}}{(1+t)^{\frac{1}{2}+\epsilon}}+C\delta_{0}\max_{\ell}\norm{\frac{\chi_{\ell,n-1}(t)\vec{u}_{c}(t)}{\langle x-y^{T_{n-1}}_{\ell,n-1}(t) \rangle^{\frac{3}{2}+\omega} }}_{L^{2}_{x}(\mathbb{R})}, 
\end{align*}
when $t\in [0,T_{n}].$
\par Consequently, since Remark \ref{stabel} implies that the projection $$P_{c,\mathrm{stab},n-1}(t)=P_{c,\sigma^{T^{n-1}}_{n-1}}(t)+P_{\mathrm{stab},\sigma^{T^{n-1}}_{n-1}}(t)$$ satisfies the same decay estimates from the statement Theorem \ref{Decesti1} that the continuous projection $P_{c,\sigma^{T^{n-1}}_{n-1}},$ we can restrict to the estimate onto
\begin{equation*}
\max_{\ell}\norm{\frac{\chi_{\ell,n-1}(t)\vec{u}_{c}(t)}{\langle x-y^{T_{n-1}}_{\ell,\sigma_{n-1}}(t) \rangle^{\frac{3}{2}+\omega} }}_{L^{2}_{x}(\mathbb{R})}
\end{equation*}
only during the time interval $[0,T_{n}]$ to control
\begin{equation*}
\max_{\ell}\norm{\frac{\chi_{\ell,n-1}(t)\vec{u}(t)}{\langle x-y^{T_{n-1}}_{\ell,\sigma_{n-1}}(t)\rangle^{\frac{3}{2}+\omega}}  }_{L^{2}_{x}(\mathbb{R})}.
\end{equation*}
As before, we will do the analysis term bu term on the right-hand side of the Duhamel expansion  for $\vec{u}_c$, see \eqref{disperpart}.
\subsubsection{$[V^{T_{n-1}}_{\ell,\sigma_{n-1}}(s)-V_{\ell,\sigma_{n-1}}(s)]\vec{u}_{*}$ term}\label{wnormaldiffV}
We first consider the expression 
\begin{multline*}
  W_{j,\ell,n}(t,x)\\=  \frac{\chi_{j,n-1}(t,x)}{(1+\vert x-y^{T_{n-1}}_{j,\sigma_{n-1}}(t)\vert)^{\frac{3}{2}+\omega}}\int_{0}^{t}\mathcal{S}(t)\circ \mathcal{S}^{{-}1}(s)P_{c,n-1}(s)\left[V^{T_{n-1}}_{\ell,\sigma_{n-1}}(s,x)-V_{\ell,\sigma_{n-1}}(s,x)\right]\vec{u}_{*}(s,x)\,ds.
\end{multline*}
 First, using the weighted decay
estimate
\begin{equation*}
\norm{\chi_{\ell,n-1}(t,x)\frac{\mathcal{S}(t)(\vec{\phi})(x)}{\langle x-y^{T_{n-1}}_{j,\sigma_{n-1}}(t)\rangle}}_{L^{\infty}_{x}(\mathbb{R})},
\end{equation*}
we can verify from H\"older's inequality and Theorem \ref{Decesti1} that
$\mathcal{S}(t)(\vec{\phi})(x)$ satisfies 
\begin{multline}\label{ww1}
    \norm{\frac{\chi_{\ell,n-1}(t,x)\mathcal{S}(t)(\vec{\phi})(x)}{(1+\vert x-y^{T_{n-1}}_{\ell,n-1}(t)\vert)^{\frac{3}{2}+\omega}}}_{L^{2}_{x}(\mathbb{R})}
    \\
    \begin{aligned}
    \leq  & \frac{K (y_{1}-y_{m}+\tau)}{(t-\tau)^{\frac{3}{2}}} \norm{\mathcal{S}(\tau)(\vec{\phi})(x)}_{L^{1}_{x}(\mathbb{R})}\\ 
&{+}\frac{K}{(t-\tau)^{\frac{3}{2}}}\max_{\ell}\norm{\chi_{\ell,n-1}(\tau,x)(1+\vert x-y^{T_{n-1}}_{\ell,n-1}(\tau)\vert )\mathcal{S}(\tau)(\vec{\phi})(x)}_{L^{1}_{x}(\mathbb{R})}\\ 
&{+}\frac{Ke^{{-}\min_{j,\ell} \alpha_{j,n-1}(T_{n-1})[y^{T_{n-1}}_{\ell,n-1}(\tau)-y^{T_{n-1}}_{\ell+1,\sigma_{n-1}}(\tau)]}\norm{\mathcal{S}(\tau)(\vec{\phi})(x)}_{L^{2}_{x}(\mathbb{R})}}{(t-\tau)^{\frac{3}{2}}}
\end{aligned}
\end{multline}
In particular, since
\begin{align}\label{ww2}
    \norm{\frac{\chi_{\ell,n-1}(t,x)\mathcal{S}(t)(\vec{\phi})(x)}{(1+\vert x-y^{T_{n-1}}_{\ell,n-1}(t)\vert)^{\frac{3}{2}+\omega}}}_{L^{2}_{x}(\mathbb{R})}\leq &\min\left(\norm{\mathcal{S}(t)(\vec{\phi})(x)}_{L^{\infty}_{x}(\mathbb{R})},\norm{\mathcal{S}(t)(\vec{\phi})(x)}_{L^{2}_{x}(\mathbb{R})}\right)\\ \nonumber
    \leq &\min\Bigg[\frac{\norm{\mathcal{S}(\tau)(\vec{\phi})(x)}_{L^{1}_{x}(\mathbb{R})}}{(t-\tau)^{\frac{1}{2}}},\norm{\mathcal{S}(\tau)(\vec{\phi})(x)}_{L^{2}_{x}(\mathbb{R})}\Bigg],
\end{align}
we can deduce from \eqref{ww1} that there exists a constant $K>1$ depending on the set $\{(v_{\ell}(0),\alpha_{\ell}(0))\}_{\ell\in[m]}$ satisfying
\begin{multline}\label{ww3}
    \norm{\frac{\chi_{\ell,n-1}(t,x)\mathcal{S}(t)(\vec{\phi})(x)}{(1+\vert x-y^{T_{n-1}}_{\ell,n-1}(t)\vert)^{\frac{3}{2}+\omega}}}_{L^{2}_{x}(\mathbb{R})}
    \\
    \begin{aligned}
    \leq  & \frac{K (y_{1}(0)-y_{m}(0)+\tau)}{(1+y^{T_{n-1}}_{1,\sigma_{n-1}}(t)-y^{T_{n-1}}_{m,\sigma_{n-1}}(t))(1+t-\tau)^{\frac{1}{2}}} \norm{\mathcal{S}(\tau)(\vec{\phi})(x)}_{L^{1}_{x}(\mathbb{R})}\\  
&{+}\frac{K}{1+(t-\tau)^{\frac{3}{2}}}\left[\max_{\ell}\norm{\langle  x-y^{T_{n-1}}_{\ell,n-1}(\tau)\rangle\chi_{\ell}(\tau,x)\mathcal{S}(\tau)(\vec{\phi})(x)}_{L^{1}_{x}(\mathbb{R})}+\norm{\mathcal{S}(\tau)(\vec{\phi})(x)}_{L^{2}_{x}(\mathbb{R})}\right]\\ 
&{+}\frac{Ke^{{-}\min_{j,\ell} \alpha_{j,n-1}(T_{n-1})[y^{T_{n-1}}_{\ell,n-1}(\tau)-y^{T_{n-1}}_{\ell+1,\sigma_{n-1}}(\tau)]}\norm{\mathcal{S}(\tau)(\vec{\phi})(x)}_{L^{2}_{x}(\mathbb{R})}}{1+(t-\tau)^{\frac{3}{2}}}.
\end{aligned}
\end{multline}

Therefore, using the estimates \eqref{dyn11} satisfied by $\sigma_{n-1}$ for all $t\geq 0,$ and Theorem \ref{Decesti1}, and the fact that $\sigma_{n-1}(0)$ does not depend on $n,$ we can verify for any $j,\,\ell\in[m]$
the existence of a constant $K>1$ depending only on $\{(v_{\ell}(0),\alpha_{\ell}(0))\}_{\ell\in [m]}$ satisfying
\begin{multline}\label{weightedVlVinfty}
  \norm{\frac{\chi_{j,n-1}(t,x)}{(1+\vert x-y^{T_{n-1}}_{j,\sigma_{n-1}}(t) \vert)^{\frac{3}{2}+\omega}}\int_{0}^{t}\mathcal{S}(t)\circ \mathcal{S}^{{-}1}(s)P_{c,n-1}(s)\left[V^{T_{n-1}}_{\ell,\sigma_{n-1}}(s,x)-V_{\ell,\sigma_{n-1}}(s,x)\right]\vec{u}_{*}(s,x)\,ds}_{L^{2}_{x}(\mathbb{R})}\\
  \begin{aligned}
  \leq & K\int_{0}^{t}\frac{y_{1}(0)-y_{m}(0)+s}{(1+t-s)^{\frac{1}{2}}(y_{1}(0)-y_{m}(0)+t)}\norm{\left[V^{T_{n-1}}_{\ell,\sigma_{n-1}}(s,x)-V_{\ell,\sigma_{n-1}}(s,x)\right]\vec{u}_{*}(s,x)}_{L^{1}_{x}(\mathbb{R})}\,ds\\
   &{+} K\max_{j\in[m]}\int_{0}^{t}\frac{\norm{\chi_{j,n-1}(s,x)\left\vert x-y^{T_{n-1}}_{j,\sigma_{n-1}}(s)\right\vert\left[V^{T_{n-1}}_{\ell,\sigma_{n-1}}(s,x)-V_{\ell,\sigma_{n-1}}(s,x)\right]\vec{u}_{*}(s)}_{L^{1}_{x}(\mathbb{R})}}{1+(t-s)^{\frac{3}{2}}}\,ds\\
   &{+} K\max_{\ell}\int_{0}^{t}\frac{\norm{\left[V^{T_{n-1}}_{\ell,\sigma_{n-1}}(s,x)-V_{\ell,\sigma_{n-1}}(s,x)\right]\vec{u}_{*}(s)}_{L^{2}_{x}(\mathbb{R})}}{1+(t-s)^{\frac{3}{2}}}\,ds
  \\&{+}K\int_{0}^{t}\frac{e^{{-}\min_{j,\ell}\alpha_{j,n-1}(T_{n-1})[y^{T_{n-1}}_{\ell,n-1}(s)-y^{T_{n-1}}_{\ell+1,\sigma_{n-1}}(s)])}\norm{\vec{u}_{*}(s)}_{L^{2}_{x}(\mathbb{R})}}{1+(t-s)^{\frac{3}{2}}}\,ds.
  \end{aligned}
\end{multline}
\par Therefore, we can verify using Lemma \ref{dinftydt}, H\"older's inequality, and \eqref{weightedVlVinfty} for any $j,\,l\in[m]$ that
\begin{multline}\label{weightedVlVinfty2}
  \norm{\frac{\chi_{j,n-1}(t,x)}{(1+\vert x-y^{T_{n-1}}_{\sigma_{n-1}}(t) \vert)^{\frac{3}{2}+\omega}}\int_{0}^{t}\mathcal{S}(t)\circ \mathcal{S}^{{-}1}(s)P_{c,n-1}(s)\left[V^{T_{n-1}}_{\ell,\sigma_{n-1}}(s,x)-V_{\ell,\sigma_{n-1}}(s,x)\right]\vec{u}_{*}(s,x)\,ds}_{L^{2}_{x}(\mathbb{R})}\\
  \begin{aligned}
 \leq & C\int_{0}^{t}\frac{\delta_{0} (y_{1}(0)-y_{m}(0)+s)}{(1+s)^{2\epsilon-1}(1+(t-s))^{\frac{1}{2}}(y_{1}(0)-y_{m}(0)+t)}\max_{\ell}\norm{\frac{\chi_{\ell,n-1}(s,x)\vec{u}_{*}(s,x)}{\langle x-y^{T_{n-1}}_{\ell,n-1}(s)\rangle^{\frac{3}{2}+\omega}}}_{L^{2}_{x}(\mathbb{R})}\,ds\\
   &{+} C\int_{0}^{t}\frac{\delta_{0}}{(1+s)^{2\epsilon-1}(1+(t-s)^{\frac{3}{2}})}\max_{\ell}\norm{\frac{\chi_{\ell,n-1}(s,x)\vec{u}_{*}(s,x)}{\langle x-y^{T_{n-1}}_{\ell,n-1}(s)\rangle^{\frac{3}{2}+\omega}}}_{L^{2}_{x}(\mathbb{R})}\,ds\\
  &{+}\frac{C \delta_{0}^{1+\frac{1}{20}}}{(1+t)^{\frac{3}{2}}}\max_{s\in [0,T_{n}]}\norm{\vec{u}_{*}(s)}_{L^{2}_{x}(\mathbb{R})},
  \end{aligned}
  \end{multline}
for some constant $C>1$ depending only on $m,\,\{(v_{\ell}(0),\alpha_{\ell}(0))\}_{\ell\in [m]}$ when $\delta\in (0,1)$ is small enough.
\par In conclusion, since $\epsilon\in (\frac{3}{4},1),$
and
\begin{equation*}
    \max_{\ell}\norm{\frac{\chi_{\ell,n-1}(s,x)\vec{u}_{*}(s,x)}{\langle x-y^{T_{n-1}}_{\ell,n-1}(s)\rangle^{\frac{3}{2}+\omega}}}_{L^{2}_{x}(\mathbb{R})}\lesssim \frac{\delta_{0}}{(1+t)^{\frac{1}{2}+\epsilon}} \text{, for any $t\in [0,T_{n}],$}
\end{equation*}
we can verify from Lemma \ref{interpol} and estimate \eqref{weightedVlVinfty2} that there exists a constant $C>1$ depending only on $\{(v_{\ell}(0),\alpha_{\ell}(0))\}_{\ell\in [m]}$ and $m$ satisfying
\begin{multline}\label{difVnormalweight}
\max_{\ell,j}\norm{\frac{\chi_{j,n-1}(t,x)}{(1+\vert x-y^{T_{n-1}}_{\sigma_{n-1}}(t) \vert)^{\frac{3}{2}+\omega}}\int_{0}^{t}\mathcal{S}(t)\circ \mathcal{S}^{{-}1}(s)P_{c,n-1}(s)\left[V^{T_{n-1}}_{\ell,\sigma_{n-1}}(s,x)-V_{\ell,\sigma_{n-1}}(s,x)\right]\vec{u}_{*}(s,x)\,ds}_{L^{2}_{x}(\mathbb{R})}\\
    \begin{aligned}
        \leq & C\delta_{0} \left[\frac{1}{(1+t)^{\frac{1}{2}+\epsilon}} \max_{\ell,s\in[0,T_{n}]}\langle s \rangle^{\frac{1}{2}+\epsilon}\norm{\frac{\chi_{\ell,n-1}(s,x)\vec{u}_{*}(s,x)}{\langle x-y^{T_{n-1}}_{\ell,n-1}(s)\rangle^{\frac{3}{2}+\omega}}}_{L^{2}_{x}(\mathbb{R})}\right] +  \frac{C \delta_{0}^{1+\frac{1}{20}}}{(1+t)^{\frac{3}{2}}}\max_{s\in [0,T_{n}]}\norm{\vec{u}_{*}(s)}_{L^{2}_{x}(\mathbb{R})}\\
        \ll & \frac{\delta_{0}}{(1+t)^{\frac{1}{2}+\epsilon}} \text{, for any $t\in [0,T_{n}],$}
    \end{aligned}
\end{multline}
since $(\vec{u}_{*},\sigma^{*})\in B_{\delta_{0},n}.$
\subsubsection{Localized nonlinear terms}
Using \eqref{quadratic function}, we consider the following function
\begin{equation*}
Q_{j,n-1}(t,x)=\frac{\chi_{j,n-1}(t,x)}{(1+\vert x-y^{T_{n-1}}_{j,\sigma_{n-1}}(t) \vert)^{\frac{3}{2}+\omega}}\int_{0}^{t}\mathcal{S}(t)\circ \mathcal{S}^{{-}1}(s)P_{c,n-1}(s)F_{2}(s,\sigma,\vec{u}_{n-1},x)\,ds,
\end{equation*}
where $j$ is an element of $\{1,2,\,...,\,m\}$ and $F_2$ is defined by \eqref{quadratic function}.  Applying the estimate \eqref{ww3}, we can verify the existence of a constant $C>1$ satisfying
\begin{multline}\label{t0}
\max_{j\in[m]}\norm{Q_{j,n-1}(t)}_{L^{2}_{x}(\mathbb{R})}\\
\begin{aligned}
\leq & C\int_{0}^{t}\frac{(y_{1}(0)-y_{m}(0)+s)\norm{F_{2}(s,\sigma,\vec{u}_{n-1},x)}_{L^{1}_{x}(\mathbb{R})}}{(y_{1}(0)-y_{m}+t)(1+t-s)^{^{\frac{1}{2}}}}\\&{+}C\max_{j\in [m]}\int_{0}^{t}\frac{\norm{\chi_{j,n-1}(s)\langle x-y^{T_{n-1}}_{j,\sigma_{n-1}}(s)\rangle F_{2}(s,\sigma,\vec{u}_{n-1},x)}_{L^{1}_{x}(\mathbb{R})}}{1+\vert t-s \vert^{\frac{3}{2}}}\,ds\\
&{+}C\int_{0}^{t}\frac{ \norm{F_{2}(s,\sigma,\vec{u}_{n-1},x)}_{L^{2}_{x}(\mathbb{R})}}{1+\vert t-s \vert^{\frac{3}{2}}}\,ds
\\&{+}C\int_{0}^{t}\frac{e^{{-}\min_{j,\ell}\alpha_{j,n-1}(T_{n-1})\left[y^{T_{n-1}}_{\ell,n-1}(s)-y^{T_{n-1}}_{\ell+1,\sigma_{n-1}}(s)\right]}}{1+\vert t-s \vert^{\frac{3}{2}}}\norm{F_{2}(s,\sigma,\vec{u}_{n-1},x)}_{L^{2}_{x}(\mathbb{R})}.
\end{aligned}
\end{multline}
Using Lemma \ref{diffcutV}, the definition of $F_{2}$ in \eqref{quadratic function}, \eqref{Fdefinition}, estimate $\norm{\vec{u}_{n-1}(t)}_{L^{\infty}_{x}(\mathbb{R})}\lesssim \delta$ obtained from \eqref{decay2} for $n-1$ by Sobolev's embedding and the fundamental theorem of calculus, the following estimate can be deduced.
\begin{multline}\label{t1}
  \max_{j\in [m]}\norm{\chi_{j,n-1}(s)\langle x-y^{T_{n-1}}_{j,\sigma_{n-1}}(s)\rangle F_{2}(s,\sigma,\vec{u}_{n-1},x)}_{L^{1}_{x}(\mathbb{R})}\\
  \begin{aligned}
  \leq & C\left(1+e^{{-}\frac{95}{100}\min_{\ell,j}\alpha_{j,n-1}(t)[y_{\ell,n-1}(t)-y_{\ell+1,n-1}(t)]}\right)\max_{\ell}\norm{\frac{\chi_{\ell,n-1}(s)\vec{u}_{n-1}(s,x)}{\langle x-y^{T_{n-1}}_{\ell,n-1}(s) \rangle^{\frac{3}{2}+\omega}}}_{L^{2}_{x}(\mathbb{R})}^{2}\\
  &{+}C e^{{-}\frac{95}{100}\min_{\ell,j}\alpha_{j,n-1}(t)[y_{\ell,n-1}(t)-y_{\ell+1,n-1}(t)]}\norm{\frac{\chi_{\ell,n-1}(s)\vec{u}_{n-1}(s,x)}{\langle x-y^{T_{n-1}}_{\ell,n-1}(s) \rangle^{\frac{3}{2}+\omega}}}_{L^{2}_{x}(\mathbb{R})} 
  \\\leq &2C \max_{\ell}\norm{\frac{\chi_{\ell,n-1}(s)\vec{u}_{n-1}(s,x)}{\langle x-y^{T_{n-1}}_{\ell,n-1}(s) \rangle^{\frac{3}{2}+\omega}}}_{L^{2}_{x}(\mathbb{R})}^{2}+C\delta_{0} \norm{\frac{\chi_{\ell,n-1}(s)\vec{u}_{n-1}(s,x)}{\langle x-y^{T_{n-1}}_{\ell,n-1}(s) \rangle^{\frac{3}{2}+\omega}}}_{L^{2}_{x}(\mathbb{R})}\leq \frac{K\delta_{0}^{2}}{(1+t)^{1+2\epsilon}},
  \end{aligned}
\end{multline}
for constants $C,K>1$ depending only on the set $\{(v_{\ell}(0),\alpha_{\ell}(0))\}$ and $m.$ In particular, the estimate above implies that
\begin{equation}\label{t2}
\norm{F_{2}(s,\sigma,\vec{u}_{n-1},x)}_{L^{1}_{x}(\mathbb{R})} \leq \frac{2K(m+1)\delta_{0}^{2}}{(1+t)^{1+2\epsilon}} .
\end{equation}
\par Furthermore, 
\begin{equation}\label{t22}
    \norm{F_{2}(s,\sigma,\vec{u}_{n-1},x)}_{L^{2}_{x}(\mathbb{R})}\lesssim\delta_{0}\max_{\ell}\norm{\frac{\chi_{\ell,n-1}(s)\vec{u}_{n-1}(s,x)}{\langle x-y^{T_{n-1}}_{\ell,n-1}(s) \rangle^{\frac{3}{2}+\omega}}}_{L^{2}_{x}(\mathbb{R})}+\norm{\frac{\chi_{\ell,n-1}(s)\vec{u}_{n-1}(s,x)}{\langle x-y^{T_{n-1}}_{\ell,n-1}(s) \rangle^{\frac{3}{2}+\omega}}}_{L^{2}_{x}(\mathbb{R})}^{2},
\end{equation}
and
\begin{align}\label{t23}
    \norm{F_{2}(s,\sigma,\vec{u}_{n-1},x)}_{H^{1}_{x}(\mathbb{R})}\lesssim &\delta_{0}\max_{\ell\in[m],j\in\{0,1\}}\norm{\frac{\chi_{\ell,n-1}(s)\vec{u}_{n-1}(s,x)}{\langle x-y^{T_{n-1}}_{\ell,n-1}(t) \rangle^{1+\frac{1-j}{2}+\frac{j(p^{*}-2)}{2p^{*}}+\omega}}}_{L^{2}_{x}(\mathbb{R})}\\&{+}\max_{\ell\in[m],j\in\{0,1\}}\norm{\frac{\chi_{\ell,n-1}(s)\vec{u}_{n-1}(s,x)}{\langle x-y^{T_{n-1}}_{\ell,n-1}(t) \rangle^{1+\frac{1-j}{2}+\frac{j(p^{*}-2)}{2p^{*}}+\omega}}}_{L^{2}_{x}(\mathbb{R})}^{2}.
\end{align}
\par Consequently, using estimates \eqref{t1}, \eqref{t2}, the definition of $\delta$ in \eqref{deltachoice}, \eqref{dyn12} for $n-1,$ we can deduce applying Lemma \ref{interpol} in the inequality \eqref{t0} that
\begin{equation}\label{quadraticweight0}
\max_{j}\norm{Q_{j,n-1}(s,x)}_{L^{2}_{x}(\mathbb{R})}\leq C\frac{\delta_{0}^{2}}{(1+s)^{\frac{1}{2}+\epsilon}}\ll \frac{\delta_{0}}{(1+s)^{\frac{1}{2}+\epsilon}}.
\end{equation}
The inequality \eqref{quadraticweight0} will be used in the next subsection to prove the Proposition \ref{Aiscontraction}.
\subsubsection{Full nonlinear term}
From the decay estimate \eqref{ww3}, it is enough to analyze
\begin{align}\label{non1}
    &\int_{0}^{t}\frac{y_{1}-y_{m}+s}{1+\vert t-s\vert^{\frac{3}{2}}}\norm{\vec{u}_{n-1}(s,x)^{2k+1}}_{L^{1}_{x}(\mathbb{R})}\,ds\\ \label{non2}
    &\int_{0}^{t} \frac{1}{1+\vert t-s\vert^{\frac{3}{2}}}\max_{\ell}\norm{\chi_{\ell,n-1}(s,x)\left\langle x-y^{T_{n-1}}_{\ell,n-1}(s)\right\rangle\vec{u}_{n-1}(s,x)^{2k+1}}_{L^{1}_{x}(\mathbb{R})}\,ds\\ \label{non3}
   & \int_{0}^{t}\frac{(1+e^{{-}\min_{j,\ell}\alpha^{T_{n-1}}_{j,\sigma_{n-1}}(t)[y^{T_{n-1}}_{\ell,n-1}(t)-y^{T_{n-1}}_{\ell+1,\sigma_{n-1}}(t)]})\norm{\vec{u}_{n-1}(s,x)^{2k+1}}_{L^{2}_{x}(\mathbb{R})}}{1+\vert t-s \vert^{\frac{3}{2}}}\,ds,
\end{align}
\par First, since $\norm{\vec{u}_{n-1}(s)}_{L^{\infty}_{x}(\mathbb{R})}\lesssim \frac{\delta_{0}}{(1+s)^{\frac{1}{2}}},$ we can verify that
\begin{equation*}
    \norm{\vec{u}_{n-1}(s,x)^{2k+1}}_{L^{1}_{x}(\mathbb{R})}\lesssim \frac{\delta_{0}^{2k-1}}{(1+s)^{k-\frac{1}{2}}}\norm{\vec{u}_{n-1}(s)}_{L^{2}_{x}(\mathbb{R})}^{2},\, \norm{\vec{u}_{n-1}(s,x)^{2k+1}}_{L^{2}_{x}(\mathbb{R})}\lesssim \frac{\delta_{0}^{2k}}{(1+s)^{k}}\norm{\vec{u}_{n-1}(s)}_{L^{2}_{x}(\mathbb{R})}.
\end{equation*}
Furthermore, using \eqref{decay4} for $n-1,$ we observe that $\vec{u}_{n-1}$ satisfies 
\begin{equation*}
    \max_{\ell}\frac{\norm{\chi_{\ell,n-1}(t,x)\vert x-y^{T_{n-1}}_{\ell,n-1}(t)\vert\vec{u}_{n-1}(t)}_{L^{2}_{x}(\mathbb{R})}}{\max_{\ell}(v_{\ell}(0)-v_{\ell+1}(0))(1+t)}\leq C\delta_{0}
\end{equation*}
for a constant $C>1$ and any $t\in[0,T_{n}].$

Therefore, using $\norm{\vec{u}_{n-1}(s)}_{L^{\infty}_{x}(\mathbb{R})}\leq \frac{\delta_{0}}{(1+t)^{\frac{1}{2}}}$ and $\norm{\vec{u}_{n-1}(s)}_{L^{2}_{x}(\mathbb{R})}\leq \delta_{0},$ it can be deduced that
 \begin{equation*}
    \max_{\ell}\norm{\chi_{\ell,n-1}(s)\left\langle x-y^{T_{n-1}}_{\ell,n-1}(s)\right\rangle\vec{u}_{n-1}(s,x)^{2k+1}}_{L^{1}_{x}(\mathbb{R})}\leq \frac{C\max_{\ell}(v_{\ell}(0)-v_{\ell+1}(0))(1+s)\delta_{0}^{2k}}{(1+s)^{k-\frac{1}{2}}}.
\end{equation*}
From Lemma \ref{interpol}, one has
\begin{equation*}
\int_{0}^{t}\frac{\delta^{2k-1}}{(1+\vert t-s \vert^{\frac{3}{2}})(1+s)^{k-\frac{3}{2}}}\,ds\leq C\left[\frac{\delta_{0}^{2k-1}}{(1+t)^{k-\frac{3}{2}}}+\frac{\delta_{0}^{2k-1}}{(1+t)^{\frac{3}{2}}}\right],
\end{equation*}
and $k>\frac{3}{2}+1,$ we deduce that \eqref{non1} and \eqref{non2} are bounded above by
\begin{equation*}
C\max_{\ell}(v_{\ell}-v_{\ell+1})\frac{\delta_{0}^{2k-1}}{(1+t)^{k-\frac{3}{2}}}\ll \frac{\delta_{0}}{(1+t)^{\frac{1}{2}+\epsilon}}.
\end{equation*}
\par Finally, using the assumptions
\begin{align*}
    \max_{t\geq 0}\norm{\vec{u}_{n-1}(t)}_{H^{1}_{x}}\leq \delta_{0},
\end{align*}
we can check that
$\max\left(\frac{\delta_{0}^{2}}{(1+t)^{k}},\frac{\delta_{0}^{2}}{(1+t)^{\frac{3}{2}}}\right)$ is an upper bound for
\begin{equation*}
     \int_{0}^{t}\frac{\norm{\vec{u}_{n-1}(s,x)^{2k+1}}_{L^{2}_{x}(\mathbb{R})}}{1+\vert t-s \vert^{\frac{3}{2}}}\,ds,
\end{equation*}
and since $2k+1>6.5,$ we can conclude that all the terms \eqref{non1}, \eqref{non2}, and \eqref{non3} are bounded above by
\begin{equation*}
    \frac{C\delta_{0}^{2}}{(1+t)^{\frac{1}{2}+\epsilon}}\ll \frac{\delta_{0}}{(1+t)^{\frac{1}{2}+\epsilon}} \text{, for any $t\in [0,T_{n}].$}
\end{equation*}
\subsubsection{Weighted $L^{2}$ estimate of multi-solitons interaction and ODEs}\label{l2normalweightforc}
The proof that 
{\footnotesize\begin{align}\label{weightedForce}
   &\int_{0}^{t}\max_{\vec{f}(s)\in\Omega \dot\sigma_{*}(s)\cup\{Int_{n-1}(s)\}\cup\{Int_{\mathrm{unst},n-1}(s,\vec{u}_{*})\}}\left(\norm{\frac{\chi_{\ell,n-1}(t)P_{c,n-1}(t)\mathcal{U}_\sigma(t,s)\vec{f}(s,x)}{\langle x-y^{T_{n-1}}_{\ell,\sigma_{n-1}}(t) \rangle^{\frac{3}{2}+\omega} }}_{L^{2}_{x}(\mathbb{R})}\right)\,ds\\
   & \quad\quad\quad\quad\quad\quad\quad \quad\quad\quad\quad\quad\quad\quad\quad\quad\quad\quad\quad\ll\frac{\delta_{0}}{(1+t)^{\frac{1}{2}+\epsilon}} 
\end{align}}
follows from Theorem \ref{Decesti1} and \eqref{b**}, and it is completely analogous to the argument used in \S \ref{interforce}, \S \ref{ssss} and \S\ref{ODEsection}, which was developed using the exponential decay of the space derivatives of $Int(t,x),\,Int_{\mathrm{unst},n-1}(t),$ and any function $\vec{f}(t,x)\in\Omega \dot\sigma_{*}(t).$
\subsubsection{Conclusion}
In conclusion, using \eqref{unequation}, we can verify from all estimates obtained in Subsection \ref{weighteduu} that \eqref{decay4} holds for $\vec{u}(t)$ for all $t\in [0,T_{n}].$
\subsection{Localized  $L^{2}$ norm of $\partial_{x}\vec{u}(t,x)$}
Now we check weighted estimates for the derivative of $\partial_{x}\vec{u}(t,x)$.
\subsubsection{Weighted $L^{2}$ norm of $\partial_{x}\mathcal{U}_\sigma(t,0)P_{c,\sigma^{T_{n-1}}_{n-1}}\vec{u}(0,x)$}
Using the hypothesis \eqref{epsilonBsmall} satisfied by $r_{0}(x),$ and the definition $\vec{u}(0,x)$ in \eqref{u0form}, we can verify from Theorem \ref{interpolation est.} that if $\epsilon>0$ is much smaller than $\delta^{2}$ and $p\in(1,2)$ is close enough to $1,$ then
\begin{equation}\label{weigtheddxu0}
 \max_{\ell\in [m]} \norm{\frac{\chi_{\ell,n-1}(t,x)\partial_{x}\mathcal{U}_\sigma(t,0)P_{c,\sigma^{T_{n-1}}_{n-1}}(0)\vec{u}(0,x)}{\langle x-y^{T_{n-1}}_{\ell,n-1}(t) \rangle^{1+\frac{p^{*}-2}{2p^{*}}+\omega}}}_{L^{2}_{x}(\mathbb{R})}\leq C(1+y_{1}(0)-y_{m}(0))\frac{\delta^{2}}{(1+t)^{\frac{1}{2}+\epsilon}}\ll \frac{\delta_{0}}{(1+t)^{\frac{1}{2}+\epsilon}}.  
\end{equation}
\subsubsection{Weighted $L^{2}$ norm of derivatives of $\vec{u}_{\mathrm{root},\sigma^{T_{n-1}}_{n-1}},\,\vec{u}_{\mathrm{unst},\sigma^{T_{n-1}}_{n-1}},\, \vec{u}_{\mathrm{stab},\sigma^{T_{n-1}}_{n-1}}$}\label{secinih}
From the definition of $\vec{u}_{\mathrm{root},\sigma^{T_{n-1}}_{n-1}}(t,x),\,\vec{u}_{\mathrm{unst},\sigma^{T_{n-1}}_{n-1}}(t,x),$ and $\vec{u}_{\mathrm{stab},\sigma^{T_{n-1}}_{n-1}}(t,x),$ it is standard for any element $h\in\{\mathrm{root},\mathrm{stab},\mathrm{unst}ab\}$ that
\begin{equation}\label{ppequiv}
  \norm{\vec{u}_{h,\sigma^{T_{n-1}}_{n-1}}}_{H^{1}_{x}(\mathbb{R})}\sim \norm{\vec{u}_{h,\sigma^{T_{n-1}}_{n-1}}}_{L^{2}_{x}(\mathbb{R})}\sim \norm{\vec{u}_{h,\sigma^{T_{n-1}}_{n-1}}}_{H^{2}_{x}(\mathbb{R})},   
\end{equation}
since each space $\Ra P_{h,\sigma^{T_{n-1}}_{n-1}}$ is finite dimensional with a basis of Schwartz functions having all of their derivatives on $x$ decaying exponentially.
\par Therefore, using Lemma \ref{blemma} and estimate \eqref{l2root}, we deduce that
\begin{equation}\label{m1}
    \max_{h\in\{\mathrm{root},\mathrm{unst}\}}\norm{\frac{\chi_{\ell,n-1}(t,x)\partial_{x}P_{h,\sigma^{T_{n-1}}_{n-1}}(t)\vec{u}(t,x)}{\langle x-y^{T_{n-1}}_{\ell,n-1}(t) \rangle^{1+\frac{p^{*}-2}{2p^{*}}+\omega}}}_{L^{2}_{x}(\mathbb{R})}\leq C(1+y_{1}(0)-y_{m}(0))\frac{\delta_{0}^{2}}{(1+t)^{\frac{1}{2}+\epsilon}}\ll \frac{\delta_{0}}{(1+t)^{\frac{1}{2}+\epsilon}}.
\end{equation}
\par Next, since $\vec{u}(t)$ satisfies the decomposition formula \ref{decomp}, we can verify from the definition of $P_{\mathrm{stab},\sigma^{T_{n-1}}_{n-1}}$ and the conclusion of Subsection \ref{weighteduu} that
\begin{equation}\label{stabl2l2}
   \norm{P_{\mathrm{stab},\sigma^{T_{n-1}}_{n-1}}(t)\vec{u}(t,x)}_{L^{2}_{x}(\mathbb{R})}\lesssim \max_{\ell} \norm{\frac{\chi_{\ell,n-1}(t,x)\vec{u}(t,x)}{\langle x-y^{T_{n-1}}_{\ell,n-1}(t) \rangle^{\frac{3}{2}+\omega}}}_{L^{2}_{x}(\mathbb{R})}\lesssim \frac{\delta_{0}^{2}}{(1+t)^{\frac{1}{2}+\epsilon}}\ll \frac{\delta_{0}}{(1+t)^{\frac{1}{2}+\epsilon}},   
\end{equation}
from which we deduce using equation \eqref{ppequiv} for $h=stab$ that
\begin{equation*}
\max_{\ell}\norm{\frac{\chi_{\ell,n-1}(t,x)\partial_{x}P_{\mathrm{stab},\sigma^{T_{n-1}}_{n-1}}(t)\vec{u}(t,x)}{\langle x-y^{T_{n-1}}_{\ell,n-1}(t) \rangle^{1+\frac{p^{*}-2}{2p^{*}}+\omega}}}_{L^{2}_{x}(\mathbb{R})}\ll \frac{\delta_{0}}{(1+t)^{\frac{1}{2}+\epsilon}},
\end{equation*}
when $t\in [0,T_{n}].$
\par Therefore, it is enough to prove for all $t\in [0,T_{n}]$ that
\begin{equation}\label{poo}
    \norm{\frac{\chi_{\ell,n-1}(t,x)P_{c,\sigma^{T_{n-1}}_{n-1}}(t)\vec{u}(t,x)}{\langle x-y^{T_{n-1}}_{\ell,n-1}(t) \rangle^{\frac{3}{2}+\omega}}}_{L^{2}_{x}(\mathbb{R})}\lesssim (1+y_{1}(0)-y_{m}(0))\frac{\delta_{0}^{2}}{(1+t)^{\frac{1}{2}+\epsilon}}\ll \frac{\delta_{0}}{(1+t)^{\frac{1}{2}+\epsilon}},
\end{equation}
to deduce that $\vec{u}(t,x)$ satisfies \eqref{decay1}.  As before, we will check estimates for $P_{c,\sigma^{T_{n-1}}_{n-1}}(t)\vec{u}(t,x)$ by studying the right-hand side of its Duhamel expansion term by term. 

\subsubsection{Weighted $L^{2}$ norm of derivative of ODEs, $Int_{n-1}(t,x)$ and $Int_{\mathrm{unst},n-1}(t,x)$}
First, we can verify that any element $\vec{f}(t,x)$ of $\Omega\dot\sigma_{*}(t)$ defined in Subsection \ref{ODEsection}, $Int_{n-1}$ defined in \eqref{interactionfunction} satisfy for all $t\geq 0$
\begin{align*}
\max_{q,\in\{1,2\},j\in\{0,1\},\ell\in [m]} \norm{\chi_{\ell,n-1}(t,x)\partial^{j}_{x}\vec{f}(t,x)\langle x-y^{T_{n-1}}_{\ell,n-1}(t) \rangle}_{L^{q}_{x}(\mathbb{R})}\lesssim  & \frac{\delta_{0}^{2}}{(1+t)^{1+2\epsilon}},\\
\max_{q,\in\{1,2\},j\in\{0,1\},\ell\in [m]}\norm{\chi_{\ell,n-1}(t,x)\partial^{j}_{x}Int_{n-1}(t,x)\langle x-y^{T_{n-1}}_{\ell,n-1}(t) \rangle}_{L^{q}_{x}(\mathbb{R})}\lesssim  & \frac{\delta_{0}^{2}}{(1+t)^{20}}.
\end{align*}
Next, using the definition of $Int_{\mathrm{unst},n-1}(t)$ in \eqref{intunstn-1b0}, Remark \ref{RRR}, estimate \eqref{b**} and Lemma \ref{interactt}, we can verify for all $t\geq 0$ that 
\begin{align*}
\max_{q,\in\{1,2\},j\in\{0,1\},\ell\in [m]}\norm{\chi_{\ell,n-1}(t,x)\partial^{j}_{x}Int_{\mathrm{unst},n-1}(t,x)\langle x-y^{T_{n-1}}_{\ell,n-1}(t) \rangle}_{L^{q}_{x}(\mathbb{R})}\lesssim & \frac{ \delta_{0}^{2} } {(1+t)^{20}}.
\end{align*}
Consequently, using Theorem \ref{interpolation est.}, we conclude that
\begin{multline*}
   \max_{\vec{f}\in \Omega\dot\sigma_{*}(t)\cup\{Int_{n-1}\}\cup\{Int_{ust,n-1}\}}\norm{\frac{\chi_{\ell}(t,x)}{\langle x-y^{T_{n-1}}_{\ell,n-1}(s) \rangle^{1+\frac{p^{*}-2}{2p^{*}}+\omega}}\int_{0}^{t}\partial_{x}\mathcal{S}(t)\circ\mathcal{S}^{{-}1}(s)P_{c,n-1}(s)\vec{f}(s,x)\,ds}_{L^{2}_{x}(\mathbb{R})}\\
   \ll \frac{\delta_{0}}{(1+t)^{\frac{1}{2}+\epsilon}},
\end{multline*}
for all $t\geq 0.$ 
\subsubsection{$\left[V^{T_{n-1}}_{\ell,\sigma_{n-1}}(s,x)-V_{\ell,\sigma_{n-1}}(s,x)\right]\vec{u}_{*}(s,x)$}
Using Lemma \ref{dinftydt}, we can deduce for any $\omega\in (0,1)$ the existence of a constant $C(\omega)>0$ satisfying 
{\footnotesize\begin{equation*}
  \max_{j\in\{0,1\}} \norm{\langle x-y^{T_{n-1}}_{\ell,n-1}(s)\rangle^{\frac{3}{2}+\omega}\frac{\partial^{j}}{\partial x^{j}}[V(x-y_{\ell,n-1}(s))-V\left(x-y^{T_{n-1}}_{\ell,n-1}(s)\right)]}_{L^{1}_{x}\cap L^{\infty}_{x}}\leq \frac{C(\omega)\delta_{0}}{(1+s)^{2\epsilon-1}} \text{, for all $s\in [0,T_{n}].$}
\end{equation*}}
\par Consequently, using Theorem \ref{interpolation est.} for $p\in (0,1)$ close enough to $1,$ and $\epsilon=\frac{3}{4}+\frac{3}{2}\left(1-\frac{2-p}{p}\right),$
we can verify for $p^{*}=\frac{p}{p-1}$ that
\begin{multline}\label{wDVu}
 \norm{\frac{\chi_{\ell}(t,x)}{\langle x-y^{T_{n-1}}_{\ell,n-1}(s) \rangle^{1+\frac{p^{*}-2}{2p^{*}}+\omega}}\int_{0}^{t}\partial_{x}\mathcal{S}(t)\circ\mathcal{S}^{{-}1}(s)P_{c,n-1}(s)\left[V(x-y_{\ell,n-1}(s))-V(x-y^{T_{n-1}}_{\ell,n-1}(s))\right]\vec{u}_{*}(s,x)\,ds}_{L^{2}_{x}(\mathbb{R})}\\
 \begin{aligned}
  \leq &\int_{0}^ {t}\frac{C(v,\alpha)}{(1+t-s)^{\frac{3}{2}(\frac{1}{p}-\frac{1}{p^{*}})}}\frac{\delta_{0}}{(1+s)^{2\epsilon-1}}\max_{j\in\{0,1\}}\norm{\frac{\partial^{j}_{x}\vec{u}_{*}(s)}{\langle x-y^{T_{n-1}}_{\ell,n-1}(s) \rangle^{1+\frac{1-j}{2}+\frac{j(p^{*}-2)}{2p^{*}}+\omega}}}_{L^{2}_{x}(\mathbb{R})}
  \\
  &{+}\int_{0}^{t} \frac{C(v,\alpha)\delta_{0} e^{{-}\min_{j,\ell}\alpha_{j,n-1}(T_{n-1})[y^{T_{n-1}}_{\ell,n-1}(s)-y^{T_{n-1}}_{\ell+1,\sigma_{n-1}}(s)]}}{(1+s)^{2\epsilon-1}(1+t-s)^{\frac{3}{2}(\frac{1}{p}-\frac{1}{p^{*}})}}\norm{\vec{u}_{*}(s)}_{L^{2}_{x}(\mathbb{R})}\,ds\\
  &{+}\int_{0}^{t}\frac{C(v,\alpha)(y_{1}(0)-y_{m}(0)+s)}{(y_{1}(0)-y_{m}(0)+t)^{\frac{3}{2}(\frac{1}{p}-\frac{1}{p^{*}})}}\frac{\delta_{0}}{(1+s)^{2\epsilon-1}}\max_{j\in\{0,1\}}\norm{\frac{\partial^{j}_{x}\vec{u}_{*}(s)}{\langle x-y^{T_{n-1}}_{\ell,n-1}(s) \rangle^{1+\frac{1-j}{2}+\frac{j(p^{*}-2)}{2p^{*}}+\alpha}}}_{L^{2}_{x}(\mathbb{R})}. 
\end{aligned}
\end{multline}
\par Similarly to the argument used in Subsection \ref{difVnormalweight}, we can verify that the remaining terms in the right-hand side of the inequality \eqref{wDVu} are bounded above by
\begin{equation*}
    \frac{C\delta_{0}}{(1+t)^{\frac{1}{2}+\epsilon}}\left[\max_{j\in\{0,1\},s\in [0,T_{n}]}\langle s\rangle^{\frac{1}{2}+\epsilon}\norm{\frac{\partial^{j}_{x}\vec{u}_{*}(s)}{\langle x-y^{T_{n-1}}_{\ell,n-1}(s) \rangle^{1+\frac{1-j}{2}+\frac{j(p^{*}-2)}{2p^{*}}+\omega}}}_{L^{2}_{x}(\mathbb{R})}\right]\ll \frac{\delta_{0}}{(1+t)^{\frac{1}{2}}},
\end{equation*}
for a constant $C>1$. For more details on how to compute the estimate above, see Lemma \ref{interpol}.
\par In conclusion, using the hypothesis $\mathrm{(H2)}$ and the choice of $\delta_{0}$ in \eqref{deltachoice}, the following estimate holds 
\begin{multline}
 \norm{\frac{\chi_{\ell,n-1}(t,x)}{\langle x-y^{T_{n-1}}_{\ell,n-1}(s) \rangle^{1+\frac{p^{*}-2}{2p^{*}}+\alpha}}\int_{0}^{t}\partial_{x}\mathcal{S}(t)\circ \mathcal{S}^{{-}1}(s)\left[V(x-y_{\ell,n-1}(s))-V(x-y^{T_{n-1}}_{\ell,n-1}(s))\right]\vec{u}_{*}(s,x)\,ds}_{L^{2}_{x}(\mathbb{R})}\\
 \begin{aligned}
    \leq &\frac{C\delta_{0}(y_{1}(0)-y_{m}(0))^{2}}{(1+t)^{\frac{1}{2}+\epsilon}}\max_{s\in[0,T_{n}],j\in\{0,1\}}\langle s\rangle^{\frac{1}{2}+\epsilon} \norm{\frac{\partial^{j}_{x}\vec{u}_{*}(s)}{\langle x-y^{T_{n-1}}_{\ell,n-1}(s) \rangle^{1+\frac{1-j}{2}+\frac{j(p^{*}-2)}{2p^{*}}+\alpha}}}_{L^{2}_{x}(\mathbb{R})}
   \\ \ll & \frac{\delta_{0}}{(1+t)^{\frac{1}{2}+\epsilon}},
\end{aligned}
\end{multline}
since $(\vec{u}_{*},\sigma^{*})\in B_{\delta,n}.$ 
\subsubsection{Localized nonlinear terms}\label{slocalH1} Using $F_{2}(t,\sigma_{n-1},\vec{u}_{n-1},x)$ defined in \eqref{quadratic function}, we consider the following function for any $t\in [0,T_{n}].$
\begin{equation}\label{Q21}
   Q_{w,2,n-1}(t)=\max_{\ell}\norm{\frac{\chi_{j}(t)}{\langle x-y^{T_{n-1}}_{\ell,n-1}(t) \rangle^{^{1+\frac{p^{*}-2}{2p^{*}}+\alpha}}}\partial_{x}\int_{0}^{t}\mathcal{S}(t)\circ \mathcal{S}^{{-}1}(s)P_{c,n-1}(s)F_{2}(s,\sigma_{n-1},\vec{u}_{n-1},x)\,ds}_{L^{2}_{x}(\mathbb{R})}.
\end{equation}
Using \eqref{est000}, the decay estimates \eqref{decay1} and \eqref{decay4} satisfied by $\vec{u}_{n-1},$ and the value of $\delta$ in \eqref{deltachoice}, we obtain the following.
\begin{equation}\label{F2estimate}
    \norm{F_{2}(s,\sigma_{n-1},\vec{u}_{n-1},x)}_{H^{1}_{x}(\mathbb{R})}\leq\frac{C\delta_{0}^{2}}{(1+t)^{\frac{1}{2}+\epsilon}}\ll \frac{\delta_{0}}{(1+t)^{\frac{1}{2}+\epsilon}}.
\end{equation}
\par Next, from the estimate \eqref{estiofF-Fderiv} of Lemma \ref{multisolitonsinteractionsize}, and estimates \eqref{odesinfty} satisfied by $\sigma_{n-1},$ we can verify using the Cauchy-Schwarz inequality and estimate $\norm{\vec{u}_{n-1}(t)}_{H^{1}_{x}(\mathbb{R})}\leq \delta$ that
\begin{multline}\label{weightedl1F2}
    \max_{\ell}\norm{ \langle x-y^{T_{n-1}}_{\ell,n-1}(t) \rangle \chi_{\ell,n-1}(t,x) \partial_{x}F_{2}(t,\sigma_{n-1},\vec{u}_{n-1},x)}_{L^{1}_{x}(\mathbb{R})}\\
    \begin{aligned}
        \leq & C  e^{{-}\frac{99}{100}\min_{\ell,j}\alpha_{j}(t)[y_{\ell,n-1}(t)-y_{\ell+1,n-1}(t)]}\max_{\ell,j\in\{0,1\}}\norm{\frac{\partial^{j}_{x}\vec{u}_{n-1}(t)}{\langle x-y^{T_{n-1}}_{\ell,n-1}(t) \rangle^{1+\frac{1-j}{2}+\frac{j(p^{*}-2)}{2p^{*}}+\alpha}}}_{L^{2}_{x}(\mathbb{R})} \\
         &{+}C \max_{\ell,j\in\{0,1\}}\norm{\frac{\partial^{j}_{x}\vec{u}_{n-1}(t)}{\langle x-y^{T_{n-1}}_{\ell,n-1}(t) \rangle^{1+\frac{1-j}{2}+\frac{j(p^{*}-2)}{2p^{*}}+\alpha}}}_{L^{2}_{x}(\mathbb{R})}^{2}. 
    \end{aligned}
\end{multline}
In particular, Minkowski's inequality implies that
\begin{multline}\label{weightedl1F21}
    \max_{\ell}\norm{  \partial_{x}F_{2}(t,\sigma_{n-1},\vec{u}_{n-1},x)}_{L^{1}_{x}(\mathbb{R})}
        \\
        \leq  Cm e^{{-}\frac{99}{100}\min_{\ell,j}\alpha_{j}(t)[y_{\ell,n-1}(t)-y_{\ell+1,n-1}(t)]}\max_{\ell,j\in\{0,1\}}\norm{\frac{\partial^{j}_{x}\vec{u}_{n-1}(t)}{\langle x-y^{T_{n-1}}_{\ell,n-1}(t) \rangle^{1+\frac{1-j}{2}+\frac{j(p^{*}-2)}{2p^{*}}+\alpha}}}_{L^{2}_{x}(\mathbb{R})} \\ \nonumber
         {+}Cm \max_{\ell,j\in\{0,1\}}\norm{\frac{\partial^{j}_{x}\vec{u}_{n-1}(t)}{\langle x-y^{T_{n-1}}_{\ell,n-1}(t) \rangle^{1+\frac{1-j}{2}+\frac{j(p^{*}-2)}{2p^{*}}+\alpha}}}_{L^{2}_{x}(\mathbb{R})}^{2}. 
    \end{multline}
\par Moreover, using the formula \eqref{quadratic function}, we can verify from the estimates \eqref{odes} satisfied by $\sigma_{n-1},$ and inequality \ref{estiofF-Fderiv} that
\begin{equation*}
    \norm{F_{2}(t,\sigma_{n-1},\vec{u}_{n-1},x)}_{L^{2}_{x}(\mathbb{R})}\leq C\norm{\vec{u}_{n-1}(t,x)}_{H^{1}_{x}(\mathbb{R})}\norm{\vec{u}_{n-1}(t,x)}_{L^{2}_{x}(\mathbb{R})}+C\frac{\delta_{0}}{(1+t)^{20}}\norm{\vec{u}_{n-1}(t,x)}_{L^{2}_{x}(\mathbb{R})},
\end{equation*}
 Consequently, we can verify the following estimate
\begin{multline}\label{H2F2}
   \int_{0}^{t}\frac{e^{{-}\min_{j,\ell}\alpha_{j,n-1}(T_{n-1})[y^{T_{n-1}}_{\ell,n-1}(s)-y^{T_{n-1}}_{\ell+1,n-1}(s)])}\norm{F_{2}(t,\sigma_{n-1},\vec{u}_{n-1},x)}_{L^{2}_{x}(\mathbb{R})}}{1+(t-s)^{\frac{3}{2}}}\,ds\\
   \leq C\frac{\delta_{0}^{2}}{(1+t)^{\frac{1}{2}+\epsilon}}\ll \frac{\delta_{0}}{(1+t)^{\frac{1}{2}+\epsilon}}. 
\end{multline}
In conclusion, since $\vec{u}_{n-1}(t)$ satisfies \eqref{decay1} and \eqref{decay4} for any $t\geq 0$ ($\vec{u}_{n-1}(t)\equiv 0,$ when $t\geq T_{n-1}$), we can can deduce applying Theorem \ref{interpolation est.} and Lemma \eqref{interpol} to the estimates above that there exists a constant $C>1$ satisfying
\begin{equation}\label{WeightedderivativeF2}
Q_{w,2,n-1}(t)
    \leq \frac{C\delta_{0}^{2}}{(1+t)^{\frac{1}{2}+\epsilon}}\ll \frac{\delta_{0}}{(1+t)^{\frac{1}{2}+\epsilon}},
\end{equation}
for $Q_{w,2,n-1}$ defined in \eqref{Q21}, $\epsilon=\frac{3}{4}+\frac{3}{2}\left(1-\frac{2-p}{p}\right)$ with $p\in (1,2)$ close enough to $1.$
\subsubsection{Full nonlinear term}
Finally, for the conclusion the estimate of the weighted norm of the derivative of $\vec{u},$ we need to estimate
\begin{equation*}
    \max_{\ell}\norm{\frac{\chi_{j}(t)}{\langle x-y^{T_{n-1}}_{j,n-1}(t) \rangle^{^{1+\frac{p^{*}-2}{2p^{*}}+\omega}}}\partial_{x}P_{c,n-1}(t)\int_{0}^{t}\mathcal{U}_\sigma(t,s)\vert \vec{u}_{n-1}(s,x)\vert^{2k}\vec{u}_{n-1}\,ds}_{L^{2}_{x}(\mathbb{R})}.
\end{equation*}
Theorem \ref{interpolation est.} implies that
\begin{multline*}
    \max_{\ell}\norm{\frac{\chi_{j,n-1}(t,x)}{\langle x-y^{T_{n-1}}_{j,n-1}(t) \rangle^{^{1+\frac{p^{*}-2}{2p^{*}}+\alpha}}}\partial_{x}P_{c,n-1}(t)\int_{0}^{t}\mathcal{U}_\sigma(t,s)\vert \vec{u}_{n-1}(s,x)\vert^{2k}\vec{u}_{n-1}\,ds}_{L^{2}_{x}(\mathbb{R})}\\
    \begin{aligned}
       \leq & C(p)\int_{0}^{t}\frac{\left[(1+s)\norm{\vert\vec{u}_{n-1}(s,x)\vert^{2k+1}}_{L^{1}_{x}(\mathbb{R})}+\max_{\ell}\norm{\langle x-y^{T_{n-1}}_{\ell,n-1}(s) \rangle\chi_{\ell}(s,x) \vert\vec{u}_{n-1} \vert^{2k}\langle \partial_{x}\rangle \vec{u}_{n-1}(s)}_{L^{1}_{x}(\mathbb{R})}\right]}{1+(t-s)^{\frac{3}{2}(\frac{1}{p}-\frac{1}{p^{*}})}}\,ds\\
       &{+}C(p)\int_{0}^{t}\frac{\norm{\vert \vec{u}_{n-1}(s) \vert^{2k}\langle \partial_{x}\rangle\vec{u}_{n-1}(s)}_{L^{2}_{x}(\mathbb{R})}}{1+ (t-s)^{\frac{3}{2}(\frac{1}{p}-\frac{1}{p^{*}})}}\,ds\\
       &{+}C(p)\int_{0}^{t}\frac{(s+y_{1}(0)-y_{m}(0))\norm{\vert \vec{u}_{n-1}(s) \vert^{2k}\langle \partial_{x}\rangle\vec{u}_{n-1}(s)}_{L^{1}_{x}(\mathbb{R})}^{\frac{2-p}{p}}\norm{\vert \vec{u}_{n-1}(s) \vert^{2k}\langle \partial_{x}\rangle\vec{u}_{n-1}(s)}_{L^{2}_{x}(\mathbb{R})}^{\frac{2(p-1)}{p}}}{(y_{1}(0)-y_{m}(0)+ t)^{\frac{3}{2}(\frac{1}{p}-\frac{1}{p^{*}})}}\,ds\\
       &{+}C(p)\int_{0}^{t}\frac{(s+y_{1}(0)-y_{m}(0))}{(t+y_{1}(0)-y_{m}(0))(1+t-s)^{\frac{1}{2}}}\norm{\vert \vec{u}_{n-1}(s) \vert^{2k} \vec{u}_{n-1}(s)}_{W^{1,1}_{x}(\mathbb{R})}\,ds\\
       &{+}C(p)\int_{0}^{t}\frac{e^{{-}\min_{\ell,j}\alpha^{T_{n-1}}_{j,n-1}(s)(y^{T_{n-1}}_{\ell,n-1}(s)-y^{T_{n-1}}_{\ell+1,n-1}(s))}}{1+(t-s)^{\frac{3}{2}(\frac{1}{p}-\frac{1}{p^{*}})}}\norm{\vert\vec{u}_{n-1}(s)\vert^{2k}\vec{u}_{n-1}(s)}_{L^{2}_{x}(\mathbb{R})}\,ds,
    \end{aligned}
\end{multline*}
for some constant $C(p)>1$ depending only on $\{(v_{\ell}(0),\alpha_{\ell}(0))\}_{\in[m]}$ and $p.$
Therefore, since $\vec{u}_{n-1}$ satisfies the following decay estimates for all $t\geq 0$ from the assumption in Proposition \ref{Aiscontraction}
\begin{align*}
    \norm{\vec{u}_{n-1}(s)}_{L^{\infty}_{x}(\mathbb{R})}\leq&\frac{\delta_{0}}{(1+s)^{\frac{1}{2}}},\,\norm{\vec{u}_{n-1}(s)}_{H^{1}_{x}(\mathbb{R})}\leq \delta_{0},
\end{align*}
for $\delta_{0}\ll 1$ defined at \eqref{deltachoice}, we deduce
\begin{multline}\label{nonlinearweightderivupp}
    \max_{\ell}\norm{\frac{\chi_{j,n-1}(t)}{\langle x-y^{T_{n-1}}_{j,n-1}(t) \rangle^{^{1+\frac{p^{*}-2}{2p^{*}}+\alpha}}}\int_{0}^{t}\partial_{x} \mathcal{S}(t)\circ \mathcal{S}^{{-}1}(s)\vert \vec{u}_{n-1}(s,x)\vert^{2k}\vec{u}_{n-1}\,ds}_{L^{2}_{x}(\mathbb{R})}\\
    \begin{aligned}
        \leq& C(p)\int_{0}^{t}\frac{\delta_{0}^{2k+1}(y_{1}(0)-y_{m}(0)+s)}{(1+s)^{k-\frac{1}{2}}[1+(t-s)^{\frac{3}{2}(\frac{1}{p}-\frac{1}{p^{*}})}]}+\frac{\delta_{0}^{2k+1}}{(1+s)^{k}[1+(t-s)^{\frac{3}{2}(\frac{1}{p}-\frac{1}{p^{*}})} ]}\,ds\\
        &{+}C(p)\int_{0}^{t}\frac{(y_{1}(0)-y_{m}(0)+s)\delta_{0}^{2k+1}}{(y_{1}(0)-y_{m}(0)+t)(1+s)^{k-\frac{1}{2}}(1+t-s)^{\frac{1}{2}}}\,ds\\
        &{+}\frac{C(p)\delta_{0}^{2k+1}}{(1+t)^{\frac{3}{2}(\frac{1}{p}-\frac{1}{p^{*}})}},
    \end{aligned}
\end{multline}
for some constant $C(p)>1$ depending only on $\{(v_{\ell}(0),\alpha_{\ell}(0))\}_{\in[m]}$ and $p.$
\par In conclusion, Lemma \ref{interpol} implies that if $k>\frac{3}{2}+\frac{5}{4}$ and $\delta\in (0,1)$ is small enough, then there exists a constant $C>1$ satisfying
\begin{multline*}
    \max_{\ell}\norm{\frac{\chi_{\ell,n-1}(t)}{\langle x-y^{T_{n-1}}_{\ell,n-1}(t) \rangle^{^{1+\frac{p^{*}-2}{2p^{*}}+\alpha}}}\partial_{x}\int_{0}^{t}\mathcal{S}(t)\circ \mathcal{S}^{{-}1}(s)P_{c,n-1}(s)\vert \vec{u}_{n-1}(s,x)\vert^{2k}\vec{u}_{n-1}\,ds}_{L^{2}_{x}(\mathbb{R})}\\
    \leq C\delta_{0}^{2}\max\left(\frac{1}{(1+t)^{k-\frac{3}{2}}},\frac{1}{(1+t)^{\frac{3}{2}(\frac{1}{p}-\frac{1}{p^{*}})}}\right).
\end{multline*}
Furthermore, 
for 
\begin{equation*}
    \epsilon(p)=\frac{3}{4}+\frac{3}{2}(1-\frac{2-p}{p})>\frac{3}{4} \text{, when $p\in(0,1)$}
\end{equation*}
since $k-\frac{3}{2}>\frac{1}{2}+\frac{3}{4}=\lim_{p\to 1}\epsilon(p),$ and
\begin{equation*}
\lim_{p\to 1}\frac{3}{2}\left(\frac{1}{p}-\frac{1}{p^{*}}\right)=\frac{3}{2},
\end{equation*}
we can find a $p\in (1,2)$ close enough to $1,$ and choose $M=\min y_{\ell}(0)-y_{\ell+1}(0)$ large enough such that 
\begin{multline*}
    \max_{j}\norm{\frac{\chi_{j,n-1}(t)}{\langle x-y^{T_{n-1}}_{j,n-1}(t) \rangle^{^{1+\frac{p^{*}-2}{2p^{*}}+\alpha}}}\partial_{x}\int_{0}^{t}\mathcal{S}(t)\circ \mathcal{S}^{{-}1}(s)P_{c,n-1}(s)\vert \vec{u}_{n-1}(s,x)\vert^{2k}\vec{u}_{n-1}\,ds}_{L^{2}_{x}(\mathbb{R})}
    \\
    \ll \frac{\delta_{0}}{(1+t)^{\frac{1}{2}+\epsilon(p)}}
\end{multline*}
is true for any $t\in [0,T_{n}].$
\subsubsection{Conclusion}
Consequently, we obtain the following.
\begin{equation*}
    \max_{\ell}\norm{\frac{\chi_{\ell,n-1}(t)}{\langle x-y^{T_{n-1}}_{\ell,n-1}(t) \rangle^{^{1+\frac{p^{*}-2}{2p^{*}}+\omega}}}\partial_{x}P_{c,n-1}(t)\vec{u}(t,x)}_{L^{2}_{x}(\mathbb{R})}
    \ll \frac{\delta_{0}}{(1+t)^{\frac{1}{2}+\epsilon}} \text{, for all $t\in [0,T_{n}]$}
\end{equation*}
from which we deduce using the conclusion of \S \ref{secinih} that \eqref{decay1} is true for all $t\in [0,T_{n}].$
\subsection{Estimate for the $H^{1}$ norm}
\par First, Lemma \ref{blemma2}, \eqref{ppequiv}, \eqref{stabl2l2} and estimate \eqref{l2root}, and the decay estimate \eqref{decay4} satisfied by $\vec{u}(t)$ imply the following inequality for a constant $C>1$
\begin{equation}\label{udiscreteH2}
    \max_{h\in\{\mathrm{root},\mathrm{stab},\mathrm{unst}\}}\norm{\vec{u}_{h,\sigma^{T_{n-1}}_{n-1}}(s)}_{H^{1}_{x}(\mathbb{R})}\leq \frac{C\delta_{0}^{2}}{(1+t)^{\frac{1}{2}+\epsilon}} \text{, for all $t\geq 0.$}
\end{equation}
\par Next, using the estimate \eqref{t23} together with the decay estimates \eqref{decay2}-\eqref{decay5} satisfied by $\vec{u}_{n-1}(t,x)$ for all $t\geq 0,$ we have that
\begin{align}\label{h1F2}
    \norm{F_{2}(s,\sigma,\vec{u}_{n-1},x)}_{H^{1}_{x}(\mathbb{R})}\leq &C\frac{\delta_{0}^{2}}{(1+s)^{\frac{1}{2}+\epsilon}}\ll \frac{\delta_{0}}{(1+t)^{\frac{1}{2}+\epsilon}},
\end{align}
for any $s\in [0,T_{n}].$ Moreover, the estimates \eqref{decay2}-\eqref{decay5} imply that
\begin{align}\label{h12kT}
    \norm{\vert\vec{u}_{n-1}(s)\vert^{2k}\vec{u}_{n-1}(s)}_{H^{1}_{x}(\mathbb{R})}\leq & Ck\left(\frac{\delta_{0}^{2k+1}}{(1+s)^{k}}\right).
\end{align}
\par Note that  Lemma \ref{dinftydt} implies that
\begin{multline}\label{h1diffVu}
    \norm{\left[V(x-y_{\ell,n-1}(s))-V(x-y^{T_{n-1}}_{\ell,n-1}(s))\right]\vec{u}_{*}(s)}_{H^{1}_{x}(\mathbb{R})}\\
    \leq  C\frac{\delta_{0}}{(1+s)^{2\epsilon-1}}\max_{j\in\{0,1\},\ell\in [m]}\norm{\frac{\chi_{\ell,n-1}(s)\frac{\partial^{j}}{\partial x^{j}}\vec{u}_{*}(s,x)}{\langle x-y^{T_{n-1}}_{\ell,n-1}(s) \rangle^{^{1+\frac{1-j}{2}+\frac{j(p^{*}-2)}{2p^{*}}+\omega}}}}_{L^{2}_{x}(\mathbb{R})}\leq C\frac{\delta_{0}^{2}}{(1+s)^{\frac{1}{2}+\epsilon}}
\end{multline}
and
\begin{equation*}
    \norm{\left[V(x-y_{\ell,n-1}(s))-V(x-y^{T_{n-1}}_{\ell,n-1}(s))\right]\vec{u}_{*}(s)}_{H^{2}_{x}(\mathbb{R})}
    \leq  C\frac{\delta_{0}}{(1+s)^{2\epsilon-1}}\norm{\vec{u}_{*}(s)}_{H^{2}_{x}(\mathbb{R})}.
\end{equation*}
\par Moreover, since $\sigma^{*}=\{(v_{\ell,*},y_{\ell,*},\alpha_{\ell,*},\gamma_{\ell,*})\}_{\ell\in [m]}$ satisfies \eqref{dotsigma*} for all $t\in [0,T_{n}],$ Proposition \ref{phrootissmall} implies for any $\vec{f}(s,x)\in\Omega\dot \sigma_{*}(s),$ see \eqref{Omegadt}, that  
\begin{equation}\label{h1Omega}
    \norm{P_{c,n-1}(s)\vec{f}(s,x)}_{H^{1}_{x}(\mathbb{R})}\sim \norm{P_{c,n-1}(s)\vec{f}(s,x)}_{L^{2}_{x}(\mathbb{R})}\leq \frac{C\delta_{0}^{2}}{(1+s)^{4\epsilon}}\ll \frac{\delta_{0}}{(1+s)^{1+\epsilon}}.
\end{equation}
\par Furthermore, using the definition of $Int_{1}(t,x)$ in \eqref{interactionfunction}, Proposition \ref{multisolitonsinteractionsize}, and the choice of $\delta$ in \eqref{deltachoice}, we can deduce the following estimate for all $s\geq 0.$
\begin{equation}\label{inth1}
    \norm{Int_{1}(s,x)}_{H^{1}_{x}(\mathbb{R})}\leq \frac{C\delta_{0}^{2}}{(1+s)^{1+\epsilon}}.
\end{equation} Similarly, using the decay estimate satisfied by $b_{h,+,*}$ in \eqref{b**}, we can verify from the definition of $Int_{\mathrm{unst},n-1}$ in \eqref{intunstn-1b0}, Lemma \ref{interactt} and the choice of $\delta\in (0,1)$ in \eqref{deltachoice} that
\begin{equation*}
    \norm{Int_{\mathrm{unst},n-1}(s,x)}_{H^{1}_{x}(\mathbb{R})}\leq \frac{C\delta_{0}^{2}}{(1+s)^{1+\epsilon}}.
\end{equation*}
\par Therefore, using estimate $\norm{\mathcal{U}_\sigma(t,s)P_{c}(s)\vec{f}}_{H^{j}_{x}(\mathbb{R})}\leq K \norm{\vec{f}}_{H^{j}_{x}(\mathbb{R})}$ for all $f\in H^{j}_{x}(\mathbb{R})$ and any $j\in\{0,1\}$ from Theorem \ref{princ11}, and estimate $\norm{\vec{u}(0,x)}_{H^{1}_{x}(\mathbb{R})}\leq C\delta^{2}\ll\delta_{0}^{2},$ for some constant $C>1,$ we can verify from the integral equation \eqref{disperpart} satisfied by $\vec{u}_{c,\sigma^{T_{n-1}}_{n-1}}(t)$ and the previous estimates in this subsection that
\begin{align*}
    \norm{\vec{u}_{c,\sigma^{T_{n-1}}_{n-1}}(t,x)}_{H^{1}_{x}(\mathbb{R})}\leq & K\delta_{0}^{2}\ll \delta_{0}.
\end{align*}
.
\par In conclusion, the estimates above and \eqref{udiscreteH2} imply that
\begin{equation*}
\norm{\vec{u}(t)}_{H^{1}_{x}(\mathbb{R})}\leq \delta,
\end{equation*}
for all $t\geq 0,$ since $\vec{u}(t)\equiv 0$ when $t\geq T_{n}.$
\subsection{Growth of weighed $L^2$ norms}\label{grl2x}
We now study estimate of $$\max_{\ell\in [m]}\norm{\chi_{\ell,n-1}(t,x)\left \vert x-y^{T_{n-1}}_{\ell,\sigma^{T_{n-1}}_{n-1}}(t)\right\vert \vec{u}(t)}_{L^{2}_{x}(\mathbb{R})}.$$ 
Similarly to the explanation at the beginning of Subsection \ref{weighteduu}, it is enough to prove that
\begin{equation}\label{finalusize}
    \max_{\ell\in [m]}\norm{\chi_{\ell,n-1}(t,x)\left \vert x-y^{T_{n-1}}_{\ell,\sigma^{T_{n-1}}_{n-1}}(t)\right\vert P_{c,\sigma^{T_{n-1}}_{n-1}}\vec{u}(t)}_{L^{2}_{x}(\mathbb{R})}\ll  \delta_{0} [\max_{\ell}\vert v_{\ell}(0)\vert+1 ](1+t) \text{, for all $t\in [0,T_{n}]$}
\end{equation}
to conclude that the first inequality of \eqref{decay3} holds for $\vec{u}.$
The main reason for this remark is because if $h\in\{\mathrm{stab},\mathrm{unst},\mathrm{root}\},$ then $P_{h,\sigma^{T_{n-1}}_{n-1},\ell}(t)\vec{u}(t)$ is a finite sum of localized Schwartz functions with exponential decay.
\par From the definition of $\vec{u}(0,x)$ in \eqref{u0form}, the hypotheses satisfied by $\vec{r}_{0}(x)$ in \eqref{u0form}, Lemma \ref{blemma2} and \eqref{l2root}, we have that
\begin{equation*}
    \max_{\ell\in [m]}\norm{\chi_{\ell,n-1}(0,x)\left \vert x-y^{T_{n-1}}_{\ell,\sigma^{T_{n-1}}_{n-1}}(0)\right\vert P_{c,\sigma^{T_{n-1}}_{n-1}}(0)\vec{u}(0)}_{L^{2}_{x}(\mathbb{R})}\lesssim \delta^{2}\ll   \delta_{0}. 
\end{equation*}
\par Next, using H\"older's inequality, Lemma \ref{dinftydt} and Corollary \ref{quadratipotentialestimate}, we deduce the following estimates.
\begin{gather*}\nonumber
  \norm{\chi_{\ell,n-1}(t,x)\left\vert x-y^{T_{n-1}}_{\ell,n-1}(t)\right\vert\vert \vec{u}_{n-1}(t,x)\vert^{2k}\vec{u}(t,x)}_{L^{2}_{x}(\mathbb{R})}
  \\ \leq  \frac{\delta_{0}^{2k}\norm{\chi_{\ell,n-1}(t,x)\left\vert x-y^{T_{n-1}}_{\ell,n-1}(t)\right\vert\vert \vec{u}_{n-1}(t,x)\vert}_{L^{2}_{x}(\mathbb{R})}}{(1+t)^{k}}\leq \frac{\delta_{0}^{2k+1}(1+\max_{\ell}\vert v_{\ell}(0)\vert)}{(1+t)^{k-1}} ,\\
  \norm{\chi_{\ell,n-1}(t,x)\left\vert x-y^{T_{n-1}}_{\ell,n-1}(t)\right\vert\left[V_{\ell,\sigma_{n}}(t,x)-V^{T_{n}}_{\ell,\sigma_{n}}(t,x)\right]\vec{u}_{*}(t,x)}_{L^{2}_{x}(\mathbb{R})}
  \leq  \frac{C\delta_{0}^{2}}{(1+t)^{3\epsilon-\frac{1}{2}}},\\
   \norm{\chi_{\ell,n-1}(t,x)\left\vert x-y^{T_{n-1}}_{\ell,n-1}(t)\right\vert F_{2}(t,\sigma_{n-1},\vec{u}_{n-1},x)}_{L^{2}_{x}(\mathbb{R})}
   \leq  \frac{C\delta_{0}^{2}}{(1+t)^{1+\epsilon}}.
\end{gather*}
Moreover, using Proposition \ref{multisolitonsinteractionsize} and the definition of $\delta$ in \eqref{deltachoice}, we can verify that the function $Int_{1}(t,x)$ defined in \eqref{interactionfunction} satisfies for some constant $c \in (0,1)$
\begin{equation*}
    \norm{\chi_{\ell,n-1}(t,x)\left\vert x-y^{T_{n-1}}_{\ell,n-1}(t)\right\vert Int_{1}(t,x)}_{L^{2}_{x}(\mathbb{R})}\leq C\frac{\delta_{0}^{2}}{(1+t)^{20}},
\end{equation*}
 due to estimate \eqref{dyn12} satisfied by $v_{\ell,n-1}(t).$ Similarly, using \eqref{intunstn-1b0}, Remark \ref{RRR} and the upper bound $\max_{\ell}\vert b_{\ell,+,*}(t) \vert\lesssim \delta,$ we can verify the following estimate
\begin{equation*}
   \norm{ \chi_{\ell,n-1}(t,x)\left\vert x-y^{T_{n-1}}_{\ell,n-1}(t)\right\vert  Int_{\mathrm{unst},n-1}(t)}_{L^{2}_{x}(\mathbb{R})}\lesssim\frac{\delta_{0}^{2}}{(1+t)^{20}}.
\end{equation*}
\par Furthermore, since $(u_{*},\sigma^{*})\in B_{\delta,n}$ and $\sigma^{*}$ satisfies \eqref{odes} for any $t\in [0,T_{n}],$ we obtain from the following inequality for any element $\vec{f}(t,x)\in\Omega\dot\sigma_{*}(t).$ 
\begin{equation*}
    \norm{\chi_{\ell,n-1}(t,x)\left\vert x-y^{T_{n-1}}_{\ell,n-1}(t)\right\vert P_{c,n-1}(t)\vec{f}(t,x)}_{L^{2}_{x}(\mathbb{R})}\leq C\frac{\delta_{0}^{2}}{(1+t)^{1+2\epsilon}}. 
\end{equation*}
\par Consequently, using estimates \eqref{h12kT}, \eqref{h1diffVu}, \eqref{h1F2}, \eqref{h1Omega}, and \eqref{inth1} of the previous subsection, we can conclude from the estimates obtained of Subsection \ref{grl2x}, Lemma \ref{interpol} and Proposition \ref{growthweightl2} that \eqref{finalusize} is true. In conclusion, if $\min_{\ell}y_{\ell}(0)-y_{\ell+1}(0)>1$ is large enough, we obtain that $\vec{u}(t)$ satisfies 
\begin{equation*}
\max_{t\in[0,T_{n}],\ell}\frac{\norm{\chi_{\ell,n-1}(t,x)\vert x-y^{T_{n-1}}_{\ell,n-1}(t)\vert\vec{u}(t,x)}_{L^{2}_{x}(\mathbb{R})}}{[\max_{\ell}\vert v_{\ell}(0)\vert+1](1+t)}\leq \delta_{0},
\end{equation*}
for all $t\in [0,T_{n}].$
\subsection{Estimate of $\vert \Lambda \dot \sigma(t)\vert $}
First, using the estimates \eqref{t2},\,\eqref{h12kT}, \eqref{inth1} \eqref{interactionfunction}, the following inequalities are obtained for all $t\in [0,T_{n}]$
\begin{align*}
    \left\vert \langle F_{2}  (t,\sigma_{n-1},\vec{u}_{n-1},x),e^{i\mathfrak{p}_3\theta_{\ell,n-1}(t,x)}\vec{Z}(\alpha_{\ell,n-1}(t),x-y_{\ell,n-1}(t))\rangle\right\vert\leq &\frac{C\delta_{0}^{2}}{(1+t)^{1+2\epsilon}}\ll \frac{\delta_{0}}{(1+t)^{\frac{1}{2}+\epsilon}},\\
    \langle \mathfrak{p}_3\vert\vec{u}_{n-1}(t,x) \vert^{2k}\vec{u}_{n-1}(t,x),\mathfrak{p}_3e^{i\mathfrak{p}_3\theta_{\ell,n-1}(t,x)}\vec{Z}(\alpha_{\ell,n-1}(t),x-y_{\ell,n-1}(t))\rangle\leq &\frac{C\delta_{0}^{2k+1}}{(1+t)^{k+\frac{1}{2}+\epsilon}}\ll \frac{\delta_{0}}{(1+t)^{\frac{1}{2}+\epsilon}},\\
    \left\vert \langle Int_{1}(t,x), \mathfrak{p}_3e^{i\mathfrak{p}_3\theta_{\ell,n-1}(t,x)}\vec{Z}(\alpha_{\ell,n-1}(t),x-y_{\ell,n-1}(t)) \rangle\right\vert\leq &\frac{C\delta_{0}^{2}}{(1+t)^{1+\epsilon}}\ll \frac{\delta_{0}}{(1+t)^{\frac{1}{2}+\epsilon}},
\end{align*}
for any function $\vec{Z}(1,x)\in \ker\mathcal{H}^{2}_{1}$ satisfying $\norm{\vec{Z}(1,x)}_{L^{2}_{x}(\mathbb{R})}\leq 1.$
\par Next, using Definition \ref{lambdasigma}, the first decay estimate in \eqref{decay3} satisfied by $\vec{u}_{*}(t),$ the inequalities \eqref{odes} satisfied for $n-1,$ and the value of $\delta$ in \eqref{deltachoice}, the following inequality holds for a constant $C>1.$ 
\begin{multline*}
    \vert \Lambda\dot\sigma_{n-1}(t) \vert \left\vert \left\langle \vec{u}_{*}(t,x),\mathfrak{p}_3\left(\partial_{t}-i\mathfrak{p}_3\partial^{2}_{x}-iV_{\ell,\sigma_{n-1}}(t,x)\right)\left[e^{i\mathfrak{p}_3\theta_{\ell,n-1}(t,x)}z(\alpha_{\ell,n-1}(t),x-y_{\ell,n-1}(t))\right] \right \rangle\right\vert\\
   \begin{aligned}
    \leq &\frac{C\delta_{0}}{(1+t)^{1+2\epsilon}}\max_{\ell}\norm{\frac{\chi_{\ell,n-1}(t)\vec{u}_{*}(t,x)}{\langle x-y^{T_{n-1}}_{\ell,n-1}(t) \rangle^{\frac{3}{2}+\omega}}}_{L^{2}_{x}(\mathbb{R})}\\
    \leq & \frac{C\delta_{0}^{2}}{(1+t)^{\frac{3}{2}+3\epsilon}}\max_{\ell}\norm{\frac{\chi_{\ell,n-1}(t)\vec{u}_{*}(t,x)}{\langle x-y^{T_{n-1}}_{\ell,n-1}(t) \rangle^{\frac{3}{2}+\omega}}}_{L^{2}_{x}(\mathbb{R})}\ll \frac{\delta_{0}}{(1+t)^{1+2\epsilon}}.
    \end{aligned}
\end{multline*}
Therefore, we can conclude applying the ordinary differential system \eqref{ODEofsigma} for any $\ell\in[m],$ and any elements $\vec{Z}$ from one basis of the subspace $\ker\mathcal{H}^{2}_{1}$ that
\begin{equation*}
    \max_{\ell}\left\vert \Lambda\dot\sigma_{\ell}(t)  \right\vert\leq \frac{C\delta_{0}^{2}}{(1+t)^{1+2\epsilon}}\ll \frac{\delta_{0}}{(1+t)^{1+2\epsilon}},
\end{equation*}
where $\Lambda\dot\sigma_{\ell}(t)$ is defined at Definition \ref{lambdasigma}.
\subsection{Proof of Proposition 
\ref{Aiscontraction} and conclusion}\label{Aiscont}
From the previous results in all the Subsections above of Section , we conclude that $(\vec{u}(t),\sigma)=A(\vec{u}_{*},\sigma^{*})\in B_{\delta,n}$ whenever $\vec{u}_{*},\sigma^{*}\in B_{\delta,n}.$ The proof that $A$ is a contraction follows by an analogous argument  that $(\vec{u},\sigma(t))$ satisfies all the estimates \eqref{decay1}-\eqref{odes}. It follows using the difference of between the equations \eqref{unequation}, \eqref{odesigma} satisfied by $A(\vec{u}_{*},\sigma^{*})$ and by the ones satisfied by $A(\vec{u}_{**},\sigma^{**})$ to compute the norm of $A(\vec{u}_{*},\sigma^{*})-A(\vec{u}_{**},\sigma^{**}).$

\section{Proof of Proposition \ref{propun-un-1}}\label{sectioncauchy}
From now on, for any $n\in\mathbb{N}_{\geq 1},$ we consider the sequence $(\vec{u}_{n},\sigma_{n})$ to be the one defined in Proposition \ref{undecays}. Moreover, from the proof of Proposition \ref{undecays} in the previous section, we can assume now that all the estimates \eqref{decay1}-\eqref{odes} are true for all $n\in\mathbb{N}.$ The following elementary proposition implies that $(\vec{u}_{n},\sigma_{n})$ satisfies  equation \eqref{unequation} for any $n\in\mathbb{N}_{\geq 1}.$
\begin{lemma}\label{lnicity}
If, for a $s\geq 0,$ $P_{c,n-1}(s)\vec{u}=P_{c}(s)\vec{w},$ $P_{\mathrm{stab},n-1}(s)\vec{u}=P_{\mathrm{stab},n-1}(s)\vec{w},$ $P_{\mathrm{unst},n-1}(s)\vec{u}=P_{\mathrm{unst},n-1}(s)\vec{w},$ and for any $\ell$ 
\begin{equation*}
    \langle \vec{u},e^{i\mathfrak{p}_3(\theta_{\ell,n-1}(s))}\vec{z}(x-y_{\ell,n-1}(s)) \rangle=\langle \vec{w},e^{i\mathfrak{p}_3\theta_{\ell,n-1}(s)}\vec{z}(x-y_{\ell,n-1}(s)) \rangle=0
\end{equation*}
for any $z\in \ker\mathcal{H}^{2}_{\ell},$ then
$\vec{u}=\vec{w}.$
\end{lemma}
In particular, from Lemma \ref{lnicity}, we can verify that if $(u_{n},\sigma_{n})$ is a fixed-point of $A_{n-1},$ then  $(u_{n},\sigma_{n})$ satisfies the following differential equations for $t\in [0,T_{n}]$ for $Forc_{\mathrm{unst},n-1}$ and $G$ defined at \eqref{Forcingtermn-1ustar} and \eqref{unequation} respectively.
\begin{gather}\label{equations for un}\tag{$\vec{u}_{n}$ system}
  i\partial_{t}\vec{u}_{n}(t,x)+\mathfrak{p}_3\partial^{2}_{x}\vec{u}_{n}(t,x)+\sum_{\ell}V_{\ell,\sigma_{n-1}}(t,x)\vec{u}_{n}(t,x)
  = G(t,\sigma_{n}(t),\sigma_{n-1}(t),\vec{u}_{n-1}),\\ \nonumber
  P_{\mathrm{unst},\ell,n-1}(t)\vec{u}_{n}(t)=i\int_{t}^{T_{n}}e^{(t-s)\vert \lambda_{\ell} \vert}P_{\mathrm{unst},\ell,n-1}(s)\left(Forc_{\mathrm{unst},n-1}(s,\sigma_{n},\vec{u}_{n})\right)\,ds,\\ \nonumber
  P_{\mathrm{stab},\ell,n-1}(t)\vec{u}_{n}(t)=e^{{-}\vert \lambda_{\ell} \vert t}P_{\mathrm{stab},\ell,n-1}(0)\vec{u}_{n}(0)-i\int_{0}^{t}e^{{-}\vert \lambda_{\ell} \vert(t-s)}P_{\mathrm{stab},\ell,n-1}(s)\left(F_{n-1}(s,\sigma_{n}(s),\vec{u}_{n}(s))\right)\,ds,\\ \nonumber
  \langle \vec{u}_{n}(t,x),\mathfrak{p}_3e^{i\theta_{\ell,n-1}(t)}\vec{z}(\alpha_{\ell,n-1}(t),x-y_{\ell,n-1}(t))\rangle=0, \text{ if $\vec{z}\in \ker\mathcal{H}^{2}_{1}.$}
\end{gather}
Next, we recall from \eqref{decomp}, the following  representation of $\vec{u}_{n}(t,x)$
\begin{align}\label{decompofun}
   \vec{u}_{n}(t)=& \vec{u}_{c,n}(t)+\sum_{\ell=1}^{m}b_{\ell,n,+}(t) e^{i\theta_{\ell}(t,x)\sigma_{3}}\alpha_{\ell}^{\frac{1}{k}}\vec{Z}_{+}(\alpha_{\ell}[x-v_{\ell}t-y_{\ell}])\\&{+}\sum_{\ell=1}^{n}b_{\ell,n,-}(t)\mathfrak{G}_{\ell}(\mathfrak{v}_{\alpha_{\ell,n-1},\overline{\lambda_{\ell}}})(t,x)+\vec{u}_{\mathrm{root},n}(t,x),
\end{align}
such that $\lambda_{\ell}=i\alpha_{\ell,n-1}(T_{n-1})^{2}\lambda_{0},$ $\vec{u}_{c}\in \Ra P_{c,n-1},$ and $\vec{u}_{\mathrm{root},n}(t,x)\in \Ra P_{\mathrm{root},\sigma^{T_{n-1}}_{n-1}}$ is the unique function such that the orthogonality condition is achieved 
\begin{equation*}
    \left\langle \vec{u}_{n}(t,x),\mathfrak{p}_3\mathfrak{G}_{\ell}(\mathfrak{z}_{\alpha_{\ell,n-1},0})(t,x) \right\rangle=0 \text{, for all $\mathfrak{z}\in \ker\mathcal{H}^{2}_{1},$ and $t\in [0,T_{n}].$}
\end{equation*}
\subsection{Equation satisfied by $\vec{u}_{n}-\vec{u}_{n-1}$}
Let $\vec{z}_{n}=\vec{u}_{n}-\vec{u}_{n-1}.$ We can verify from \eqref{equations for un} and the definition of $\vec{u}_{n}$ that $\vec{z}_{n}$ is a strong solution of a equation of the form
\begin{multline}\label{diffequ}
i\partial_{t}\vec{z}_{n}(t)+\mathfrak{p}_3\partial^{2}_{x}\vec{z}_{n}(t)+\sum_{j=1}^{m}V_{\ell,\sigma_{n-1}}(t,x)\vec{z}_{n}(t)\\
\begin{aligned}
    =&{-}F\left(\sum_{\ell}e^{i\theta_{\ell,n-1}}\phi_{\ell}(x-y_{\ell,n-1}(t))\right)
    {+}F\left(\sum_{\ell}e^{i\theta_{\ell,n-2}}\phi_{\ell}(x-y_{\ell,n-2}(t))\right)
    \\&{+}\sum_{\ell}F\left(e^{i\theta_{\ell,n-1}}\phi_{\ell}(x-y_{\ell,n-1}(t))\right)
    {-}\sum_{\ell}F\left(e^{i\theta_{\ell,n-2}}\phi_{\ell}(x-y_{\ell,n-2}(t))\right)\\
    &{-}\sum_{\ell}\left[V_{\ell,\sigma_{n-1}}(t,x)-V_{\ell,\sigma_{n-2}}(t,x)\right]\vec{u}_{n-1}(t,x)\\
    &{-}\left[N(\sigma_{n-1},\vec{u}_{n-1})-N(\sigma_{n-2},\vec{u}_{n-2})\right]\\
    &{-}\sum_{\ell}\Lambda \dot\sigma_{\ell,n-2}(t)\left[e^{i\mathfrak{p}_3\theta_{\ell,n-1}(t,x)}\vec{\mathcal{E}}_{\ell}(\alpha_{\ell,n-1}(t),x-y_{\ell,n-1}(t))-e^{i\mathfrak{p}_3\theta_{\ell,n-2}(t,x)}\vec{\mathcal{E}}_{\ell}(\alpha_{\ell,n-2}(t),x-y_{\ell,n-2}(t))\right]\\&{-}\sum_{\ell}\left(\Lambda \dot\sigma_{\ell,n-1}(t)-\Lambda \dot\sigma_{\ell,n-2}(t)\right)e^{i\mathfrak{p}_3\theta_{\ell,n-1}(t,x)}\vec{\mathcal{E}}_{\ell}(\alpha_{\ell,n-1}(t),x-y_{\ell,n-1}(t))\\
    =&Diff_{n,n-1}(t), 
\end{aligned}
\end{multline}
such that all $\vec{\mathcal{E}}_{\ell}$ is an element of the subspace $\ker\mathcal{H}^{2}_{1},$ and the function $N$ is defined by
\begin{multline}\label{Nonlineartermdiff}
    N(\sigma_{j},\vec{u}_{j})={-}\vert \vec{u}_{j}(t) \vert^{2k}\vec{u}_{j}(t)-F^{\prime}\left(\sum_{\ell}e^{i\theta_{\ell,j}}\phi_{\ell}(x-y_{\ell,j}(t))\right)\vec{u}_{j}(t)+\sum_{\ell}F^{\prime}\left(e^{i\theta_{\ell,j}}\phi_{\ell}(x-y_{\ell,j}(t))\right)\vec{u}_{j}(t)\\{-}\Bigg[F\left(\sum_{\ell}e^{i\theta_{\ell,j}}\phi_{\ell}(x-y_{\ell,j}(t))+u_{j}(t)\right)-F\left(\sum_{\ell}e^{i\theta_{\ell,j}}\phi_{\ell}(x-y_{\ell,j}(t))\right)\\{-}F^{\prime}\left(\sum_{\ell}e^{i\theta_{\ell,j}}\phi_{\ell}(x-y_{\ell,j}(t))\right)\vec{u}_{j}(t)-\vert \vec{u}_{j}(t) \vert^{2k}\vec{u}_{j}(t)\Bigg],
\end{multline}
see \eqref{unequation} for more details.
\par Next, in \S \ref{teo1.2}, it was verified that if Proposition \ref{propun-un-1} is true until $N\in\mathbb{N}_{\geq 1}$ and Proposition \ref{undecays} is true, then the following inequality
\begin{equation}\label{distance}
    \max_{i\in\{0,...,\,N+1\}} T_{i}\left(\max_{s\in[0,T_{i}]}\vert \alpha_{i+1}(s)-\alpha_{i}(s) \vert+\max_{s\in[0,T_{i}]}\vert v_{i+1}(s)-v_{i}(s) \vert\right)\leq A<{+}\infty
\end{equation}
is true for a constant $A>1.$ In particular, since \eqref{distance} was checked for $i=0$ in Section \ref{teo1.2}, we can assume from now on that \eqref{distance} is true for any $i\in\{1,\,...,\,n-2\}.$ Consequently, for any $i\in\{1,\,...,\,n-2\},$ the following estimate holds.
\begin{equation}\label{differenceofyiyi+1}
    \max_{\ell\in[m],\,s\in[0,T_{i}]} \vert y_{\ell,i}(s)-y_{\ell,i+1}(s) \vert\lesssim A.
\end{equation}
\par Before starting the proof of Proposition \ref{propun-un-1}, the computations, we consider the following propositions. They will be useful in the estimates in the next subsections.
\begin{proposition}\label{shititeraction}
Let $\sigma_{n}$ and $\sigma_{n-1}$ satisfy the hypotheses of Proposition \ref{undecays}. 
Let \begin{align*}
    Int_{n-1}(t,x)\coloneqq &F\left(\sum_{\ell}e^{i\theta_{\ell,n-1}(t,x)}\phi_{\ell,n-1}(x-y_{\ell,n-1}(t))\right)\\&{-}\sum_{\ell}F\left(e^{i\theta_{\ell,n-1}(t,x)}\phi_{\ell,n-1}(x-y_{\ell,n-1}(t))\right)
\end{align*}
Assuming for any $n$ that $\max_{\ell,t\in[0,T_{n-1}]}\vert y_{\ell,n-1}(t)- y_{\ell,n-2}(t) \vert<A$ for a constant $A,$ there exists $C(\alpha_{\ell}(0))>1$ such that for any $t\in[0,T_{n-1}]$ the following estimate holds
\begin{multline*}
    \max_{h\in\{0,1,2\},s\in\{1,2\},j\in[m]}\norm{\langle x-y_{j,n-1}(t) \rangle^{h} \left[Int_{n-1,j}(t,x)-Int_{n-2,j}(t,x)\right]}_{H^{s}_{x}(\mathbb{R})}\\
    \leq C(A,\alpha_{\ell}(0))e^{{-}\frac{99}{100}(\min_{\ell}\alpha_{\ell}(0)) [\min_{\ell}y_{\ell,n-1}(t)-y_{\ell+1,n-1}(t)]}\max_{M\in\{y,v,\gamma,\alpha\},\ell}\vert M_{\ell,n-1}(t)-M_{\ell,n-2}(t) \vert.
\end{multline*}
\end{proposition}
\begin{proof}
    First, from the fundamental theorem of calculus, and identity $F(0)=0,$ we can deduce the following equation.
    \begin{multline*}
     Int_{n-1}(t,x)-Int_{n-2}(t,x)\\
    \begin{aligned}
=&\int_{0}^{1}F^{\prime}\left(\beta \sum_{\ell=1}^{m} e^{i\theta_{\ell,n-1}(t,x)}\phi(\alpha_{\ell,n-1}(t),x-y_{\ell,n-1}(t))\right)\\
     &\times \left[\sum_{\ell=1}^{m}e^{i\mathfrak{p}_3\theta_{\ell,n-1}(t,x)}\phi(\alpha_{\ell,n-1}(t),x-y_{\ell,n-1}(t))\right]\,d\beta
    \\
    &{-}\sum_{\ell=1}^{m}\int_{0}^{1}F^{\prime}\left(\beta e^{i\theta_{\ell,n-1}(t,x)}\phi(\alpha_{\ell,n-1}(t),x-y_{\ell,n-1}(t))\right)e^{i\mathfrak{p}_3\theta_{\ell,n-1}(t,x)}\phi(\alpha_{\ell,n-1}(t),x-y_{\ell,n-1}(t))
    \\&{-}\int_{0}^{1}F^{\prime}\left(\sum_{\ell=1}^{m}\beta e^{i\theta_{\ell,n-2}(t,x)}\phi(\alpha_{\ell,n-2}(t),x-y_{\ell,n-2}(t))\right)\\
     &\times \left[\sum_{\ell=1}^{m}e^{i\mathfrak{p}_3\theta_{\ell,n-2}(t,x)}\phi(\alpha_{\ell,n-2}(t),x-y_{\ell,n-2}(t))\right]\,d\beta\\
     &{+}\sum_{\ell=1}^{m}\int_{0}^{1}F^{\prime}\left(\beta e^{i\theta_{\ell,n-2}(t,x)}\phi(\alpha_{\ell,n-2}(t),x-y_{\ell,n-2}(t))\right)e^{i\mathfrak{p}_3\theta_{\ell,n-2}(t,x)}\phi(\alpha_{\ell,n-2}(t),x-y_{\ell,n-2}(t))
    \end{aligned}.
    \end{multline*}
     From now on, we consider the following functions that interpolate $\sigma_{n-1}(t)$ and $\sigma_{n-2}(t).$
    \begin{equation}\label{sigmadiffe}
        \begin{bmatrix}
            \alpha_{\ell,\beta,n-1,n-2}(t)\\
            v_{\ell,\beta,n-1,n-2}(t)\\
            y_{\ell,\beta,n-1,n-2}(t)\\
            \gamma_{\ell,\beta,n-1,n-2}(t)
        \end{bmatrix}=\begin{bmatrix}
            \alpha_{\ell,n-2}(t)+\beta[\alpha_{\ell,n-1}(t)-\alpha_{\ell,n-2}(t)]\\
            v_{\ell,n-2}(t)+\beta[v_{\ell,n-1}(t)-v_{\ell,n-2}(t)]\\
            y_{\ell,\beta,n-2}(t)+\beta[y_{\ell,n-1}(t)-y_{\ell,n-2}(t)]\\
            \gamma_{\ell,n-2}(t)+\beta[\gamma_{\ell,n-1}(t)-\gamma_{\ell,n-1}(t)]
        \end{bmatrix} \text{, for any $\beta \in [0,1],$}
    \end{equation}
    and
    \begin{equation}\label{thetaellbeta}
        \theta_{\ell,\beta,n-1,n-2}(t,x)=\frac{v_{\ell,\beta,n-1,n-2}(t)x}{2}+\gamma_{\ell,\beta,n-1,n-2}(t).
    \end{equation}
    Consequently, using the fundamental theorem of calculus again, we can verify the following identity
\begin{multline}\label{diffInt1-2first}
Int_{n-1}(t,x)-Int_{n-2}(t,x)\\
\begin{aligned}
=&\sum_{\ell=1}^{m}\int_{0}^{1}\int_{0}^{1}F^{\prime\prime}\left(\beta e^{i\theta_{\ell,n-1}(t,x)}\phi_{\alpha_{\ell,n-1}(t)}(x-y_{\ell,n-1}(t))+\beta\beta_{1}\sum_{j=1,j\neq \ell}^{m}e^{i\theta_{j,n-1}(t,x)}\phi_{\alpha_{j,n-1}(t)}(x-y_{j,n-1}(t))\right)\\
    &\times \beta \Big[e^{i\mathfrak{p}_3\theta_{\ell,n-1}(t,x)}\phi_{\alpha_{\ell,n-1}(t)}(x-y_{\ell,n-1}(t))\Big]\Big[\sum_{j=1,j\neq \ell}^{m}e^{i\theta_{j,n-1}(t,x)}\phi_{\alpha_{j,n-1}(t)}(x-y_{j,n-1}(t))\Big]\,d\beta d\beta_{1}
    \\&{-}\sum_{\ell=1}^{m}\int_{0}^{1}\int_{0}^{1}F^{\prime\prime}\left(\beta e^{i\theta_{\ell,n-2}(t,x)}\phi_{\alpha_{\ell,n-2}(t)}(x-y_{\ell,n-2}(t))+\beta\beta_{1}\sum_{j=1,j\neq \ell}^{m}e^{i\theta_{j,n-2}(t,x)}\phi_{\alpha_{j,n-2}(t)}(x-y_{j,n-2}(t))\right)\\
    &\times \beta \Big[e^{i\mathfrak{p}_3\theta_{\ell,n-2}(t,x)}\phi_{\alpha_{\ell,n-2}(t)}(x-y_{\ell,n-2}(t))\Big]\Big[\sum_{j=1,j\neq \ell}^{m}e^{i\theta_{j,n-2}(t,x)}\phi_{\alpha_{j,n-2}(t)}(x-y_{j,n-2}(t))\Big]\,d\beta d\beta_{1}.
\end{aligned}
\end{multline}
\par Therefore, using the functions \eqref{sigmadiffe}, \eqref{thetaellbeta}, and the fact that $F$ defined in \eqref{Fdefinition} is in $C^{4},$ the identity $F^{\prime\prime}(0)=0$ for all $k>2,$ and the elementary identity below
\begin{multline*}
F^{\prime\prime} \left(\sum_{\ell=1}^{m}[\beta(1-\delta^{\ell}_{j})+\delta^{\ell}_{j}] e^{i\theta_{\ell,n-1}(t,x)}\phi(\alpha_{\ell,n-1}(t),x-y_{\ell,n-1}(t))\right)\\
-F^{\prime\prime} \left(\sum_{\ell=1}^{m}[\beta(1-\delta^{\ell}_{j})+\delta^{\ell}_{j}] e^{i\theta_{\ell,n-2}(t,x)}\phi(\alpha_{\ell,n-2}(t),x-y_{\ell,n-2}(t))\right)\\
=\sum_{\ell=1,\ell\neq j}^{m}\int_{0}^{1}\sum_{h\in\{v,y,\alpha,\gamma\}}F^{\prime\prime\prime}\left(\sum_{\ell=1}^{m}[\beta(1-\delta^{\ell}_{j})+\delta^{\ell}_{j}] e^{i\theta_{\ell,\beta,n-1,n-2}(t,x)}\phi(\alpha_{\ell,\beta,n-1,n-2}(t),x-y_{\ell,\beta,n-1,n-2}(t))\right)\\
\times \beta (h_{\ell,n-1}(t)-h_{\ell,n-2}(t)) \frac{\partial}{\partial h}\left[e^{i\theta_{\ell,\beta_{1},n-1,n-2}(t,x)\mathfrak{p}_3}\phi(\alpha_{\ell,\beta,n-1,n-2}(t),x-y_{\ell,\beta_{1},n-1,n-2})\right]\,d\beta_{1}\\
{+}\int_{0}^{1}\sum_{h\in\{v,y,\alpha,\gamma\}}F^{\prime\prime\prime}\left(\sum_{\ell=1}^{m}[\beta(1-\delta^{\ell}_{j})+\delta^{\ell}_{j}] e^{i\theta_{j,\beta,n-1,n-2}(t,x)}\phi(\alpha_{j,\beta,n-1,n-2}(t),x-y_{\ell,\beta,n-1,n-2}(t))\right)\\
\times (h_{j,n-1}(t)-h_{j,n-2}(t)) \frac{\partial}{\partial h}\left[e^{i\theta_{j,\beta_{1},n-1,n-2}(t,x)\mathfrak{p}_3}\phi(\alpha_{j,\beta,n-1,n-2}(t),x-y_{j,\beta_{1},n-1,n-2}(t))\right]\,d\beta_{1},
\end{multline*}
 we can deduce using Lemma \ref{interactt} and the estimates
 \begin{equation*}
     \left\vert \frac{\partial^{\ell}}{\partial x^{\ell}}\phi_{\alpha}(x) \right\vert\lesssim_{\ell,\alpha} e^{{-}\alpha \vert x \vert} \text{, for all $\alpha>0,\,\ell\in\mathbb{N},$}
  \end{equation*} and \eqref{differenceofyiyi+1} the following inequality
\begin{multline}\label{firstintn-1n-2diff}
 \max_{d\in\{0,1,2\},h\in\{n-1,n-2\}} \Bigg\vert\Bigg\vert \langle x-y_{j,h}(t) \rangle^{d}\left[Int_{n-1}(t,x)-Int_{n-2}(t,x)\right] \Bigg\vert\Bigg\vert_{H^{2}_{x}(\mathbb{R})}\\
\begin{aligned}
\lesssim & \max_{h\in\{0,1\},\ell}\vert y_{\ell,n-1}(t)-y_{\ell-1,n-1}(t)\vert ^{3}e^{{-}\alpha_{\ell-h,n-1}(t)\vert y_{\ell,n-1}(t)-y_{\ell-1,n-1}(t)\vert}\\
&\times \left[\max_{h\in\{y,v,\alpha,\gamma\}}\vert h_{\ell,n-1}(t)-h_{\ell,n-2}(t) \vert\right].
\end{aligned}
\end{multline}
In conclusion, from the assumption of the hypothesis $\mathrm{(H2)},$ estimates \eqref{dyn11}, \eqref{dyn12}, and \eqref{firstintn-1n-2diff}, we obtain the result of Proposition \ref{shititeraction}.
\end{proof}
\begin{proposition}\label{diffprop1}
Let
\begin{align*}
    \theta_{\ell,n}(t,x)=\frac{v_{\ell,n}(t)x}{2}+\gamma_{\ell,n}(t),\,\, \theta_{\ell,n}(t,x)=\frac{v_{\ell,n}(t)x}{2}+\gamma_{\ell,n}(t).
\end{align*}
If $\tau \in (0,1),$ the following inequality holds for any Schwartz function $W$
\begin{align}\label{W43}
    &\max_{q\in[1,{+}\infty]}\max_{t\in[0,T_{n}]}\norm{ e^{i\theta_{\ell,n}(t,x)}W(x-y_{\ell,n}(t))-e^{i\theta_{\ell,n-1}(t,x)}W(x-y_{\ell,n-1}(t))}_{L^{q}_{x}(\mathbb{R})}\\ \nonumber
    & \lesssim_{\tau} \max_{s\in[0,T_{n}],M\in\{v,\alpha,D\}}\langle s \rangle^{1+\tau}\vert M_{\ell,n}(s)-M_{\ell,n-1}(s) \vert{+}\max_{s\in[0,T_{n}]} \vert \hat{\gamma}_{\ell,n}(s)-\hat{\gamma}_{\ell,n-1}(s) \vert
\end{align}
where
\begin{equation*}
     \hat{\gamma}_{\ell,n}(t):=
    \gamma_{\ell,n}(t)+\frac{v_{\ell,n}(t)y_{\ell,n}(t)}{2}.   
\end{equation*}
Moreover, if $V(\alpha,x)$ and $W(\alpha,x)$ are two smooth functions satisfying for all $\alpha>0$ and $x\in\mathbb{R}$
\begin{equation}\label{mmmmm1}
 \max_{n\in\{0,1\},z\in\{x,\alpha\}}\vert \frac{\partial^{n}}{\partial z^{n}} W(\alpha,x) \vert+ \vert V(\alpha,x) \vert\lesssim_{\alpha} e^{\frac{{-}999\alpha\vert x\vert }{1000}},
\end{equation}
then
\begin{align}\nonumber
\max_{q\in[1,2],h\in[m]}\max_{t\in[0,T_{n}]}\Bigg\vert\Bigg\vert &\chi_{h,n}(t,x) \langle x-y_{h,n}(t) \rangle V(\alpha_{j,n}(t),x-y_{j,n}(t))\Big[ e^{i\theta_{\ell,n}(t,x)}W(\alpha_{\ell,n}(t),x-y_{\ell,n}(t))\\ \label{intpartt}
&{-}e^{i\theta_{\ell,n-1}(t,x)}W(\alpha_{\ell,n-1}(t),x-y_{\ell,n-1}(t))\Big]\Bigg\vert\Bigg\vert_{L^{q}_{x}(\mathbb{R})}\\ \nonumber
    & \lesssim_{\tau} e^{{-}\min_{\ell,j}\frac{99\alpha_{\ell,n}(t)(y_{j,n}(t)-y_{j+1,n}(t))}{100}}\max_{s\in[0,T_{n}],M\in\{v,\alpha,D\}}\langle s \rangle^{1+\tau}\vert M_{\ell,n}(s)-M_{\ell,n-1}(s) \vert\\ \nonumber
    &{+} e^{{-}\min_{\ell,j}\frac{99\alpha_{\ell,n}(t)(y_{j,n}(t)-y_{j+1,n}(t))}{100}}\max_{s\in[0,T_{n}]} \vert \hat{\gamma}_{\ell,n}(s)-\hat{\gamma}_{\ell,n-1}(s) \vert
\end{align}
\end{proposition}
\begin{remark}\label{remarkVnvn-1}
 In particular, for any $t\in [0,T_{n}],$ we can deduce from the proof of Proposition \ref{diffprop1} the following estimate holds
 \begin{align*}
    &\left\vert e^{i\theta_{\ell,n}(s,x)}W(x-y_{\ell,n}(s))-e^{i\theta_{\ell,n-1}(s,x)}W(x-y_{\ell,n-1}(s)) \right\vert\\
   & \lesssim  \max_{M\in\{v,\alpha\}}\langle t \rangle\vert M_{\ell,n}(s)-M_{\ell,n-1}(s) \vert{+} \vert \hat{\gamma}_{\ell,n}(t)-\hat{\gamma}_{\ell,n-1}(t) \vert{+} \vert y_{\ell,n}(t)-y_{\ell,n-1}(t) \vert.
\end{align*}
\end{remark}
\begin{corollary}\label{diffprop2}
For any $\tau\in (0,1),$ if $W$ is a Schwartz function, the following estimate holds for any $t\in [0,T_{n}].$
\begin{multline*}
    \max_{q\in[1,{+}\infty]}\norm{ e^{i\theta_{\ell,n}(t,x)}W(x-y_{\ell,n}(t))-e^{i\theta_{\ell,n-1}(t,x)}W(x-y_{\ell,n-1}(t))}_{L^{q}_{x}(\mathbb{R})}\\
     \lesssim_{\tau} \max_{M\in\{v,\alpha,D,\Gamma\}}\langle t \rangle^{1+\tau} \max_{s\in [0,t]} \langle s  \rangle^{1+\tau}\vert \dot M_{\ell,n}(s)-\dot M_{\ell,n-1}(s) \vert,
\end{multline*}
where
\begin{equation*}
   \dot \Gamma_{\ell,n}(t)\coloneqq \dot\gamma_{\ell,n}(t)-\alpha_{\ell,n}(t)^{2}-\frac{v_{\ell,n}(t)^{2}}{4}+\frac{y_{\ell,n}(t)\dot v_{\ell,n}(t)}{2}.
\end{equation*}
Furthermore, for any $\tau,\,\tau_{1}\in (0,1),$
\begin{equation*}
   \vert y_{\ell,n}(t)-y_{\ell,n-1}(t) \vert+\vert \gamma_{\ell,n}(t)-\gamma_{\ell,n-1}(t) \vert\lesssim_{\tau,\tau_{1}} \max_{M\in\{v,\alpha,D,\Gamma\}}\langle t \rangle^{1+\tau_{1}} \max_{s\in [0,t]} \langle s  \rangle^{1+\tau}\vert \dot M_{\ell,n}(s)-\dot M_{\ell,n-1}(s) \vert. 
\end{equation*}
\end{corollary}
\begin{proof}[Proof of Corollary \ref{diffprop2}]
First, we can verify that
\begin{multline}\label{firstdgamma}
  \left\vert \gamma_{\ell,n}(t)+\frac{v_{\ell,n}(t)y_{\ell,n}(t)}{2}-\gamma_{\ell,n-1}(t)-\frac{v_{\ell,n-1}(t)y_{\ell,n-1}(t)}{2} \right\vert\\
  \begin{aligned}
      \leq & \int_{0}^{t}\left\vert(\dot\gamma_{\ell,n}(s)-\dot \gamma_{\ell,n-1}(s))+\left(\frac{\dot v_{\ell,n}(s)y_{\ell,n}(s)-\dot v_{\ell,n-1}(s)y_{\ell,n-1}(s)}{2}\right)\right\vert\,ds\\
      &{+}\int_{0}^{t}\frac{v_{\ell,n}(s)^{2}-v_{\ell,n-1}(s)^{2}}{2}\,ds\\&{+}\int_{0}^{t}\left[\vert v_{\ell,n}(s) \vert+\vert v_{\ell,n-1}(s)\vert\right]\left\vert(\dot y_{\ell,n}(s)-v_{\ell,n}(s)-\dot y_{\ell,n-1}(s)+v_{\ell,n-1}(s))\right\vert\,ds\\
      &{+}\int_{0}^{t}\vert v_{\ell,n}(s)-v_{\ell,n-1}(s) \vert \max\left(\vert \dot y_{\ell,n}(s)-v_{\ell,n}(s)\vert,\vert \dot y_{\ell,n-1}(s)-v_{\ell,n-1}(s)\vert\right)\,ds.
    \end{aligned}
 \end{multline}
In particular, using
\begin{equation*}
    d\gamma_{\ell,n}(s):=\dot \gamma_{\ell,n}(s)+\frac{y_{\ell,n}(s)\dot v_{\ell,n}(s)}{2}-\alpha_{\ell,n}(s)^{2}+\frac{v_{\ell,n}(s)^{2}}{4},
\end{equation*}
we can verify that the first two integrals on the right-hand side of the inequality above are bounded above by
\begin{equation*}
 C_{\tau}\left[(\max_{s\in [0,t]}\langle s \rangle^{1+\tau}\vert d\gamma_{\ell,n}(s)-d\gamma_{\ell,n-1}(s)\vert) + \langle t \rangle^{1+\tau}\max_{M\in\{\alpha,v\},s\in[0,t]}\langle s\rangle^{1+\tau} \vert\dot M_{\ell,n}(s)-\dot M_{\ell,n-1}(s)\vert\right].    
\end{equation*}
Therefore, we can deduce from \eqref{firstdgamma} that
\begin{multline}\label{cor222}
    \left\vert \gamma_{\ell,n}(t)+\frac{v_{\ell,n}(t)y_{\ell,n}(t)}{2}-\gamma_{\ell,n-1}(t)-\frac{v_{\ell,n-1}(t)y_{\ell,n-1}(t)}{2} \right\vert\\
    \leq C_{\tau}\left[(\max_{s\in [0,t]}\langle s \rangle^{1+\tau}\vert d\gamma_{\ell,n}(s)-d\gamma_{\ell,n-1}(s)\vert) + \langle t \rangle^{1+\tau}\max_{M\in\{\alpha,v\},s\in[0,t]}\langle s\rangle^{1+\tau} \vert\dot M_{\ell,n}(s)-\dot M_{\ell,n-1}(s)\vert\right]\\
    {+}C_{\tau}\max_{s\in[0,t]}\langle s\rangle^{1+\tau}\left[\vert \dot v_{\ell,n}(s)-\dot v_{\ell,n-1}(s) \vert+\vert \dot \alpha_{\ell,n}(s)-\alpha_{\ell,n-1}(s) \vert \right] 
\end{multline}
\par Next, using the Fundamental Theorem of Calculus, we can deduce that
\begin{equation}\label{cor111}
\max_{s\in[0,T_{n}],M\in\{v,\alpha,D\}}\left\vert M_{\ell,n}(s)-M_{\ell,n-1}(s) \right\vert\lesssim_{\tau} \max_{s\in[0,T_{n}],M\in\{v,\alpha,D\}}\langle s\rangle^{1+\tau_{1}}\left\vert \dot M_{\ell,n}(s)-\dot M_{\ell,n-1}(s) \right\vert.
\end{equation}
In conclusion, Corollary \ref{diffprop2} follows from Proposition \ref{diffprop1}, and estimates \eqref{cor111}, and \eqref{cor222}.
\end{proof}
\begin{proof}[Proof of Proposition \ref{diffprop1}]
 The proof is similar to the proof of Proposition $4.8$ of \cite{KriegerSchlag}. More precisely, from the definition of $\theta_{\ell,n}$ we have that
 \begin{multline*}
     \theta_{\ell,n}(t,x)-\theta_{\ell,n-1}(t,x)=\frac{(v_{\ell,n}(t)-v_{\ell,n-1}(t))(x-y_{\ell,n}(t))}{2}+\frac{v_{\ell,n-1}(y_{\ell,n-1}-y_{\ell,n})}{2}\\
     {+}\left[\gamma_{\ell,n}(t)+\frac{v_{\ell,n}(t)y_{\ell,n}(t)}{2}-\gamma_{\ell,n-1}(t)-\frac{v_{\ell,n-1}(t)y_{\ell,n-1}(t)}{2}\right],
 \end{multline*}
 we also recall that for any $\tau\in (0,1)$
 \begin{align*}
     \vert y_{\ell,n}(t)-y_{\ell,n-1}(t)\vert \leq& \vert D_{\ell,n}(t)-D_{\ell,n-1}(t)\vert+ \int_{0}^{t}\vert v_{\ell,n}(s)-v_{\ell,n-1}(s) \vert\,ds \\
     \leq & C_{\tau}\max_{s\in [0,t]}\langle s \rangle^{1+\tau}\vert v_{\ell,n}(s)-v_{\ell,n-1}(s) \vert+\max_{s\in [0,t]}\vert D_{\ell,n}(s)-D_{\ell,n-1}(s)\vert.  
 \end{align*}
 Moreover, the Fundamental Theorem of Calculus implies that
 \begin{equation*}
  \left\vert W(x-y_{\ell,n}(t))-W(x-y_{\ell,n-1}(t))\right\vert\leq \vert y_{\ell,n}(t)-y_{\ell,n-1}(t) \vert \max_{h\in [0,1]} \left\vert W^{\prime}(x-h y_{\ell,n}(t)-(1-h)y_{\ell,n-1}(t)) \right\vert.  \end{equation*}
In conclusion, since $W$ is a Schwartz function, it is not difficult to verify that the three estimates above imply the inequality \eqref{W43}.
\par Similarly, we can verify from the fundamental theorem of calculus that if $V(\alpha,x),\,W(\alpha,x)$ are in $C^{1}(\mathbb{R}_{>0}\times\mathbb{R})$ then
\begin{multline*}
    \left\vert \chi_{h,n}(t,x) \langle x-y_{h,n}(t)\rangle V(\alpha_{j,n}(t),x-y_{j,n}(t)) [W(\alpha_{\ell,n}(t),x-y_{\ell,n}(t))-W(\alpha_{\ell,n-1}(t),x-y_{\ell,n-1}(t))]\right\vert\\
    \lesssim  \max_{h\in [0,1],z\in\{\alpha,x\}} \left\vert \partial_{z} W((1-h)\alpha_{\ell,n-1}(t)+h(\alpha_{\ell,n-1}(t)),x-hy_{\ell,n}(t)-(1-h)y_{\ell,n-1}(t)) \right\vert\\
    \times  \left\vert\chi_{h,n}(t,x) \langle x-y_{h,n}(t)\rangle V(\alpha_{j,n}(t),x-y_{j,n}(t)) V(\alpha_{j,n}(t),x-y_{j,n}(t))\right\vert\\
    \times \left[ \vert y_{\ell,n}(t)-y_{\ell,n-1}(t) \vert +\vert \alpha_{\ell,n}(t)-\alpha_{\ell,n-1}(t) \vert \right].
\end{multline*}
In particular, from the Definition \ref{cutlinearpath} and Proposition \ref{undecays}, we can verify using \eqref{mmmmm1} for any $h\in [m]$ that
\begin{equation*}
    \left\vert\chi_{h,n}(t,x) \langle x-y_{h,n}(t)\rangle V(\alpha_{j,n}(t),x-y_{j,n}(t)) V(\alpha_{j,n}(t),x-y_{j,n}(t))\right\vert
    \lesssim e^{\frac{{-}999\alpha_{j,n}(t) \vert x -y_{j,n}(t)\vert}{1000}}.
\end{equation*}
Consequently, since $W(\alpha,x)$ satisfy \eqref{mmmmm1} for any $\alpha>0$ and $x\in\mathbb{R},$ using the decay estimates \eqref{odes} in Proposition \ref{undecays} and Lemma \ref{interactt}, we can verify that the estimate \eqref{intpartt} holds for all $t\geq 0.$

\end{proof}
\subsection{Estimate of unstable components}
In this subsection,  we will estimate $P_{\mathrm{unst},\ell,n-1}\vec{z}_{n}$.
\begin{proposition}\label{hyperbolicerror}
If
\begin{gather*}
   \vec{u}_{n}(0)=\vec{r}_{0}\\{+}\sum_{\ell}h_{\ell,n-1}(0)e^{i\theta_{\ell}(0,x)\mathfrak{p}_3}\vec{Z}_{+}(\alpha_{\ell,n-1}(T_{n}),x-y_{\ell}(0))+\sum_{\ell}e^{i\theta_{\ell}(0,x)\mathfrak{p}_3}\vec{E}_{\ell,n-1}(\alpha_{\ell,n-1}(T_{n}),x-y_{\ell}(0)),\\
   \vec{u}_{n-1}(0)=\vec{r}_{0}\\{+}\sum_{\ell}h_{\ell,n-2}(0)e^{i\theta_{\ell}(0,x)\mathfrak{p}_3}\vec{Z}_{+}(\alpha_{\ell,n-2}(T_{n-1}),x-y_{\ell}(0))+\sum_{\ell}e^{i\theta_{\ell}(0,x)\mathfrak{p}_3}\vec{E}_{\ell,n-2}(\alpha_{\ell,n-2}(T_{n-1}),x-y_{\ell}(0))
\end{gather*}
such that $\vec{E}_{\ell,h}\in \ker\mathcal{H}^{2}_{1}$ for all $\ell\in [m],$ $h\in\{n-1,n-2\},$ and
\begin{align*}
  P_{\mathrm{unst},\ell,n}\left(\vec{u}_{n}(0)\right)=&i\int_{0}^{T_{n}}e^{{-}s \lambda_{0} \alpha_{\ell,n-1}(T_{n-1})^{2} }P_{\mathrm{unst},\ell,n-1}(s)\left(Forc_{\mathrm{unst},n-1}(s,\sigma_{n},\vec{u}_{n})\right)\,ds=b_{\ell,+,n}(0),\\
   P_{\mathrm{unst},\ell,n-1}\left(\vec{u}_{n-1}(0)\right)=&i\int_{0}^{T_{n-1}}e^{{-}s\lambda_{0} \alpha_{\ell,n-2}(T_{n-2})^{2}}P_{\mathrm{unst},\ell,n-2}(s)\left(Forc_{\mathrm{unst},n-2}(s,\sigma_{n-1},\vec{u}_{n-1})\right)\,ds\\
   =&b_{\ell,+,n-1}(0),
\end{align*}
then
\begin{align*}
  \max_{\ell} \left\vert h_{\ell,n-1}(0)-h_{\ell,n-2}(0)  \right\vert\lesssim & \delta_{0}\norm{(\Pi_{n-1}-\Pi_{n-2},\vec{u}_{n}-\vec{u}_{n-1})}_{Y_{n-1}}\\&{+}\max_{\ell}\norm{ P_{\mathrm{unst},\ell,n-1}\left( \vec{u}_{n}(0)-\vec{u}_{n-1}(0)\right) }_{L^{2}_{x}(\mathbb{R})}{+}\frac{\delta_{0}}{T_{n-1}^{2\epsilon}}.   
\end{align*}
\end{proposition}
\begin{proof}
First, 
the functions $\vec{E}_{\ell,n-1},\,\vec{E}_{\ell,n-2}$ of $ \ker\mathcal{H}^{2}_{1}$ can be rewritten as
\begin{equation*}
    \vec{E}_{\ell,n-1}=H_{n-1}((h_{\ell,n-1})_{\ell},\vec{r}(0)),\, \vec{E}_{\ell,n-2}=H_{n-2}((h_{\ell,n-2})_{\ell},\vec{r}(0)),
\end{equation*}
for unique bilinear continuous functions $H_{n-1},H_{n-2}:\mathbb{C}^{m}\times L^{2}_{x}(\mathbb{R})\to L^{2}_{x}(\mathbb{R})$ to allow $\vec{u}_{n}$ and $\vec{u}_{n-1}$ satisfy
\begin{align}\label{orthogun}
    \langle \vec{u}_{n}(0),\mathfrak{p}_3e^{i\mathfrak{p}_3\theta_{\ell}(0,x)}z(\alpha_{\ell}(0),x-y_{\ell}(0)) \rangle=  \langle \vec{u}_{n-1}(0),\mathfrak{p}_3e^{i\mathfrak{p}_3\theta_{\ell}(0,x)}z(\alpha_{\ell}(0),x-y_{\ell}(0)) \rangle=0.
\end{align}
 
From, the definition of $\vec{u}_{n}$ and $\vec{u}_{n-1},$ we have that
\begin{align}\label{pq1}
    \vec{u}_{n}(0)-\vec{u}_{n-1}(0)=&\sum_{\ell}\left[h_{\ell,n-1}(0)-h_{\ell,n-2}(0)\right]e^{i\theta_{\ell}(0,x)\mathfrak{p}_3}\vec{Z}_{+}\left(\alpha_{\ell,n-1}(T_{n}),x-y_{\ell}(0))\right)\\ \nonumber
    &{+}\sum_{\ell}h_{\ell,n-2}(0)e^{i\theta_{\ell}(0,x)\mathfrak{p}_3}\left[\vec{Z}_{+}\left(\alpha_{\ell,n-1}(T_{n}),x-y_{\ell}(0))\right)-Z_{+}\left(\alpha_{\ell,n-2}(T_{n-1}),x-y_{\ell}(0))\right)\right]\\ \nonumber
    &{+}\sum_{\ell}e^{i\theta_{\ell}(0,x)\mathfrak{p}_3}\left[H_{n-1,\ell}((h_{j,n-1})_{j},\vec{r}(0))-H_{n-2,\ell}((h_{j,n-2})_{j},\vec{r}(0))\right].
\end{align}
\par Next, since $\vec{r}_{0}$ satisfies \eqref{r0cond}, we can deduce using Proposition \ref{undecays} that 
\begin{equation*}
    \norm{\vec{u}_{n}(0)}_{L^{2}}+\norm{\vec{u}_{n-1}(0)}_{L^{2}}+\max_{\ell,j\in\{n-1,n-2\}}\vert h_{\ell,j}(0) \vert\lesssim \delta_{0}.
\end{equation*}
As a consequence, $\max_{\ell}\vert h_{\ell,n-1}(0)  \vert+\vert h_{\ell,n-2}(0) \vert\lesssim \delta_{0},$ from which we can verify using Corollary \ref{diffprop2} that
\begin{multline}\label{estttt1}
   \max_{j}\vert h_{\ell,n-2}(0)\vert\norm{\vec{Z}_{+}\left(\alpha_{\ell,n-1}(T_{n}),x)\right)-\vec{Z}_{+}\left(\alpha_{\ell,n-2}(T_{n-1}),x\right)}_{L^{2}_{x}(\mathbb{R})}\\
   \lesssim \delta_{0}\Big[\max_{s\in [0,T_{n-2}]}\langle s \rangle^{1+\frac{\epsilon}{2}-\frac{3}{8}}\vert \dot\alpha_{\ell,n-1}(s)-\dot\alpha_{\ell,n-2}(s) \vert+\int_{T_{n-2}}^{T_{n-1}}\vert \dot \alpha_{\ell,n-1}(s) \vert \,ds\Big].  
\end{multline}
\par Consequently, we obtain from \eqref{pq1} the following estimate
\begin{multline}\label{popopopo1}
 \begin{aligned}
    \max_{\ell} \vert h_{\ell,n-1}(0)-h_{\ell,n-2}(0) \vert\lesssim & \max_{\ell}\norm{P_{\mathrm{unst},\ell,n-1}\left[\vec{u}_{n}(0)-\vec{u}_{n-1}(0)\right]}_{L^{2}_{x}(\mathbb{R})}\\ &{+}\delta_{0} \left[\max_{\ell}\vert \alpha_{\ell,n-1}(T_{n})-\alpha_{\ell,n-2}(T_{n-1}) \vert\right]\\
    &{+}\delta_{0}\max_{\ell} \norm{H_{n-1,\ell}((h_{j,n-1})_{j},\vec{r}(0))-H_{n-2,\ell}((h_{j,n-2})_{j},\vec{r}(0))}_{L^{2}_{x}(\mathbb{R})}.
\end{aligned}
\end{multline}
Next, we recall that each function $\vec{E}_{\ell,n-1}\in L^{2}_{x}(\mathbb{R},\mathbb{C}^{2})$ is a linear combination of at most $4$ functions, since $\dim \ker\mathcal{H}^{2}_{1}=4,$ the same conclusion holds for $\mathcal{E}_{\ell,n-2}.$ Moreover, using the following decay estimate from Proposition \ref{undecays}
\begin{equation*}
    \vert \dot \alpha_{\ell,n-1}(t) \vert\leq \frac{\delta_{0}}{(1+t)^{1+2\epsilon}} \text{ for all $t\geq 0,,$}
\end{equation*}
we can consider the following inequality
\begin{equation}\label{alphainvT}
    \vert \alpha_{\ell,n-1}(T_{n})-\alpha_{\ell,n-2}(T_{n-1}) \vert\lesssim \vert \alpha_{\ell,n-1}(T_{n-1})-\alpha_{\ell,n-2}(T_{n-1}) \vert +\frac{\delta_{0}}{(1+T_{n-1})^{2\epsilon}}.
\end{equation}
\par Furthermore, we can deduce from the difference of the two orthogonal equations in \eqref{orthogun} and identity \eqref{pq1}, and the the fundamental theorem of calculus that
\begin{multline}\label{esttt2}
    \norm{\vec{E}_{\ell,n-1}(\alpha_{\ell,n-1}(T_{n}),x-y_{\ell}(0))-\vec{E}_{\ell,n-2}(\alpha_{\ell,n-2}(T_{n-1}),x-y_{\ell}(0))}_{L^{2}_{x}(\mathbb{R})}\\
    \lesssim  \delta_{0} \norm{\vec{u}_{n}(0)-\vec{u}_{n-1}(0)}_{L^{2}_{x}(\mathbb{R})}+ \delta_{0} \vert \alpha_{\ell,n-1}(T_{n})-\alpha_{\ell,n-2}(T_{n-1})\vert\\
    {+} e^{{-}\min_{j}\alpha_{j\pm 1}\vert y_{j}(0)- y_{j\pm 1}(0)\vert}\max_{\ell}\vert h_{\ell,n-1}(0)-h_{\ell,n-2}(0)\vert
    \\{+}\delta_{0} \max_{\ell} \vert h_{\ell,n-1}(0)-h_{\ell,n-2}(0)\vert.
\end{multline}
In particular, using the definitions of $\vec{u}_{n}(0),\,\vec{u}_{n-1}(0)$ and \eqref{alphainvT}, we can improve the estimate \eqref{esttt2} by the following inequality
\begin{multline}\label{esttt22}
    \norm{\vec{E}_{\ell,n-1}(\alpha_{\ell,n-1}(T_{n}),x-y_{\ell}(0))-\vec{E}_{\ell,n-2}(\alpha_{\ell,n-2}(T_{n-1}),x-y_{\ell}(0))}_{L^{2}_{x}(\mathbb{R})}\\
    \lesssim  \delta_{0} \max_{\ell}\vert h_{\ell,n-1}(0)-h_{\ell,n-2}(0) \vert+ \delta_{0} \vert \alpha_{\ell,n-1}(T_{n-1})-\alpha_{\ell,n-2}(T_{n-1})\vert+\frac{\delta_{0}}{(1+T_{n-1})^{2\epsilon}}\\
    {+} e^{{-}\min_{j}\alpha_{j\pm 1}\vert y_{j}(0)- y_{j\pm 1}(0)\vert}\max_{\ell}\vert h_{\ell,n-1}(0)-h_{\ell,n-2}(0)\vert.
\end{multline}
 
 \par In conclusion, since $\alpha_{\ell,n-1}(0)=\alpha_{\ell,n-2}(0)$ for any $\ell\in [m],$ the result of the Proposition \ref{hyperbolicerror} can be obtained with the fundamental theorem of calculus, identity \eqref{pq1}, and estimates \eqref{popopopo1}, \eqref{esttt22}.
\end{proof}
Next, we consider from now on
\begin{equation*}
 b_{\ell,n,n-1}(s)=P_{\mathrm{unst},\ell,n-1}\left(\vec{u}_{n}(s)-\vec{u}_{n-1}(s)\right).   
\end{equation*}
We recall the decomposition \eqref{decomp} satisfied by each function $\vec{u}_{n}(t)$ defined in Proposition \ref{undecays}, which, using the notation in Definitions \ref{Hyperbolicspace}, \ref{rootspace} and \ref{unstspace}, can be rewritten as
\begin{align*}
  \vec{u}_{n}(t)=&P_{c,\sigma^{T_{n-1}}_{n-1}}(t)\vec{u}(t)+\sum_{\ell=1}^{n}b_{\ell,n,+}(t)e^{i\theta^{T_{n-1}}_{\ell,\sigma_{n-1}}(t,x)\sigma_{3}}\vec{Z}_{+}\left(\alpha_{\ell}(T_{n-1}),x-y^{T_{n-1}}_{\ell,\sigma_{n-1}}(t)\right)\\&{+}\sum_{\ell=1}^{n}b_{\ell,-,n}(t)\mathfrak{G}_{\ell}(\mathfrak{v}_{\alpha_{\ell,n-1},\overline{\lambda_{\ell}}})(t,x){+}P_{\mathrm{root},\sigma^{T_{n-1}}_{n-1}}(t)\vec{u}(t,x),   
\end{align*}
such that
\begin{equation*}
    P_{\mathrm{stab},\sigma^{T_{n-1}}_{n-1}}(t)\vec{u}_{n}(t)=\sum_{\ell=1}^{n}b_{\ell,-,n}(t)\mathfrak{G}_{\ell}(\mathfrak{v}_{\alpha_{\ell,n-1},\overline{\lambda_{\ell}}})(t,x) \text{, for all $t\geq 0.$}
\end{equation*}
Let
\begin{align}\nonumber
   Int_{\mathrm{unst},\sigma,n-1}(s)=&{-}\sum_{h=1}^{m}\sum_{j=1,j\neq h}^{m} b_{h,n,+}(s)V^{T_{n-1}}_{j,\sigma_{n-1}}(s,x)e^{i\theta^{T_{n-1}}_{\sigma,\ell}(s,x)}\vec{Z}_{+}\left(\alpha_{h}(T_{n-1}),x-y^{T_{n-1}}_{h,\sigma_{n-1}}(s)\right)\\ \label{intunstn-1b}
    &{-}\sum_{j=1}^{m}V^{T_{n-1}}_{j,\sigma_{n-1}}(s,x)[P_{c,\sigma^{T_{n-1}}_{n-1}}(s)\vec{u}_{n}(s)-P_{c,j,\sigma^{T_{n-1}}_{n-1}}(s)\vec{u}_{n}(s)].
\end{align}
From the initial condition satisfied by $\vec{u}_{n-1}(0)$ and $\vec{u}_{n}(0),$ we can verify using \eqref{diffequ} that the following estimate holds for any $t\in[0,T_{n}]$
\begin{multline}\label{blnn-1tform}
  b_{\ell,n,n-1}(t)=e^{i\alpha_{\ell,n-1}(T_{n-1})^{2}\lambda_{0}t}b_{\ell,n,n-1}(0)-i \int_{0}^{t}e^{\alpha_{\ell,n-1}(T_{n-1})^{2}\lambda_{0}(t-s)}P_{\mathrm{unst},\ell,n-1}\left(Diff_{n-1,n-2}(s)\right)\,ds\\
{+}i\int_{0}^{t}e^{\alpha_{\ell,n-1}(T_{n-1})^{2}\lambda_{0}(t-s)}P_{\mathrm{unst},\ell,n-1}\Bigg([\sum_{j}V_{j,\sigma_{n-1}}(t,x){-}V^{T_{n-1}}_{j,\sigma_{n-1}}(t,x)
  ](\vec{u}_{n}-\vec{u}_{n-1})\Bigg)\,ds\\
  {-}i\int_{0}^{t}e^{\alpha_{\ell,n-1}(T_{n-1})^{2}\lambda_{0}(t-s)}P_{\mathrm{unst},\ell,n-1}\left[Int_{\mathrm{unst},\sigma,n-1}(s)-Int_{\mathrm{unst},\sigma,n-2}(s)\right].
\end{multline}
Moreover, Proposition \ref{undecays} implies that
\begin{equation*}
    \max_{q\in\{2,\infty\}}\norm{\vec{u}_{n}(T_{n-1})-\vec{u}_{n-1}(T_{n-1})}_{L^{q}_{x}(\mathbb{R})}\leq 2\max_{q\in\{2,\infty\}}\max\left(\norm{\vec{u}_{n}(T_{n-1})}_{L^{q}_{x}(\mathbb{R})},\norm{\vec{u}_{n-1}(T_{n-1})}_{L^{q}_{x}(\mathbb{R})}\right)\leq 2\delta.
\end{equation*}
From now on, let we say that $f(x)=O_{L^{\infty}\cap L^{2}}(c)$ for a $c>0$ if
\begin{equation*}
    \norm{f(x)}_{L^{2}_{x}(\mathbb{R})}+\norm{f(x)}_{L^{\infty}_{x}(\mathbb{R})}\lesssim c.
\end{equation*}
Consequently, we can deduce from \eqref{blnn-1tform} the following estimate 
\begin{multline*}
\begin{aligned}
    b_{\ell,n,n-1}(0)=&O_{L^{\infty}\cap L^{2}}\left(\delta e^{{-}T_{n-1}\alpha_{\ell,n-1}(T_{n-1})^{2}\lambda_{0}}\right)\\
    &{-}i\int_{0}^{T_{n-1}}e^{{-}s\alpha_{\ell,n-1}(T_{n-1})^{2}\lambda_{0}}P_{\mathrm{unst},\ell,n-1}\Bigg([\sum_{j}V_{j,\sigma_{n-1}}(t,x){-}V^{T_{n-1}}_{j,\sigma_{n-1}}(t,x)
  ](\vec{u}_{n}-\vec{u}_{n-1})\Bigg)\,ds\\
    &{+}i\int_{0}^{T_{n-1}}e^{{-}s\alpha_{\ell,n-1}(T_{n-1})^{2}\lambda_{0}}P_{\mathrm{unst},\ell,n-1}\left(Diff_{n-1,n-2}(s)\right)\,ds\\
     &{-}i\int_{0}^{t}e^{\alpha_{\ell,n-1}(T_{n-1})^{2}\lambda_{0}(t-s)}P_{\mathrm{unst},\ell,n-1}\left[Int_{\mathrm{unst},\sigma,n-1}(s)-Int_{\mathrm{unst},\sigma,n-2}(s)\right]\\=&O_{L^{\infty}\cap L^{2}}\left(\delta e^{{-}T_{n-1}\alpha_{\ell,n-1}(T_{n-1})^{2}\lambda_{0}}\right)\\&{+}i\int_{0}^{T_{n-1}}e^{{-}s\alpha_{\ell,n-1}(T_{n-1})^{2}\lambda_{0}}P_{\mathrm{unst},\ell,n-1}M_{n,n-1}(s)\,ds,
\end{aligned}
\end{multline*}
where $M_{n,n-1}(s)$ is the following function
\begin{align*}
  M_{n,n-1}(s):=&  {-}[\sum_{j}V_{j,\sigma_{n-1}}(t,x){-}V^{T_{n-1}}_{j,\sigma_{n-1}}(t,x)
  ](\vec{u}_{n}-\vec{u}_{n-1})\\
  &{-}\left[Int_{\mathrm{unst},\sigma,n-1}(s)-Int_{\mathrm{unst},\sigma,n-2}(s)\right]
  \\&{+}Diff_{n-1,n-2}(s),
\end{align*}
for $Diff_{n-1,n-2}(s)$ defined in \eqref{diffequ}, and $Int_{\mathrm{unst},\sigma,n-2}(s)$ is defined in \eqref{intunstn-1b}.
In particular, using \eqref{blnn-1tform} again, we can rewrite the equation satisfied by $b_{\ell,n,n-1}$ for any $t\in[0,T_{n-1}]$ by
\begin{equation}\label{bnn-1eq}
    b_{\ell,n,n-1}(t)=O_{L^{\infty}\cap L^{2}}\left(\delta_{0} e^{(t-T_{n-1})\alpha_{\ell,n-1}(T_{n-1})^{2}\lambda_{0}}\right)+i\int_{t}^{T_{n-1}}e^{(t-s)\alpha_{\ell,n-1}(T_{n-1})^{2}\lambda_{0}}P_{\mathrm{unst},\ell,n-1}M_{n,n-1}(s)\,ds.
\end{equation}
As a consequence of \eqref{bnn-1eq}, we can deduce the following proposition.
\begin{proposition}\label{propbnndecays}
 The functions $b_{\ell,n,n-1}$ satisfy the following decay estimates.
 \begin{align*}
     \vert b_{\ell,n,n-1}(t) \vert\lesssim \min\left(\delta_{0} e^{{-}\frac{(T_{n-1}-t)\alpha_{\ell,n-1}(T_{n-1})^{2}\lambda_{0}}{2}}+\max_{s\geq t}\norm{P_{\mathrm{unst},\ell,n-1}\left(M_{n,n-1}(s)\right)}_{L^{2}_{x}(\mathbb{R})},\frac{\delta_{0}}{(1+t)^{\frac{1}{2}+\epsilon}}\right).
 \end{align*}
\end{proposition}
\begin{proof}[Proof of Proposition \ref{propbnndecays}]
First, since $\vec{u}_{n}(t)=0$ when $t>T_{n},$ Proposition \ref{undecays} implies when $t \geq 0$ for any $\ell \in [m]$ that
\begin{align*}
    \max_{j\in\{n,n-1\}}\norm{P_{\mathrm{unst},\ell,j-1}(t)\vec{u}_{j}(t,x)}_{L^{2}_{x}(\mathbb{R})}\lesssim\max_{\ell}\norm{\frac{\chi_{\ell,j-1}(t)\vec{u}_{j}(t,x)}{\langle x-y^{T_{n-1}}_{\ell,j-1}(t) \rangle^{\frac{3}{2}+\omega}}}_{L^{2}_{x}(\mathbb{R})}\lesssim \frac{\delta_{0}}{(1+t)^{\frac{1}{2}+\epsilon}}.
\end{align*}
Next, using the formula \eqref{bnn-1eq},  from the fundamental theorem of calculus  if $t\in [0,T_{n-1}],$ then
\begin{equation*}
    \max_{j\in\{n,n-1\}} \vert b_{\ell,n,n-1}(t) \vert\lesssim [1+\delta_{0}]e^{{-}\frac{(T_{n-1}-t)\alpha_{\ell,n-1}(T_{n-1})^{2}\lambda_{0}}{2}}+\max_{s\geq t}\norm{P_{\mathrm{unst},\ell,n-1}\left(M_{n,n-1}(s)\right)}_{L^{2}_{x}(\mathbb{R})}.
\end{equation*}
In conclusion, the estimate in the statement of Proposition \ref{propbnndecays} follows from the two estimates above.
\end{proof}
\begin{corollary}\label{bcorol}
If $t\in [0,T_{n}],$ the following estimate holds
\begin{multline*}
\begin{aligned}
\vert b_{\ell,n,n-1}(t) \vert\lesssim \mathcal{Q}(t)\coloneqq \min\Bigg(&[1+\delta_{0}]e^{{-}\frac{(T_{n-1}-t)\vert \lambda_{\ell,n}(T_{n-1})\vert}{2}}\\ &{+}\max_{s\in[t,T_{n-1}],\ell}\frac{C\delta_{0}}{(1+s)^{\frac{1}{2}+\epsilon}}\max_{h\in\{y,\hat{\gamma}\}}\vert h_{\ell,n-1}(s)-h_{\ell,n-2}(s)\vert \\
&{+}\max_{s\in[t,T_{n-1}],\ell}\frac{C\delta_{0}}{(1+s)^{\frac{1}{2}+\epsilon}}\max_{h\in\{v,\alpha\}}\langle s \rangle\vert h_{\ell,n-1}(s)-h_{\ell,n-2}(s)\vert 
\\ &   {+}\max_{s\in [t,T_{n-1}],\ell}\frac{\delta_{0}}{(1+s)^{2\epsilon-1}}\norm{\frac{\chi_{\ell,n-1}(s)[\vec{u}_{n}(s)-\vec{u}_{n-1}(s)]}{\langle x-y^{T_{n-1}}_{\ell,n-1}(s)\rangle}}_{L^{\infty}_{x}(\mathbb{R})}
    \\
  &{+}\max_{s\in [t,T_{n-1}],\ell}\frac{\delta_{0}}{(1+s)^{\frac{1}{2}+\epsilon}}\norm{\frac{\chi_{\ell,n-2}(s)[\vec{u}_{n-1}(s)-\vec{u}_{n-2}(s)]}{\langle x-y^{T_{n-2}}_{\ell,n-2}(s) \rangle}}_{L^{\infty}_{x}(\mathbb{R})}
    \\
    &{+}\max_{\ell,s\in[t,T_{n-1}]}\frac{\delta_{0}}{(1+s)^{2\epsilon-1}}\max_{\ell}\left\vert \Lambda \dot \sigma_{\ell,n-1}(s)-\Lambda \dot \sigma_{\ell,n-2}(s)\right\vert,\frac{\delta_{0}}{(1+t)^{\frac{1}{2}+\epsilon}}\Bigg).
\end{aligned}
\end{multline*}
\end{corollary}
In particular, Corollary \ref{bcorol} and Proposition \ref{hyperbolicerror} imply the following lemma.
\begin{lemma}\label{hyp+z(0)}
  If
\begin{gather*}
   \vec{u}_{n}(0)=\vec{r}_{0}\\{+}\sum_{\ell}h_{\ell,n-1}(0)e^{i\theta_{\ell}(0,x)\mathfrak{p}_3}\vec{Z}_{+}(\alpha_{\ell,n-1}(T_{n}),x-y_{\ell}(0))+\sum_{\ell}e^{i\theta_{\ell}(0,x)\mathfrak{p}_3}\vec{E}_{\ell,n-1}(\alpha_{\ell,n-1}(T_{n}),x-y_{\ell}(0)),\\
   \vec{u}_{n-1}(0)=\vec{r}_{0}\\{+}\sum_{\ell}h_{\ell,n-2}(0)e^{i\theta_{\ell}(0,x)\mathfrak{p}_3}\vec{Z}_{+}(\alpha_{\ell,n-2}(T_{n-1}),x-y_{\ell}(0))+\sum_{\ell}e^{i\theta_{\ell}(0,x)\mathfrak{p}_3}\vec{E}_{\ell,n-2}(\alpha_{\ell,n-2}(T_{n-1}),x-y_{\ell}(0))
\end{gather*}
such that $\vec{E}_{\ell,h}\in \ker\mathcal{H}^{2}_{1}$ for all $\ell\in [m],$ $h\in\{n-1,n-2\},$   
then
\begin{align*}
    \max_{\ell} \left\vert h_{\ell,n-1}(0)-h_{\ell,n-2}(0)  \right\vert\lesssim & \delta_{0}\norm{(\Pi_{n-1}-\Pi_{n-2},\vec{u}_{n-1}-\vec{u}_{n-2})}_{Y_{n-1}}+ \delta_{0}\norm{(\Pi_{n}-\Pi_{n-1},\vec{u}_{n}-\vec{u}_{n-1})}_{Y_{n}}\\
  &{+}\frac{1}{T_{n-1}^{\frac{1}{2}+\epsilon}}.
\end{align*}
\end{lemma}
\begin{proof}[Proof of Lemma \ref{hyp+z(0)} using Corollary \ref{bcorol}]
First, it is not difficult to verify that Corollary \ref{bcorol} implies the following estimate
\begin{align}\label{minb}
 \vert b_{\ell,n,n-1}(t) \vert\lesssim \min \Bigg(&\frac{1}{T_{n-1}^{1+\epsilon}}{+}\delta_{0}\norm{(\Pi_{n-1}-\Pi_{n-2},\vec{u}_{n-1}-\vec{u}_{n-2})}_{Y_{n-1}}\\ \nonumber
 &{+} \delta_{0}\norm{(\Pi_{n}-\Pi_{n-1},\vec{u}_{n}-\vec{u}_{n-1})}_{Y_{n}},\frac{\delta_{0}}{T_{n-1}^{\frac{1}{2}+\epsilon}} \Bigg),   
\end{align}
since for all $t\in [0,\frac{T_{n-1}}{2}]$
\begin{align*}
    e^{{-}\frac{\alpha_{\ell,n-1}(T_{n-1})^{2}\lambda_{0}(T_{n-1}-t)}{2}}\lesssim & \frac{1}{T_{n-1}^{\frac{1}{2}+\epsilon}},\\
    \max_{s\in[t,T_{n-1}],\ell}\frac{1}{(1+s)^{\frac{1}{2}+\epsilon}}\max_{h\in\{y,\hat{\gamma}\}}\vert h_{\ell,n-1}(s)-h_{\ell,n-2}(s)\vert\lesssim &  \norm{(\Pi_{n-1}-\Pi_{n-2},\vec{u}_{n-1}-\vec{u}_{n-2})}_{Y_{n-1}},\\
     \max_{s\in[t,T_{n-1}],\ell}\frac{1}{(1+s)^{\frac{1}{2}+\epsilon}}\max_{h\in\{v,\alpha\}}\langle s\rangle \vert h_{\ell,n-1}(s)-h_{\ell,n-2}(s)\vert\lesssim & \norm{(\Pi_{n-1}-\Pi_{n-2},\vec{u}_{n-1}-\vec{u}_{n-2})}_{Y_{n-1}}.
\end{align*}
\par Therefore, \eqref{minb} implies the result of Lemma \ref{hyp+z(0)}.
\end{proof}
\begin{proof}[Proof of Corollary \ref{bcorol}.]
It is enough to estimate the $L^{2}$ norm of 
\begin{equation*}
    P_{\mathrm{unst},\ell,n-1}\left(M_{n,n-1}(s)\right).
\end{equation*}
First, Lemma \ref{dinftydt} implies that
\begin{multline*}
    \norm{P_{\mathrm{unst},\ell,n-1}\left[V_{\ell,\sigma_{n-1}}(t,x)-V^{T_{n-1}}_{\ell,\sigma_{n-1}}(t,x)\right]\vec{z}_{n}}_{L^{2}_{x}(\mathbb{R})}\\
    \lesssim \frac{\delta_{0}}{(1+t)^{2\epsilon-1}}\max_{\ell}\norm{\frac{\chi_{\ell,n-1}(t,x)\vec{z}_{n}(t)}{\langle x-y^{T_{n-1}}_{\ell,n-1}(t)\rangle}}_{L^{\infty}_{x}(\mathbb{R})}.
\end{multline*}
\par Next, we consider $Diff_{n,n-1}(t)$ which is the right-hand side of \eqref{diffequ}. From Proposition \ref{shititeraction}, we can verify the following estimate
\begin{multline}\label{interactp+}
\Bigg\vert\Bigg\vert Int_{n-1}(t,x)-Int_{n-2}(t,x)\Bigg\vert\Bigg\vert_{L^{2}_{x}(\mathbb{R})}\\
\begin{aligned}
\lesssim & \max_{q\in\{1,2\}}\max_{\ell\in[m],\mathfrak{M}\in\{v,y,\gamma,\alpha\}}\vert \mathfrak{M}_{\ell,n-1}(t)-\mathfrak{M}_{\ell,n-2}(t) \vert e^{{-}\min_{j_{1},j_{2}\in [m]}\frac{99\alpha_{j_{1}}(0)}{100}(y_{j_{2},n-q}(t)-y_{j_{2}+1,n-q}(t))}\\
\lesssim & \max_{\ell\in[m],\mathfrak{M}\in\{v,y,\gamma,\alpha\}}\vert \mathfrak{M}_{\ell,n-1}(t)-\mathfrak{M}_{\ell,n-2}(t) \vert \frac{\delta_{0}}{(1+t)^{20}}.
\end{aligned}
\end{multline}
\par Next, using estimate \eqref{intpartt} of Proposition \ref{diffprop1} and the upper bound 
\begin{equation*}
\max_{t\geq 0,j\in\{{-}1,0\}}\norm{\vec{u}_{n+j}(t)}_{L^{2}_{x}(\mathbb{R})}+\max_{\ell,j\in\{0,{-}1\}}\vert b_{\ell,n+j,+}(t)\vert\lesssim \delta_{0},    
\end{equation*}
 we can verify that the following estimate holds.
\begin{multline}\label{estintunstscatt}
\norm{Int_{\mathrm{unst},\sigma,n-1}(t)-Int_{\mathrm{unst},\sigma,n-2}(t)}_{H^{1}_{x}(\mathbb{R})}\\+\max_{\ell\in [m]} \norm{\chi_{\ell,n-1}(t)\langle x-y_{\ell,n-1} \rangle[Int_{\mathrm{unst},\sigma,n-1}(t)-Int_{\mathrm{unst},\sigma,n-2}(t)]}_{L^{1}_{x}(\mathbb{R})}\\\begin{aligned}
    \lesssim & \max_{q\in\{1,2\},\ell}\vert b_{\ell,n,+}(t)-b_{\ell,n-1,+}(t) \vert e^{{-}\min_{j_{1},j_{2}\in [m]}\frac{99\alpha_{j_{1}}(0)}{100}(y_{j_{2},n-q}(t)-y_{j_{2}+1,n-q}(t))}\\
    &{+}\norm{\vec{z}_{n}(t)}_{L^{2}_{x}(\mathbb{R})}e^{{-}\min_{j_{1},j_{2}\in [m]}\frac{99\alpha_{j_{1}}(0)}{100}(y_{j_{2},n-q}(t)-y_{j_{2}+1,n-q}(t))}\\
    &{+} \delta_{0} \max_{q\in\{1,2\}}\max_{\ell\in[m],\mathfrak{M}\in\{v,y,\gamma,\alpha\}}\vert \mathfrak{M}_{\ell,n-1}(t)-\mathfrak{M}_{\ell,n-2}(t) \vert e^{{-}\min_{j_{1},j_{2}\in [m]}\frac{99\alpha_{j_{1}}(0)}{100}(y_{j_{2},n-q}(t)-y_{j_{2}+1,n-q}(t))}\\
    \lesssim &
   \frac{\max_{\ell} \vert b_{\ell,n,+}(t)-b_{\ell,n-1,+}(t) \vert \delta_{0}^{2}}{(1+t)^{20}}
    {+}\frac{\norm{\vec{z}_{n}(t)}\delta_{0}^{2}}{(1+t)^{20}}
    \\&{+} \max_{q\in\{1,2\}}\max_{\ell\in[m],\mathfrak{M}\in\{v,y,\gamma,\alpha\}}\vert \mathfrak{M}^{T_{n-1}}_{\ell,n-1}(t)-\mathfrak{M}^{}_{\ell,n-2}(t) \vert \frac{\delta_{0}^{2}}{(1+t)^{20}}.
\end{aligned}  
\end{multline}
Moreover, using the decomposition formula for $\vec{u}_{n}(t)$ and $\vec{u}_{n-1}(t)$ in \eqref{decompofun} and the estimates
\begin{equation}
 \norm{\vec{z}_{n}(t)}_{L^{2}_{x}(\mathbb{R})}\leq \max_{t\geq 0}\norm{\vec{u}_{n}(t)}_{L^{2}_{x}(\mathbb{R})}+ \norm{\vec{u}_{n-1}(t)}_{L^{2}_{x}(\mathbb{R})}\lesssim \delta_{0}
\end{equation} 
obtained from Proposition \ref{undecays}, we can verify that
\begin{equation}\label{themostimportantupperbounddiffb}
\max_{\ell}\vert b_{\ell,n,+}(t)-b_{\ell,n-1,+}(t)\vert\lesssim\max_{\ell}\vert b_{\ell,n,n-1}(t)\vert+\delta_{0}  \max_{\ell\in[m],\mathfrak{M}\in\{v,y,\gamma,\alpha\}}\vert \mathfrak{M}_{\ell,n-1}(t)-\mathfrak{M}_{\ell,n-2}(t) \vert.   
\end{equation}
Consequently, we obtain from \eqref{themostimportantupperbounddiffb} estimate that
\begin{align}\label{fff}
\norm{Int_{\mathrm{unst},\sigma,n-1}(t)-Int_{\mathrm{unst},\sigma,n-2}(t)}_{L^{2}_{x}(\mathbb{R})}\lesssim & \frac{\delta_{0} \vert b_{\ell,n,n-1}(t) \vert}{(1+t)^{20}}
{+}\frac{\delta_{0}^{2}\norm{\vec{z}_{n}(t)}_{L^{2}_{x}(\mathbb{R})}}{(1+t)^{20}}
\\& {+} \max_{q\in\{1,2\}}\max_{\ell\in[m],\mathfrak{M}\in\{v,y,\gamma\}}\vert \mathfrak{M}_{\ell,n-1}(t)-\mathfrak{M}_{\ell,n-2}(t) \vert \frac{\delta_{0}^{2}}{(1+t)^{20}},
\end{align}
\par Moreover, we can deduce using Remark \ref{remarkVnvn-1} of Proposition \ref{diffprop1} 
\begin{align*}
\norm{\left[V_{\ell,n-1}(t,x)-V_{\ell,n-2}(t,x)\right]\vec{u}_{n-1}(t)}_{L^{2}_{x}(\mathbb{R})}
  \lesssim & \frac{\max_{\ell,\mathfrak{M}\in\{y,\hat{\gamma}\}}\vert \mathfrak{M}_{\ell,n-1}(t)-\mathfrak{M}_{\ell,n-2}(t) \vert \delta_{0}}{(1+t)^{\frac{1}{2}+\epsilon}}\\
  &{+}\frac{\max_{\ell,\mathfrak{M}\in\{v,\alpha\}}\langle t \rangle\vert \mathfrak{M}_{\ell,n-1}(t)-\mathfrak{M}_{\ell,n-2}(t) \vert \delta_{0}}{(1+t)^{\frac{1}{2}+\epsilon}}.
\end{align*}
\par Furthermore, if $\vert y_{\ell,n-1}(s)-y_{\ell,n-2}(s)\vert\leq A$ when $s\in [0,T_{n}],$ then
\begin{multline}\label{phypdifferencefullnonlinear}
\norm{P_{\mathrm{unst},\ell,n-1}(t)\left[N(\sigma_{n-1}(t),\vec{u}_{n-1}(t))-N(\sigma_{n-2}(t),\vec{u}_{n-2}(t))\right]}_{L^{2}_{x}(\mathbb{R})}\\
\begin{aligned}
\lesssim & \frac{\delta_{0}^{2k}\norm{\chi_{\ell,n-2}(t)\frac{\vec{u}_{n-1}(t)-\vec{u}_{n-2}(t)}{\langle x-v_{\ell,n-2}(T_{n-1})t-D_{\ell,n-2}(T_{n-1}) \rangle}}_{L^{\infty}_{x}(\mathbb{R})}}{(1+t)^{k+2k \epsilon}}\\
&{+}\max_{\ell}\frac{\delta_{0}}{(1+t)^{\frac{1}{2}+\epsilon}}\norm{\frac{\chi_{\ell,n-2}(t)[\vec{u}_{n-1}(t)-\vec{u}_{n-2}(t)]}{\langle x-v_{\ell,n-2}(T_{n-1})s-D_{\ell,n-2}(T_{n-1}) \rangle}}_{L^{\infty}_{x}(\mathbb{R})}\\
&{+}\frac{\delta_{0}^{2}\max_{\ell,\mathfrak{M}\in\{y,\hat{\gamma}\}}\vert \mathfrak{M}_{\ell,n-1}(t)-\mathfrak{M}_{\ell,n-2}(t) \vert }{(1+t)^{1+2\epsilon}}\\
&{+}\frac{\delta_{0}^{2}\max_{\ell,\mathfrak{M}\in\{v,\alpha\}}\langle t \rangle\vert \mathfrak{M}_{\ell,n-1}(t)-\mathfrak{M}_{\ell,n-2}(t) \vert }{(1+t)^{1+2\epsilon}}.
\end{aligned}
\end{multline}
\par In particular, the right-hand side of \eqref{phypdifferencefullnonlinear} can also be obtained from the estimate of the $L^{2}$ norm
of $P_{\mathrm{unst},\ell,n-1}$ applied to each of the terms
\begin{align*}
 \vert\vec{u}_{n-1}(t)\vert^{2k}\vec{u}_{n-1}(t)-\vert\vec{u}_{n-2}(t)\vert^{2k}\vec{u}_{n-2}(t),\, V\left(e^{i\theta_{\ell,n-1}(t,x)}\phi_{\ell}(x-y_{\ell,n-1}(t))\right)[\vec{u}_{n-1}(t)-\vec{u}_{n-2}(t)],\\
 V\left(e^{i\theta_{\ell,n-1}(t,x)}\phi_{\ell}(x-y_{\ell,n-1}(t))\right)\left[\vert \vec{u}_{n-1}(t)\vert^{2+\alpha}-\vert \vec{u}_{n-2}(t) \vert^{2+\alpha}\right] \text{, for an $\alpha\in\mathbb{N}_{\geq 0}.$}
\end{align*}
The remaining terms of the right-hand side of \eqref{bnn-1eq} are
\begin{gather*}
  \norm{\Lambda \dot \sigma_{\ell,n-2}(t)P_{\mathrm{unst},\ell,n-1}\left[\mathcal{E}_{\ell,n-1}(\alpha_{\ell,n-1}(T_{n}),x-y_{\ell,n-1}(t))-\mathcal{E}_{\ell,n-2}(\alpha_{\ell,n-2}(T_{n-1}),x-y_{\ell,n-2}(t))\right]}_{L^{2}_{x}(\mathbb{R})}\\
  \lesssim \left[\max_{\ell,\mathfrak{M}\in\{y,\hat{\gamma}\}}\vert \mathfrak{M}_{\ell,n-1}(t)-\mathfrak{M}_{\ell,n-2}(t)\vert+\max_{\ell,\mathfrak{M}\in\{v,\alpha\}}\langle t\rangle\vert \mathfrak{M}_{\ell,n-1}(t)-\mathfrak{M}_{\ell,n-2}(t)\vert\right] \frac{\delta_{0}^{2}}{(1+t)^{1+2\epsilon}},\\  
  \norm{\left[\Lambda \dot \sigma_{\ell,n-1}(t)- \Lambda \dot \sigma_{\ell,n-2}(t)\right]P_{\mathrm{unst},\ell,n-1}\mathcal{E}_{\ell,n-1}(x-y_{\ell,n-1}(t))}_{L^{2}_{x}(\mathbb{R})}\lesssim \frac{\delta_{0}}{(1+t)^{2\epsilon-1}} \left\vert \Lambda \dot \sigma_{\ell,n-1}(t)- \Lambda \dot \sigma_{\ell,n-2}(t)\right\vert. 
\end{gather*}
Using all the estimates above, and the 
elementary estimate
\begin{equation*}
   \norm{\int_{t}^{T_{n-1}}e^{i\alpha_{\ell,n-1}(T_{n-1})^{2}\lambda_{0}(t-s)}g(s)\,ds}_{L^{2}_{x}(\mathbb{R})}\lesssim \max_{s\in [t,T_{n-1}]}\norm{g(s)}_{L^{2}_{x}(\mathbb{R})},
\end{equation*}
which is implied by
fact that $\alpha_{\ell,n-1}(T_{n-1})=\alpha_{\ell}(0)+O(\delta_{0}),$ we can verify
from the Proposition \ref{propbnndecays} that 
\begin{equation*}
\vert b_{\ell,n,n-1}(t)\vert\lesssim \mathcal{Q}(t)+\delta\int_{t}^{T_{n-1}}e^{{-}\beta s}\vert b_{\ell,n,n-1}(s)\vert\,ds
\end{equation*}
for all $t\in [t,T_{n-1}].$
As a consequence, we can deduce from Gronwall lemma and
estimate
\begin{equation*}
\vert b_{\ell,n,n-1}(t)\vert\lesssim\frac{\delta_{0}}{(1+t)^{\frac{1}{2}+\epsilon}}
\end{equation*}
that
Corollary \ref{bcorol} holds.
\end{proof}
\subsection{Difference of initial data}\label{zn0subsection}
In this subsection, we estimate the difference of initial data  $\norm{\vec{u}_{n}(0)-\vec{u}_{n-1}(0)}_{L^{2}_{x}(\mathbb{R})}$. 

First, we consider the basis of $\ker\mathcal{H}^{2}_{1}$ denoted in \eqref{ker2basis}.
\begin{equation*}
    Basis_{2}=\left\{\begin{bmatrix}
            \partial_{x}\phi_{1}(x)\\
            \partial_{x}\phi_{1}(x)
        \end{bmatrix},\, \begin{bmatrix}
            i\phi_{1}(x)\\
            {-}i\phi_{1}(x)
        \end{bmatrix},\,
        \begin{bmatrix}
        ix\phi_{1}(x)\\
        {-}ix\phi_{1}(x)
        \end{bmatrix},\,
        \begin{bmatrix}
        \partial_{\alpha}\phi_{1}(x)\\
        \partial_{\alpha}\phi_{1}(x)
        \end{bmatrix}
        \right\}.
\end{equation*}
Moreover, we recall that the functions $\vec{u}_{n}(0,x)$ and $\vec{u}_{n-1}(0,x)$ satisfy the following identities;
\begin{align*}
    \vec{u}_{n}(0,x)=&\vec{r}_{0}(x)\\&{+}\sum_{\ell=1}^{m}h_{\ell,n-1}(0)e^{i\theta_{\ell}(0,x)\mathfrak{p}_3}\vec{Z}_{+}(\alpha_{\ell,n-1}(T_{n-1}),x-y_{\ell}(0))\\&{+}\sum_{\ell=1}^{m}\sum_{\vec{w}\in Basis_{2}}p_{\ell,n-1}(0)e^{i\theta_{\ell}(0,x)\mathfrak{p}_3}\vec{w}(\alpha_{\ell,n-1}(T_{n-1}),x-y_{\ell}(0)),\\
    \vec{u}_{n-1}(0,x)=&\vec{r}_{0}(x)\\&{+}\sum_{\ell=1}^{m}h_{\ell,n-2}(0)e^{i\theta_{\ell}(0,x)\mathfrak{p}_3}\vec{Z}_{+}(\alpha_{\ell,n-2}(T_{n-2}),x-y_{\ell}(0))\\&{+}\sum_{\ell=1}^{m}\sum_{\vec{w}\in Basis_{2}}p_{\ell,n-2}(0)e^{i\theta_{\ell}(0,x)\mathfrak{p}_3}\vec{w}(\alpha_{\ell,n-2}(T_{n-2}),x-y_{\ell}(0)),
\end{align*}
and $p_{\ell,n-1}(0)$ and $p_{\ell,n-1}(0)$ are the unique complex numbers such that
\begin{align*}
    \langle \vec{u}_{n}(0,x),\mathfrak{p}_3e^{i\theta_{\ell}(0,x)\mathfrak{p}_3}\vec{w}(\alpha_{\ell}(0),x-y_{\ell}(0)) \rangle=\langle \vec{u}_{n-1}(0,x),\mathfrak{p}_3e^{i\theta_{\ell}(0,x)\mathfrak{p}_3}\vec{w}(\alpha_{\ell}(0),x-y_{\ell}(0)) \rangle=0,
\end{align*}
for all $\vec{w}\in\ker\mathcal{H}^{2}_{1}.$ We recall that $\sigma_{\ell,n}(0)=\sigma_{\ell}(0)$ for any $n\in\mathbb{N}.$
Therefore, $\vec{z}_{n}(0,x)$ satisfies the following equation for all $\ell\in [m],$ and all $\vec{w}\in\ker\mathcal{H}^{2}_{1}$
\begin{equation}\label{orthoconditiont=0}
  \langle \vec{z}_{n}(0,x),\mathfrak{p}_3 e^{i\theta_{\ell}(0,x)\mathfrak{p}_3}\vec{w}(\alpha_{\ell}(0),x-y_{\ell}(0)) \rangle=0.  
\end{equation}
 Moreover, from \eqref{orthoconditiont=0}, we obtain the following equation for any $\vec{w}\in Basis_{2}$ and $j\in [m].$
\begin{multline}\label{eqjp}
  \sum_{\ell=1}^{m}\sum_{\vec{z}\in Basis_{2}}\left[ p_{\ell,n-1}(0)-p_{\ell,n-2}(0)\right]\langle e^{i\theta_{\ell}(0,x)\mathfrak{p}_3}\vec{z}(\alpha_{\ell,n-1}(T_{n-1}),x-y_{\ell}(0)) ,\mathfrak{p}_3 e^{i\theta_{j}(0,x)\mathfrak{p}_3}w(\alpha_{j}(0),x-y_{j}(0))\rangle\\
  \begin{aligned}
      =&{-}\sum_{\ell=1}^{m}\sum_{\vec{z}\in Basis_{2}}p_{\ell,n-2}(0)\\
      &\times\Bigg[\langle e^{i\theta_{\ell}(0,x)\mathfrak{p}_3}\vec{z}(\alpha_{\ell,n-1}(T_{n-1}),x-y_{\ell}(0)) ,\mathfrak{p}_3 e^{i\theta_{j}(0,x)\mathfrak{p}_3}w(\alpha_{j}(0),x-y_{j}(0))\rangle\\
      &{-} \langle e^{i\theta_{\ell}(0,x)\mathfrak{p}_3}\vec{z}(\alpha_{\ell,n-2}(T_{n-2}),x-y_{\ell}(0)) ,\mathfrak{p}_3 e^{i\theta_{j}(0,x)\mathfrak{p}_3}w(\alpha_{j}(0),x-y_{j}(0))\rangle  \Bigg]\\
      &{-}\sum_{\ell=1}^{m}(h_{\ell,n-1}(0)-h_{\ell,n-2}(0))\langle e^{i\theta_{\ell}(0,x)\mathfrak{p}_3}\vec{Z}_{+}\left(\alpha_{\ell,n-1}(T_{n-1}),x-y_{\ell}(0)\right), \mathfrak{p}_3 e^{i\theta_{j}(0,x)\mathfrak{p}_3}w(\alpha_{j}(0),x-y_{j}(0))\rangle\\
      &{-}\sum_{\ell=1}^{m}h_{\ell,n-2}(0)\times\\
      &\Bigg[\langle e^{i\theta_{\ell}(0,x)\mathfrak{p}_3}\vec{Z}_{+}\left(\alpha_{\ell,n-1}(T_{n-1}),x-y_{\ell}(0)\right), \mathfrak{p}_3 e^{i\theta_{j}(0,x)\mathfrak{p}_3}w(\alpha_{j}(0),x-y_{j}(0))\rangle\\
      &{-}\langle e^{i\theta_{\ell}(0,x)\mathfrak{p}_3}\vec{Z}_{+}\left(\alpha_{\ell,n-2}(T_{n-2}),x-y_{\ell}(0)\right), \mathfrak{p}_3 e^{i\theta_{j}(0,x)\mathfrak{p}_3}w(\alpha_{j}(0),x-y_{j}(0))\rangle\Bigg].
  \end{aligned}
\end{multline}
\par Consequently, since
\begin{equation*}
    \max_{\ell}\vert p_{\ell,n-2}(0) \vert+\vert h_{\ell,n-2}(0)\vert\lesssim \norm{\vec{u}_{n-2}(0)}_{L^{2}_{x}(\mathbb{R})}\leq \delta,
\end{equation*} and \eqref{dyn12} implies the following inequality
\begin{equation*}
   \left\vert \alpha_{\ell,n-1}(T_{n-2})-\alpha_{\ell,n-1}(T_{n-1})\right\vert\lesssim \int_{T_{n-2}}^{\vert}\frac{\delta_{0}}{(1+s)^{1+2\epsilon}}\lesssim \frac{\delta_{0}}{T_{n-1}^{2\epsilon}}\ll \frac{\delta_{0}}{T_{n-1}^{\frac{1}{2}+\epsilon}},
\end{equation*}
we can deduce from the system of equations \eqref{eqjp} for any $j\in[m]$ that
\begin{align}\label{ppt1}
    \max_{\ell}\vert p_{\ell,n-1}(0)-p_{\ell,n-2}(0) \vert\lesssim & \delta_{0} \max_{\ell,s\in [0,T_{n-2}]}\vert \alpha_{\ell,n-1}(s)-\alpha_{\ell,n-2}(s) \vert+\frac{\delta_{0}^{2}}{T_{n-1}^{\frac{1}{2}+\epsilon}}\\
    &{+}\max_{\ell}\vert h_{\ell,n-1}(0)-h_{\ell,n-2}(0) \vert.
\end{align}
Therefore, Lemma \ref{hyp+z(0)} and estimate \eqref{ppt1} imply that
\begin{align}\label{diffpker2}
   \max_{\ell}\vert p_{\ell,n-1}(0)-p_{\ell,n-2}(0) \vert\lesssim & \delta_{0} \left[\norm{(\Pi_{n-1}-\Pi_{n-2},\vec{u}_{n-1}-\vec{u}_{n-2})}_{Y_{n-1}}+ \norm{(\Pi_{n}-\Pi_{n-1},\vec{u}_{n}-\vec{u}_{n-1})}_{Y_{n}}\right]\\ \nonumber
   &{+}\frac{1}{T_{n-1}^{1+\epsilon}}+\frac{\delta_{0}}{T_{n-1}^{\frac{1}{2}+\epsilon}}.
\end{align}
\par In conclusion, since $\vec{u}_{n}(0)-\vec{u}_{n-1}(0)$ is a finite sum of localized Schwartz functions with exponential decay, we deduce from \eqref{diffpker2}, Lemma \ref{hyp+z(0)}, hypothesis $\mathrm{(H2)},$ and the Minkowski inequality that
\begin{multline}\label{zn(0)size}
    \norm{\vec{z}_{n}(0)}_{H^{2}_{x}(\mathbb{R})}+\max_{q\in\{1,2\},\ell\in[m]}\norm{\chi_{\ell,n-1}(0,x)\langle x-y_{\ell}(0)\rangle\vec{z}_{n}(0,x)}_{W^{1,q}_{x}(\mathbb{R})}\\
    \begin{aligned}
    \lesssim & \delta_{0} \left[\norm{(\Pi_{n-1}-\Pi_{n-2},\vec{u}_{n-1}-\vec{u}_{n-2})}_{Y_{n-1}}+ \norm{(\Pi_{n}-\Pi_{n-1},\vec{u}_{n}-\vec{u}_{n-1})}_{Y_{n}}\right]\\
   &{+}\frac{1}{T_{n-1}^{1+\epsilon}}+\frac{\delta_{0}}{T_{n-1}^{\frac{1}{2}+\epsilon}}.
    \end{aligned}
\end{multline}
\subsection{Estimates of $L^2$ and $L^\infty$ norms of the difference} In this subsection, we estimate  $$\langle t\rangle^{{-}1} \norm{P_{c,n-1}\left(u_{n}-u_{n-1}\right)(t)}_{L^{2}_{x}}\, \text{and}\, \langle t\rangle^{{-}\frac{1}{4}} \norm{\chi_{\ell}(t)P_{c,n-1}\left(u_{n}-u_{n-1}\right)(t)}_{L^{\infty}_{x}}.$$
First, we recall that $\vec{z}_{n}=\vec{u}_{n}-\vec{u}_{n-1}$ is a strong solution of  equation \eqref{diffequ}. Using the Duhamel integral formula, we can verify that $P_{c,n-1}\vec{z}_{n}$ satisfies the following integral equation.
\begin{align}\label{pczn-1}
    P_{c,n-1}\vec{z}_{n}(t)=&\mathcal{U}_\sigma(t,0)P_{c,n-1}(0)\vec{z}_{n}(0)-i\int_{0}^{t}\mathcal{U}_\sigma(t,s)P_{c,n-1}(s)\left[Diff_{n,n-1}(s)\right]\,ds\\ \nonumber
    &{+}i\int_{0}^{t}\mathcal{U}_\sigma(t,s)P_{c,n-1}(s)\sum_{\ell=1}^{m}V_{\ell,\sigma_{n-1}}(s,x)\vec{z}_{n}(s)\,ds\\ \nonumber
    &{-}i\int_{0}^{t}\mathcal{U}_\sigma(t,s)P_{c,n-1}(s)\sum_{\ell=1}^{m}V^{T_{n-1}}_{\ell,\sigma_{n-1}}(s,x)\vec{z}_{n}(s)\,ds\\
    &{-}i\int_{0}^{t}\mathcal{U}_\sigma(t,s)P_{c,n-1}(s)\left[Int_{\mathrm{unst},\sigma,n-1}(s)-Int_{\mathrm{unst},\sigma,n-2}(s)\right]\,ds.
\end{align}
\subsubsection{$L^{2}$ estimate of $P_{c,n-1}\vec{z}_{n}$}
First, using Theorem \ref{Decesti1} and Lemma \ref{dinftydt}, we can verify from the integral equation \eqref{pczn-1} that
\begin{align*}
  \norm{P_{c,n-1}\vec{z}_{n}(t)}_{L^{2}_{x}(\mathbb{R})}\lesssim & \norm{P_{c,n-1}(0)\vec{z}_{n}(0)}_{L^{2}_{x}(\mathbb{R})}+\int_{0}^{t}\norm{P_{c,n-1}(s)Diff_{n,n-1}(s)}_{L^{2}_{x}(\mathbb{R})}\,ds\\&{+}\int_{0}^{t}\frac{\delta_{0}}{(1+s)^{2\epsilon-1}}\max_{\ell}\norm{\frac{\chi_{\ell,n-1}(s)\vec{z}_{n}(s)}{\langle x-v_{\ell,n-1}(T_{n})s-D_{\ell,n-1}(T_{n}) \rangle}}_{L^{\infty}_{x}(\mathbb{R})}\,ds\\
  &{+}\int_{0}^{t}\norm{Int_{\mathrm{unst},\sigma,n-1}(s)-Int_{\mathrm{unst},\sigma,n-2}(s)}_{L^{2}_{x}(\mathbb{R})}\,ds.  
\end{align*}
\par Moreover, Proposition \ref{undecays} and the assumption that $\vert y_{\ell,n-1}(t)-y_{\ell,n-2}(t) \vert\leq A$ when $t\in[0,T_{n-1}]$ implies for any function $W\in C^{1}$ satisfying $W(0)=0$ that
\begin{gather}\label{bullshit}
    \left\vert \vert\vec{u}_{n-1}(t)\vert^{2k}  \vec{u}_{n-1}(t)-\vert\vec{u}_{n-2}(t)\vert^{2k}  \vert\vec{u}_{n-2}(t) \right\vert\lesssim  \left[\vert \vec{u}_{n-1}(t) \vert+\vert \vec{u}_{n-2}(t) \vert\right]^{2k}\vert \vec{u}_{n-1}(t)-\vec{u}_{n-2}(t) \vert,\\
    \norm{e^{i\theta_{\ell,n-1}(t)} W(\phi_{\ell}(x-y_{\ell,n-1}(t)))\vert \vec{u}_{n-1}(t) \vert^{2}-e^{i\theta_{\ell,n-2}(t)}W(\phi_{\ell}(x-y_{\ell,n-2}(t)))\vert \vec{u}_{n-2}(t) \vert^{2}}_{L^{2}_{x}(\mathbb{R})}\\ 
    \begin{aligned}
    \lesssim &\frac{\delta_{0}}{(1+t)^{\frac{1}{2}+\epsilon}} \norm{\frac{\chi_{\ell,n-1}(t,x)\vec{z}_{n}(t,x)}{\langle x-v_{\ell,n-1}(T_{n-1})t-D_{\ell,n-1}(T_{n-1}) \rangle}}_{L^{\infty}_{x}}\\&{+}\norm{[ e^{i\theta_{\ell,n-1}(t)}W(\phi_{\ell}(x-y_{\ell,n-2}(t)))-e^{i\theta_{\ell,n-2}(t)}W(\phi_{\ell}(x-y_{\ell,n-1}(t)))]\langle x-y^{T_{n-1}}_{\ell,\sigma_{n-1}}(t) \rangle^{3+2\omega}}_{L^{2}_{x}(\mathbb{R})}\frac{\delta_{0}^{2}}{(1+t)^{1+2\epsilon}}.
    \end{aligned}
\end{gather}
In addition, we recall from our previous estimate \eqref{estintunstscatt} that 
\begin{align*}
\norm{Int_{\mathrm{unst},\sigma,n-1}(t)-Int_{\mathrm{unst},\sigma,n-2}(t)}_{L^{2}_{x}(\mathbb{R})}\lesssim &  \frac{\delta_{0}\vert b_{\ell,n,n-1}(t) \vert}{(1+t)^{20}}
{+}\frac{\delta_{0}\norm{\vec{z}_{n}(t)}_{L^{2}_{x}(\mathbb{R})}}{(1+t)^{20}}
\\& {+} \max_{\ell\in[m],\mathfrak{M}\in\{v,y,\gamma,\alpha\}}\vert \mathfrak{M}_{\ell,n-1}(t)-\mathfrak{M}_{\ell,n-2}(t) \vert \frac{\delta_{0}}{(1+t)^{20}}.
\end{align*}
Consequently, we can deduce from Corollaries \ref{diffprop2}, \ref{bcorol}, Proposition \ref{undecays}, and equations \eqref{pczn-1}, \eqref{diffequ} that if $t\in [0,T_{n}],$ then
\begin{multline}\label{l2weightzn}
\begin{aligned}
\frac{\norm{P_{c,n-1}\vec{z}_{n}(t)}_{L^{2}_{x}}}{1+t}\lesssim & \frac{\norm{P_{c,n-1}\vec{z}_{n}(0)}_{L^{2}_{x}(\mathbb{R})}}{1+t}+\delta_{0}^{2k}\left[\max_{s\in[0,t]}\frac{\norm{\vec{z}_{n-1}(s)}_{L^{2}_{x}}}{1+s}\right]+\delta_{0}^{2}\left[\max_{s\in[0,t]}\frac{\norm{\vec{z}_{n}(s)}}{1+s}\right]\\
&{+}\max\left((1+t)^{\frac{3}{4}-\epsilon},1\right)\max_{s\in[0,t]}\frac{\delta_{0}}{\langle s \rangle^{\frac{1}{4}}}\norm{\frac{\chi_{\ell,n-1}(t,x)\vec{z}_{n}(t,x)}{\langle x-v_{\ell,n-1}(T_{n-1})t-D_{\ell,n-1}(T_{n-1}) \rangle}}_{L^{\infty}_{x}}\\
&{+}\frac{\delta_{0}^{2}}{T^{\frac{1}{2}+\epsilon}_{n}}
\\&{+}\delta_{0} \max_{s\in[0,t]}\langle s\rangle^{1+\frac{\epsilon}{2}-\frac{3}{8}}\left\vert \dot \Lambda \sigma_{n-1}(t)-\dot \Lambda \sigma_{n-2}(t) \right\vert,
\end{aligned}
\end{multline}
where for the last expression on the right-hand side of the inequality above we used Proposition \ref{shititeraction} and Corollary \ref{diffprop2}.
\par In conclusion, estimates \eqref{zn(0)size} and \eqref{l2weightzn} imply for all $t\in [0,T_{n}]$ that
\begin{align}\label{l2weightznfinal}
    \frac{\norm{P_{c,n-1}\vec{z}_{n}(t)}_{L^{2}_{x}}}{1+t}\lesssim & \delta_{0} \left[\norm{(\vec{u}_{n-1}-\vec{u}_{n-2},\sigma_{n-1}-\sigma_{n-2})}_{Y_{n-1}}
  + \norm{(\vec{u}_{n}-\vec{u}_{n-1},\sigma_{n}-\sigma_{n-1})}_{Y_{n}}\right]
    \\ \nonumber
    &{+}\frac{1}{T_{n}^{\frac{1}{2}+\epsilon}}.
\end{align}
\subsubsection{$L^{\infty}$ decay of $P_{c,n-1}\vec{z}_{n}$}
Next, from Theorem \ref{Decesti1}, it is not difficult to verify using hypotheses $\mathrm{(H1)},\,\mathrm{(H2)}$ and estimate \eqref{dyn11} the following inequalities when $\min_{\ell}v_{\ell}(0)-v_{\ell+1}(0)$ is large enough.
\begin{align}\label{weight11}
\norm{\frac{\mathcal{U}_\sigma(t,\tau)P_{c}\overrightarrow{\psi_{0}}}{(1+\vert x-y_{\ell}-v_{\ell}t\vert)}}_{L^{\infty}_{x}(\mathbb{R})}\leq & \frac{K (y_{1}(0)-y_{m}(0)+\tau)}{(t-\tau)^{\frac{3}{2}}} \norm{P_{c}(\tau)\overrightarrow{\psi_{0}}(x)}_{L^{1}_{x}(\mathbb{R})}\\  \nonumber
&{+}\frac{K}{(t-\tau)^{\frac{3}{2}}}\max_{\ell}\norm{(1+\vert x-y_{\ell}-v_{\ell}\tau\vert )\chi_{\ell}(\tau,x)P_{c}(\tau)\overrightarrow{\psi_{0}}(x)}_{L^{1}_{x}(\mathbb{R})}\\  \nonumber
&{+}\frac{K e^{{-}\min_{j,\ell}\alpha_{j}((v_{\ell}-v_{\ell+1})\tau+y_{\ell}-y_{\ell+1})}\norm{P_{c}(\tau)\overrightarrow{\psi}_{0}(x)}_{L^{2}_{x}(\mathbb{R})}}{(t-\tau)^{\frac{3}{2}}},\\ 
\label{Weigth12}
\norm{\frac{\mathcal{U}_\sigma(t,\tau)P_{c}\overrightarrow{\psi_{0}}}{(1+\vert x-y_{\ell}-v_{\ell}t\vert)}}_{L^{\infty}_{x}(\mathbb{R})}\leq &\frac{K}{(t-\tau)^{\frac{1}{2}}}\norm{P_{c}(\tau)\vec{\psi}_{0}(x)}_{L^{1}_{x}(\mathbb{R})}\\&{+}\frac{K}{(t-\tau)^{\frac{1}{2}}}e^{{-}\frac{\min_{j,\ell}\alpha_{j}((v_{\ell}-v_{\ell+1})\tau+y_{\ell}-y_{\ell+1})}{2}}\norm{P_{c}(\tau)\vec{\psi}_{0}(x)}_{L^{2}_{x}(\mathbb{R})},\\ 
\label{weight13}
\norm{\frac{\mathcal{U}_\sigma(t,\tau)P_{c}\overrightarrow{\psi_{0}}}{(1+\vert x-y_{\ell}-v_{\ell}t\vert)}}_{L^{\infty}_{x}(\mathbb{R})}\leq &K\norm{P_{c}(\tau)\overrightarrow{\psi_{0}}}_{H^{1}_{x}(\mathbb{R})}.
\end{align}
Consequently, we can verify the following inequality for all $t\geq 0.$\begin{align}\label{wdiffff}
    \norm{\frac{\chi_{\ell}(t,x)\mathcal{U}_\sigma(t,0)P_{c,n-1}(0)\vec{z}_{n}(0,x)}{\langle x-y^{T_{n-1}}_{\ell,\sigma_{n-1}}(t) \rangle}}_{L^{\infty}_{x}(\mathbb{R})}\lesssim & \frac{y_{1}(0)-y_{m}(0)}{(y_{1}(0)-y_{m}(0)+t)(1+t)^{\frac{1}{2}}}\left[\norm{P_{c,n-1}\vec{z}_{n}(0)}_{H^{1}_{x}(\mathbb{R})}\right]\\ \nonumber
    &{+}\frac{y_{1}(0)-y_{m}(0)}{(y_{1}(0)-y_{m}(0)+t)(1+t)^{\frac{1}{2}}}\left[\norm{P_{c,n-1}\vec{z}_{n}(0)}_{L^{1}_{x}(\mathbb{R})}\right]\\ \nonumber
    &{+}\frac{1}{(1+t)^{\frac{3}{2}}}\max_{j\in[m]}\norm{\chi_{j,n-1}(0,x)\vert x-y_{j,n-1}(0)\vert\vec{z}_{n}(0,x)}_{L^{1}_{x}(\mathbb{R})}.
\end{align}
Moreover, $\vec{u}_{n}(0)-\vec{u}_{n-1}(0)$ is a finite sum of localized Schwartz functions with exponential decay from its formula in Proposition \ref{hyperbolicerror}. Therefore, we deduce that
\begin{align}\label{u(t,0)zn0}
    \norm{\frac{\chi_{\ell}(t,x)\mathcal{U}_\sigma(t,0)P_{c,n-1}(0)\vec{z}_{n}(0,x)}{\langle x-y^{T_{n-1}}_{\ell,\sigma_{n-1}}(t) \rangle}}_{L^{\infty}_{x}(\mathbb{R})}\lesssim_{y_{1}(0)-y_{m}(0),\{(\alpha_{\ell}(0),v_{\ell}(0))\}_{[m]}} \frac{\norm{\vec{z}_{n}(0,x)}_{L^{2}_{x}(\mathbb{R})}}{(1+t)^{\frac{3}{2}}} \text{, for all $t\geq 0$}
\end{align}
\par Next, using estimates \eqref{weight11}, \eqref{Weigth12} and \eqref{weight13} and Lemma \ref{dinftydt}, we can verify that
\begin{multline*}
    \max_{\ell}\norm{\int_{0}^{t}\frac{\chi_{\ell,n-1}(s,x)}{\langle x-y^{T_{n-1}}_{\ell,n-1}(s)\rangle}\mathcal{U}_\sigma(t,s)P_{c,n-1}(s)\left[V^{T_{n-1}}_{\ell,\sigma_{n-1}}(s,x)-V_{\ell,\sigma_{n-1}}(s,x)\right]\vec{z}_{n}(s,x)}_{L^{\infty}_{x}(\mathbb{R})}\\
    \begin{aligned}
    \lesssim & \int_{0}^{\max(0,t-1)}\frac{\delta_{0}(y_{1}-y_{m}+s)}{(y_{1}-y_{m}+t)(1+t-s)^{\frac{1}{2}}(1+s)^{2\epsilon-1}}\max_{\ell}\norm{\frac{\chi_{\ell,n-1}(s,x)\vec{z}_{n}(s)}{\langle x-y^{T_{n-1}}_{\ell,n-1}(s)\rangle}}_{L^{\infty}_{x}(\mathbb{R})}\,ds\\ 
    & {+}\int_{0}^{\max(0,t-1)}\frac{\delta_{0}}{(1+s)^{20}(1+t-s)^{\frac{3}{2}}}\norm{\vec{z}_{n}(s)}_{L^{2}_{x}(\mathbb{R})}\\
    &{+}\int_{t-1}^{t}\frac{\delta_{0}}{(t-s)^{\frac{1}{2}}(1+s)^{2\epsilon-1}}\max_{\ell}\norm{\frac{\chi_{\ell,n-1}(s,x)\vec{z}_{n}(s)}{\langle x-y^{T_{n-1}}_{\ell,n-1}(s)\rangle}}_{L^{\infty}_{x}(\mathbb{R})}\,ds.
    \end{aligned}
\end{multline*}
Consequently, using Lemma \ref{interpol} and the fact that $\epsilon>\frac{3}{4}$, we obtain the following
\begin{multline}\label{weightV-Vinftyz}
    \frac{1}{\langle t \rangle^{\frac{1}{4}} }\max_{\ell}\norm{\int_{0}^{t}\frac{\chi_{\ell,n-1}(t,x)}{\langle x-y^{T_{n-1}}_{\ell,n-1}(t)\rangle}\mathcal{U}_\sigma(t,s)P_{c,n-1}(s)\left[V^{T_{n-1}}_{\ell,\sigma_{n-1}}(s,x)-V_{\ell,\sigma_{n-1}}(s,x)\right]\vec{z}_{n}(s,x)}_{L^{\infty}_{x}(\mathbb{R})}\\
   \begin{aligned}
    \lesssim & \delta_{0} \max_{s\in [0,t]}\frac{1}{\langle s \rangle^{\frac{1}{4}}} \max_{\ell}\norm{\frac{\chi_{\ell,n-1}(s,x)\vec{z}_{n}(s)}{\langle x-y^{T_{n-1}}_{\ell,n-1}(s)\rangle}}_{L^{\infty}_{x}(\mathbb{R})}{+}\delta_{0}\max_{s\in [0,t]}\frac{\norm{\vec{z}_{n}(s)}_{L^{2}_{x}(\mathbb{R})}}{\langle s \rangle}. 
   \end{aligned}
\end{multline}
\par Moreover, we can verify from Proposition \ref{shititeraction}, Corollary \ref{diffprop2}, the definition of $\delta_{0}\in (0,1)$ in \eqref{deltachoice} and hypothesis $\mathrm{(H2)}$ that
\begin{multline*}
    \frac{1}{\langle t \rangle^{\frac{1}{4}}}\max_{\ell}\norm{\int_{0}^{t}\frac{\chi_{\ell,n-1}(t,x)}{\langle x-y^{T_{n-1}}_{\ell,n-1}(t) \rangle}\mathcal{U}_\sigma(t,s)P_{c,n-1}(s)\left[Int_{n-1}(s,x)-Int_{n-2}(s,x)\right]}_{L^{\infty}_{x}(\mathbb{R})}\\
    \begin{aligned}
        \lesssim & \frac{1}{\langle t \rangle^{\frac{1}{4}} }\int_{0}^{\max(0,t-1)}\frac{\delta_{0}\langle s \rangle^{1+\frac{1}{100}}}{\langle s\rangle^{20}} \frac{(y_{1}(0)-y_{m}(0)+s)}{(y_{1}(0)-y_{m}(0)+t)(1+t-s)^{\frac{1}{2}}} \max_{\tau \in [0,s]}\langle \tau \rangle^{1+\frac{\epsilon}{2}-\frac{3}{8}}\left\vert \Lambda \dot \sigma_{n-1}(\tau)-\Lambda \dot \sigma_{n-2}(\tau) \right\vert\,ds\\
        &{+}\frac{1}{\langle t \rangle^{\frac{1}{4}}}\int_{\max(0,t-1)}^{t}
        \frac{\delta_{0}\langle s \rangle^{1+\frac{1}{100}}}{\langle s\rangle^{20}(t-s)^{\frac{1}{2}}}\max_{\tau \in [0,s]}\langle \tau \rangle^{1+\frac{\epsilon}{2}-\frac{3}{8}}\left\vert \Lambda \dot \sigma_{n-1}(\tau)-\Lambda \dot \sigma_{n-2}(\tau) \right\vert\,ds.
    \end{aligned}
\end{multline*}
As a consequence, we can obtain using Lemma \ref{interpol} the following weighted estimate on $Int_{n-1}(s,x)-Int_{n-2}(s,x).$
\begin{multline}\label{weightofint1-int2}
    \frac{1}{\langle t \rangle^{\frac{1}{4}}}\max_{\ell}\norm{\int_{0}^{t}\frac{\chi_{\ell,n-1}(t,x)}{\langle x-y^{T_{n-1}}_{\ell,n-1}(t) \rangle}\mathcal{U}_\sigma(t,s)P_{c,n-1}(s)\left[Int_{n-1}(s,x)-Int_{n-2}(s,x)\right]}_{L^{\infty}_{x}(\mathbb{R})}\\
    \lesssim \frac{\delta_{0}}{\langle t\rangle^{\frac{1}{4}}} \max_{\tau \in [0,t]}\langle \tau \rangle^{1+\frac{\epsilon}{2}-\frac{3}{8}}\left\vert \Lambda \dot \sigma_{n-1}(\tau)-\Lambda \dot \sigma_{n-2}(\tau) \right\vert.
\end{multline}
\par Next, using the estimate \eqref{estintunstscatt} and the decay estimates  \eqref{weight11}-\eqref{weight13}, we can deduce that
\begin{multline*}
\frac{1}{t^{\frac{1}{4}}}\max_{\ell}\norm{\int_{0}^{t}\frac{\chi_{\ell,n-1}(s,x)}{\langle x-y^{T_{n-1}}_{\ell,n-1}(s)\rangle}\mathcal{U}_{\sigma}(t,s)\left[Int_{\mathrm{unst},n-1}(s,x)-Int_{\mathrm{unst},n-2}(s,x)\right]\,ds}_{L^{\infty}_{x}(\mathbb{R})}\\
\begin{aligned}
    \lesssim & \frac{1}{t^{\frac{1}{4}}} \int_{0}^{t}\frac{\delta_{0}(y_{1}(0)-y_{m}(0)+s)}{(1+s)^{20}(y_{1}(0)-y_{m}(0)+t)(1+t-s)^{\frac{1}{2}}}[\max_{\ell} \vert b_{\ell,n,+}(s)-b_{\ell,n-1,+}(s) \vert  
    {+}\norm{\vec{z}_{n}(s)}_{L^{2}_{x}(\mathbb{R})}]\,ds
    \\&{+}\frac{1}{t^{\frac{1}{4}}} \int_{0}^{t} \frac{\delta_{0}(y_{1}(0)-y_{m}(0)+s)}{(1+s)^{20}(y_{1}(0)-y_{m}(0)+t)(1+t-s)^{\frac{1}{2}}}\max_{q\in\{1,2\}}\max_{\ell\in[m],\mathfrak{M}\in\{v,y,\gamma,\alpha\}}\vert \mathfrak{M}^{T_{n-1}}_{\ell,n-1}(s)-\mathfrak{M}^{T_{n-2}}_{\ell,n-2}(s) \vert \,ds.
\end{aligned}
\end{multline*}
Consequently, we can verify using Corollary \ref{bcorol} and \eqref{themostimportantupperbounddiffb} that
\begin{multline}\label{unstintlinfty}
    \frac{1}{\langle t\rangle ^{\frac{1}{4}}}\max_{\ell}\norm{\int_{0}^{t}\frac{\chi_{\ell,n-1}(s,x)}{\langle x-y^{T_{n-1}}_{\ell,\sigma_{n-1}}(s)\rangle}\mathcal{U}_{\sigma}(t,s)P_{c,n-1}(s)\left[Int_{\mathrm{unst},n-1}(s,x)-Int_{\mathrm{unst},n-2}(s,x)\right]\,ds}_{L^{\infty}_{x}(\mathbb{R})}\\
    \lesssim \frac{\delta_{0}}{\langle t\rangle ^{\frac{1}{4}}} \frac{1}{T_{n}^{\frac{1}{2}+\epsilon}}+\delta_{0}  \max_{\tau \in [0,t]}\langle \tau \rangle^{1+\frac{\epsilon}{2}-\frac{3}{8}}\left\vert \Lambda \dot \sigma_{n-1}(\tau)-\Lambda \dot \sigma_{n-2}(\tau) \right\vert+\frac{\delta_{0}}{\langle t \rangle^{\frac{1}{4}}} \max_{s\in[0,t]}\frac{ \norm{\vec{z}_{n}(s)}_{L^{2}_{x}(\mathbb{R})}}{\langle s \rangle }. 
\end{multline}
\par Moreover, we can verify using Corollary \ref{diffprop2} with $\tau_{1}=\epsilon-\frac{3}{4}\in (0,1),$ and estimates \eqref{weight11}, \eqref{Weigth12} that
\begin{multline*}
     \frac{1}{\langle t \rangle^{\frac{1}{4}}}\max_{\ell}\norm{\int_{0}^{t}\frac{\chi_{j,n-1}(t,x)}{\langle x-y^{T_{n-1}}_{j,n-1}(t) \rangle}\mathcal{U}_\sigma(t,s)P_{c,n-1}(s)\sum_{\ell}\left[V_{\ell,\sigma_{n-1}}(s,x)-V_{\ell,\sigma_{n-2}}(s,x)\right]\vec{u}_{n-1}(s,x)}_{L^{\infty}_{x}(\mathbb{R})}\\
     \begin{aligned}
     \lesssim &\int_{0}^{\frac{t}{2}} \frac{(y_{1}(0)-y_{m}(0))(1+s)^{2+\epsilon-\frac{3}{4}}}{\langle t \rangle^{\frac{1}{4}}(1+t-s)^{\frac{3}{2}}} \max_{\tau \in [0,s],\ell}\langle \tau \rangle^{1+\frac{\epsilon}{2}-\frac{3}{8}}\left\vert \Lambda \dot \sigma_{n-1}(\tau)-\Lambda \dot \sigma_{n-2}(\tau) \right\vert\norm{\frac{\chi_{\ell,n-1}(s,x)\vec{u}_{n-1}(s,x)}{\langle x-y^{T_{n-1}}_{\sigma_{n-1}}(s) \rangle^{\frac{3}{2}+\omega}}}_{L^{2}_{x}(\mathbb{R})}\,ds\\
     &{+}\frac{1}{\langle t \rangle^{\frac{1}{4}}}\int_{\frac{t}{2}}^{t}\frac{(1+s)^{1+\epsilon-\frac{3}{4}}}{(t-s)^{\frac{1}{2}}} \max_{\tau \in [0,s]}\langle \tau \rangle^{1+\frac{\epsilon}{2}-\frac{3}{8}}\left\vert \Lambda \dot \sigma_{n-1}(\tau)-\Lambda \dot \sigma_{n-2}(\tau) \right\vert\max_{\ell}\norm{\frac{\chi_{\ell,n-1}(s,x)\vec{u}_{n-1}(s,x)}{\langle x-y^{T_{n-1}}_{\sigma_{n-1}}(s) \rangle^{\frac{3}{2}+\omega}}}_{L^{2}_{x}(\mathbb{R})}\,ds.
     \end{aligned}
\end{multline*}
As a consequence, we can deduce using Lemma \ref{interpol} and estimate \eqref{decay4} of Proposition \ref{undecays} satisfied by $\vec{u}_{n-1}$ the following inequality.
\begin{multline}\label{Vn-1Vn-2un-1w}
     \frac{1}{\langle t \rangle^{\frac{1}{4}}}\max_{j}\norm{\int_{0}^{t}\frac{\chi_{j,n-1}(t,x)}{\langle x-y^{T_{n-1}}_{j,n-1}(t) \rangle}\mathcal{U}_\sigma(t,s)P_{c,n-1}(s)\sum_{\ell}\left[V_{\ell,\sigma_{n-1}}(s,x)-V_{\ell,\sigma_{n-2}}(s,x)\right]\vec{u}_{n-1}(s,x)}_{L^{\infty}_{x}(\mathbb{R})}\\
     \lesssim \delta_{0}   \max_{\tau \in [0,t]}\langle \tau \rangle^{1+\frac{\epsilon}{2}-\frac{3}{8}}\left\vert \Lambda \dot \sigma_{n-1}(\tau)-\Lambda \dot \sigma_{n-2}(\tau) \right\vert.
\end{multline}
In particular, the reason for the choice of the weight $\langle t \rangle^{{-}\frac{1}{4}}$ in the localized $L^{\infty}$ norm is to obtain the sharper quantity
\begin{equation}\label{cccc}
    \delta_{0} \norm{(\vec{u}_{n-1}-\vec{u}_{n-2},\sigma_{n-1}-\sigma_{n-2})}_{Y_{n-1}}
\end{equation}
in the right-hand side of inequality \eqref{Vn-1Vn-2un-1w}. 
Using any power of $\langle t\rangle^{q}$ with $q>{-}\frac{1}{4}$ would make the right-hand side of \eqref{Vn-1Vn-2un-1w} much larger than \eqref{cccc}, and this value would be able to diverge as $t\to{+}\infty.$
\par Next, similarly to the estimate of the second inequality of \eqref{bullshit}, we can verify from the hypothesis $\mathrm{(H2)}$ and  estimate \eqref{decay4} of Proposition \ref{undecays} the following inequality for any $C^{1}$ function $W$ satisfying $W(0)=0.$
\begin{multline}\label{bullshit2}
    \max_{\ell}\Bigg\vert\Bigg\vert\chi_{\ell,n-1}(t,x)\langle x-y^{T_{n-1}}_{\ell,n-1}(t)\rangle\Big[e^{i\theta_{\ell,n-1}(t)} W(\phi_{\ell}(x-y_{\ell,n-1}(t)))\vert \vec{u}_{n-1}(t) \vert^{2}\\{-}e^{i\theta_{\ell,n-2}(t)}W(\phi_{\ell}(x-y_{\ell,n-2}(t)))\vert \vec{u}_{n-2}(t) \vert^{2}\Big]\Bigg\vert\Bigg\vert_{L^{1}_{x}(\mathbb{R})}\\ 
    \begin{aligned}
    \lesssim &\frac{\delta_{0}}{(1+t)^{\frac{1}{2}+\epsilon}} \norm{\frac{\chi_{\ell,n-1}(t,x)\vec{z}_{n}(t,x)}{\langle x-v_{\ell,n-1}(T_{n-1})t-D_{\ell,n-1}(T_{n-1}) \rangle}}_{L^{\infty}_{x}}\\&{+}\frac{\delta_{0}^{2}\norm{[ e^{i\theta_{\ell,n-1}(t)}W(\phi_{\ell}(x-y_{\ell,n-2}(t)))-e^{i\theta_{\ell,n-2}(t)}W(\phi_{\ell}(x-y_{\ell,n-1}(t)))]\langle x-y^{T_{n-1}}_{\ell,\sigma_{n-1}}(t) \rangle^{3+2\omega}}_{L^{2}_{x}(\mathbb{R})}}{(1+t)^{1+2\epsilon}}.
    \end{aligned}
\end{multline}
Therefore, using Corollary \ref{quadratipotentialestimate}, estimates \eqref{bullshit}, \eqref{bullshit2} and the Fundamental Theorem of calculus, we can verify for any $k>2$ the following estimate below for all $t\in [0,T_{n}].$
\begin{multline*}\max_{\ell}\langle t \rangle^{{-}\frac{1}{4}}\norm{\int_{0}^{t}\frac{\chi_{\ell,n-1}(t,x)\mathcal{U}_\sigma(t,s)\left[ N(\sigma_{n-1}(s),\vec{u}_{n-1}(s))-N(\sigma_{n-2}(s),\vec{u}_{n-2}(s))\right]}{\langle x-y^{T_{n-1}}_{\ell,\sigma_{n-1}}(s) \rangle}\,ds}_{L^{\infty}_{x}(\mathbb{R})}\\
    \begin{aligned}
        \lesssim & \frac{1}{\langle t \rangle^{\frac{1}{4}}}\int_{0}^{\frac{t}{2}}\frac{\delta_{0} (y_{1}(0)-y_{m}(0)+s)}{(1+s)^{\frac{1}{2}+\epsilon}(y_{1}(0)-y_{m}(0)+t) (1+t-s)^{\frac{1}{2}}}\Bigg[\max_{\ell}\norm{\chi_{\ell,n-1}(s)\frac{\vec{u}_{n-1}(s,x)-\vec{u}_{n-2}(s)}{\langle x-y^{T_{n-1}}_{\ell,\sigma_{n-1}}(s)\rangle}}_{L^{\infty}_{x}(\mathbb{R})}\\&{+}\max_{\tau\in[0,s]}\langle \tau \rangle^{1+\frac{\epsilon}{2}-\frac{3}{8}}\left\vert \Lambda \dot \sigma_{n-1}(\tau)-\Lambda \dot \sigma_{n-2}(\tau) \right\vert \Bigg]\,ds\\
         &{+}\frac{1}{\langle t \rangle^{\frac{1}{4}}}\int_{0}^{\frac{t}{2}}\frac{\delta_{0}^{2k}(y_{1}(0)-y_{m}(0)+s)}{(y_{1}(0)-y_{m}(0)+t) (1+t-s)^{\frac{1}{2}}(1+s)^{k-\frac{1}{2}}}\norm{\vec{u}_{n-1}(s)-\vec{u}_{n-2}(s)}_{L^{2}_{x}(\mathbb{R})}\,ds
        \\
        &{+}\frac{1}{\langle t \rangle^{\frac{1}{4}}}\int_{0}^{\frac{t}{2}}\frac{\delta_{0}(y_{1}(0)-y_{m}(0)+s)}{(1+s)^{20}(y_{1}(0)-y_{m}(0)+t)(1+t-s)^{\frac{1}{2}}}\Bigg[\norm{\vec{u}_{n-1}(s)-\vec{u}_{n-2}(s)}_{L^{2}_{x}(\mathbb{R})}\\&{+}\max_{\tau\in[0,s]}\langle \tau \rangle^{1+\frac{\epsilon}{2}-\frac{3}{8}}\left\vert \Lambda \dot \sigma_{n-1}(\tau)-\Lambda \dot \sigma_{n-2}(\tau) \right\vert\Bigg]\,ds\\
        &{+}\int_{\frac{t}{2}}^{t}\frac{1}{\langle t\rangle^{\frac{1}{4}}(t-s)^{\frac{1}{2}}}\left[\frac{\delta_{0}^{2k}}{(1+s)^{k-\frac{1}{2}}}\norm{\vec{u}_{n-1}(t)-\vec{u}_{n-2}(t)}_{L^{2}_{x}(\mathbb{R})}\right]\\&{+}\int_{\frac{t}{2}}^{t}\frac{1}{\langle t\rangle^{\frac{1}{4}}(t-s)^{\frac{1}{2}}}\left[\frac{\delta_{0}}{(1+s)^{\frac{1}{2}+\epsilon}}\max_{\ell}\norm{\frac{\chi_{\ell,n-1}(s,x)[\vec{u}_{n-1}(s)-\vec{u}_{n-2}(s)]}{\langle x-y^{T_{n-1}}_{\sigma_{n-1}}(s)\rangle}}_{L^{\infty}_{x}(\mathbb{R})}\right]\,ds\\
        &{+}\int_{\frac{t}{2}}^{t}\frac{\delta_{0}}{\langle t\rangle^{\frac{1}{4}} (1+s)^{\frac{1}{2}+\epsilon}(t-s)^{\frac{1}{2}}}\left[\max_{\tau\in[0,s]}\langle \tau \rangle^{1+\frac{\epsilon}{2}-\frac{3}{8}}\left\vert \Lambda \dot \sigma_{n-1}(\tau)-\Lambda \dot \sigma_{n-2}(\tau) \right\vert\right]\,ds
    \end{aligned}.
\end{multline*}
Consequently, using Lemma \ref{interpol} and condition $k>\frac{11}{4},$ we conclude the following estimate.
\begin{multline}\label{DiffNolUt}\max_{\ell}\langle t \rangle^{{-}\frac{1}{4}}\norm{\int_{0}^{t}\frac{\chi_{\ell,n-1}(t,x)\mathcal{U}_\sigma(t,s)\left[ N(\sigma_{n-1}(s),\vec{u}_{n-1}(s))-N(\sigma_{n-1}(s),\vec{u}_{n-1}(s))\right]}{\langle x-y^{T_{n-1}}_{\ell,\sigma_{n-1}}(s) \rangle}\,ds}_{L^{\infty}_{x}(\mathbb{R})}\\
    \begin{aligned}
    \lesssim & \delta_{0}\left[\max_{s\in [0,t]} \frac{1}{\langle s\rangle^{\frac{1}{4}}}\max_{\ell}\norm{\chi_{\ell,n-1}(s)\frac{\vec{u}_{n-1}(s,x)-\vec{u}_{n-2}(s)}{\langle x-y^{T_{n-1}}_{\ell,\sigma_{n-1}}(s)\rangle}}_{L^{\infty}_{x}(\mathbb{R})}+\max_{s\in[0,t]}\langle s \rangle^{1+\frac{\epsilon}{2}-\frac{3}{8}}\left\vert \Lambda \dot \sigma_{n-1}(s)-\Lambda \dot \sigma_{n-2}(s) \right\vert\right] \\
    &{+}\delta_{0}\left[ \max_{s\in[0,t]}\frac{\norm{\vec{u}_{n-1}(s)-\vec{u}_{n-2}(s)}_{L^{2}_{x}(\mathbb{R})}}{\langle s \rangle }\right].
    \end{aligned}
\end{multline}
\par Next, from \eqref{ker2basis} and the elementary estimate for any $n\in\mathbb{N}$
\begin{equation*}
    \left\vert \frac{d^{n}}{dx^{n}}\phi_{1}(x) \right\vert\lesssim_{n} e^{{-}\vert x \vert },
\end{equation*}
we can verify using the assumption $\mathrm{(H2)}$ on $\{y_{\ell}(0)\}_{\ell\in[m]}$ and estimates \eqref{dyn11} and \eqref{dyn12} that any element $\vec{z}\in\ker\mathcal{H}^{2}_{1}$ satisfies
\begin{equation}\label{ss1}
    \max_{j,\ell\in[m],\,h\in\{0,{-}1,{-}2\}}\norm{\langle x-y^{T_{n-1}}_{j,\sigma_{n-1}}(t)\rangle^{10+4\omega}\vec{z}(\alpha_{\ell,n+h}(t),x-y_{\ell,n}(t))}_{H^{2}_{x}(\mathbb{R})}\lesssim_{\norm{\vec{z}_{1}}_{L^{2}_{x}(\mathbb{R})}} 1.
\end{equation}
Therefore, we can deduce the following estimates
\begin{multline}\label{dmodweight1}
    \norm{e^{i\mathfrak{p}_3\theta_{\ell,n-1}(t,x)}\vec{\mathcal{E}}_{\ell}(\alpha_{\ell,n-1}(t),x-y_{\ell,n-1}(t))}_{H^{2}_{x}(\mathbb{R})}\\{+}\max_{\ell,j}\norm{\chi_{j,n-1}(t,x)\langle x-y^{T_{n-1}}_{\sigma_{n-1}}(t)\rangle e^{i\mathfrak{p}_3\theta_{\ell,n-1}(t,x)}\vec{\mathcal{E}}_{\ell}(\alpha_{\ell,n-1}(t),x-y_{\ell,n-1}(t))}_{L^{1}_{x}(\mathbb{R})}
    \lesssim 1, 
\end{multline}
for any function $\vec{\mathcal{E}}_{\ell}\in (\ker \mathcal{H}_{1})^{4}$ defined in \eqref{diffequ}.
Furthermore, using Corollary \ref{diffprop2}, we can deduce the following inequality.
\begin{multline}\label{dmodweight2}
    \max_{\ell,j\in[m]}\Bigg\vert\Bigg\vert\chi_{j,n-1}(t)\langle x-y^{T_{n-1}}_{j,\sigma_{n-1}}(t) \rangle \Big[e^{i\mathfrak{p}_3\theta_{\ell,n-1}(t,x)}\vec{\mathcal{E}}_{\ell}(\alpha_{\ell,n-1}(t),x-y_{\ell,n-1}(t))\\{-}e^{i\mathfrak{p}_3\theta_{\ell,n-2}(t,x)}\vec{\mathcal{E}}_{\ell}(\alpha_{\ell,n-2}(t),x-y_{\ell,n-2}(t))\Big]\Bigg\vert\Bigg\vert_{L^{1}_{x}(\mathbb{R})}\\
    {+}\norm{e^{i\mathfrak{p}_3\theta_{\ell,n-1}(t,x)}\vec{\mathcal{E}}_{\ell}(\alpha_{\ell,n-1}(t),x-y_{\ell,n-1}(t))-e^{i\mathfrak{p}_3\theta_{\ell,n-2}(t,x)}\vec{\mathcal{E}}_{\ell}(\alpha_{\ell,n-2}(t),x-y_{\ell,n-2}(t))}_{H^{2}_{x}(\mathbb{R})}\\
    \lesssim \langle t \rangle^{1+\frac{\epsilon}{100}-\frac{3}{800}} \max_{s\in[0,t]}\langle s \rangle^{1+\frac{\epsilon}{2}-\frac{3}{8}}\left\vert \Lambda \dot \sigma_{n-1}(s)-\Lambda \dot \sigma_{n-2}(s) \right\vert.
\end{multline}
Next, to simplify our reasoning, we consider the following notation{\footnotesize
\begin{gather*}
    W_{1}(t,x)=\sum_{\ell}\left(\Lambda \dot\sigma_{\ell,n-1}(t)-\Lambda \dot\sigma_{\ell,n-2}(t)\right)e^{i\mathfrak{p}_3\theta_{\ell,n-1}(t,x)}\vec{\mathcal{E}}_{\ell}(\alpha_{\ell,n-1}(t),x-y_{\ell,n-1}(t)),\\
    W_{2}(t,x)
    =\sum_{\ell}\Lambda \dot\sigma_{\ell,n-2}(t)\left[e^{i\mathfrak{p}_3\theta_{\ell,n-1}(t,x)}\vec{\mathcal{E}}_{\ell}(\alpha_{\ell,n-1}(t),x-y_{\ell,n-1}(t))-e^{i\mathfrak{p}_3\theta_{\ell,n-2}(t,x)}\vec{\mathcal{E}}_{\ell}(\alpha_{\ell,n-2}(t),x-y_{\ell,n-2}(t))\right].
\end{gather*}}
Consequently, we can deduce using \eqref{dmodweight1}, \eqref{dmodweight2}, and estimates \eqref{weight11} and \eqref{weight13} that
\begin{align}\label{finalest}
    &\langle t \rangle^{{-}\frac{1}{4}}\max_{\ell\in[m],j\in\{1,2\}}\norm{\int_{0}^{t}\frac{\chi_{\ell,n-1}(t,x)P_{c,n-1}(t)\mathcal{U}_\sigma(t,s)W_{j}(t,x)}{\langle x-y^{T_{n-1}}_{\ell,\sigma_{n-1}}(t) \rangle}}_{L^{\infty}_{x}(\mathbb{R})}\\
    &\lesssim \delta_{0} \max_{s\in[0,t]}\langle s \rangle^{1+\frac{\epsilon}{2}-\frac{3}{8}}\left\vert \Lambda \dot \sigma_{n-1}(s)-\Lambda \dot \sigma_{n-2}(s) \right\vert.
\end{align}
 Therefore, using the estimates \eqref{u(t,0)zn0}, \eqref{weightV-Vinftyz}, 
\eqref{weightofint1-int2}, \eqref{unstintlinfty}, \eqref{Vn-1Vn-2un-1w}, \eqref{DiffNolUt} and \eqref{finalest}, we deduce for all $t\geq 0$ (when $t>T_{n},\, \vec{z}_{n}(t)\equiv 0$)
\begin{align}\nonumber
    \frac{1}{\langle t \rangle^{\frac{1}{4}} }\max_{\ell} \norm{\frac{\chi_{\ell}(t,x)P_{c,n-1}(t)\vec{z}_{n}(t,x)}{\langle x-y^{T_{n-1}}_{\ell,\sigma_{n-1}}(t) \rangle}}_{L^{\infty}_{x}(\mathbb{R})}\lesssim &\frac{\norm{P_{c,n-1}\vec{z}_{n}(0)}_{L^{2}_{x}}}{\langle t \rangle^{\frac{7}{4}}}+\delta_{0} \norm{(\vec{u}_{n}-\vec{u}_{n-1},\sigma_{n}-\sigma_{n-1})}_{Y_{n}}\\ \label{conclusion1weightednormzn}
    &{+}\delta_{0} \norm{(\vec{u}_{n-1}-\vec{u}_{n-2},\sigma_{n-1}-\sigma_{n-2})}_{Y_{n-1}}+\frac{\delta_{0}}{T^{\frac{1}{2}+\epsilon}_{n}}. 
\end{align}
In conclusion, estimates \eqref{zn(0)size} and \eqref{conclusion1weightednormzn} imply that
\begin{align}\label{conclusion1weightednormzn1}
    \frac{1}{\langle t \rangle^{\frac{1}{4}} }\max_{\ell} \norm{\frac{\chi_{\ell}(t,x)P_{c,n-1}(t)\vec{z}_{n}(t,x)}{\langle x-y^{T_{n-1}}_{\ell,\sigma_{n-1}}(t) \rangle}}_{L^{\infty}_{x}(\mathbb{R})}\lesssim &\delta_{0} \Bigg[\norm{(\vec{u}_{n-1}-\vec{u}_{n-2},\sigma_{n-1}-\sigma_{n-2})}_{Y_{n-1}}\\&{+}\norm{(\vec{u}_{n}-\vec{u}_{n-1},\sigma_{n}-\sigma_{n-1})}_{Y_{n}}\Bigg]\\ \nonumber
    &{+}\frac{1}{T_{n-1}^{\frac{1}{2}+\epsilon}}. 
\end{align}
\subsection{Root space}
We recall that each $\vec{u}_{n}$ has a unique representation of the form
\begin{align*}
\vec{u}_{n}(t,x)=&P_{c,n-1}(t)\vec{u}_{n}(t,x)\\&{+}\sum_{\ell=1}^{m}b_{\ell,n,+}(t)\left[e^{i\mathfrak{p}_3\theta^{T_{n-1}}_{\ell,n-1}(t,x)}\vec{Z}_{+,n-1}(\alpha_{\ell,n-1}(T_{n-1}),x-y^{T_{n-1}}_{\ell,\sigma_{n-1}}(t))\right]\\&{+}\sum_{\ell}\sum_{w\in Basis_{2}}\vec{A}_{\ell,w}(\vec{u}_{n,c},\vec{b}_{n}(t))e^{i\mathfrak{p}_3\theta^{T_{n-1}}_{\ell,n-1}(t,x)}\vec{w}(\alpha_{\ell,n-1}(T_{n-1}),x-y^{T_{n-1}}_{\ell,\sigma_{n-1}}(t)),
\end{align*}
such that the functions $\vec{A}_{\ell,w}\in\mathbb{C}^{2}$ are uniquely determined to satisfy for all $\vec{w}\in \ker\mathcal{H}^{2}_{1}$
\begin{equation*}
    \langle \vec{u}_{n}(t,x),\mathfrak{p}_3e^{i\theta_{\ell,n-1}(t,x)\mathfrak{p}_3}\vec{w}(\alpha_{\ell,n-1}(t),x-y_{\ell,n-1}(t)) \rangle=0 \text{ for any $t\in [0,T_{n}].$}
\end{equation*}
To simplify more our notation, we consider
\begin{multline}\label{diffb}
    diff b_{\ell,+,n}(t) Z:= b_{\ell,n,+}(t)\left[e^{i\mathfrak{p}_3\theta^{T_{n-1}}_{\ell,n-1}(t,x)}\vec{Z}_{+}(\alpha_{\ell,n-1}(T_{n-1}),x-y^{T_{n-1}}_{\ell,\sigma_{n-1}}(t))\right]\\
    {-} b_{\ell,n-1,+}(t)\left[e^{i\mathfrak{p}_3\theta^{T_{n-1}}_{\ell,n-1}(t,x)}\vec{Z}_{+}(\alpha_{\ell,n-2}(T_{n-2}),x-y^{T_{n-2}}_{\ell,\sigma_{n-2}}(t))\right].
\end{multline}
We can find a similar decomposition for $\vec{z}_{n}=\vec{u}_{n}-\vec{u}_{n-1}$ given by
\begin{align}\label{znformula2}
   \vec{z}_{n}(t,x)=& P_{c,n-1}(t)\vec{z}_{n}(t,x)\\&{+}\sum_{\ell=1}^{m}b_{\ell,n,n-1,+}(t)\left[e^{i\mathfrak{p}_3\theta^{T_{n-1}}_{\ell,n-1}(t,x)}\vec{Z}_{+}(\alpha_{\ell,n-1}(T_{n-1}),x-y^{T_{n-1}}_{\ell,\sigma_{n-1}}(t))\right]\\&{+}\sum_{\ell}\sum_{w\in Basis_{2}}\vec{a}_{\ell,w}(t)e^{i\mathfrak{p}_3\theta^{T_{n-1}}_{\ell,n-1}(t,x)}\vec{w}(\alpha_{\ell,n-1}(T_{n-1}),x-y^{T_{n-1}}_{\ell,\sigma_{n-1}}(t)),
\end{align}
where the functions $b_{\ell,n,n-1}$ were already estimated in the previous subsection. In particular, using the fact that $\langle \vec{z}_{\pm}(\alpha,x),\mathfrak{p}_3\vec{z}_{0}(\alpha,x) \rangle=0$ for any $z_{\pm}\in \ker (\mathcal{H}_1\mp i\lambda_{0}\mathrm{Id})$
and $z_{0}\in \ker \mathcal{H}_{1}^2,$ we can verify the following proposition.

\begin{proposition}\label{kernel2proj}
If $t\in [0,T_{n}],$ the following estimate holds    
\begin{align*}
    \max_{\ell,t\in [0,T_{n}]}\frac{\vert \vec{a}_{\ell,\vec{w}}(t) \vert}{(1+t)^{\frac{1}{4}}}\lesssim & \delta_{0}\left[\norm{(\vec{u}_{n}-\vec{u}_{n-1},\sigma_{n}-\sigma_{n-1})}_{Y_{n}}+\norm{(\vec{u}_{n-1}-\vec{u}_{n-2},\sigma_{n-1}-\sigma_{n-2})}_{Y_{n-1}}\right]\\
    &{+}\frac{1}{T_{n-1}^{\frac{1}{2}+\epsilon}}.
\end{align*}
 \end{proposition}
\begin{proof}
First, we recall for all $\vec{w}\in\ker\mathcal{H}^{2}_{1}$ that 
\begin{align}\label{ottt1}
  \langle \vec{u}_{n}(t,x),\mathfrak{p}_3e^{i\mathfrak{p}_3\theta_{\ell,n-1}(t,x)}\vec{w}(\alpha_{\ell,n-1}(t),x-y_{\ell,n-1}(t)) \rangle=&0,\\ \label{ottt2}
  \langle \vec{u}_{n-1}(t,x),\mathfrak{p}_3e^{i\mathfrak{p}_3\theta_{\ell,n-2}(t,x)}\vec{w}(\alpha_{\ell,n-2}(t),x-y_{\ell,n-2}(t)) \rangle=&0.
\end{align}
Moreover, estimates \eqref{dyn11}, \eqref{dyn12} imply for all $t\in [0,T_{n}]$ that
\begin{multline*}
 \langle e^{i\mathfrak{p}_3\theta_{\ell,n-1}(t,x)}\vec{w}(\alpha_{\ell,n-1}(t),x-y_{\ell,n-1}(t)),e^{i\mathfrak{p}_3\theta^{T_{n-1}}_{\ell,n-1}(t,x)}\vec{w}(\alpha_{\ell,n-1}(T_{n-1}),x-y^{t_{n-1}}_{\ell,\sigma_{n-1}}(t)) \rangle\\
= \norm{\vec{w}(\alpha_{\ell}(0),x)}_{L^{2}_{x}(\mathbb{R})}+O\left(\delta_{0}\right).
\end{multline*}
Consequently, using identities \eqref{ottt1} and \eqref{ottt2}, we can obtain from estimating 
\begin{equation*}
    \langle \vec{z}_{n}(t,x),\mathfrak{p}_3e^{i\mathfrak{p}_3\theta_{\ell,n-1}(t,x)}\vec{w}(\alpha_{\ell,n-1}(t),x-y_{\ell,n-1}(t)) \rangle
\end{equation*}
 that  
\begin{multline*}
\max_{\ell,\vec{w}\in Basis_{2}}\left\vert a_{\ell,\vec{w}}(t)\right\vert\\
\begin{aligned}
\lesssim \max_{\vec{w}\in Basis_{2}} & \left\vert \langle P_{c,n-1}(t)\vec{z}_{n}(t), \mathfrak{p}_3e^{i\mathfrak{p}_3\theta_{n-1}(t,x)}\vec{w}(\alpha_{\ell,n-1}(t),x-y_{\ell,n-1}(t))\rangle \right\vert\\
&{+} \frac{\delta_{0}}{(1+t)^{2\epsilon-1}} \max_{\ell} \vert b_{\ell,n,n-1}(t)\vert\\
&{+}\left\vert\langle \vec{u}_{n-1}(t),e^{i\mathfrak{p}_3\theta_{n-1}(t,x)}\vec{w}(\alpha_{\ell,n-1}(t),x-y_{\ell,n-1}(t))-e^{i\mathfrak{p}_3\theta_{n-2}(t,x)}\vec{w}(\alpha_{\ell,n-2}(t),x-y_{\ell,n-2}(t))\rangle\right\vert. 
\end{aligned}
\end{multline*}
Consequently, we can deduce from Lemma \ref{dinftydt}, Corollary \ref{diffprop2}, hypothesis $\mathrm{(H2)}$ and the inequality
\begin{equation*}
    \max_{\ell}\norm{\frac{\chi_{\ell,n-2}(t,x)\vec{u}_{n-1}(t,x)}{\langle x-y^{T_{n-2}}_{\ell,\sigma_{n-2}} \rangle^{\frac{3}{2}+\omega}}}\leq \frac{\delta_{0}}{(1+t)^{\frac{1}{2}+\epsilon}} \text{, for all $t\geq 0$} 
\end{equation*}
that
\begin{align}\label{almostaa}
    \max_{\ell,\vec{w}\in Basis_{2}}\left\vert a_{\ell,\vec{w}}(t)\right\vert\lesssim & \frac{\delta_{0}}{(1+t)^{2\epsilon-1}}\norm{\vec{z}_{n}(t)}_{L^{2}_{x}(\mathbb{R})}\\
    &{+}\delta_{0}\max_{s\in [0,T_{n}]} \langle s  \rangle^{1+\frac{\epsilon}{2}-\frac{3}{8}}\left\vert \Lambda\dot \sigma_{n-1}(s)-\Lambda\dot \sigma_{n-2}(s)(s) \right\vert{+}\frac{\delta_{0}}{T_{n-1}^{\frac{1}{2}+\epsilon}}.
\end{align}
In conclusion, since $\epsilon>\frac{3}{4},$ inequality \eqref{almostaa} implies the result of Proposition \ref{kernel2proj}.
\end{proof}
\subsection{Estimate of $\Lambda\dot\sigma_{n}-\Lambda\dot\sigma_{n-1}$}
First, the function $G(t,\sigma(t),\sigma_{n-1}(t),\vec{u}_{n-1})$ defined in \eqref{unequation}, and the identity obtained from \eqref{ODEofsigma} implies that the following equation
\begin{multline}\label{ODEofsigma2}
 \left\langle {-}iG(t,\sigma(t),\sigma_{n_{1}-1}(t),\vec{u}_{n_{1}-1}),\mathfrak{p}_3e^{i\mathfrak{p}_3\left(\frac{v_{\ell,n_{1}-1}(t)x}{2}+\gamma_{\ell,n_{1}-1}(t)\right)}\vec{w}(\alpha_{\ell,n_{1}-1}(t),x-y_{\ell,n_{1}-1}(t))\right\rangle\\
 {+}\left\langle \vec{u}_{n_{1}}(t,x),\mathfrak{p}_3\left(\partial_{t}-i\mathfrak{p}_3\partial^{2}_{x}-iV_{\ell,\sigma_{n_{1}-1}}(t,x)\right)\left[e^{i\mathfrak{p}_3\left(\frac{v_{\ell,n_{1}-1}(t)x}{2}+\gamma_{\ell,n_{1}-1}(t)\right)}\vec{w}(\alpha_{\ell,n_{1}-1}(t),x-y_{\ell,n_{1}-1}(t))\right] \right \rangle\\
 {+} \left\langle \vec{u}_{n_{1}-1}(t,x),{-}i\mathfrak{p}_3\left[\sum_{j\neq \ell}V_{j,\sigma_{n_{1}-1}}(t,x)\right]\left[e^{i\mathfrak{p}_3\left(\frac{v_{\ell,n_{1}-1}(t)x}{2}+\gamma_{\ell,n_{1}-1}(t)\right)}\vec{w}(\alpha_{\ell,n_{1}-1}(t),x-y_{\ell,n_{1}-1}(t))\right] \right \rangle=0,
\end{multline}
holds for all $\vec{w}\in \ker\mathcal{H}^{2}_{1},$ any $\ell\in[m],\,n_{1}\in\mathbb{N}.$ 
\par Therefore, computing the difference between the equations \eqref{ODEofsigma2} satisfied by $n_{1}=n$ and $n_{1}=n-1$ for any $\vec{w}\in Basis_{2},$ we can verify using \eqref{unequation}, Lemma \ref{interactt}, Definition \ref{lambdasigma} and Corollary \ref{diffprop2} the following estimate.
\begin{multline}\label{dfff}
 \max_{f\in\{v,y,\alpha,\gamma\},\ell\in [m]}\vert  \Lambda \dot f_{\ell,n}(t)-\Lambda \dot f_{\ell,n-1}(t) \vert\\
 \begin{aligned}
   \lesssim  & \max_{n_{1}\in\{n-1,n-2\},\ell\in[m]}\norm{\frac{\chi_{\ell,n-1}(t)}{\langle x-y^{T_{n-1}}_{\ell,n-1}(t) \rangle }\vec{z}_{n}(t,x)}_{L^{\infty}_{x}(\mathbb{R})}\left\vert \dot \Lambda \sigma_{n_{1}}(t) \right\vert\\
    &{+} \max_{n_{1}\in\{n-1,n-2\},\ell\in[m]}\norm{\frac{\chi_{\ell,n-2}(t)}{\langle x-y^{T_{n-2}}_{\ell,n-2}(t) \rangle }\vec{z}_{n-1}(t,x)}_{L^{\infty}_{x}(\mathbb{R})}\left[\left\vert \dot \Lambda \sigma_{n_{1}}(t) \right\vert +\delta e^{{-}t}\right]\\
    &{+}\max_{n_{1}\in\{n,n-1\}}\left\vert \dot \Lambda \sigma_{n_{1}-1}(t) \right\vert\norm{\frac{\chi_{\ell,n_{1}-1}(t) \vec{u}_{n_{1}}(t)}{\langle x-y^{T_{n_{1}-1}}_{\ell,n_{1}-1}(t) \rangle^{\frac{3}{2}+\omega} }}_{L^{2}_{x}(\mathbb{R})}\langle t\rangle^{1+\frac{\epsilon}{50}} \norm{(\vec{u}_{n-1}-\vec{u}_{n-2},\sigma_{n-1}-\sigma_{n-2})}_{Y_{n-1}}\\
    &{+}\left(\max_{n_{1}\in\{n,n-1\}}\vert \dot\Lambda\sigma_{n_{1}}(t) \vert\right) \langle t\rangle^{1+\frac{\epsilon}{2}-\frac{3}{8}}\left\vert \Lambda \dot \sigma_{n-1}(t)- \Lambda \dot\sigma_{n-2}(t) \right\vert\\
    &{+}\max_{n_{1}\in\{n,n-1\}}\norm{\frac{\chi_{\ell,n_{1}-1}(t) \vec{u}_{n_{1}}(t)}{\langle x-y^{T_{n_{1}-1}}_{\ell,n_{1}-1}(t) \rangle^{\frac{3}{2}+\omega} }}_{L^{2}_{x}(\mathbb{R})}\left\vert  \Lambda \dot \sigma_{n-1}(t)- \Lambda \dot\sigma_{n-2}(t) \right\vert\\
    &{+}\max_{\ell,\vec{w}\in Basis_{2}}\Bigg\vert \langle Int_{n-1}(t,x),\mathfrak{p}_3e^{i\mathfrak{p}_3\theta_{n-1}(t,x)}\vec{w}(\alpha_{\ell,n-1}(t),x-y_{\ell,n-1}(t))\rangle \\&{-}\langle Int_{n-2}(t,x),\mathfrak{p}_3e^{i\mathfrak{p}_3\theta_{n-2}(t,x)}\vec{w}(\alpha_{\ell,n-2}(t),x-y_{\ell,n-2}(t))\rangle\Bigg\vert\\
    &{+}\max_{\ell,\vec{w}\in Basis_{2}}\Bigg\vert \langle N(\sigma_{n-1},\vec{u}_{n-1}),\mathfrak{p}_3e^{i\mathfrak{p}_3\theta_{n-1}(t,x)}\vec{w}(\alpha_{\ell,n-1}(t),x-y_{\ell,n-1}(t))\rangle \\ &{-}\langle N(\sigma_{n-2},\vec{u}_{n-2}),\mathfrak{p}_3e^{i\mathfrak{p}_3\theta_{n-2}(t,x)}\vec{w}(\alpha_{\ell,n-2}(t),x-y_{\ell,n-2}(t))\rangle\Bigg\vert,
\end{aligned}
\end{multline}
such that $Int_{n}(t,x)$ and $N(\sigma_{j},\vec{u}_{j})$ are defined respectively in \eqref{interactionfunction} and \eqref{Nonlineartermdiff}.
\par Moreover, Proposition \ref{shititeraction} and Corollary \ref{diffprop2} imply that
\begin{multline*}
    \max_{\ell,\vec{w}\in Basis_{2},t\in [0,T_{n}]}\Bigg\vert \langle Int_{n-1}(t,x),\mathfrak{p}_3e^{i\mathfrak{p}_3\theta_{n-1}(t,x)}\vec{w}(\alpha_{\ell,n-1}(t),x-y_{\ell,n-1}(t))\rangle \\{-}\langle Int_{n-2}(t,x),\mathfrak{p}_3e^{i\mathfrak{p}_3\theta_{n-2}(t,x)}\vec{w}(\alpha_{\ell,n-2}(t),x-y_{\ell,n-2}(t))\rangle\Bigg\vert\\
        \lesssim \delta_{0} \max_{t\in [0,T_{n}]}\frac{1}{(1+t)^{20}}\left\vert \Lambda \dot\sigma_{n-1}(t)-\Lambda \dot \sigma_{n-2}(t) \right\vert,  
\end{multline*}
from which we deduce the following estimate using \eqref{odes}
\begin{multline}\label{int1Dsigma}
    \max_{\ell,\vec{w}\in Basis_{2},t\in [0,T_{n}]}\Bigg\vert \langle Int_{n-1}(t,x),\mathfrak{p}_3e^{i\mathfrak{p}_3\theta_{n-1}(t,x)}\vec{w}(\alpha_{\ell,n-1}(t),x-y_{\ell,n-1}(t))\rangle \\{-}\langle Int_{n-2}(t,x),\mathfrak{p}_3e^{i\mathfrak{p}_3\theta_{n-2}(t,x)}\vec{w}(\alpha_{\ell,n-2}(t),x-y_{\ell,n-2}(t))\rangle\Bigg\vert\\
        \lesssim \delta_{0}   \norm{(\vec{u}_{n-1}-\vec{u}_{n-2},\sigma_{n-1}-\sigma_{n-2})}_{Y_{n-1}}+\frac{\delta_{0}}{(1+T_{n-1})^{2\epsilon}}.  
\end{multline}
\par Next, using the estimates in \eqref{bullshit} and Corollary \ref{diffprop2}, we can verify the following inequality.
\begin{multline*}
    \max_{\ell,\vec{w}\in Basis_{2}}\Bigg\vert \langle N(\sigma_{n-1},\vec{u}_{n-1}),\mathfrak{p}_3e^{i\mathfrak{p}_3\theta_{n-1}(t,x)}\vec{w}(\alpha_{\ell,n-1}(t),x-y_{\ell,n-1}(t))\rangle \\{-}\langle N(\sigma_{n-2},\vec{u}_{n-2}),\mathfrak{p}_3e^{i\mathfrak{p}_3\theta_{n-2}(t,x)}\vec{w}(\alpha_{\ell,n-2}(t),x-y_{\ell,n-2}(t))\rangle\Bigg\vert
    \\
    \lesssim \frac{\delta_{0}}{(1+t)^{\frac{1}{2}+\epsilon}} \norm{\frac{\chi_{\ell,n-1}(t,x)\vec{z}_{n}(t,x)}{\langle x-v_{\ell,n-1}(T_{n-1})t-D_{\ell,n-1}(T_{n-1}) \rangle}}_{L^{\infty}_{x}}\\{+}\frac{\delta_{0}^{2}}{(1+t)^{1+2\epsilon}} \langle t \rangle^{1+\frac{\epsilon}{50}}\max_{s\in [0,t]}\langle  s\rangle^{1+\frac{\epsilon}{2}-\frac{3}{8}} \left\vert \Lambda \dot\sigma_{n-1}(s)-\Lambda \dot \sigma_{n-2}(s) \right\vert\\
    {+}\max_{j\in\{n-1,n-2\}}\norm{N(\sigma_{j}(t),\vec{u}_{j}(t))}_{L^{2}_{x}(\mathbb{R})}\langle t \rangle^{1+\frac{\epsilon}{50}}\max_{s\in [0,t]}\langle  s\rangle^{1+\frac{\epsilon}{2}-\frac{3}{8}} \left\vert \Lambda \dot\sigma_{n-1}(s)-\Lambda \dot \sigma_{n-2}(s) \right\vert.
\end{multline*}
Therefore, we can conclude using \eqref{odes} for all $t\in [0,T_{n}]$
\begin{multline}\label{Noldsigmapart}
     \max_{\ell,\vec{w}\in Basis_{2}}\Bigg\vert \langle N(\sigma_{n-1},\vec{u}_{n-1}),\mathfrak{p}_3e^{i\mathfrak{p}_3\theta_{n-1}(t,x)}\vec{w}(\alpha_{\ell,n-1}(t),x-y_{\ell,n-1}(t))\rangle \\{-}\langle N(\sigma_{n-2},\vec{u}_{n-2}),\mathfrak{p}_3e^{i\mathfrak{p}_3\theta_{n-2}(t,x)}\vec{w}(\alpha_{\ell,n-2}(t),x-y_{\ell,n-2}(t))\rangle\Bigg\vert
    \\
    \lesssim   \delta_{0}   \left[\norm{(\vec{u}_{n-1}-\vec{u}_{n-2},\sigma_{n-1}-\sigma_{n-2})}_{Y_{n-1}}+ \norm{(\vec{u}_{n}-\vec{u}_{n-1},\sigma_{n}-\sigma_{n-1})}_{Y_{n}}+\frac{1}{T_{n-1}^{\frac{1}{2}+\epsilon}}\right].
\end{multline}
\par In conclusion, applying Proposition \ref{undecays} in \eqref{dfff}, and using estimates \eqref{int1Dsigma} and \eqref{Noldsigmapart}, we deduce the following inequality.
\begin{align}\label{sigmann-1final}
    \max_{t\in [0,T_{n}]}\langle t \rangle^{\frac{\epsilon}{2}-\frac{3}{8}}\left\vert \Lambda\dot\sigma_{n}(t)-\Lambda\dot\sigma_{n-1}(t) \right\vert\lesssim & \delta_{0} \norm{(\vec{u}_{n-1}-\vec{u}_{n-2},\sigma_{n-1}-\sigma_{n-2})}_{Y_{n-1}}\\
    & {+} \delta_{0}\norm{(\vec{u}_{n}-\vec{u}_{n-1},\sigma_{n}-\sigma_{n-1})}_{Y_{n}}+\frac{\delta_{0}}{(1+T_{n-1})^{\frac{1}{2}+\epsilon}}.
\end{align}
\subsection{Conclusion of the proof of Proposition \ref{propun-un-1}}
First, let $\min_{\ell}y_{\ell}(0)-y_{\ell+1}(0)>1$ be large enough and take $\delta_{0}$ defined in \eqref{deltachoice} to be small enough.
\par As a consequence, for $\delta_{0}\in (0,1)$ small enough, we can conclude using the formula \eqref{znformula2} satisfied by $\vec{z}_{n}$, and the estimates \eqref{minb}, \eqref{l2weightznfinal},  \eqref{conclusion1weightednormzn1} \eqref{almostaa}, \eqref{sigmann-1final} that $(\vec{z}_{n}(t),\sigma_{n}(t)-\sigma_{n-1}(t))$ should satisfy Proposition \ref{propun-un-1} for any $n\in \mathbb{N}.$

\appendix

\section{Proof of Proposition \ref{growthweightl2} }\label{B}
\par First, let $\chi:\mathbb{R}\to[0,1]$ be a smooth cut-off function satisfying the following condition for a small $\varepsilon\in (0,1).$
\begin{align*}
\chi(x)=
\begin{cases}
 0 \text{, if $x\leq {-}\frac{1}{2}-2\varepsilon,$}\\
 1 \text{, if $x\geq {-}\frac{1}{2} -\varepsilon.$}
\end{cases}
\end{align*}
Moreover, we set for any $\ell\in \{2,\,....,\,m-1\}$ the smooth cut-off function $\chi_{\ell,\sigma}:\mathbb{R}_{\geq 0}\times \mathbb{R}\to [0,1]$ to be
\begin{equation}\label{chisi}
  \chi_{\ell,\sigma}(t,x)=\chi\left(\frac{x-v_{\ell}t-y_{\ell}}{[y_{\ell}-y_{\ell+1}+(v_{\ell}-v_{\ell+1})t]}\right)-\chi\left(\frac{x-v_{\ell-1}t-y_{\ell-1}}{[y_{\ell-1}-y_{\ell}+(v_{\ell-1}-v_{\ell})t]}\right),
\end{equation}
and
\begin{equation*}
    \chi_{1,\sigma}(t,x)=\chi\left(\frac{x-v_{1}t-y_{1}}{[y_{1}-y_{2}+(v_{1}-v_{2})t]}\right),\,\chi_{m,\sigma}(t,x)=1-\chi\left(\frac{x-v_{m}t-y_{m}}{[y_{m-1}-y_{m}+(v_{m-1}-v_{m})t]}\right).
\end{equation*}
In particular, it is not difficult to verify that the following estimates hold for all $n\in\mathbb{N}.$ 
\begin{equation}\label{decaypartialchi}
\max_{n_{1}+n_{2}=j}\left\vert\frac{\partial^{n_{1}+n_{2}}}{\partial x^{n_{1}}\partial t^{n_{2}}}\chi_{\ell,\sigma}(t,x)\right\vert\left\vert x-v_{\ell}t-y_{\ell}\right\vert^{n}\lesssim \max_{\ell,\pm} \left\vert y_{\ell\pm1}-y_{\ell}+(v_{\ell\pm1}-v_{\ell})t \right\vert^{n-j},  
\end{equation}
for any $j\in\{1,2\}.$
\par Next, for any $\vec{f}\in L^{2}_{x}(\mathbb{R},\mathbb{C}^{2}),$ we consider the following function.
\begin{equation}\label{Mll}
  M_{\ell}(t)=\left\langle \vert x-v_{\ell}t-y_{\ell} \vert^{2}\chi_{\ell,\sigma}(t,x)\mathcal{S}(\vec{\phi})(t),\mathcal{S}(\vec{\phi})(t) \right\rangle.  
\end{equation}
Moreover, we can verify that the function $\vec{\upsilon}(t,x)=\mathcal{S}(\vec{\phi})(t,x)$ satisfies the following partial differential equation.
\begin{multline}\label{linearPDEB0}
  \partial_{t}\vec{\upsilon}(t,x)-i\mathfrak{p}_3\partial^{2}_{x}\vec{\upsilon}(t,x)-i\sum_{\ell}V_{\ell,\sigma}(t,x)\vec{\upsilon}(t,x)\\
  \begin{aligned}
  =&{-}\sum_{\ell=1}^{m}V_{\ell,\sigma}(t,x)\Bigg[\mathcal{S}\left(\vec{\phi}(t)\right)(t,x)\\&{-}e^{i\left(\frac{v_{\ell}x}{2}-\frac{v_{\ell}^{2}t}{4}+\omega_{\ell}t+\gamma_{\ell}\right)\sigma_{3}}\hat{G}_{\omega_{\ell}}\left(
   e^{{-}it(k^{2}+\omega_{\ell})\sigma_{3}}e^{{-}i\gamma_{\ell}\sigma_{3}} \begin{bmatrix}
e^{iy_{\ell}k}\phi_{1,\ell}\left(t,k+\frac{v_{\ell}}{2}\right)\\
e^{iy_{\ell}k}\phi_{2,\ell}\left(t,k-\frac{v_{\ell}}{2}\right)
    \end{bmatrix}\right)(x-y_{\ell}-v_{\ell}t)\Bigg]  
\end{aligned}
\end{multline}
such that
\begin{equation*}
  V_{\ell,\sigma}(t,x)=    
    \begin{bmatrix}
        {-}(k+1)\phi^{2k}_{\alpha_{\ell}}(x-v_{\ell}t-y_{\ell}) & {-}k e^{i\theta_{\ell}(t,x)}\phi^{2k}_{\alpha_{\ell}}(x-v_{\ell}t-v_{\ell}t))\\
        k e^{{-}i\theta{\ell}(t,x)}\phi^{2k}_{\alpha_{\ell}}(x-v_{\ell}t-y_{\ell}) &  {-}(k+1)\phi^{2k}_{\alpha_{\ell}}(x-v_{\ell}t-y_{\ell}),
    \end{bmatrix}
\end{equation*}
with 
\begin{equation*}
    \theta_{\ell}(t,x)=\frac{v_{\ell}x}{2}-\frac{v_{\ell}^{2}t}{4}+\alpha_{\ell}^{2}t+\gamma_{\ell}.
\end{equation*}
Consequently, we can verify using Remark \ref{RRR}, Lemma \ref{interactt}, and the exponential decay rate of the potential functions $V_{\ell,\sigma}(t,x)$ that $\vec{\upsilon}(t,x)=\mathcal{S}(\vec{\phi})(t,x)$ satisfies
\begin{multline}\label{RR2}
 \max_{\ell\in[m],j\in\{0,1\}}\norm{ \chi_{\ell,\sigma}(t,x)\langle x-v_{\ell}t-y_{\ell}\rangle^{2} \frac{\partial^{j}}{\partial x^{j}}[\partial_{t}\vec{\upsilon}(t,x)-i\mathfrak{p}_3\partial^{2}_{x}\vec{\upsilon}(t,x)-i\sum_{\ell}V_{\ell,\sigma}(t,x)\vec{\upsilon}(t,x)]}_{L^{2}_{x}(\mathbb{R})}\\
 \lesssim e^{{-}\frac{1}{2}\min_{\ell,j}\alpha_{j}(y_{\ell}-y_{\ell+1}+(v_{\ell}-v_{\ell+1})t)}\norm{\mathcal{S}(\vec{\phi})(t,x)}_{L^{2}_{x}(\mathbb{R})}.   
\end{multline}
\par Furthermore, hypothesis $\mathrm{(H2)}$ and the exponential decay of $V_{\ell,\sigma}(t,x)$ implies that
\begin{align}\label{Vpart}
 &\max_{j,\ell\in [m]}\left\vert\left\langle \vert x-v_{\ell}t-y_{\ell}\vert^{2}\chi_{\ell,\sigma}(t,x)V_{j,\sigma}(t,x)\mathcal{S}(\vec{\phi})(t,x),\mathcal{S}(\vec{\phi})(t,x)\right\rangle\right\vert\\
 &\quad\quad\quad\quad\quad\quad\quad\quad\quad\quad\quad\quad\quad\quad\lesssim \norm{\mathcal{S}(\vec{\phi})(t,x)}_{L^{2}_{x}(\mathbb{R})}^{2}\lesssim \norm{\mathcal{S}(\vec{\phi})(s,x)}_{L^{2}_{x}(\mathbb{R})}^{2}.
\end{align}
Consequently, we can verify using integration by parts, \eqref{chisi}, \eqref{decaypartialchi} and \eqref{RR2} that
\begin{align}\label{almostgrothb}
 \left\vert\frac{d}{dt}M_{\ell}(t)\right\vert\lesssim &\max_{\ell}\left[(1+\vert v_{\ell}\vert)\langle t\rangle+\vert y_{\ell} -y_{\ell+1}\vert\right]\norm{\mathcal{S}(\vec{\phi})(s,x)}_{H^{1}_{x}(\mathbb{R})}^{2}\\ \nonumber
 & 
 {+}O\left(\left\vert \left\langle (x-v_{\ell}t-y_{\ell})\chi_{\ell,\sigma}(t,x)\partial_{x}\mathcal{S}(\vec{\phi})(t,x),\mathfrak{p}_3\mathcal{S}(\vec{\phi})(t,x) \right\rangle \right\vert\right)\\
 &{+}O\left(\vert v_{\ell} \vert  \left\vert\left\langle (x-v_{\ell}t-y_{\ell})\chi_{\ell,\sigma}(t,x)\mathcal{S}(\vec{\phi})(t,x) ,\mathfrak{p}_3\mathcal{S}(\vec{\phi})(t,x)  \right\rangle \right\vert\right).  
\end{align}

\par Furthermore, if $\vec{\psi}(t)\in H^{2}_{x}(\mathbb{R}),$ we can verify using integration by parts, hypothesis $\mathrm{(H2)},$ the exponential decay of $V_{\ell,\sigma}(t,x),$ and \eqref{RR2} that
\begin{equation}\label{bb11}
    \left\vert \frac{d}{dt} \left\langle (x-v_{\ell}t-y_{\ell})\chi_{\ell,\sigma}(t,x)\partial_{x}\mathcal{S}(\vec{\phi})(t,x),\mathfrak{p}_3\mathcal{S}(\vec{\phi})(t,x)\right\rangle \right\vert
    \lesssim (1+\vert v_{\ell} \vert) 
    \norm{\mathcal{S}(\vec{\phi})(s,x)}_{H^{1}_{x}(\mathbb{R})}^{2},
\end{equation}
and
\begin{equation}\label{bb22}
    \left\vert \frac{d}{dt} \left\langle (x-v_{\ell}t-y_{\ell})\chi_{\ell,\sigma}(t,x)\mathcal{S}(\vec{\phi})(t,x),\mathfrak{p}_3\mathcal{S}(\vec{\phi})(t,x)\right\rangle \right\vert
    \lesssim (1+\vert v_{\ell} \vert)  
    \norm{\mathcal{S}(\vec{\phi})(s,x)}_{H^{1}_{x}(\mathbb{R})}^{2}.
\end{equation}
In conclusion, using the density of $H^{2}_{x}(\mathbb{R},\mathbb{C}^{2})$ on $H^{1}_{x}(\mathbb{R},\mathbb{C}^{2}),$ we can obtain the result of Proposition \ref{growthweightl2} from \eqref{almostgrothb}, \eqref{bb11}, \eqref{bb22} and a direct integration in time via the fundamental theorem of calculus.
\section{Proof of Corollary \ref{corollll}}\label{ApB}
First, Theorem \ref{Tsigma(t)} implies that the solution $\psi(t,x)$ has the following representation for all $t\geq 0$
\begin{equation*}
\psi(t,x)=\sum_{\ell=1}^{m}e^{i\left(\frac{v_{\ell}(t)x}{2}+\gamma_{\ell}(t)\right)}\phi_{\alpha_{\ell}(t)}\left(x-y_{\ell}(t)\right)+u(t),
\end{equation*}
such all the inequalities \eqref{decay1}-\eqref{decay5} and \eqref{odes} are true. Moreover, using \eqref{odes}, we can verify from the fundamental theorem of calculus that there exist real constants $v_{\ell,\infty},\,\alpha_{\ell,\infty},\,y_{\ell,\infty}$ and $\gamma_{\ell,\infty}$ for any $\ell\in [m]$ satisfying
 \begin{align*}
     \max_{\ell}\left\vert \gamma_{\ell}(t)-\alpha_{\ell,\infty}^{2}t+\frac{v_{\ell,\infty}^{2}t}{4}-\gamma_{\ell,\infty}\right \vert+\max_{\ell}\vert y_{\ell}(t)-v_{\ell,\infty}t-y_{\ell,\infty}  \vert\lesssim & \frac{\delta_{0}}{(1+t)^{2\epsilon-1}},\\
     \max_{\ell}\vert v_{\ell}(t)-v_{\ell,\infty}\vert+\max_{\ell}\vert \alpha_{\ell}(t)-\alpha_{\ell,\infty}\vert\lesssim & \frac{\delta_{0}}{(1+t)^{2\epsilon}}, 
 \end{align*}
for all $t\geq 0.$
\par Next, setting 
\begin{equation*}
    \vec{u}(t,x)=\begin{bmatrix}
        u(t,x)\\
        \overline{u(t,x)}
    \end{bmatrix}
\end{equation*}
and using Theorem \ref{stablecase}, we can verify that $\vec{u}(t,x)$ has the following representation
\begin{align}\label{repruinfty}
   \vec{u}(t,x)=& \mathcal{S}
\left(\vec{\phi}(t,k)\right)(t,x)+\sum_{\ell=1}^{m}\sum_{j=1}^{\dim \ker \mathcal{H}_{\ell,\infty}^{2} } b_{j,\ell,0}\mathfrak{G}_{\ell}(\mathfrak{z}_{\ell})(t,x)\\&{+}\sum_{\ell=1}^{m}\sum_{\lambda\in \sigma_{d,\mathrm{stab}}(\mathcal{H}_{\ell,\infty})}b_{\ell,\lambda}(t)\mathfrak{G}_{\ell}(\mathfrak{v}_{\alpha_{\ell,\infty},\lambda})(t,x)
\\&{+}\sum_{\ell=1}^{m}b_{\ell,+}(t)e^{i\theta_{\ell}(t,x)\sigma_{3}}\alpha_{\ell,\infty}^{\frac{1}{k}}\vec{Z}_{+}\left(\alpha_{\ell,\infty}[x-v_{\ell,\infty}t-y_{\ell,\infty}]\right),  
\end{align}
where $\mathcal{S}$ is the dispersive map defined in Definition \ref{s0def} for the set $\sigma_{\infty}=\{v_{\ell,\infty},\,\alpha_{\ell,\infty},\,y_{\ell,\infty}\}_{\ell\in[m]},$ and
\begin{equation*}
\mathcal{H}_{\ell,\infty}=\begin{bmatrix}
    {-}\partial^{2}_{x}+\alpha_{\ell,\infty}^{2}-(k+1)\phi^{2k}_{\alpha_{\ell,\infty}}(x) &  {-}k \phi^{2k}_{\alpha_{\ell,\infty}}(x)\\
    k \phi^{2k}_{\alpha_{\ell,\infty}}(x) & \partial^{2}_{x}-\alpha_{\ell,\infty}^{2}+(k+1)\phi^{2k}_{\alpha_{\ell,\infty}}(x)
    \end{bmatrix}.
\end{equation*}
\par In particular, using the local $L^{2}$ decay estimate of Theorem \ref{Tsigma(t)}, 
we can verify the following inequality
\begin{align}\label{discreteupper}
    \max_{\ell}\vert b_{\ell,+}(t)\vert+\max_{\ell,j}\vert b_{j,\ell,0}(t)\vert +\max_{\ell, \lambda\in \sigma_{d,\mathrm{stab}}(\mathcal{H}_{\ell,\infty})}\vert b_{\ell,\lambda}(t)\vert\lesssim & \frac{\delta_{0}}{(1+t)^{\frac{1}{2}+\epsilon}},
\end{align}
for all $t\geq 0.$
\par Next, let
\begin{equation*}
    \theta_{\ell,\infty}(t,x)=\frac{v_{\ell,\infty}x}{2}+\gamma_{\ell,\infty}+\alpha_{\ell,\infty}^{2}t-\frac{v_{\ell,\infty}^{2}t}{4},\,\, \, \theta_{\ell}(t,x)=\frac{v_{\ell}(t)x}{2}+\gamma_{\ell}(t),
\end{equation*}
and
\begin{align}
    V_{\ell,\infty}(t,x)=&\begin{bmatrix}
        {-}(k+1)\phi^{2k}_{\alpha_{\ell,\infty}}(x-v_{\ell,\infty}t-y_{\ell,\infty}) & {-}k e^{i\theta_{\ell,\infty}(t,x)}\phi^{2k}_{\alpha_{\ell}(t)}(x-v_{\ell,\infty}t-y_{\ell,\infty})\\
        k e^{{-}i\theta_{\ell,\infty}(t,x)}\phi^{2k}_{\alpha_{\ell}(t)}(x-v_{\ell,\infty}t-y_{\ell,\infty}) &  (k+1)\phi^{2k}_{\alpha_{\ell}(t)}(x-v_{\ell,\infty}t-y_{\ell,\infty})
    \end{bmatrix},\\
    V_{\ell}(t,x)=&\begin{bmatrix}
        {-}(k+1)\phi^{2k}_{\alpha_{\ell}(t)}(x-y_{\ell}(t)) & {-}k e^{i\theta_{\ell}(t,x)}\phi^{2k}_{\alpha_{\ell}(t)}(x-y_{\ell}(t))\\
        k e^{{-}i\theta_{\ell}(t,x)}\phi^{2k}_{\alpha_{\ell}(t)}(x-y_{\ell}(t)) &  (k+1)\phi^{2k}_{\alpha_{\ell}(t)}(x-y_{\ell}(t))
    \end{bmatrix}.
\end{align}
Using
equations \eqref{gqq} and \eqref{unequation}, we can verify that $\mathcal{S}(\vec{\phi}(t,k))(t,x)$ satisfies the following integral equation.
\begin{align}\label{integralSSS}
    \mathcal{S}(\vec{\phi}(t,k))(t,x)=&\mathcal{S}(\vec{\phi}(0,k))(t,x)-i\int_{0}^{t} \mathcal{S}(t)\circ\mathcal{S}^{{-}1}(s)P_{c,\sigma}(s)G(s,\sigma(t),\sigma(t),\vec{u})\,ds\\
    &{-}i\int_{0}^{t} \mathcal{S}(t)\circ \mathcal{S}^{{-}1}(s)P_{c,\sigma}(s)[\sum_{\ell}[V_{\ell,\infty}(s,x)-V_{\ell}(s,x)]\vec{u}(s)]\,ds
    \\ \nonumber
    &{+}i\int_{0}^{t}\mathcal\mathcal{S}(t)\circ\mathcal{S}(s)P_{c,\sigma}(s)\\&\times \sum_{h=1}^{m}\sum_{j=1,j\neq h}^{m}V_{j,\infty}(s,x) b_{h,+}(s)e^{i\theta_{\ell,\infty}(s,x)}\vec{Z}_{+}\left(\alpha_{\ell,\infty},x-y_{\ell,\infty}-v_{\ell,\infty}s\right)\,ds\\
     &{+}i\int_{0}^{t} \mathcal{S}(t)\circ \mathcal{S}^{{-}1}(s)P_{c,\sigma}\left[\sum_{j=1}^{m}V_{j,\infty}(s,x)[P_{c,n-1}(s)\vec{u}(s)-P_{c,j,n-1}(s)\vec{u}(s)]\right]\,ds.
     \end{align}
\par Furthermore, we can deduce similarly to Lemma \ref{dinftydt} using Theorem \ref{Tsigma(t)} that
\begin{equation*}
  \max_{j\in\{0,1\},\ell\in[m]}\norm{\frac{\partial^{j}}{\partial x^{j}}[V_{\ell,\infty}(s,x)-V_{\ell}(s,x)]}_{L^{\infty}_{x}(\mathbb{R})}\lesssim \frac{\delta_{0}}{(1+s)^{2\epsilon-1}},  
\end{equation*}
where $\epsilon=\frac{3}{4}+\frac{3}{2}\left(1-\frac{2-p}{p}\right)>\frac{3}{4}.$ Consequently, we can verify using the $H^{1}$ local decay estimate of Theorem \ref{Tsigma(t)} that
\begin{equation*}
    \max_{\ell\in[m]}\norm{[V_{\ell,\infty}(s,x)-V_{\ell}(s,x)]\vec{u}(s,x)}_{H^{1}_{x}(\mathbb{R})}\lesssim \frac{\delta_{0}}{(1+s)^{3\epsilon-\frac{1}{2}}}\ll \frac{\delta_{0}}{(1+s)^{\frac{7}{4}}} \text{, for all $s\geq 0.$}
\end{equation*}
\par Moreover, using the decay estimates satisfied by $\vec{u}(t)$ and $\sigma$ in Theorem \ref{Tsigma(t)}, we can verify using that $k>\frac{11}{4}$ that
\begin{align*}
\norm{G(s,\sigma(s),\sigma(s),\vec{u}(s))}_{H^{1}_{x}(\mathbb{R})}\lesssim & \frac{\delta_{0}}{(1+s)^{1+\epsilon}} \text{, for all $s\geq 0.$}
\end{align*}
The proof of the estimate above is completely similar to the reasoning in \S \ref{slocalH1}.
\par Next, using Lemma \ref{interactt} and Remark \ref{RRR}, we can verify from $\norm{\vec{u}(s,x)}_{H^{1}_{x}(\mathbb{R})}\lesssim \delta_{0}$ the following decay estimates
\begin{align*}
    \max_{j\in [m]}\norm{V_{j,\infty}(s,x)[P_{c,n-1}(s)\vec{u}(s)-P_{c,j,n-1}(s)\vec{u}(s)]}_{H^{1}_{x}(\mathbb{R})}\lesssim \frac{\delta_{0}}{(1+s)^{20}} \text{, for all $s\geq 0.$}
\end{align*}

In particular, Lemma \ref{interactt} and estimate \eqref{discreteupper} imply the following inequality
\begin{equation*}
   \max_{j\neq \ell;\,j,\ell\in [m]}\norm{ V_{j,\infty}(s,x) b_{h,+}(s)e^{i\theta_{\ell,\infty}(s,x)}\vec{Z}_{+}\left(\alpha_{\ell,\infty},x-y_{\ell,\infty}-v_{\ell,\infty}s\right)}_{L^{2}_{x}(\mathbb{R})}\lesssim \frac{\delta_{0}}{(1+s)^{20}} \text{, for all $s\geq 0.$}
\end{equation*}
\par In conclusion, we can verify that the function
\begin{align}
    \vec{\phi}_{\infty}(k)=&\vec{\phi}(0,k)-i\int_{0}^{\infty} \mathcal{S}^{{-}1}(s)P_{c,\sigma}(s)G(s,\sigma(t),\sigma(t),\vec{u})\,ds\\
    &{-}i\int_{0}^{\infty} \mathcal{S}^{{-}1}(s)P_{c,\sigma}(s)[\sum_{\ell}[V_{\ell,\infty}(s,x)-V_{\ell}(s,x)]\vec{u}(s)]\,ds
    \\ \nonumber
    &{+}i\int_{0}^{\infty}\mathcal\mathcal{S}^{{-}1}(s)P_{c,\sigma}(s)\\&\times \sum_{h=1}^{m}\sum_{j=1,j\neq h}^{m}V_{j,\infty}(s,x) b_{h,+}(s)e^{i\theta_{\ell,\infty}(s,x)}\vec{Z}_{+}\left(\alpha_{\ell,\infty},x-y_{\ell,\infty}-v_{\ell,\infty}s\right)\,ds\\
     &{+}i\int_{0}^{\infty} \mathcal{S}^{{-}1}(s)P_{c,\sigma}\left[\sum_{j=1}^{m}V_{j,\infty}(s,x)[P_{c,n-1}(s)\vec{u}(s)-P_{c,j,n-1}(s)\vec{u}(s)]\right]\,ds
     \end{align}
is well-defined in $L^{2}_{k}(\mathbb{R},\mathbb{C}^{2}),$ and using the integral equation \eqref{integralSSS} and the inequality
\begin{equation*}
    \norm{\mathcal{S}(t)(\vec{\phi})}_{H^{1}_{x}(\mathbb{R},\mathbb{C}^{2})}\sim  \norm{\mathcal{S}(0)(\vec{\phi})}_{H^{1}_{x}(\mathbb{R},\mathbb{C}^{2})}
\end{equation*}
proved in \cite{dispanalysis1} and \cite{dispanalysis2}, we can conclude using the fundamental theorem of calculus and the previous estimates in this section that
\begin{equation*}
    \norm{\mathcal{S}(\vec{\phi}_{\infty}(k))(t,x)-\mathcal{S}(\vec{\phi}(t,k))(t,x)}_{H^{1}_{x}(\mathbb{R})}\lesssim \frac{\delta_{0}}{(1+t)^{\frac{3}{4}}} \text{, for all $t\geq0,$}
\end{equation*}
which is equivalent to the statement of Corollary \ref{corollll}.
\bibliographystyle{plain}
\bibliography{ref}
\end{document}